\documentclass[A4paper,reqno]{amsart}

\usepackage{hyperref}
\usepackage{amsmath, amssymb, latexsym, graphicx,tikz, stmaryrd, ifsym, cleveref, enumerate, multicol} 

\usepackage{stackengine,scalerel}
\usepackage{esvect}
\usepackage[shortlabels]{enumitem} 
\usepackage[latin1]{inputenc}
\usepackage{amssymb}
\usepackage{amsfonts}
\usepackage{amscd}
\usepackage{multicol}
\usepackage{mathrsfs}
\usepackage{enumerate}
\usepackage[T1]{fontenc}
\usepackage[english]{babel}
\usepackage{marvosym}
\usepackage{amsthm}
\usepackage{multicol}
\usepackage[all]{xy}
\usepackage{todonotes}
\usepackage{graphicx,tikz}
\usepackage{tikz-cd} 
\usepackage{mathbbol}
\usepackage{verbatim}
\usepackage{xcolor} 

\usepackage{stackengine}

\newcommand{\tikzdiagh}[2][]{\tikz[#1,very thick,baseline={([yshift=1ex+#2]current bounding box.center)}]}
\tikzstyle{tikzdot}=[fill, circle, inner sep=2pt]
\newcommand{\fdot}[3][]{ \node  [anchor = center, fill=white, draw=black,circle,inner sep=2pt] at (#3) {} ; \node[xshift=-.065cm, yshift=-.035cm,anchor = south west] at (#3){\small $#1$}; \node[xshift=-.065cm, yshift=.065cm,anchor = north west] at (#3){\small $#2$}; }
\newcommand{\plusspacing}{\llap{\phantom{\small ${+}1$}}}

\usepackage{pb-diagram}

\usepackage{bm}

\usepackage{wasysym}
\usepackage{color}

\theoremstyle{plain}
\newtheorem{thm}{Theorem}[section]
\newtheorem*{thm*}{Theorem}
\newtheorem{lem}[thm]{Lemma}

\newtheorem{prop}[thm]{Proposition}

\newtheorem{coro}[thm]{Corollary}
\newtheorem{belief}[thm]{Belief}

\theoremstyle{definition}
\theoremstyle{remark}
\newtheorem{term}[thm]{Terminology}
\newtheorem{rk}[thm]{Remark}
\newtheorem{df}[thm]{Definition}
\newtheorem{ex}[thm]{Example}

\newtheorem{nota}[thm]{Notation}
\numberwithin{equation}{section}

\def\Pol{\mathrm{Pol}}
\def\hPol{{\mathrm{EPol}}}
\def\NH{\mathrm{NH}}
\def\hNH{\mathrm{ENH}}
\def\hhNH{\mathrm{\mathbb{E}NH}}
\def\hR{\mathrm{ER}}

\def\1c{\mathrm{1c}}

\def\ui{{\mathbf{i}}}
\def\uj{{\mathbf{j}}}

\def\ut{{\mathbf{t}}}

\def\bbC{\mathbb{C}}

\def\bbR{\mathbb{R}}

\def\bbV{\mathbb{V}}

\def\bbZ{\mathbb{Z}}

\def\frakS{\mathcal{S}}

\def\calA{\mathcal{A}}
\def\calB{\mathcal{B}}

\def\calF{\mathcal{F}}
\def\calG{\mathcal{G}}

\def\calK{\mathcal{K}}
\def\calL{\mathcal{L}}

\def\calO{\mathcal{O}}

\def\calU{\mathcal{U}}
\def\calV{\mathcal{V}}

\def\calY{\mathcal{Y}}
\def\calZ{\mathcal{Z}}
\def\ccT{T}

\def\frakb{\mathfrak{b}}

\def\frakg{\mathfrak{g}}

\def\frakm{\mathfrak{m}}
\def\frakn{\mathfrak{n}}
\def\frakp{\mathfrak{p}}

\def\bfa{\mathbf{a}}

\def\bfd{\mathbf{d}}

\def\bfk{\mathbf{k}} 

\def\bfn{\mathbf{n}}

\def\bfZ{\mathbf{Z}}

\def\bfZ{\mathbf{Z}}

\def\point{{\{\rm pt\}}}

\def\WW{\frakS}

\newcommand\restr[2]{{
  \left.\kern-\nulldelimiterspace 
  #1 
  \vphantom{\big|} 
  \right|_{#2} 
  }}

\def\grdim{\operatorname{grdim}}

\def\homo{\operatorname{\it \mathscr{H}\kern-.25em om}}
\def\ext{\operatorname{\it \mathscr{E}\kern-.25em xt}}
\def\edo{\operatorname{\it \mathscr{E}\kern-.25em nd}}
\def\der{\operatorname{\it \mathscr{D}\kern-.25em er}}

\def\Hom{\mathrm{Hom}}

\def\End{\mathrm{End}}

\def\Ker{\mathrm{Ker}}

\def\Id{\mathrm{Id}}

\def\loc{\mathrm{loc}}

\def\p{\overrightarrow{p}}

\def\op{\operatorname}

\def\rmZ{\mathrm{Z}}
\def\rmT{\mathrm{T}}
\def\rmU{\mathrm{U}}
\def\rmG{\mathrm{G}}
\def\rmB{\mathrm{B}}
\def\rmP{\mathrm{P}}
\def\oX{\overline{X}}
\def\oY{\overline{Y}}

\def\bV{\mathbb{V}}
\def\bU{\mathbb{U}}

\def\rel{\mathrm{rel}}

\def\IC{\mathrm{IC}}

\def\inv{\mathrm{inv}}

\def\tF{\widetilde{\calF}}

\usetikzlibrary{decorations.pathreplacing,backgrounds,decorations.markings,arrows}
\tikzset{wei/.style={draw=red,double=red!40!white,double distance=1.5pt,thin}}
\tikzset{bdot/.style={fill,circle,color=blue,inner sep=3pt,outer sep=0}}
\tikzset{dir/.style={postaction={decorate,decoration={markings,mark=at position .8 with {\arrow[scale=1.3]{>}}}}}}

\newcommand{\mZ}{\mathbb{Z}}

\newcommand{\cB}{\mathcal{B}}
\newcommand{\cP}{\mathcal{P}}
\newcommand{\cQ}{\mathcal{Q}}

\newcommand{\cA}{\mathcal{A}}

\newcommand{\cF}{\mathcal{F}}

\newcommand{\bi}{\mathbf{i}}

\newcommand{\la}{\lambda}

\makeatletter
\newcommand{\mywedge}{\@ifnextchar^\@extp{\@extp^{\,}}}
\def\@extp^#1{\mathop{\bigwedge\nolimits^{\!\!#1}}}
\makeatother
\renewcommand*\labelenumi{\rm{\theenumi.)}}

\begin{document}
\title[Geometric categorification of Verma modules]{Geometric categorifications of Verma modules: 
Grassmannian Quiver Hecke algebras}
\author[R. Maksimau]{Ruslan Maksimau}
\author[C. Stroppel]{Catharina Stroppel}

\begin{abstract}
Naisse and Vaz defined an extension of KLR algebras to categorify Verma modules. We realise these algebras geometrically as convolution algebras in Borel--Moore homology. For this we introduce \emph{Grassmannian--Steinberg quiver flag varieties}. They generalize Steinberg quiver flag varieties in a non-obvious way, reflecting the diagrammatics from the Naisse--Vaz construction.  Using different kind of stratifications we provide geometric explanations of the rather mysterious algebraic and diagrammatic basis theorems. 

A geometric categorification of Verma modules was recently found in the special case of $\mathfrak{sl}_2$  by Rouquier. Rouquier's construction uses coherent sheaves on certain quasi-map spaces to flag varieties (zastavas), whereas our construction is implicitly based on perverse sheaves. Both should be seen as parts (on dual sides) of a general geometric framework for the Naisse--Vaz approach. 

We first treat the (substantially easier) $\mathfrak{sl}_2$ case in detail and construct as a byproduct a geometric dg-model of the nil-Hecke algebras. The extra difficulties we encounter in general require the use of more complicated Grassmannian--Steinberg quiver flag varieties. Their definition arises from combinatorially defined \emph{diagram varieties} which we assign to each Naisse--Vaz basis diagram. Our explicit analysis here might shed some light on categories of coherent sheaves on more general zastava spaces studied by Feigin--Finkelberg--Kuznetsov--Mirkovi{\'c} and Braverman,  which we expect to occur in a generalization of Rouquier's construction away from $\mathfrak{sl}_2$.
\end{abstract}

\maketitle
\def\ccX{{X}}
\def\ccT{\mathtt{T}}
\newcommand{\Zr}{\Delta^r} 
\newcommand{\x}{\iota}
\newcommand{\y}{\gamma}
\newcommand{\z}{\pi}
\newcommand{\catP}{Q}

\newcommand{\mA}{{\mathbb{A}}}
\newcommand{\mI}{{\mathbb{I}}}
\newcommand{\hh}{\hat{h}}
\newcommand{\hF}{\hat\calF}
\newcommand{\VWp}{\bV^{\perp W}}
\renewcommand{\bfn}{{\underline{n}}}
\renewcommand{\bfk}{{\underline{k}}}
\newcommand{\ccF}{\calF}
\newcommand{\Jbfn}{J_\bfn}
\newcommand{\Xequal}{\tikz[scale=0.5,baseline={([yshift=-0.5ex]current bounding box.center)}]{
		\draw   (-.35,-.35)-- (.35,.35);
		\draw   (.35,-.35)-- (-.35,.35);
		\draw   (-0.1,-.3)-- (.1,-0.3);
		\draw   (-0.1,-.2)-- (.1,-0.2);
	}  }
	\newcommand{\Xright}{\tikz[scale=0.5,baseline={([yshift=-0.5ex]current bounding box.center)}]{
		\draw   (-.35,-.35)-- (.35,.35);
		\draw   (.35,-.35)-- (-.35,.35);
		\draw [->]   (-0.2,-.3)-- (.2,-0.3);
	}  }
		\newcommand{\Xleft}{\tikz[scale=0.5,baseline={([yshift=-0.5ex]current bounding box.center)}]{
		\draw   (-.35,-.35)-- (.35,.35);
		\draw   (.35,-.35)-- (-.35,.35);
		\draw [<-]   (-0.2,-.3)-- (.2,-0.3);
	}  }

\stackMath
\newcommand\tsup[2][2]{%
 \def\useanchorwidth{T}%
  \ifnum#1>1%
    \stackon[-.5pt]{\tsup[\numexpr#1-1\relax]{#2}}{\scriptscriptstyle\sim}%
  \else%
    \stackon[.5pt]{#2}{\scriptscriptstyle\sim}%
  \fi%
}
\def\tV{\tsup[1]{\mathbb{V}}}
\def\ttV{\tsup{\mathbb{V}}}
\def\tW{\tsup[1]{W}}
\newcommand{\J}{\calF}
\newcommand{\JJ}{\tilde\calF}

\newcommand{\tA}{\mathtt{A}}
\newcommand{\tB}{\mathtt{B}}
\renewcommand{\oX}{\mathtt{X}}
\renewcommand{\oY}{\mathtt{Y}}
\newcommand{\tR}{\mathtt{R}}
\newcommand{\tZ}{\mathtt{Z}}
\newcommand{\Zdouble}{Steinberg extended flag variety }
\newcommand{\Zone}{extended Steinberg quiver flag variety }
\newcommand{\ovS}{\mathbf{S}}
\newcommand{\tS}{\mathtt{S}}

\makeatletter
\DeclareRobustCommand{\cev}[1]{%
  {\mathpalette\do@cev{#1}}%
}
\newcommand{\do@cev}[2]{%
  \vbox{\offinterlineskip
    \sbox\z@{$\m@th#1 x$}%
    \ialign{##\cr
      \hidewidth\reflectbox{$\m@th#1\vec{}\mkern4mu$}\hidewidth\cr
      \noalign{\kern-\ht\z@}
      $\m@th#1#2$\cr
    }%
  }%
}
\makeatother
\newcommand{\vechR}{\vec{\mathrm{E}}\mathrm{R}}
\newcommand{\vecI}{\vv{I}}
\newcommand{\cevI}{\cev{I}}
\newcommand{\cevhR}{\cev{\mathrm{E}}\mathrm{R}}
\newcommand{\cevhPol}{\cev{\mathrm{E}}\mathrm{Pol}}
\newcommand{\cevGGZ}{\cev{\ddot{\rmZ}}} 
\newcommand{\cevZnaiv}{\cev{\Znaiv}}

\newcommand\up{\mathrm{up}}
\newcommand\down{\mathrm{low}}

\newcommand{\colJ}{\mathcal{O}} 
\newcommand{\plusJ}{\mathbf{O}} 
\newcommand{\naivJ}{O} 

\newcommand{\dvp}{\op{pr}} 

\newcommand{\compT}{{{\underline\rmT}}} 
\newcommand{\St}{{\mathrm{St}}} 

\newcommand{\ccY}{\calY} 
\def\hY{\widehat{\ccY}} 

\def\rmY{\calY} 
\def\bfY{\mathbf{Y}} 
\newcommand{\Ytw}{\bfY^\phi} 


\newcommand{\ccZ}{\ddot{{\calZ}}} 
\newcommand{\ccZplus}{\dot{{\calZ}}} 
\def\hZ{\widehat{\ccZ}}
\newcommand{\hZplus}{\widehat{\ccZplus}}

\newcommand{\GGZ}{\ddot{\rmZ}} 
\newcommand{\Znaiv}{{\dot{\rmZ}}} 

\newcommand{\Ztw}{{\GGZ^\phi}} 
\newcommand{\Znaivtw}{{\Znaiv^\phi}} 

\newcommand{\Zplus}{\dot{\calZ}}
\newcommand{\ZplusGG}{\ddot{\calZ}} 
\newcommand{\ZplusGGpr}{\ddot{\calZ}} 

\newcommand{\cevYtw}{\cev{\bfY}{}^\phi}
\newcommand{\cevZtw}{\cev{\GGZ}{}^\phi}

\newcommand{\abZ}{Z}
\newcommand{\abY}{Y}
\newcommand{\abX}{X}

\newcommand{\Za}{WRONGZa}
\newcommand{\Zcolor}{ZCOLOR}
%

\newcommand{\bfZplus}{{WRONG}}



\newcommand{\subse}{%
	\mathrel{\ooalign{$\subset$\cr\hss\raisebox{0.2ex}{\hspace{0.2ex}\scalebox{0.6}{$=$}}\hss}}%
}
\newcommand{\supse}{%
	\mathrel{\ooalign{$\supset$\cr\hss\raisebox{0.2ex}{\hspace{-0.2ex}\scalebox{0.6}{$=$}}\hss}}%
}

\vspace{-5mm}

\tableofcontents

\section*{Introduction}
Let $\Gamma$ be a quiver without loops, and let $\frakg$ be the associated Kac-Moody Lie algebra with corresponding quantum group $U_q(\frakg)$.  To categorify $U_q(\frakg)$, or more precisely its negative part $U_q(\frakg)^-$,
Khovanov--Lauda \cite{KL1} and Rouquier \cite{Rou2KM} introduced in their seminal works important algebras which are now called \emph{KLR algebras}. These algebras depend on the quiver $\Gamma$ and a dimension vector $\bfn$ (which encodes the weight space which gets categorified). Given moreover a dominant weight $\Lambda$ of  $\frakg$, the algebra $R_\bfn$ has a remarkable finite-dimensional quotient $R_\bfn^\Lambda$, called the \emph{cyclotomic quotient} attached to $\Lambda$. It was proven by Kang--Kashiwara \cite{KangKash}, that the cyclotomic quotients $R_\bfn^\Lambda$ yield a categorification of the simple $U_q(\frakg)$-module $L_q(\Lambda)$ of highest weight $\Lambda$.  This result finally settled the \emph{cyclotomic categorification conjecture} from \cite{KL1} in full generality; only special cases of this conjecture had been proven before, \cite{BS3}, \cite{BKdecnumb}, \cite{Webster}. The simple module $L_q(\Lambda)$ is the unique simple quotient of the Verma module $V_q(\Lambda)$. It is therefore natural to ask how to categorify the Verma modules themselves and not only $U_q(\frakg)^-$. An answer to this question is interesting from a structural point of view, but also in the context of categorified link invariants,  see \cite{NaisseVaz2}. Finding such a categorification turned out to be quite difficult (merely because of the asymmetry in the behaviour of the positive versus negative Chevalley generators) to find such a categorification. It was achieved by Naisse and Vaz in \cite{NVsl2} for $\mathfrak{sl}_2$ and in \cite{NaisseVaz} in general. Let us briefly sketch their construction.
\subsection*{The Naisse--Vaz algebras}
Depending still on $\Gamma$ and $\bfn$, Naisse and Vaz introduced, see \cite{NVapproach}, \cite{NVsl2},  \cite{NaisseVaz}, a new $(\bbZ,\bbZ_{\geqslant 0})$-bigraded algebra $\hR_\bfn$, whose degree zero component for the second grading is the $\bbZ$-graded KLR algebra $R_\bfn$,  see \cite{Vazsurvey} for a survey.  They then show that categories of modules for $\hR_\bfn$ can be used to categorify the \emph{universal Verma module} $V_q$ (with generic highest weight). Moreover, for each $\Lambda$ as above, they construct a special operator $\bfd_\Lambda\colon\hR_\bfn\to \hR_\bfn$ of degree $-1$ (for the second grading) turning the algebra $\hR_\bfn$ into a dg-algebra. They prove that the dg-algebra $(\hR_\bfn,\bfd_\Lambda)$ has homology concentrated in degree zero and that this homology is isomorphic to $R_\bfn^\Lambda$. This means that $(\hR_\bfn,\bfd_\Lambda)$ is a dg-model, or a sort of resolution, of the cyclotomic quotient $R_\bfn^\Lambda$. It is a particularly nice model, since every term of this resolution is free over $R_\bfn$. The isomorphism $H_*(\hR_\bfn,\bfd_\Lambda)\cong R_\bfn^\Lambda$ is also of interest, since it categorifies the following two procedures simultaneously: the specialization $V_q\to V_q(\Lambda)$ of the generic weight to $\Lambda$, and the procedure of taking the simple quotient $V_q(\Lambda)\to L_q(\Lambda)$.

The Naisse--Vaz algebra $\hR_\bfn$ is defined algebraically and in terms of diagrams which extend the diagrammatic description of  \cite{KL1} in a subtle and intriguing way. The main new ingredients, the \emph{floating dots}, are key to the construction. A conceptual explanation of their relations and their twists seems however so far missing.  
\subsection*{Geometric construction?}
In case of the KLR algebra $R_\bfn$, the generators and relations as well as the grading have natural geometric explanations. In the geometric construction due to Rouquier \cite{Rou2Lie} and Varagnolo-Vasserot  \cite{VV}, which was extended to positive characteristics in \cite{Mak1}, the KLR algebra $R_\bfn$ is realised as a \emph{quiver Hecke algebra} that is the equivariant Borel--Moore homology of the corresponding Steinberg quiver flag variety, equipped with the convolution product. 

The natural question which arises now is whether this geometric construction can be extended to give a geometric construction of the Naisse--Vaz algebra in terms of some equivariant Borel--Moore homology. This question is the main motivation of our paper, and we will give an affirmative answer. We will introduce  a \emph{Grassmannian Quiver Steinberg variety} $\Zplus$ with a natural group action $\rmG$ and define  
a (non-obvious!) convolution product on $H_*^\rmG(\Zplus)$. For the resulting  \emph{Grassmannian quiver Hecke algebra} $H_*^\rmG(\Zplus)$ (attached to $\Gamma$ and $\bfn$) we show then our main result: 
\begin{thm*}
There is an isomorphism of graded algebras $H_*^\rmG(\Zplus)\cong\hR_\bfn$.
\end{thm*}
The idea to incorporate extra Grassmannians into the Quiver Steinberg variety construction came from the fact that categorified universal Verma modules should contain at least the data of all categorified finite dimensional irreducible modules. For the easiest case of quantum $\mathfrak{sl}_2$, the (already very powerful) theory of categorified finite dimensional irreducible representations goes back to \cite{CR04}, \cite{FKS} and is by now well established, see e.g. for some context  \cite{Rou2Lie},  \cite{ICMStr}, \cite{Webster}. Importantly, Grassmannians appear hereby naturally;  each weight space is categorified via the category of graded modules over the cohomology ring of some Grassmannians. For further motivation we refer to   \cite{Vazsurvey}.

\subsection*{Outline of the construction for the quiver with one vertex}
We start our outline of the construction with the case $\frakg=\mathfrak{sl}_2$ (i.e., $\Gamma$ has one vertex and no arrows). In this case, the KLR algebra $R_\bfn$ is the nil-Hecke algebra $\NH_n$ which, by definition, acts faithfully on $\Pol_n=\Bbbk[X_1,\ldots,X_n]$. This action can be extended to an $\NH_n$-action on $\hPol_n=\Pol_n\otimes \mywedge^\bullet(\omega_1,\ldots,\omega_n)$, see \S\ref{subs:hNH}. The algebra $\hR_\bfn$ is in case of $\frakg=\mathfrak{sl}_2$ easy to describe as what we call the \emph{exterior power extended nil-Hecke algebra} $\hNH_n$, see Definition~\ref{df:hNH}. It is the subalgebra  of linear endomorphisms of $\hPol_n$, generated by the operators coming from $\NH_n$ together with the multiplications by $\omega_1,\ldots,\omega_n$. (It is in fact enough to only take  $\omega_1$, the other $\omega$'s will then be included). For a presentation via generators and relations see \Cref{lem:gen-rel-hhNH}.

The Grassmannian--Steinberg variety $\ccZplus$ is also easy to describe here: 
Let $\calF$ be the variety of full flags in $\bbC^n$ and set $\rmG=\op{GL}_n(\bbC)$. Then the algebra $\NH_n$ can be identified with the equivariant Borel-Moore homology $H_*^\rmG(\calF\times\calF)$ (as the Steinberg quiver flag variety is $\calF\times\calF$) and the representation $\Pol_n$ can be identified with $H_*^\rmG(\calF)$, \cite{VV}. To upgrade $\Pol_n$ to $\hPol_n$ with a geometric action of $\hNH_n$, we add some Grassmannian varieties  $\op{Gr}_k(\bbC^n)$ to the construction, schematically,
\begin{equation*}
\Pol_n\leadsto \hPol_n\text{ corresponds geometrically to } \calF\leadsto \calF\times\calG:= \calF\times\coprod_{k=0}^n \op{Gr}_k(\bbC^n).
\end{equation*}
    Indeed, since we can identify $H_*^\rmG(\calF\times\calG)$ with $\hPol_n$ as vector space, adding the direct factor $\calG$ results in adding $\mywedge^\bullet(\omega_1,\ldots,\omega_n)$ to the vector space $\Pol_n$. 

To include now the multiplications by $\omega_1,\ldots,\omega_n$ into the algebra $\NH_n\cong H_*^\rmG(\calF\times\calF)$, one might naively add a copy of $\calG$.  But then one gets stuck, since there is no apparent convolution product on $H_*^\rmG(\calF\times\calF\times\calG)$.
     Alternatively, we can add two copies of $\calG$ to obtain an algebra $\hhNH_n:=H_*^\rmG(\calF\times\calF\times \calG\times\calG)$ which indeed has a convolution product and acts faithfully on $\hPol_n\cong H_*^\rmG(\calF\times\calG)$.  This algebra $\hhNH_n$ is however, bigger than the algebra $\hNH_n$ we want to reconstruct geometrically.  We call this algebra the \emph{doubly extended nil-Hecke algebra}, see ~\Cref{df:hhNH}, since
    $\hhNH_n$ not only contains the desired multiplications by $\omega_1,\ldots,\omega_n$, but additionally some annihilation operators $\omega^-_1,\ldots,\omega^-_n$. (This observation touches an intriguing aspect at the core of the Naisse--Vaz construction: having the creation operators and not the annihilation operators creates an asymmetry which is designed to address the aforementioned asymmetry for Verma modules.)  
    
    To get a geometric construction of precisely the algebra $\hNH_n$, we combine the two naive approaches. Namely, we extend the Steinberg quiver flag variety $\calF\times\calF$ by adding the variety $\calG$ of Grassmannians, and then define the \emph{Grassmannian--Steinberg quiver flag variety} $\ccZplus:=\calF\times\calF\times \calG$, see \Cref{Zplus}.  In \S\ref{subs:geom-hNH}, we identify the vector space $H_*^\rmG(\calF\times\calF\times \calG)$ with a sub\emph{space} of the algebra $\hhNH_n=H_*^\rmG(\calF\times\calF\times \calG\times\calG)$. We show that this sub\emph{space} is in fact a sub\emph{algebra} and that there is an isomorphism of algebras $\hNH_n\cong H_*^\rmG(\calF\times\calF\times \calG).$ 
 The crucial insight hereby is that $H_*^\rmG(\calF\times\calF\times \calG)$ has no natural convolution product on itself, but via an embedding into the algebra $\hhNH_n$ it inherits a (non-obvious) product. 
 \subsection*{Cyclotomic quotients via equivariant intersection cohomology}
The doubly extended nil-Hecke algebra $\hhNH_n$ encodes moreover the differential on $\hNH_n$. Given $N\in\bbZ_{\geqslant 0}$, there is an operator $\bfd_N\colon\hNH_n\to \hNH_n$ turning $\hNH_n$ into a dg-algebra whose homology is concentrated in degree zero and isomorphic to the cyclotomic quotient $\NH^N_n$ of $\NH_n$, see \S\ref{subs:DG-on-hNH}. We show in \Cref{lem:dg} that there is an element $\mathbb{d}_N\in\hhNH_n$ such that the differential $\bfd_N$ is the (super) commutator with $\mathbb{d}_N$. The action of $\mathbb{d}_N$ on the faithful representation induces an operator $d_N\colon \hPol_n\to \hPol_n$ which has a nice geometric meaning, see \Cref{differentials}.  We realise $\NH^N_n$ in geometric terms in \Cref{cyclogeo}, and explain how the resolution $(\hNH_n,\bfd_N)$ of $\NH_n^N$ relates to the resolution of $H^*(\op{Gr}_n(\bbC^N))$ by the equivariant intersection cohomology of orbit closures in $\Hom(\bbC^n,\bbC^N)$, see \Cref{IC}.

The case $\frakg=\mathfrak{sl}_2$ is not only interesting on its own, but also contains key ideas of the general construction which we describe next. 
 We want to stress however that the general case appears to be substantially more involved.  Since however a precise description of the difficulties requires a good understanding of the basic $\mathfrak{sl}_2$ case, we treat this first in detail, often with different proofs as in the general case. 
  
 This phenomenon of substantially increasing difficulty should be compared with the situation on the coherent sheaves side. There, a geometric reconstruction of the Naisse--Vaz categorified Verma module already exists in case of $\frakg=\mathfrak{sl}_2$, \cite{Rou2Verma}. Rouquier used in \cite{Rou2Verma} (categories of coherent sheaves on) spaces (so-called zastavas) of quasi-maps  to flag varieties, \cite{FFKM}, to provide a geometric categorification with the action of Chevalley generators given by correspondences.  Adding a superpotential and passing to matrix factorizations allows in this framework to construct a simple quotient of the categorified Verma module, see \cite{Rou2Verma} for details.  In the special case of  $\frakg=\mathfrak{sl}_2$, these spaces are particularly nice and smooth. In  general, these zastavas have a rich complicated structure, \cite{FFKM}, \cite{Rou2Verma} which was used in \cite{FFKM} to give a geometric construction of the (not categorified) Verma modules themselves. 

\subsection*{Outline of the construction for general quivers}
Now, let $\Gamma$ be an arbitrary finite quiver without loops. We like to upgrade the geometric construction of the KLR-algebra $R_\bfn$ to $\hR_\bfn$. The case $\frakg=\mathfrak{sl}_2$ suggests the following naive strategy which is considered in \S\ref{subs:Z'-subalg-Z}. We just add the variety $\calG$ of Grassmannians to each vertex of the quiver.  We then take this enhancement of the geometric construction of $R_\bfn$,  form the corresponding generalized Steinberg variety  and take its equivariant homology. (The sizes of the Grassmannian space is encoded by an idempotent in $\hR_\bfn$). Inside this doubly extended algebra we might hope to find exactly as before a subalgebra isomorphic to  $\hR_\bfn$. Naively mimicking this approach  has however a problem that did  not show up for $\frakg=\mathfrak{sl}_2$: the Grassmannians interact with the maps given by the representation of the quiver. Adding Grassmannians naively creates an algebra which is by far too large. We have to make additional modification of our variety to get exactly the algebra $\hR_\bfn$ that we need (and not just a bigger algebra). Solving this problem is the most complicated part of the paper. 

For the vertices that are sources of the quiver we see in \S\ref{subs:source}, that there is however an easy solution: it suffices to add the additional \emph{Grassmannian--Steinberg condition}  that the Grassmannian spaces that we have added are in the kernel of the quiver representation maps, see \Cref{sinkcondition}. We obtain in \Cref{geomsourcealgebra} a geometric construction of a natural subalgebra of $\hR_\bfn$, namely the \emph{source algebra} where the Grassmannian dimension vectors are supported at sources of the quiver. However, as shown in  \Cref{ex:contr-target}, it is in general not possible to reduce the size of this algebra far enough by adding extra conditions on the Grassmannian spaces. In general, we have to do something more sophisticated. 

In general, we add a second copy of the quiver representation space and do a twist mixing the two copies, \Cref{defWtwist}, \Cref{exWtwist}. This trick is motivated by the \emph{diagram varieties} discussed in \Cref{app:diag-var}. The idea behind this is that the mysterious Naisse-Vaz basis elements, see \Cref{prop:NVbasis-gen}, in $\hR_\bfn$ should be realised in terms of fundamental classes of certain varieties whose definition is predicted by the diagrammatics. For $\mathfrak{sl}_2$ this works very well, in full general the principal  idea is still correct, but has to be adapted, mainly due to some non-transversality problems. We believe however that they do not to occur for the source algebra. We introduce the aforementioned twists to resolve these non-transversality issues, see \Cref{app:diag-var} for a detailed treatment and examples.

Incorporating the crucial \emph{twisted Springer conditions} from \Cref{crucialdef} to the naive construction, we obtain in \Cref{geomsinkalgebra} a geometric construction of another natural subalgebra of $\hR_\bfn$, namely the \emph{sink algebra} where the Grassmannian dimension vectors are supported at sinks of the quiver. In \S\ref{sec:isom-thm-GqHecke-gen} we are  finally able to define the general Grassmannian--Steinberg varieties  $\Zplus$ in \Cref{DefGrassSt}. As in the geometric construction of KLR, \cite{VV}, we consider now coloured (with colours being the vertices of $\Gamma$)  versions of $\calF$, $\calG$, see \S\ref{sec:coloured}.  

For the construction of $\Zplus$  we then proceed in steps, see \S\ref{111}, \begin{enumerate}[(i)]
\item \label{szero} We add the variety of Grassmannians to the usual quiver flag variety and obtain the \emph{Grassmannian Springer variety}, see \Cref{Grassmannian Springer variety}, 
\begin{equation*}
\quad\quad\quad\rmY=\tF\times \calG\subset E_\bfn\times\calF\times\calG, \;\text{where}\;\tF=\{(\alpha,\bV)\mid \text{$\alpha$ preserves $\bV$}\}\subset E_\bfn\times\calF.
\end{equation*}
\item\label{sone} We add a second copy of the representation space and twist to obtain $\Ytw$, see \Cref{crucialdef}.
\item\label{stwo} We build the Steinberg variety $\Ztw=\Ytw\times _{E_\bfn\oplus E_\bfn}\Ytw$.
\item\label{sthree} We consider a (natural) asymmetric subvariety $\Znaivtw=\Ytw\times _{E_\bfn\oplus E_\bfn}\Ytw_0\subset \Ztw.$
\item We define the \emph{Grassmannian--Steinberg variety} as the subvariety $\Zplus\subset\Znaivtw$ 
of points $(\alpha,\beta,\bV,\tV,W)$ satisfying the \emph{Grassmannian--Steinberg conditions} 
\begin{equation*} 
\alpha_h((\bV^r\cap W_i)^\perp\cap W_i)\subset (\bV^r\cap W_j)^\perp\cap W_j,
\end{equation*}   
for any arrow $h\colon i\to j$ in $\Gamma$ and any $r\in [0;n]$, see \eqref{eq:cond-bfZ'-gen}.
\end{enumerate}
In case of  $\mathfrak{sl}_2$, step \ref{szero}  is the passage from $\calF$ to $\calF\times\calG$, step \ref{sone} does nothing. Step \ref{stwo} gives then $\calF\times\calF\times\calG\times\calG$ with the subvariety  $\calF\times\calF\times\calG$ obtained in step \ref{sthree}. The last step does again nothing and we obtain $\Zplus=\calF\times\calF\times\calG$. 

Apart from constructing a variety and showing that its homology has the correct size, the mayor challenge is to show that it inherits a convolution product. This is done similarly to $\mathfrak{sl}_2$, but is technically more involved, since we do not always have nice pavings for all the involved varieties in general. 

We finish this introduction with remarks putting our construction in context.
\subsection*{Further results and applications}
\subsubsection*{Pavings and Gells} 
To control the size and the grading of the involved Borel--Moore homologies we introduce several types of pavings. In case of $\mathfrak{sl}_2$ one can work with well-known pavings coming from Schubert cells, and in fact many decompositions would do the job for us. In the general case, the situation is much more tricky. There, we need to have a paving which is adapted to both, the asymmetric subvariety we define and to the mysterious algebraic and diagrammatic bases \cite[Thm.~3.16]{NaisseVaz}, \cite[Prop.~3.11]{Vazsurvey}. We therefore introduce first for  $\mathfrak{sl}_2$, see \S\ref{sec:flags-cells},  and then in general in \S\ref{subs:col-cells}, more clever decompositions of the varieties into \emph{Gells}. These are analogs of the Schubert cell decomposition of flag varieties, but they in addition incorporate some interaction with the added Grassmannian varieties  and with the quiver representations in a more clever way. We believe that the combinatorics of these Gells is interesting on its own. For some explicit examples see \Cref{app:diag-var}.
\subsubsection*{Dg-model of cyclotomic quotients} 
A geometric construction of cyclotomic quotients of  $R_\bfn^\Lambda$  is often related to adding shadow vertices to quiver representations, possibly with stability conditions, see e.g. \cite{Webster}, \cite{stroppel2014quiver}.
Building on this idea, see \Cref{shadow}, we construct in case of $\frakg=\mathfrak{sl}_2$ in \S\ref{subs:geom-mean-resol} not only the cyclotomic quotient, but instead the whole resolution $(\hNH_n,\bfd_N)$ of $R_\bfn^\Lambda$ geometrically and provide  as a byproduct a new geometric perspective on the cyclotomic KLR algebras. 
\subsubsection*{Diagram varieties} 
To understand the rational behind our definitions it might be helpful to have a look at \Cref{app:diag-var}, where we indicate how the diagrams from  \cite{NaisseVaz} encode the varieties. We believe that the notion of diagram varieties is a powerful notion and interesting to study further.
\subsubsection*{Higher floating dots}
In \S\ref{Higherfloatingdots} we treat the higher floating dots and in particular realise in \Cref{highdot} their action on the faithful representation geometrically. The construction also provides a geometric interpretation of the recursion formulas.
\subsubsection*{Geometric Naisse-Vaz program} 
As mentioned already, another approach towards a geometric realisation of the Naisse--Vaz constructions is given in  \cite{Rou2Verma} by the construction of a 
  categorification of Verma modules for $\frakg=\mathfrak{sl}_2$. The two constructions are quite different and from our point of view complementary (and most likely dual to each other).  
  In the paper \cite{Rou2Verma}, Rouquier provides directly a geometric categorification of Verma modules using the geometry of zastavas. 
  What we do, is the geometric construction of the algebra $\hR_\bfn$ itself,  without doing the categorification explicitly. It would be nice to have a common framework for both constructions. Our approach might help thereby to make the explicit translation to the original (algebraic and diagrammatic) construction.

 \subsection*{\it Acknowledgments}
We thank Pedro Vaz for his explanation of how the Naisse-Vaz algebras work and Alexandre Minets for useful discussions on how to do computations in convolution algebras, the main steps in the generalization to positive characteristics are due to him.  
R. M.  is grateful for the support and hospitality of the MPI for Mathematics in Bonn, where a big part of this work was done. C. S. is supported by Gottfried Wilhelm Leibniz-Preis of the DFG.

\subsection*{Conventions and notation} We fix throughout the paper a ground field $\Bbbk$ which we assume for simplicity to be of characteristic $0$. 
If not stated otherwise, tensor products, homomorphisms etc. are taken over $\Bbbk$; for instance $\otimes=\otimes_\Bbbk$, $\End(V)=\End_\Bbbk(V)$. By an \emph{algebra} we mean an associative algebra (over $\Bbbk$, if not specified otherwise) with unit $1$. For $a,b\in\bbZ$ such that $a\leqslant b$, we abbreviate $[a;b]=\{a,a+1,a+2,\ldots,b\}$ and denote by $\frakS_n$ the symmetric group generated by the simple transpositions $s_i=(i,i+1)$, $1\leq n-1$ with composition e.g. $(2,3)(1,2)=(1,3,2)$. For $1\leq k\leq n$ we denote by $w_{0,k}$ the longest element in $\frakS_k\subset \frakS_n$. If not stated otherwise, \emph{grading} always means \emph{$\mZ$-grading}.

\addtocontents{toc}{{\textbf{Preliminaries: The main players}}}
\begingroup
\section{KLR algebras and their extended versions}
\label{sec:KLR}
\endgroup

Let $\Gamma=(I, A)$ be a finite quiver without loops.  We denote by $I$ the set of vertices and by $A$ the set of arrows of $\Gamma$. For $a\in A$ let $s(a)\in I$ be its source and $t(a)\in I$ its target. For $i, j \in I$ let $h_{i,j}$ be the number of arrows from $i$ to $j$ in $\Gamma$.  For $i\ne j$ define the polynomials $\cQ_{i,j}(u,v)=(u-v)^{h_{i,j}}(v-u)^{h_{j,i}}$. A \emph{word} in $I$ is a finite sequence (including the empty sequence) of elements in $I$.  We will define  a category with objects words in $I$ and morphisms between words given by diagrams with $I$-labelled strands built from 
\begin{equation}\label{KLRgeneratorsdiag}
	\tikz[very thick,baseline={([yshift=1.0ex]current bounding box.center)}]{
	          \draw (-2,-.5) node[below] {\small $i$} -- (-1,.5);
	          \draw (-1,-.5) node[below] {\small $j$} -- (-2,.5) 
	} 
	 \quad\quad\text{and} \quad\quad
	\tikz[very thick,baseline={([yshift=1.0ex]current bounding box.center)}]{
	          \draw (0,-.5) node[below] {\small $i$} -- (0,.5)   node [midway,fill=black,circle,inner sep=2pt]{};;
	   } \quad\quad
\text{and} \quad\quad
	\tikz[very thick,baseline={([yshift=1.0ex]current bounding box.center)}]{
	          \draw (0,-.5) node[below] {\small $i$} -- (0,.5) ;
	   }
\end{equation}
We always read such diagrams from bottom to top. 

\subsection{KLR algebras}
\label{subs:KLR-def}
Consider the  $\Bbbk$-linear strict monoidal category  $\mathcal{R}^{\op{free}}$ defined as follows: on the level of objects it is generated by the elements in $I$, i.e., objects are just word in $I$ and the tensor product on objects is concatenation of words, with the empty word being the unit. On the level of morphisms $\mathcal{R}^{\op{free}}$  is generated by morphisms $ij\rightarrow ji$, $ i\rightarrow i$ (and the identity morphisms $i\rightarrow i$) for any $i,j\in I$ given by the diagrams displayed in \eqref{KLRgeneratorsdiag}. The label at the bottom indicates the source (and determines the target). We interpret the label as a  \emph{colour} on each strand. The tensor product of morphisms is given by placing diagrams horizontally next to each other.  Composition of morphisms means stacking diagrams vertically (the top morphisms follows the bottom morphism, and the composition is defined whenever the colours of the strands match, and is by convention zero otherwise.) Thus an arbitrary morphism is a linear combination of diagrams obtained by stacking vertically and horizontally the generating diagrams \eqref{KLRgeneratorsdiag} such that colours match. 

\smallskip
We use the usual diagrammatical abbreviations, such as 
\begin{equation*}
	\tikz[very thick,baseline={([yshift=1.0ex]current bounding box.center)}]{  
	          \draw (-2,-.5) node[below] {\small $i$} -- (-1,.5); 	         
	           \draw (-1,-.5) node[below] {\small $j$} -- (-2,.5)
	           node [near end,fill=black,circle,inner sep=2pt]{};}
	           \quad:=\quad
\tikz[very thick,baseline={([yshift=1.0ex]current bounding box.center)}]{  
	          \draw (-2,-.5) node[below] {\small $i$} -- (-1,.5);
	          \draw (-1,-.5) node[below] {\small $j$} -- (-2,.5); 
                  \draw (-1,.5) -- (-1,1.5) ;
	          \draw (-2,0.5)-- (-2,1.5)  node [midway,fill=black,circle,inner sep=2pt]{};
} \quad=\quad
\tikz[very thick,baseline={([yshift=1.0ex]current bounding box.center)}]{  
 \draw (-1, -0.5) node[below] {\small $i$} -- (-1,.5)   node at (-0.5,0){$\circ$};
	          \draw (-2,-0.5) node[below] {\small $j$} -- (-2,.5) node [midway,fill=black,circle,inner sep=2pt]{};
	          \draw (0,-.5) node[below] {\small $i$} -- (1,.5);
	          \draw (1,-.5) node[below] {\small $j$} -- (0,.5); 
	      
} 
\quad and\quad 
\tikz[very thick,baseline={([yshift=1.0ex]current bounding box.center)}]{
\node at (0.5,0) {$p(X_1,X_2)$};
          \draw (-.55,-.4) -- (-.55,.4); \draw (1.55,-.4) -- (1.55,.4);
          \draw (-.55,-.4) -- (1.55,-.4); \draw (-.55,.4) -- (1.55,.4); 
          \draw (0,-.75) node[below] {\small $i$} -- (0,-.4);\draw (0,.75) -- (0,.4);
          \draw (1,-.75) node[below] {\small $j$} -- (1,-.4);\draw (1,.75) -- (1,.4);
	}
\end{equation*}
where $p(X_1,X_2)\in \Bbbk[X_1,X_2]$.  This polynomial should be viewed as a $\Bbbk$-linear combination of monomials and $X_1^aX_2^b$ is interpreted as the identity diagram with $a$ dots on the first strand (coloured $i$) and $b$ dots on the second strand (labelled $j$).
\begin{df} The \emph{diagrammatic KLR category}  $\mathcal{R}$ is the quotient of the $\Bbbk$-linear strict monoidal category  $\mathcal{R}^{\op{free}}$ by the local relations \eqref{eq:KLRR2} - \eqref{eq:R3_2} on morphisms.

\begin{equation}\label{eq:KLRR2}
	\tikz[very thick,baseline={([yshift=1.0ex]current bounding box.center)}]{
	      \draw  +(0,-.75) node[below] {\small $i$} .. controls (1,0) ..  +(0,.75);
	      \draw  +(1,-.75) node[below] {\small $i$} .. controls (0,0) ..  +(1,.75); 
	}  \mspace{2mu} = \mspace{5mu} 0 \mspace{40mu}\text{and}\mspace{40mu}
        \tikz[very thick,baseline={([yshift=1.0ex]current bounding box.center)}]{
	      \draw  +(0,-.75) node[below] {\small $i$} .. controls (1,0) ..  +(0,.75);
	      \draw  +(1,-.75) node[below] {\small $j$} .. controls (0,0) ..  +(1,.75); 
	}  \mspace{2mu} = \mspace{5mu} 
	\tikz[very thick,baseline={([yshift=1.0ex]current bounding box.center)}]{
\node at (0.5,0) {$\cQ_{i,j}(X_1,X_2)$};
          \draw (-.55,-.4) -- (-.55,.4); \draw (1.55,-.4) -- (1.55,.4);
          \draw (-.55,-.4) -- (1.55,-.4); \draw (-.55,.4) -- (1.55,.4); 
          \draw (0,-.75) node[below] {\small $i$} -- (0,-.4);\draw (0,.75) -- (0,.4);
          \draw (1,-.75) node[below] {\small $j$} -- (1,-.4);\draw (1,.75) -- (1,.4);
	}\mspace{40mu}\text{if }i\neq j
\end{equation}
        \begin{align}
	\tikz[very thick,baseline={([yshift=1.0ex]current bounding box.center)}]{
	          \draw (0,-.5) node[below] {\small $i$} -- (1,.5);
	          \draw (1,-.5) node[below] {\small $j$} -- (0,.5) node [near start,fill=black,circle,inner sep=2pt]{};
	} 
	&\quad = \quad
	\tikz[very thick,baseline={([yshift=1.0ex]current bounding box.center)}]{
	          \draw (0,-.5) node[below] {\small $i$} -- (1,.5) ;
	          \draw (1,-.5) node[below] {\small $j$} -- (0,.5)  node [near end,fill=black,circle,inner sep=2pt]{};
	}  \mspace{10mu}
&
	\tikz[very thick,baseline={([yshift=1.0ex]current bounding box.center)}]{
	          \draw (0,-.5) node[below] {\small $i$} -- (1,.5) node [near start,fill=black,circle,inner sep=2pt]{};
	          \draw (1,-.5) node[below] {\small $j$} -- (0,.5);
	} 
	&\quad = \quad
	\tikz[very thick,baseline={([yshift=1.0ex]current bounding box.center)}]{
	          \draw (0,-.5) node[below] {\small $i$} -- (1,.5)node [near end,fill=black,circle,inner sep=2pt]{};
	          \draw (1,-.5) node[below] {\small $j$} -- (0,.5);
	}  \mspace{10mu}
	 & \text{ if }i&\ne j,
\end{align}

\begin{equation}\label{eq:nh1}
	\tikz[very thick,baseline={([yshift=1.0ex]current bounding box.center)}]{
	          \draw (0,-.5) node[below] {\small $i$} -- (1,.5);
	          \draw (1,-.5) node[below] {\small $i$} -- (0,.5)  node [near end,fill=black,circle,inner sep=2pt]{};
	}  
	\quad - \quad
	\tikz[very thick,baseline={([yshift=1.0ex]current bounding box.center)}]{
	          \draw (0,-.5) node[below] {\small $i$} -- (1,.5); 
	          \draw (1,-.5) node[below] {\small $i$} -- (0,.5) node [near start,fill=black,circle,inner sep=2pt]{};
	} \quad = \quad 
	\tikz[very thick,baseline={([yshift=1.0ex]current bounding box.center)}]{
	          \draw (0,-.5) node[below] {\small $i$} -- (0,.5);
	          \draw (1,-.5) node[below] {\small $i$} -- (1,.5);
	}     \quad = \quad
        	\tikz[very thick,baseline={([yshift=1.0ex]current bounding box.center)}]{
	          \draw (0,-.5) node[below] {\small $i$} -- (1,.5) node [near start,fill=black,circle,inner sep=2pt]{};
	          \draw (1,-.5) node[below] {\small $i$} -- (0,.5);
	}  	 \quad - \quad
	\tikz[very thick,baseline={([yshift=1.0ex]current bounding box.center)}]{
	          \draw (0,-.5) node[below] {\small $i$} -- (1,.5)node [near end,fill=black,circle,inner sep=2pt]{};
	          \draw (1,-.5) node[below] {\small $i$} -- (0,.5);
		}
\end{equation}

\begin{align} \label{eq:R3_1}
	\tikz[very thick,scale=.75,baseline={([yshift=1.0ex]current bounding box.center)}]{
		     	 \draw  +(1,-1) node[below] {\small $j$}  .. controls (0,0) ..  +(1,1);
		 \draw (0,-1) node[below] {\small $i$} -- (2,1); 
	          \draw (0,1)-- (2,-1) node[below] {\small $k$} ; 
	 } 
&\quad = \quad
	\tikz[very thick,scale=.75,baseline={([yshift=1.0ex]current bounding box.center)}]{
		     	 \draw  +(1,-1) node[below] {\small $j$}  .. controls (2,0) ..  +(1,1);
		 \draw (0,-1) node[below] {\small $i$} -- (2,1);
	          \draw (0,1)-- (2,-1) node[below] {\small $k$};
	 } 
&\text{unless $i=k\neq j$,} 
\end{align}
\begin{align} \label{eq:R3_2} 
	\tikz[very thick,scale=.75,baseline={([yshift=1.0ex]current bounding box.center)}]{
		     	 \draw  +(1,-1) node[below] {\small $j$}  .. controls (0,0) ..  +(1,1);
		 \draw (0,-1) node[below] {\small $i$} -- (2,1);
	          \draw (0,1)-- (2,-1) node[below] {\small $i$} ; 
	 }  \  - \  
	 \tikz[very thick,scale=.75,baseline={([yshift=1.0ex]current bounding box.center)}]{
		     	 \draw  +(1,-1) node[below] {\small $j$}  .. controls (2,0) ..  +(1,1);
		 \draw (0,-1) node[below] {\small $i$} -- (2,1); 
	          \draw (0,1)-- (2,-1) node[below] {\small $i$};
	 }
 &= \mspace{15mu}
	\tikz[very thick,scale=.75,baseline={([yshift=1.0ex]current bounding box.center)}]{
\node at (1,0) {\small $\dfrac{\cQ_{i,j}(X_3,X_2)-\cQ_{i,j}(X_1,X_2)}{X_3-X_1}$};
          \draw (-2,-.75) -- (-2,.75); \draw (3.9,-.75) -- (3.9,.75);
          \draw (-2,-.75) -- (3.9,-.75); \draw (-2,.75) -- (3.9,.75); 
          \draw (0,-1) node[below] {\small $i$} -- (0,-.75);\draw (0,.75) -- (0,1);
          \draw (1,-1) node[below] {\small $j$} -- (1,-.75);\draw (1,.75) -- (1,1);
          \draw (2,-1) node[below] {\small $i$} -- (2,-.75);\draw (2,.75) -- (2,1);
	}
& \text{if $i \neq j $.} 
\end{align}
Here, a monomial $X_1^{a_1}\cdots X_n^{a_n}$ in a box above $n$ strands abbreviates the configuration of $a_i$ dots on the $i$th strands; a polynomial is a linear combination of such.
\end{df}
Let now  $\bfn\in \bbZ_{\geqslant 0}I$ be a dimension vector of $\Gamma$. We can view $\bfn$ as a formal $\bbZ_{\geqslant 0}$-linear combination of the vertices $i\in I$, and  write $\bfn=\sum_{i\in I}n_i\cdot i$, where $n_i$ is the dimension at the vertex $i$.  We set $n=|\bfn|=\sum_{i\in I}n_i$.   We denote by  $I^\bfn$ the subset of $I^n$ consisting all ordered sequences \emph{having dimension vector}  $\bfn$, that is all
$\ui = (i_1,i_2, \dots, i_n)$ with $i_r \in I$ such that $i$ appears $n_i$ times in the sequence. Interpreting sequences as objects in $\mathcal{R}$ we have that $\op{Hom}_{\mathcal{R}}(\ui,\uj)=\{0\}$, if $\ui$ and $\uj$ have not the same  dimension vector.
\begin{df}
The \emph{KLR algebra} $R_\bfn=R_\bfn(\Gamma)$ is the $\Bbbk$-algebra   $\bigoplus_{\ui,\uj\in I^\bfn}\op{Hom}_{\mathcal{R}}(\ui,\uj)$.
\end{df}
For each $\ui=(i_1,\ldots, i_n)\in I^\bfn$  we have in $R_\bfn$  the idempotent 
\begin{equation}\label{idempotent}
1_{\ui}\, = 
        \tikz[very thick,xscale=1.5,baseline={([yshift=.8ex]current bounding box.center)}]{
          \draw (-.5,-.5) node[below] {\small $i_1$} -- (-.5,.5); 
          \draw (0,-.5) node[below] {\small $i_2$} -- (0,.5); 
          \node at (0.75,0) {$\dotsm$};
          \draw (1.5,-.5) node[below] {\small $i_n$} -- (1.5,.5); 
        } .
\end{equation}

\begin{rk}
More concretely, see \cite{KL1}, elements in the algebra  $R_\bfn$ are linear combinations of \emph{KLR diagrams} of weight $\bfn$, that is of planar diagrams containing $|\bfn|$ strands such that the following holds: the strands connect $|\bfn|$ points on a horizontal line at the bottom with $|\bfn|$ points on another horizontal line located above. Each strand is labelled by an element of $I$ and there are $n_i$ strands with label $i$. Two strands are allowed to cross transversally, but there are no triple-intersections. Strands are allowed to carry dots, but only away from crossing. These diagrams are considered modulo isotopies (in particular, a dot is allowed to move along a strand, but not slide past a crossing) and modulo the defining relations \eqref{eq:KLRR2} -- \eqref{eq:R3_2} of $\mathcal{R}$.  
\end{rk}
We name some morphisms from  $R_\bfn$, since they will play an important role: 
\begin{equation}\label{diagdotcross}
X_{r}1_{\bi}:=\tikz[very thick,xscale=1,baseline={([yshift=.8ex]current bounding box.center)}]{
          \draw (-.4,-.5) node[below] {\small $i_1$} -- (-.4,.5); 
           \draw (0.4,-.5)  node[below] {\small $i_{r-1}$} -- (0.4,.5); 
          \node at (0,0) {$\dotsm$};
          \draw (1,-.5)  node[below] {\small $i_r$} -- (1,.5); 
          \node at (1,0){\textbullet};
          \node at (2,0) {$\dotsm$};
              \draw (1.6,-.5)  node[below] {\small $i_{r+1}$} -- (1.6,.5); 
          \draw (2.4,-.5) node[below] {\small $i_n$} -- (2.4,.5); 
        } ,
\mspace{20mu}
\tau_{r}1_{\bi} :=\tikz[very thick,baseline={([yshift=+.6ex]current bounding box.center)}]{
    \draw (-1.4,-.5) node[below] {\small $i_1$} -- (-1.4,.5);
    \draw   (0,-.5)-- (1,.5);
    \draw   (1,-.5)-- (0,.5);
     \draw   (1.5,-.5)-- (1.5,.5);
 \draw   (-0.5,-.5)-- (-0.5,.5);
    \node at (-0.5,-.75){$i_{r-1}$};
    \node at (0,-.75){$i_r$};
    \node at (0.9,-.75){$i_{r+1}$};
    \node at (1.6,-.75){$i_{r+2}$};
    \draw (2.5,-.5) node[below] {\small $i_n$} -- (2.5,.5);
    \node at (-.9,0) {$\dotsm$};\node at (2,0) {$\dotsm$};
} .
\end{equation}

\begin{df}\label{gradingKLR} We consider $R_\bfn$ as a graded algebra by setting, for $r\in[1;n-1]$, $\op{deg}(X_{r}1_{\bi})=2$ and $\op{deg}(\tau_{r}1_{\bi} )=2h_{i_r,i_{r+1}}-2\delta_{i_r,i_{r+1}}$, 	where $\delta$ is the Kronecker delta.
\end{df}
\begin{rk}
The convention for $\op{\deg}(\tau_r1_\ui)$ is taken here to have the degree equal to the degree of the operator acting on the polynomial representation. 
\end{rk}
\begin{df}\label{tau}
Given a reduced expression $w=s_{r_1}\ldots s_{r_t}$ in $\frakS_n$, let $\tau_w=\tau_{r_1}\ldots \tau_{r_t}$. Whenever we write $\tau_w$ we assume a reduced expression has been chosen for $w\in\frakS_n$.
\end{df}

\subsection{Naisse-Vaz extensions of KLR algebras}
\label{subs:des-KLR-NV}
To define the algebras $\hR_\bfn=\hR_\bfn(\Gamma)$ from~\cite[\S3]{NaisseVaz} we use an extension of the monoidal category $\mathcal{R}^{\op{free}}$ from \S\ref{subs:KLR-def}.

Consider the $\Bbbk$-linear strict monoidal supercategory\footnote{That is a monoidal category enriched in vector superspaces, i.e. each morphism space is a superspace and composition induces an even linear map, see e.g. \cite[Def. 1.4]{BrundanEllis} for a definition. We also use the convention that when we work over a field of characteristic $2$, then we assume additionally that each odd endomorphism of the tensor unit squares to zero.}  $\mathcal{C}$ defined as $\mathcal{R}^{\op{free}}$
but with additional odd endomorphisms  $\Huge{\circ_i^a}$ of the tensor unit for $i\in I$, $a\in \bbZ_{\geqslant 0}$. These morphisms are taken modulo the relations \eqref{eq:KLRR2}-\eqref{eq:R3_2} from $\mathcal{R}$ and additionally 
\begin{align}
\label{eq:fdots}
	\tikzdiagh{0}{
	          \draw (0,-.5) node[below] {\small $i$} -- (0,.5);
		  \fdot[a]{j}{.5,0};
	} 
	& \ = \ 
\begin{cases}
	\tikzdiagh{0}{
	          \draw (0,-.5) node[below] {\small $i$} -- (0,.5);
		 \fdot[a-1]{i}{-1.1,0};
	}
	\quad  - \quad
	\tikzdiagh{0}{
	          \draw (0,-.5) node[below] {\small $i$} -- (0,.5) node [midway,tikzdot]{};
		 \fdot[a-1]{i}{.5,0};
	}
&\text{ if $i=j$ and $a > 0$, } \\
	\sum\limits_{t,r} (-1)^r
	q_{ij}^{tr}\quad 
	\tikzdiagh{0}{
	          \draw (0,-.5) node[below] {\small $i$} -- (0,.5) node [midway,tikzdot]{} node [midway, xshift=1.5ex, yshift=.75ex] {\small $t$};
		 \fdot[a+r]{j}{-1.25,0};
	}  
	 & \text{ if $ i \neq j$, }
\end{cases}
\end{align}
\begin{align}\label{eq:ExtR2}
	\tikzdiagh[scale=1]{0}{
	      	\draw  (0,-.75) node[below] {\small $i$} .. controls (0,-.5) and (1,-.5) .. (1,0) .. controls (1,.5) and (0, .5) .. (0,.75);
 	 	\draw  (1,-.75) node[below] {\small $j$} .. controls (1,-.5) and (0,-.5) .. (0,0) .. controls (0,.5) and (1, .5) .. (1,.75);
		\fdot[a]{j}{.4,.05};
	} 
	\  &= \  
	 \tikzdiagh{0}{
	      \draw  (0,-.75) node[below] {\small $i$} -- (0,.75);
	      \draw  (1,-.75) node[below] {\small $j$} -- (1,.75);
		\fdot[a]{j}{1.5,0}
	}
	 \  +  \sum_{t,r} q_{ij}^{tr} \sum_{\substack{h+\ell=\\r-1}} (-1)^h\ 
	\tikzdiagh{0}{
	      \draw  (0,-.75) node[below] {\small $i$} -- (0,.75)  node [midway,tikzdot]{} node[midway,xshift=1.5ex,yshift=.75ex] {\small $t$};
	      \draw  (1,-.75) node[below] {\small $j$} -- (1,.75)  node [midway,tikzdot]{} node[midway,xshift=1.5ex,yshift=.75ex] {\small $\ell$};	
		\fdot[a+h]{j}{-1.2,0};
	}& 
	 \quad \text{if } i \ne j,
\end{align}
where the $q_{ij}^{tr}\in\Bbbk$ are defined by the expansion $\cQ_{i,j}(u,v)=\sum_{t,r}q_{ij}^{tr}u^tv^r$. 
\begin{df}
The  {$\circ_i^a$} are called \emph{floating dots}. We call $i$ its \emph{colour} and $a$ its \emph{twist}.
We moreover use the abbreviation $\circ_i=\circ_i^0$  for floating dots with twist zero. Denote
\begin{equation}\label{diagfloatdot}
 \Omega_{r,i}^a 1_{\ui}:=\tikz[very thick,baseline={([yshift=.5ex]current bounding box.center)},
  decoration={markings, mark=at position 0.5 with {\arrow{>}}}] {
   \draw (-2,-.5) node[below] {\small $i_1$} -- (-2,.5);
     \node at (-1.5,0) {$\dotsm$};
  \draw (-1,-.5) node[below] {\small $i_r$} -- (-1,.5);
  \fdot[a]{i}{-.5,0};
  \draw (0,-.5) node[below] {\small $i_{r+1}$} -- (0,.5);
  \node at (.5,0) {$\dotsm$};
  \draw (1,-.5) node[below] {\small $i_n$} -- (1,.5);
} ,
\quad\quad
\Omega 1_\ui:= \Omega_{1,i_1}^0 1_{\ui}:=\tikz[very thick,baseline={([yshift=.5ex]current bounding box.center)},
decoration={markings, mark=at position 0.5 with {\arrow{>}}}] {
	\draw (-2,-.5) node[below] {\small $i_1$} -- (-2,.5);
	\fdot[]{i_1}{-1.5,0};
	\draw (-1,-.5) node[below] {\small $i_2$} -- (-1,.5);
	\node at (-.5,0) {$\dotsm$};
	\draw (0,-.5) node[below] {\small $i_n$} -- (0,.5);
} 
\end{equation}
\end{df}
\begin{df}
\label{df:ext-KLR-cated}

 The \emph{diagrammatic Naisse--Vaz category} $\mathcal{ER}$  is the quotient of  $\mathcal{C}$ by the ideal given by the additional relation  that a floating dot on the left is zero:
\begin{align*}
\tikzdiagh[xscale=.75]{0}{
	\fdot[a]{i}{-0.5,0.5};
	\draw (0,0) node[below] {\plusspacing \small $i_1$} -- (0,1);
	\draw (1,0) node[below] {\plusspacing \small $i_2$} -- (1,1);
	\node at(2,.5) {$\dots$};
	\draw (3,0) node[below] {\plusspacing \small $i_n$} -- (3,1);
}
\ = \ 
0.
\end{align*}
\end{df}

\begin{df}
\label{df:ext-KLR-alg}
We denote by $\hR_\bfn$ the $\Bbbk$-algebra $\bigoplus_{\ui,\uj\in I^\bfn}\op{Hom}_{\mathcal{ER}}(\ui,\uj)$.
\end{df}

Observe that, as for $\mathcal{R}$, nonzero morphisms only exist between objects having the same dimension vector. By definition of a monoidal supercategory, the floating dots anti-commute; in particular the $\Omega_{r,i}^a 1_{\ui}$ square to zero, diagrammatically \eqref{eq:ExtR2-tight}.

Elements in $\hR_\bfn$  are linear combinations of diagrams like the KLR diagrams from $R_\bfn$ except that they might have some additional small circles arising from the floating dot generators. 
Finitely many floating dots can appear in each diagram, but \eqref{eq:fdots} and \eqref{eq:ExtR2} allow to inductively move floating dots to the left. Namely, one can verify, see \cite[Lemma 3.11]{NaisseVaz}, that the following relation holds:
\begin{equation}
\label{eq:extra-rel}
\tikzdiagh{0}{
	\draw (0,0) node[below]{\small $i$} -- (0,1);
	\draw (1,0) node[below]{\small $i$} -- (1,1);
	\fdot[a]{i}{1.35,.5};
}
\ = \ 
\tikzdiagh[yscale=1.5]{0}{
	\draw (0,0) node[below]{\small $i$} ..controls (0,.15) and (1,.15) .. (1,.5)
		 ..controls (1,.85) and (0,.85) .. (0,1) ;
	\draw (1,0) node[below]{\small $i$} ..controls (1,.15) and (0,.15) .. (0,.5)
		..controls (0,.85) and (1,.85) .. (1,1)  node [near end,tikzdot]{};
	\fdot[a]{i}{0.4,0.5};
} 
\ - \ 
\tikzdiagh[yscale=1.5]{0}{
	\draw (0,0) node[below]{\small $i$} ..controls (0,.15) and (1,.15) .. (1,.5) node [near start,tikzdot]{}
		 ..controls (1,.85) and (0,.85) .. (0,1);
	\draw (1,0) node[below]{\small $i$} ..controls (1,.15) and (0,.15) .. (0,.5)
		..controls (0,.85) and (1,.85) .. (1,1);
	\fdot[a]{i}{0.4,0.5};
} 
\end{equation}

 In the resulting nonzero diagrams, floating dots only occur directly to the right of the leftmost strand and carry the same colour as this strand. We often omit indicating this unique interesting colour, see e.g. \eqref{eq:ExtR2-tight}, \eqref{eq:extrarel-tight}. By the super relations one can reduce to diagrams with at most one circle in each region and obtain a   more minimalistic (and concrete) presentation of  $\hR_\bfn$, see \cite[Cor.~3.17]{NaisseVaz}:
\begin{lem} \label{onlyomega1}As an (ordinary) algebra,  $\hR_\bfn$ is isomorphic to the algebra generated by $R_\bfn$ and $\Omega1_{\ui}$ for $\ui\in I^\bfn$ modulo the relations \eqref{eq:ExtR2-tight} and \eqref{eq:extrarel-tight}. 
\end{lem}
In \Cref{onlyomega1}, products are set to be zero in case the labels do not match.  
   \begin{equation}\label{eq:ExtR2-tight}
	\tikz[very thick,baseline={([yshift=1.5ex]current bounding box.center)}]{
		\draw  (1,-.5) node[below] {\small $i_1$} -- (1,1);
		\fdot[]{}{1.5,0} \fdot[]{}{1.5,0.5} 
		\draw  (2,-.5) node[below] {\small $i_2$} -- (2,1);
		\node at (2.5,0.25) {$\dotsm$};
		\draw  (3,-.5) node[below] {\small $i_n$} -- (3,1);
	}
	\  = \ 0 , 
	\end{equation}
	\begin{equation} \label{eq:extrarel-tight}
	\tikz[very thick,baseline={([yshift=0ex]current bounding box.center)}]{
		\draw  (0,-.75) node[below] {\small $i$} .. controls (0,-.5) and (1,-.5) .. (1,0) .. controls (1,.5) and (0, .5) .. (0,.75);
		\draw  (1,-.75) node[below] {\small $j$} .. controls (1,-.5) and (0,-.5) .. (0,0) .. controls (0,.5) and (1, .5) .. (1,.75);
		\fdot{}{.4,0.1}; 
		\fdot{}{.45,.95}; 
		\node at (1.3,0) {$\dotsm$};
	} = - \quad 
	\tikz[very thick,baseline={([yshift=1ex]current bounding box.center)}]{
		\draw  (0,-.75) node[below] {\small $i$} .. controls (0,-.5) and (1,-.5) .. (1,0) .. controls (1,.5) and (0, .5) .. (0,.75);
		\draw  (1,-.75) node[below] {\small $j$} .. controls (1,-.5) and (0,-.5) .. (0,0) .. controls (0,.5) and (1, .5) .. (1,.75);
		\fdot{}{.4,0.1}; 
		\fdot{}{.45,-.75}; 	 
		\node at (1.3,0) {$\dotsm$};
	} .
	\end{equation}

\begin{df}\label{DefgradhR} 
	We consider $\hR_\bfn$ as a $\bbZ$-graded algebra by putting
	$$
	\op{deg}(1_\ui)=0,\; \op{deg}(X_r)=2,\; \op{deg}(\tau_r1_\ui)=2h_{i_r,i_{r+1}}-2\delta_{i_r,i_{r+1}},\; \op{deg}(\Omega1_\ui)=2(n_{i_1}-1).
	$$ 
\end{df}
\begin{df}
We refer to $\hR_\bfn$, viewed as graded algebra via \Cref{DefgradhR}, as the \emph{Naisse--Vaz-algebra}  (associated with $\bfn$). 
\end{df}

Our goal is to provide a geometric construction of the Naisse--Vaz algebras.


\section{Convolution algebras}
\label{subs:conv}
We pass now to geometry and recall here the fundamental notion of a convolution algebra from \cite[\S 2.7]{CG97}. We do an equivariant version, \cite{EdidinGraham}, \cite{Graham}.
	
Let $\abY$ be a smooth complex manifold and let $\abX$ be a complex algebraic variety (possibly singular). Let $L$ be a complex algebraic group acting smoothly on $\abY$ and algebraically on $\abX$. 
Denote by $H_*^L(\abX)$ the $L$-equivariant Borel--Moore homology of $\abX$. By the assumption on $\abY$, its (equivariant) Borel--Moore homology $H_*^L(\abY)$ can be identified with its (equivariant) cohomology  $H^*_L(\abY)$. Let $\tR_L=H^*_L(\point)$ be the $L$-equivariant cohomology of a point.

\begin{df}
Assume that there is an $L$-invariant proper map $p\colon \abY\to \abX$. Set
\begin{equation}\label{eq:Z-abstract-fibre}
Z=\abY \times_X \abY.
\end{equation}
Let $p_1,p_2\colon \abY\times \abY\to \abY$ be the projections to the first and the second component respectively. Consider the maps $p_{1,2}$, $p_{1,3}$ and $p_{2,3}$ from $\abY\times \abY\times \abY$ to $\abY\times \abY$ forgetting respectively the third, the second and the first factor. 
\end{df}
\begin{lem}\label{lem:conv}Assume the setup of the opening paragraph of this section.
\begin{enumerate}
\item There is an algebra structure on $H_*^L(Z)$ with multiplication given by 
\begin{eqnarray}\label{conv}
\qquad a \star b&=&(p_{1,3})_*(p_{1,2}^*(a)\cap p_{2,3}^*(b)), \qquad \mbox{for  }a,b\in H_*^L(Z).\qquad
\end{eqnarray}
The unit is the fundamental class of the diagonal (isomorphic to $\abY$) in $Z$.
\item The algebra $H_*^L(Z)$ acts on $H_*^L(\abY)$ by 
\begin{eqnarray}\label{conv2}
\qquad a \star b&=&(p_{1})_*(a\cap p_{2}^*(b)),\qquad \mbox{for } a\in H_*^L(Z),~b\in H_*^L(\abY).\qquad
\end{eqnarray}
\end{enumerate}
\end{lem}
\begin{ex}
	
One might pick $\abX=\point$ and $\abY$ any complex projective variety. Then $Z=\abY\times \abY$ and  $H_*^L(Z)$ can be identified with $\End_{\tR_L}(H^*_L(\abY))$. The $H_*^L(Z)$-action \eqref{conv} on $H_*^L(\abY)\cong H^*_L(\abY)$ is the natural action of $\End_{\tR_L}(H^*_L(\abY))$ on $H^*_L(\abY)$.
\end{ex}

\subsection{Torus fixed points localisation}
Assume that $G$ is a reductive complex algebraic group and $\rmT\subset G$ is a maximal torus. 
Let $\abX$ be a $G$-variety.
Then by \cite[\S III, Prop. 1]{Hsiang}, there is an action of the Weyl group $W$ on $H_\rmT^*(\abX)$ and we have\footnote{Here, the assumption that the characteristic of $\Bbbk$ is zero is crucial!}
  $H^G_*(\abX)\simeq (H^\rmT_*(\abX))^W$. This allows us to include $H^G_*(\abX)$ into $H^\rmT_*(\abX)$. 
  
Now, assume that $\abX$ has an action of $\rmT$ (not obligatory of $G$). We consider $\rmT$-equivariant Borel--Moore homology, see e.g. \cite{AndersonFulton} for details.

View $H^\rmT_*(\abX)$ as $\tR_\rmT$-module and let $H^\rmT_*(\abX)_{\rm loc}={\rm Frac}(\tR_\rmT)\otimes_{\tR_\rmT}H^\rmT_*(\abX)$ be its localisation. Assume that the set $\abX^\rmT$ of $\rmT$-fixed points in $\abX$ is non-empty and finite. By the localisation theorem, the push-forward of the natural inclusion $\abX^\rmT\subset \abX$ induces an isomorphism $H^\rmT_*(\abX)_{\rm loc}\cong H^\rmT_*(\abX^\rmT)_{\rm loc}$. 

\begin{df}  \label{locmap} 
We call the  obvious map $\op{loc}\colon H^\rmT_*(\abX)\hookrightarrow H^\rmT_*(\abX)_{\rm loc}\simeq H^\rmT_*(\abX^\rmT)_{\rm loc}$ the \emph{localisation map}. If  $H^\rmT_*(\abX)$ is free over $\tR_\rmT$, then this map is in fact injective. 
\end{df}
The localisation map allows us to do computation using $\rmT$-fixed points. For $x\in \abX ^\rmT$, let $[x]\in H_*^\rmT(\abX)$ be the push-forward of the fundamental class of the point $x$ to $\abX$. Then $\{[x],\,x\in\abX ^\rmT\}$ is a basis of $H^\rmT_*(\abX)_{\rm loc}$ over ${\rm Frac}(\tR_\rmT)$. We will often do homology calculations using this basis.
Now, assume that $\abX$ and $\abY$ are $G$-varieties and that $\abY$ is smooth. As in \eqref{eq:Z-abstract-fibre}, set $Z=\abY\times_{\abX}\abY$. Assuming that $\abY^\rmT$ is non-empty and finite, we have $Z^\rmT\subset \abY^\rmT\times \abY^\rmT$, and we can label $\rmT$-fixed points of $Z$ by pairs $(x,y)$ with $x,y\in\abY^\rmT$. 

For a finite-dimensional representation $M$ of $\rmT$ we denote by $\op{eu}(M)\in \tR_\rmT$ the character of $\Lambda^{\dim M}M$. For $x\in\abY^\rmT$ we set $\op{eu}(\abY,x)=\op{eu}(T_x\abY)$.

The following lemma is crucial for doing computation on the $\rmT$-fixed points, see  e.g. \cite[Prop. 2.10, 2.11]{MakMin} and \cite[Lemma 2.19 $(b)$]{VV}.

\begin{lem}\label{lem:comp-in-loc}
The localisation maps satisfy the following for $x,y,z\in \abY^\rmT$.
\begin{enumerate} 
	\item \label{1} The inclusion $H^\rmT_*(Z)\hookrightarrow H^\rmT_*(Z^\rmT)_{\rm loc}$ is an algebra homomorphism for the product $[(x,y)]\star [(y',z)]=\delta_{y,y'}\op{eu}(\abY,y)[(x,z)]$ on $H_*^\rmT(Z^\rmT)_{\rm loc}$. 	

\item \label{2} With the action map from \Cref{lem:conv} the following diagram commutes	
\[
	\begin{CD}
	H_*^\rmT(Z)\otimes H_*^\rmT(\abY) @>>> H_*^\rmT(\abY)\\
	@VVV @VVV\\
	H_*^\rmT(Z^\rmT)_{\rm loc}\otimes H_*^\rmT(\abY^\rmT)_{\rm loc} @>>> H_*^\rmT(\abY^\rmT)_{\rm loc},
	\end{CD}
	\]
	where $H_*^\rmT(Z^\rmT))_{\rm loc}$ acts on  $H_*^\rmT(\abY^\rmT)_{\rm loc}$ by $[(x,y)]\star [y']=\delta_{y,y'}\op{eu}(\abY,y)[x]$. 
	\end{enumerate}
In particular, the $H_*^{G}(Z)$-action on $H_*^{G}(\abY)$ is faithful. 
\end{lem}
\begin{rk}
\label{rk:faith-from-strat}
\Cref{lem:comp-in-loc} holds in more general settings, where for instance the space is not algebraic, as long as 
we have $H_*^{\rmT}(\abY) \subset H_*^{\rmT}(\abY)_{\rm loc}\cong H_*^{\rmT}(\abY ^\rmT)_{\rm loc}$  
and $H_*^{\rmT}(Z) \subset H_*^{\rmT}(Z)_{\rm loc}\cong H_*^{\rmT}(Z ^\rmT)_{\rm loc}$.
 In practice,  this can often be achieved by constructing nice pavings of $\abY$ and $Z$. 
\end{rk}

Assuming that $\bfZ\subset Z$ is closed and that
\begin{eqnarray*}
\bfZ\circ\bfZ:=\{(y_1,y_3)\in \abY\times\abY\mid (y_1,y_2), (y_2,y_3)\in \bfZ  \mbox{ for some } y_2\in \abY  \}\subset \bfZ,
\end{eqnarray*}
the following is a consequence of the definitions.
\begin{lem}\label{convcirc}
The rule \eqref{conv} turns $H_*^{G}(\bfZ)$ into an algebra, and the push-forward of the inclusion $\bfZ\subset Z$ induces an algebra homomorphism $H_*^{G}(\bfZ)\to H_*^{G}(Z)$.
\end{lem}

\begin{rk}
	The push-forward $H_*^{G}(\bfZ)\to H_*^{G}(Z)$ may not be injective in general. Thanks to the following commutative diagram this map is however injective if the localisation map $H_*^{G}(\bfZ)\to H_*^{T}(\bfZ ^\rmT)_{\rm loc}$ is injective, see \Cref{locmap} (to ensure that, it suffices for instance to show the existence of a nice paving of $\bfZ$). 
	$$
	\begin{CD}
	H_*^{G}(\bfZ)@>>> H_*^{G}(Z)\\
	@VVV                   @VVV\\
	H_*^{T}(\bfZ ^\rmT)_{\rm loc}@>>> H_*^{T}(Z ^\rmT)_{\rm loc}.
	\end{CD}
	$$
\end{rk}

\addtocontents{toc}{{\textbf{Part I: The case of $\mathfrak{sl}_2$} (the quiver with one vertex and no arrows)}}
\section{Nil-Hecke algebra and its extended versions}
\label{sec:hNH}
We start with the first part of the paper, namely the treatment of the $\mathfrak{sl}_2$ case, i.e., we assume that our quiver $\Gamma$ has only one vertex and no arrows. The dimension vector $\bfn$ is then just a positive integer $n$. The KLR algebra $R_\bfn$ is in this case the nil-Hecke algebra $\NH_n$. We briefly recall this algebra and its extension $\hNH_n$.

We will define several  $\Bbbk$-algebras via some faithful representations. Equivalent definitions in terms of generators and relations can by found in \Cref{app:gen-rel}.

\subsection{Nil-Hecke algebra}
\label{sec:Pol}
First, we recall some generalities about the nil-Hecke algebras, see \cite[\S 4]{KostKum} and \cite[\S 2.2]{KL} for more details.

Let $n$ be a positive integer. Set $\Pol_n=\Bbbk[X_1,\ldots,X_n]$. The symmetric group $\frakS_n$ acts (from the left) on $\Pol_n$ such that $s_r$ exchanges $X_r$ with $X_{r+1}$. 

\begin{df}\label{actionNH}
The \emph{nil-Hecke algebra} $\NH_n$ is the subalgebra
of $\End(\Pol_n)$ generated by the following endomorphism of  $\Pol_n$ for $i\in[1;n]$ and $r\in[1;n-1]$: 
\begin{itemize}
\item the element $X_i\in \NH_n$ acting as multiplication with $X_i\in \Pol_n$ on  $\Pol_n$,
\item the element $T_r\in \NH_n$ acting by the Demazure operator $\partial_r=\frac{1-s_r}{X_r-X_{r+1}}$. 
\end{itemize}
\end{df}
We view $\Pol_n$ as a graded algebra by putting the generators in degree $2$. Note that this induces a grading on $\NH_n$  with $X_i$ of degree $2$ and  $T_r$ of degree $-2$.

\subsection{The extended nil-Hecke algebra $\hNH_n$}
\label{subs:hNH}

Set $\hPol_n=\Pol_n\otimes \mywedge_n$, where $\mywedge_n=\mywedge^\bullet(\omega_1,\ldots,\omega_n)$ is the exterior algebra in $n$ generators.
The $\frakS_n$-action on $\Pol_n$ can be extended to an $\frakS_n$-action on $\hPol_n$ by algebra automorphisms such that
\begin{eqnarray}\label{Esr}
s_r(\omega_i)&=&
\begin{cases}
\omega_i& \mbox{ if }i\ne r,\\
\omega_r+(X_r-X_{r+1})\omega_{r+1} &\mbox{ if }i=r
\end{cases}
\end{eqnarray}
and we obtain well-defined \emph{Demazure operators} $\partial_r=\frac{1-s_r}{X_r-X_{r+1}}$ on $\hPol_n$. Note that $\partial_r(\omega_r)=-\omega_{r+1}$, and $\partial_r(\omega_i)=0$ if $i\ne r$, and thus the $\NH_n$-action on $\Pol_n$ extends to $\hPol_n$. (This can be verified by checking the relations in \Cref{lem:gen-rel-NH}.)

\begin{nota}\label{notnew}
Consider the set $\Lambda(n)=\cup_{k=0}^n\Lambda_k(n)$ with
$$
\Lambda_k(n)=\{\lambda=(\lambda_1,\lambda_2,\ldots,\lambda_n)\in \{0,1\}^n\mid\sum_{r=1}^n\lambda_r=k\}. 
$$
For each $\lambda\in\Lambda(n)$ we set $\omega_\lambda=\omega_{i_k}\wedge \omega_{i_{k-1}}\wedge\ldots\wedge\omega_{i_1}$ where $i_1<\ldots<i_k$ are the numbers of the positions $i$ such that $\lambda_i=1$. 
We denote $|\lambda|=\sum_{r=1}^n\lambda_r$ and often identify $\lambda$ with the set $\{i_1,i_2,\ldots,i_k\}$ of cardinality $k$. 
\end{nota}
\begin{rk}
The inversion in the indices of $\omega_\lambda$ may look strange. We chose this convention to avoid signs in the geometric construction of the creation operators in \Cref{lem:creation+n} and \Cref{lem:creation+r}. The choice creates signs for the annihilation operators, but they are less important for this paper. Viewing the elements  of $\Lambda(n)$ as subsets $\lambda$ of $\{1,2,\ldots,n\}$  is useful, because it allows to use set-theoretical operations auch as unions, intersections, etc. as we tacitly did already in \Cref{notnew}. For instance, $|\lambda|$ is the cardinality of $\lambda$. 
\end{rk}

\begin{df}\label{df:hNH}
The \emph{extended nil-Hecke algebra $\hNH_n$} is the subalgebra of $\End(\hPol_n)$ generated by the following endomorphisms of $\hPol_n$ for $i\in[1;n]$ and $r\in[1;n-1]$: 

\begin{itemize}
\item the element $X_i\in \hNH_n$ acting as left multiplication with $X_i\in \hPol_n$,
\item the element $\omega_i\in \hNH_n$ acting as left multiplication with $\omega_i\in \hPol_n$,
\item the element $T_r\in \hNH_n$ acting by the Demazure operator $\partial_r$ from \eqref{Esr}. 
\end{itemize}
\end{df}

\begin{rk}
The extended nil-Hecke algebra is  the algebra from \Cref{df:ext-KLR-alg} in case the quiver has one vertex and no arrows. 
In this case $\omega_r\in\hNH_n$ equals $\Omega_r\:=\Omega 1_\ui\in \hR_\bfn= \hR_n$. For general quivers the two elements differ and  $\Omega_r1_\ui$ acts on the polynomial representation by a more complicated operator, see \Cref{Rkonembedding}.
\end{rk}

\begin{rk}
	\label{rk:gr-hNH}
	Viewing the exterior algebra as a graded algebra with $\omega_i$ in degree $2(n-i)$, the algebra $\hNH_n$ inherits a $\bbZ$-grading:
	$$
	\deg(X_i)=2,\quad \deg(\omega_i)=2(n-i),\quad \deg(T_r)=-2.
	$$ 
We pick $\deg(\omega_i)$ different from \cite{NaisseVaz} to fit it better with the geometry. This is just an unimportant renormalisation which could have been chosen already in \cite{NaisseVaz}.  
\end{rk}

Elements from $\hNH_n$ are in fact already determined by their action on $\Pol_n$:

\begin{lem}
\label{pol-rep-plus-hNH}
Given a nonzero element $h$ of the extended nil-Hecke algebra $\hNH_n$, there exists a polynomial $P\in \Pol_n\subset\hPol_n$ such that $hP\ne 0$.
\end{lem}

The subalgebras $\NH_n$ and $\mywedge_n$ of $\hNH_n$ provide a PBW-type decomposition:
\begin{lem}
\label{coro:basis-hNH}
There is an isomorphism $\hNH_n\cong \NH_n\otimes \mywedge_n$  of vector spaces.
\end{lem}
This isomorphism is induced from the multiplication. A proof of \Cref{pol-rep-plus-hNH} and \Cref{coro:basis-hNH} is given in \Cref{app:gen-rel} where we also give a presentation of $\hNH_n$.

\subsection{The doubly extended nil-Hecke algebra $\hhNH_n$}
We extend the action of the algebra $\hNH_n$ on $\hPol_n$ to an action of a bigger algebra $\hhNH_n$ on the same vector space. 
We do this by viewing multiplication on the left by $\omega_i$ as \emph{creation operator} $\omega^+_i\in\End(\hPol_n)$
and additionally define the \emph{annihilation operator} $\omega^-_i$  as 
\begin{equation*}
\omega^-_i(Q)=0\quad\text{and}\quad  \omega^-_i(\omega_i Q)=Q
\end{equation*}
for any $Q=P\omega_{i_1}\ldots\omega_{i_k}$ with $P\in \Pol_n$ and distinct $i_1,\ldots,i_k$ not equal to $i$. 
\begin{df}
\label{df:hhNH}
The \emph{doubly extended nil-Hecke algebra} $\hhNH_n$ is the subalgebra of $\End(\hPol_n)$, generated by the following endomorphisms:
\begin{itemize}
\item the element $X_i\in \hhNH_n$ acting as left multiplication by $X_i\in \hPol_n$,
\item the element $\omega^+_i\in \hhNH_n$ acting by the creation operator $\omega^+_i$,
\item the element $\omega^-_i\in \hhNH_n$ acting by the annihilation operator $\omega^-_i$,
\item the element $T_r\in \hhNH_n$ acting by the Demazure operator $\partial_r$. 
\end{itemize}
\end{df}

Note that $\Pol_n\subset \hhNH_n$ commutes with $\omega^+_i$ and with $\omega^-_i$.

\begin{nota}
Similarly to the definition of the monomial $\omega_\lambda\in\hNH_n$ in \Cref{notnew}, we can define monomials $\omega^\pm_\lambda\in\hhNH_n$ in the following way:
$$
\omega^+_\lambda=\omega^+_{i_k}\wedge \omega^+_{i_{k-1}}\wedge\ldots\wedge\omega^+_{i_1},\qquad \omega^-_\lambda=\omega^-_{i_1}\wedge \omega^-_{i_{2}}\wedge\ldots\wedge\omega^-_{i_k}.
$$
The opposite conventions ensure $\omega^+_\lambda(1)=\omega_\lambda$ and $\omega^-_\lambda(\omega_\lambda)=1$ without signs.
\end{nota}

The following justifies the name {\it doubly extended nil-Hecke algebra}.
\begin{prop}
\label{lem:basis-hhNH}
The following set is a basis of $\hhNH_n$:
$$
\{T_wX_1^{a_1}\ldots X_n^{a_n}\omega^+_\lambda\omega^-_\mu ;~w\in\frakS_n,a_i\in \bbZ_{\geqslant 0}, \lambda,\mu\in \Lambda(n) \}.
$$
\end{prop}
A proof of this result is given in \Cref{app:gen-rel}. It directly implies the following.
\begin{coro}
\label{coro:decomp-hhNH}
Multiplication induces isomorphisms of vector spaces 
\begin{center}
$
\hhNH_n\cong \hNH_n\otimes \mywedge_n \cong \NH_n\otimes \mywedge_n\otimes \mywedge_n.
$
\end{center}
\end{coro}

\begin{nota}
The algebra $\hhNH_n$ contains idempotents $1_k$, $k\in[0;n]$ projecting to the components of $\hPol_n$ given by the exterior algebra degree $k$, in formulas
\begin{eqnarray*}
1_0=\omega^-_1\ldots\omega^-_n\omega^+_n\ldots\omega^+_1,
&&
1_k=\sum_{\lambda\in \Lambda_k(n)}\omega^+_\lambda 1_0 \omega^-_\lambda.
\end{eqnarray*}
\end{nota}

\subsection{Application: A dg-enrichment of $\hNH_n$ via the superalgebra $\hhNH_n$}
\label{subs:DG-on-hNH}
Throughout this section we fix a nonnegative integer $N$. 

\begin{nota}\label{DefbfnN}
Let $\bfd_N\colon \hNH_n\to \hNH_n$ be the unique  linear operator which satisfies $\bfd_N(X_i)=0$, $\bfd_N(T_r)=0$, $\bfd_N(\omega_1)=X_1^N$ and the \emph{graded Leibniz rule} $\bfd_N(PQ)=\bfd_N(P)Q+(-1)^{\deg P}P\bfd_N(Q)$ for  homogeneous $P,Q\in \NH_n$, see \Cref{lem:gen-rel-hNH-omega1}.
\end{nota}
If we equip (for a moment) the algebra $\hNH_n$ with the  $\bbZ$-grading  given by $\deg(X_i)=\deg (T_r)=0$ and $\deg(\omega_i)=1$, then $\bfd_N$ turns $\hNH_n$ into a dg-algebra $(\hNH_n, \bfd_N)$ with differential $\bfd_N$ of degree $-1$. By \cite[Prop. 4.14]{NaisseVaz}, its homology is particularly nice, recovering an important algebra:
\begin{prop}
	\label{prop:resol-cycl-NH}
	Let $r\in \bbZ$. 
	Then   
	$$
	H_r(\hNH_n,\bfd_N)=
	\begin{cases}
	  \NH_n^N & \mbox{ if }r= 0,\\
	  0 & \mbox{ if }r\ne 0.
	\end{cases}
	$$
where $\NH^N_n$ denotes the \emph{cyclotomic nil-Hecke algebra} corresponding to $N$, that is the quotient of $\NH_n$ by the ideal generated by $X_1^N$.
\end{prop} 

\begin{rk}
If $n>N$, then the algebra $\NH_n^N$ is zero, see \cite[Thm. 2.34]{HuLiang}. 
\end{rk}

We get a new point of view on the differential  $\bfd_N$ when we realize $\hNH_n$ as a subalgebra of $\hhNH_n$, but with $\hhNH_n$ viewed as a \emph{super}algebra with even generators $X_i,T_r$ and odd generators $\omega_i^{\pm}$ (this is well-defined by \S\ref{sec:KLR} or more explicitly by \Cref{lem:gen-rel-hhNH}). Let $[\bullet,\bullet]$ denote the supercommutator in $\hhNH_n$.
\begin{df}\label{Ps}
Define inductively polynomials $P_1,P_2,\ldots,P_n\in \Pol_n$ by $P_1=X_1^N$ and $P_{r+1}=-\partial_r(P_r)$ for $r\in [1;n-1]$, and set $\mathbb{d}_N=\sum_{i=1}^n P_i\omega_i^-\in\hhNH_n$.
\end{df}
\begin{prop}\label{lem:dg}
Let $h\in \hNH_n$. Then $\bfd_N(h)=[\mathbb{d}_N,h]$ as elements in  $\hhNH_n$.	
\end{prop}
\begin{proof}
Since both, $\bfd_N$ and the super commutator $[\mathbb{d}_N,\bullet]$, satisfy the graded Leibniz rule (by definition respectively thanks to the super Jacobi identity) it is enough to check the statement on the generators $h\in\{X_i,T_r,\omega_i\}$ of $\hNH_n$.  

\noindent\emph{Case $h=X_i$}: Since the $\omega_r^-$ commute with polynomials,  
$[\mathbb{d}_N,X_i]=0=\bfd_N(X_i)$. \hfill\\
\emph{Case $h=T_r$}: Note that $T_r$ commutes with each summand in $\mathbb{d}_N$ except of $P_r\omega_r^-$ and $P_{r+1}\omega_{r+1}^-$, and also with $\omega_r^-$ and $P_{r+1}$. Thus $[\mathbb{d}_N,T_r]=[P_r,T_r]\omega_r^-+P_{r+1}[\omega_{r+1}^-,T_r]$. For $P\in \Pol_n$, we have $[P,T_r]=\partial_r(P)((X_{r+1}-X_r)T_r-1)$. We get in particular $[P,T_r]\stackrel{\ref{Ps}}=P_{r+1}((X_{r}-X_{r+1})T_r+1)$ and 
$
[\omega_{r+1}^-,T_r]\stackrel{\ref{lem:gen-rel-hhNH}}=[X_r\omega_r^-,T_r]=[X_r,T_r]\omega_r^-=((X_{r}-X_{r+1})T_r-1)\omega_r^-.
$
Thus, $[\mathbb{d}_N,T_r]=0=\bfd_N(T_r).$\hfill\\
\emph{Case $h=\omega_1$}: Only the first summand $X_1^N\omega_1^-$ in $\mathbb{d}_N$ does not commute with $\omega_1^+$. But then we have
$
[\mathbb{d}_N,\omega^+_1]=[X_1^N\omega_1^-,\omega^+_1]=X_1^N[\omega_1^-,\omega^+_1]=X_1^N=\bfd_N(\omega_1).
$
\end{proof}
\begin{nota}\label{defdN}
Denote by $d_N\colon \hPol_n\to \hPol_n$ the linear map given by the action of $\mathbb{d}_N\in\hhNH_n$. It is characterized  by obeying the graded Leibniz rule, commuting with the action of $\NH_n$ and satisfying $\bfd_N(\omega_1)=X_1^N$. 
\end{nota}
\Cref{lem:dg} provides a dg-enrichment of the $\hNH_n$-module $\hPol_n$:
\begin{coro} \label{dg} 
$(\hNH_n,\bfd_N)$ is a dg-algebra and $(\hPol_n, d_N)$ a dg-module.
\end{coro}

\section{Construction of the Grassmannian quiver Hecke algebra }
\label{sec:sl2-geom}
In this section we discuss geometric constructions of the nil-Hecke algebra and their (doubly) extended versions. For $\NH_n$ this is well-known and will be recalled first. For the extended versions, the constructions are new and provide in particular geometric versions of creation and annihilation operators.
\subsection{The setup and important varieties}
\label{setupsl2}
Let $V$ be an $n$-dimensional complex vector space. Let $\rmG=\mathrm{GL}(V)$ with a choice $\rmT\subset \rmB\subset\rmG$ of a maximal torus and a Borel subgroup. Consider the variety $\calF=\rmG/\rmB$ of full flags in $V$. We will usually denote an element of $\calF$ by $\bV=\bV^\bullet=(0\subset \bV^1\subset \bV^2\subset\ldots\subset \bV^{n-1}\subset V)$. 

There is an isomorphism of algebras $H^*_\rmG(\calF)\cong\Bbbk[X_1,\ldots,X_n]$, where $X_k$ is the equivariant Chern class of the line bundle $\bV^k/\bV^{k-1}$. We also have an isomorphism of algebras $\tR_\rmT=H^*_\rmT(\point)\cong \Bbbk[\ccT_1,\ldots,\ccT_n]$, where $\ccT_1,\ldots,\ccT_n$ are the basic characters of $\rmT$.
We consider the usual $\frakS_n$-action, denoted by $w(P)(\ccT_1,\ldots,\ccT_n)=P(\ccT_{w(1)},\ldots,\ccT_{w(n)})$ for $w\in\frakS_n$,  on $\tR_\rmT=\Bbbk[\ccT_1,\ldots,\ccT_n]$.

For $k\in [0;n]$ let $\calG_k=\op{Gr}_k(V)$ be the Grassmannian variety of $k$-dimensional vector subspaces in $V$ 
and let $\ccY_k=\calF\times \calG_k$.  

\begin{df}\label{extendedflagvariety}
We set $\calG=\coprod_{k=0}^n \calG_k$ and call it the \emph{variety of Grassmannians}. We set  $\ccY=\calF\times \calG=\coprod_{k=0}^n\ccY_k$ and call it the \emph{(Grassmannian) extended flag variety}. These varieties come equipped with the obvious (diagonal) $\rmG$-action. \end{df}
The following geometric construction of the nil-Hecke algebra is well-known connecting  \Cref{lem:conv}  with \Cref{actionNH}.
\begin{prop}
\label{prop:geom-NH}
There is an isomorphism of algebras $H_*^\rmG(\calF\times\calF)\cong\NH_n$ identifying the $H_*^\rmG(\calF\times\calF)$-action with the $\NH_n$-action on $H_*^\rmG(\calF)=\Pol_n$ from \S\ref{sec:Pol}.
\end{prop}

We consider the \emph{Steinberg extended flag variety}  $\ccZ=\ccY\times \ccY$, $\ccZ_{k_1,k_2}=\ccY_{k_1}\times \ccY_{k_2}$. We would like to describe the convolution algebra $H_*^\rmG(\ccZ)$ with its action on $H_*^\rmG(\ccY)$.

\begin{nota} Whenever we write $\bV$, $\tV$ or $W, \tW$ without specification we mean a point in $\cF$ respectively in $\calG$. A point in $\ccZ$ is denoted as $(\bV,\tV,W,\tW)$. Hereby $\bV$ always refers to the first and $\tV$ to the second flag, similar for $W$ and $\tW.$  
\end{nota}

\subsection{The representation $H_*^\rmG(\ccY)$}
\label{subs:rep-H(FG)}

\begin{lem}
\label{lem:isom-FG/G-G/T}
For any $\rmG$-variety $X$, there is a canonical isomorphism
\begin{equation}\label{iso-FG/G-G/T}
H^*_\rmG(\calF\times X)\cong H^*_\rmT(X).
\end{equation}
\end{lem}
\begin{rk}
The left hand side of \eqref{iso-FG/G-G/T} has an action of $H^*_\rmG(\calF)=\Bbbk[X_1,\ldots,X_n]$ and the right have side has an action of $\tR_\rmT=H^*_\rmT(\point)=\Bbbk[\ccT_1,\ldots,\ccT_n]$. The isomorphism identifies the multiplications by $X_r$ on the left and by $\ccT_r$ on the right.
\end{rk}
\begin{proof}
The isomorphism of varieties
\begin{equation}
\label{eq:isom-G/B*Gr}
\rmG\times_\rmB X\cong \rmG/\rmB\times X,\qquad (g,x)\mapsto (gB, gx).
\end{equation}
induces 
$
H^*_\rmG(\calF\times X)\cong H^*_\rmG(\rmG\times_\rmB X)=H^*_\rmB(X)=H^*_\rmT(X).
$
\end{proof}

\begin{df}
For $\lambda\in \Lambda_k(n)$, see \Cref{notnew}, let 
$$
C_\lambda=\{(\bV,W)\in \ccY_k\mid \dim((\bV^r\cap W)/(\bV^{r-1}\cap W))=\lambda_r,~\forall r\in[1;n]\}.
$$
Denote by $\tS_\lambda$ the class $[\overline{C_\lambda}]\in H_*^G(\ccY_k)$, the \emph{equivariant Schubert class}\footnote{Note that $\tS_\lambda$  corresponds via \eqref{iso-FG/G-G/T} to the usual,  see e.g. \cite{Graham},  equivariant Schubert class on $\calG$.} for $\lambda$.
\end{df}

Setting $X=\calG$ in \eqref{eq:isom-G/B*Gr} gives the following (see also \cite[Prop. 4.12]{GKS}):

\begin{coro}\label{Schurandomega}
There is an isomorphism $H_*^\rmG(\ccY)\cong\Bbbk[X_1,\ldots,X_n]\otimes \mywedge^\bullet(\omega_1,\ldots,\omega_n)$ 
of $H_{\rmG}^*(\calF)=\Bbbk[X_1,\ldots,X_n]$-modules sending $\tS_\lambda$ to $\omega_\lambda$ for each $\lambda\in\Lambda_k(n)$.
\end{coro}
From now on we identify $H_*^\rmG(\ccY)$ with $\Bbbk[X_1,\ldots,X_n]\otimes \mywedge^\bullet(\omega_1,\ldots,\omega_n)$ as a module over $H_{\rmG}^*(\calF)=\Bbbk[X_1,\ldots,X_n]$ such that the element $\tS_\lambda$ corresponds to $\omega_\lambda$.

\subsection{Euler classes in (extended) flag varieties}
We will now do explicit calculations with fixed points and classes in ${\rm Frac}(\tR_\rmT)=\Bbbk(\ccT_1,\ldots, \ccT_n)$, the fraction field of 
 $\tR_\rmT=\Bbbk[\ccT_1,\ldots, \ccT_n]$, in particular with Euler classes in extended flag varieties.
 
 For this fix a basis $e_1, e_2,\ldots,e_n$ in $V$ such that $\rmB$ and $\rmT$ are the upper triangular respectively the diagonal matrices. Let $\bU=(0\subset \bU^1\subset \bU^2\subset\cdots\subset \bU^{n-1}\subset V)$ be the full flag in $V$ whose stabilizer is $\rmB$, i.e., $\bU^{j}=\langle e_1,e_2,\ldots,e_{j}\rangle$.

Each permutation $w\in \frakS_n$ can be seen as permutation matrix $(a_{ij})\in \rmG$, where $(a_{ij})=\delta_{i,w(j)}$. Then $\calF^\rmT=\{f_w\mid w\in\frakS_n\}$, where $f_w=w(\bU)$. 

Let $\rmB_w=w\rmB w^{-1}$ be the stabilizer of the flag $f_w\in\calF$ and $\frakb_w$ its Lie algebra. Let $\frakn_w$ be the nilpotent radical of $\frakb_w$. Set also $\frakn_{w_1,w_2}=\frakn_{w_1}\cap \frakn_{w_2}$ and $\frakm_{w_1,w_2}=\frakn_{w_1}/\frakn_{w_1,w_2}$. Let $\rmP_{w_1,w_2}$ be the minimal parabolic subgroup of $\rmG$ containing both $\rmB_{w_1}$ and $\rmB_{w_2}$, and let $\frakp_{w_1,w_2}$ be its Lie algebra. We record a direct consequence:
\begin{lem} We have
$
\frakp_{w,ws_r}/\frakb_w\cong \frakm_{ws_r,w}$ and $\frakm_{w,ws_r}^*\cong \frakm_{ws_r,w}
$
for any $w\in \frakS_n$, $r\in [1;n-1]$. Moreover, $\op{eu}(\frakm_{w,ws_r})=\ccT_{w(r)}-\ccT_{w(r+1)}$.
\end{lem}

Via the inclusion $H^*_\rmG(\calF)\subset H^*_\rmT(\calF)_{\loc}$ from \S\ref{subs:conv}, any $P\in \Bbbk[X_1,\ldots,X_n]=H^*_\rmG(\calF)$ can be viewed as an element in $H^*_\rmT(\calF)_{\loc}$ and then be written as 
\begin{equation}\label{Pinfpts}
P=\sum_{w\in\WW_n}P_w\tA^{-1}_w[f_w],\quad \text{where}\quad \tA_w=\op{eu}(\calF,f_w).
\end{equation} 

\begin{lem}
For any $w\in\frakS_n$, the Euler class $\tA_w=\op{eu}(\calF,f_w)$ equals
\begin{equation}\label{Aw}
\tA_w=\prod_{1\leqslant i<j\leqslant n}(\ccT_{w(j)}-\ccT_{w(i)}).
\end{equation}
\end{lem}
\begin{proof}
Via the isomorphism $\rmG/\rmB_w\cong \calF$, $g\rmB_w\mapsto gf_w$, the tangent space $T_{f_w}(\calF)$ is isomorphic to $\frakg/\frakb_w\cong (\frakn_w)^*$ and clearly $\op{eu}(\frakn_w)=\prod_{1\leqslant i<j\leqslant n}(\ccT_w(i)-\ccT_w(j))$.
\end{proof}

The $\rmT$-fixed points in $\calF\times\calF$ are the pairs $f_{w_1,w_2}:=(f_{w_1},f_{w_2})$ of fixed points in $\calF$. For $w\in \frakS_n$ let $O^w$ be the diagonal $\rmG$-orbit of $f_{\Id,w}$ in $\calF\times \calF$ and $\overline{O^w}$ its closure.

\begin{lem}\label{eunm}
For any simple transposition $s\in \frakS_n$ we have
\begin{equation*}
\op{eu}(\overline{O^s},f_{w,w})=\op{eu}((\frakn_w)^*\oplus \frakm_{ws,w})\quad\text{and}\quad
\op{eu}(\overline{O^s},f_{w,ws})=\op{eu}((\frakn_w)^*\oplus \frakm_{w,ws}).
\end{equation*}
\end{lem}
\begin{proof}
Consider the isomorphism
$
\rmG\times_{\rmB_w} \rmP_{w,ws}/\rmB_
w\to \overline{O^s}, (g,p)\mapsto (gf_w,gpf_{w}).
$
It realizes $\overline{O^s}$ as a fibration over $\rmG/\rmB_w$ with fibre $\rmP_{w,ws}/\rmB$. The tangent spaces to $\rmP_{w,ws}/\rmB_w$ at $f_w\rmB_w$ and $f_{ws}\rmB_w$ are isomorphic to $\frakm_{ws,w}$ respectively $\frakm_{w,ws}$.
\end{proof}
\label{subs:more-Euler}
The $\rmT$-fixed points in $\calG_k$ are canonically parametrized by $\Lambda_k(n)$. Namely, the fixed point corresponding to $\lambda\in\Lambda_k(n)$ is the vector subspace $g_\lambda$ of $V$ spanned by the basis vectors  $e_i$ such that $\lambda_i=1$. We obtain directly:
\begin{lem}
The Euler class $\tA_\lambda=\op{eu}(\calG_k,g_\lambda)$ is explicitly 
\begin{equation}\label{Alambda}
\tA_\lambda=\prod_{i\in\lambda,j\not\in\lambda}(\ccT_j-\ccT_i).
\end{equation}
\end{lem}
We parametrise the  $\rmT$-fixed points in $\ccY_k=\calF\times \calG_k$ by $\frakS_n\times \Lambda_k(n)$ by denoting
$x_{w,\lambda}=(f_w,g_{w(\lambda)})$. Here, the twist of $\lambda$ by $w$ is chosen such that $x_{w,\lambda}\in C_\lambda$. Denote
\begin{equation}\label{DefAwmu}
\tA_{w,\mu}:=\tA_w\cdot \tA_{w(\mu)}=\op{eu}(\ccY_k,x_{w,\mu})\in \Bbbk[\ccT_1,\ldots,\ccT_n].
\end{equation}
We will later compute the $H_*^\rmG(\ccY\times\ccY)$-action on $H_*^\rmG(\ccY)$ using the localisation theorem and repeatedly the following lemma, which is a special case of \Cref{lem:comp-in-loc}~\ref{2}.). We denote by  $x_{w_1,w_2,\mu_1,\mu_2}=(x_{w_1,\mu_1},x_{w_2,\mu_2})$ the $\rmT$-fixed points of $\ccY\times\ccY$.
\begin{lem}
The $H_*^\rmG(\ccY\times\ccY)$-action on $H_*^\rmG(\ccY)$ satisfies
\begin{equation}\label{unclear}
[x_{w_1,w_2,\mu_1,\mu_2}] * [x_{w,\mu}]=
\begin{cases}
\tA_{w,\mu}[x_{w_1,\mu_1}] & \mbox{ if } w_2=w, \mu_2=\mu,\\
0& \mbox{ otherwise.}
\end{cases}
\end{equation}
\end{lem}

Consider the polynomial $\tS_\lambda^w(\mu)\in \Bbbk[\ccT_1,\ldots,\ccT_d]\cong H^*_\rmT(\{x_{w,\mu}\})$ defined as the pull-back of $\tS_\lambda\in H_*^\rmG(\ccY_k)\subset  H_*^\rmT(\ccY_k)$ along the inclusion $\{x_{w,\mu}\}\to\ccY_k$. In the basis of $\rmT$-fixed points,  $\tS_\lambda=\sum_{w,\mu} \tS^w_\lambda(\mu) \tA_{w,\mu}^{-1}\cdot[x_{w,\mu}]$. 

The following characterization of the $\tS^w_\lambda(\mu)$ is a reformulation of \cite[Lemma 3.8]{GKS}. To state it consider a partial order on $\Lambda_k(n)$ given by $\lambda\geqslant\mu$ if $\sum_{i=1}^r\lambda_i\geqslant \sum_{i=1}^r\mu_i$ for each $r\in [1;n]$. Set $\inv(\lambda)=\{(i,j)|\mbox{ }1\leqslant i<j\leqslant n,\lambda_i>\lambda_j\}$. 

\begin{lem}
\label{lem:caract-S}
Fix $\lambda\in \Lambda_k(n)$ and $w\in\frakS_n$. There is a unique family of polynomials $\tS^w_\lambda(\mu)\in \Bbbk[\ccT_1,\ldots,\ccT_n]$, parametrized by $\mu\in \Lambda_k(n)$, such that $\tS^w_\lambda(\mu)\in H^*_\rmT(\{x_{w,\mu}\})$ is the  restriction of some class $x\in H_{G}^*(\ccY_k)\subset H_{\rmT}^*(\ccY_k)$ to the $\rmT$-fixed point $x_{w,\mu}$, and the following properties hold: 
\begin{enumerate}[\rm{(}i\rm{)}]
\item The support of $\tS^w_\lambda$ lies above $\lambda$.
\item $\tS^w_\lambda(\lambda)=\tA_w\cdot w(\prod_{(i,j)\in \inv(\lambda)}(\ccT_{j}-\ccT_{i}))$.
\item For each $\mu\geqslant\lambda$, the degree of $\tS^w_\lambda(\mu)$ is $\#(\inv(\lambda))+\binom{n}{2}$.
\end{enumerate}
\end{lem}
\begin{rk}
\label{rk-Sw}
Note that by convention, (iii) allows $\tS^w_\lambda(\mu)$ to be zero, since zero is of any degree. \Cref{lem:caract-S} directly implies that $\tS^w_\lambda(\mu)=w(\tS^\Id_\lambda(\mu))$. 
\end{rk}
\begin{proof}
We claim that $\tS^w_\lambda(\mu)=\tA_w\widetilde \tS^w_{w(\lambda)}(w(\mu))$, where $\widetilde \tS^w_{\lambda}(\mu)=\op{eu}(\overline{\rmB_wg_\lambda},g_\mu)$ is the polynomial denoted $\tS^w_\lambda(\mu)$ in \cite{GKS}. Consider, similar to \eqref{eq:isom-G/B*Gr}, isomorphisms
\begin{equation}\label{GKSclaim}
\ccY_k=\calF\times\calG_k\simeq \rmG/\rmB_w\times \calG_k\simeq \rmG\times_{\rmB_w}\calG_k,
\end{equation}
where the last one is $(g\rmB_w,W)\mapsto (g,g^{-1}(W))$.
The composition \eqref{GKSclaim} identifies the subvariety $\overline{C_{\lambda}}\subset \ccY_k$ with  $\rmG\times _{\rmB_w}\overline{\rmB_wg_{w(\lambda)}}\subset \rmG\times _{\rmB_w}\calG_k$. We obtain then 
\begin{equation}
\label{eq:S-vs-Stilda}
\tS^w_\lambda(\mu)=\op{eu}(\overline{C_{\lambda}},x_{w,\mu})=\op{eu}(\rmG/\rmB_w,f_w)\cdot \op{eu}(\overline{\rmB_wg_{w(\lambda)}},g_{w(\mu)})=\tA_w\widetilde \tS^w_{w(\lambda)}(w(\mu)).
\end{equation}
Now we apply \cite[Lemma 3.8]{GKS} to the polynomials $\tS^w_\lambda(\mu)\tA_w^{-1}$ (noting that $\tS^w_\lambda(\mu)$ is automatically divisible by $\tA_w$) and the lemma follows.
\end{proof}

The following "wall-crossing formula" is proved in \cite[Prop.~3.12]{GKS}.
\begin{lem}
\label{lem:Schubert-change-w}
Let $w\in\frakS_n$ and $r\in[1;n-1]$. Then the following holds
$$
\tS^{ws_r}_\lambda(\mu)=
\begin{cases}
\tS^w_\lambda(s_r(\mu))+(\ccT_{w(r)}-\ccT_{w(r+1)})\tS^w_{s_r(\lambda)}(s_r(\mu)) &\mbox{ if }\lambda_r>\lambda_{r+1},\\
\tS^w_\lambda(s_r(\mu)) &\mbox{ otherwise.}
\end{cases}
$$
\end{lem}

\begin{rk}\label{rk:loc-X-vs-FX}
The proof of \Cref{lem:caract-S} and also its relation to ordinary equivariant Schubert classes might become more transparent when we have a closer look at the isomorphism $H^*_\rmG(\calF\times X)\cong H^*_\rmT(X)$ in \Cref{lem:isom-FG/G-G/T} under localisation. Assume for this that $X$ is a smooth $\rmG$-variety such that  $X^\rmT=\{x_\nu\}_\nu$ is finite. Note that the Weyl group $\frakS_n$ of $\rmG$ acts on $X^\rmT$. Assume moreover that the localisation map $H^*_\rmT(X)\subset H^*_\rmT(X)_{\rm loc}$ is injective. 

Let now $h\in H^*_\rmT(X)$. Then we can write $h=\sum_{\nu}h_\nu[x_\nu]$ for some $h_\nu\in {\rm Frac}(\tR_\rmT)$. On the other hand, we can view $h$ as an element of $H^*_\rmG(\calF\times X)$ and do the localisation for $\calF\times X$. We consider $H^*_\rmG(\calF\times X)\subset H^*_\rmT(\calF\times X)\subset H^*_\rmT(\calF\times X)_{\rm loc}$. In the basis of fixed points $x_{w,\nu}:=(f_w,x_{w(\nu)})$ of $\calF\times X$, the element $h$ becomes
$$
h=\sum_{w,\nu}w(h_\nu)\tA^{-1}_w[x_{w,\nu}]\in H^*_\rmT(\calF\times X)_{\rm loc}.
$$
Now let's specialise to $X=\calG$ and $h=\tS_\lambda\in H^*_\rmT(\calG)$. Its localisation in $H^*_\rmT(\calG)_{\rm loc}$ looks like $\tS_\lambda=\sum_{\mu}\widetilde{\tS}^\Id_\lambda(\mu)\tA^{-1}_\mu[g_\mu]$, where $\widetilde{\tS}$ is as in the proof of \Cref{lem:caract-S}. We then get indeed  $\tS^w_\lambda(\mu)=w(\widetilde{\tS}^\Id_\lambda(\mu))\tA_w$, since for $\tS_\lambda\in H^*_\rmG(\calF\times \calG)$ we have 
$$
\tS_\lambda=\sum_{w,\mu}\tS^w_\lambda(\mu)\tA^{-1}_{w,\mu}[x_{w,\mu}].
$$
This is also compatible with \eqref{eq:S-vs-Stilda} and \Cref{rk-Sw}.
\end{rk}

\subsection{The $\WW_n$-action on $H_*^\rmG(\ccY)$}
Recall that $H^*_\rmG(\calG)=H^*_\rmT(\calG)^{\frakS_n}$. The $\frakS_n$-action here extends to an action on $H^*_\rmT(\calG)_{\loc}\supset H^*_\rmT(\calG)$, explicitly $w(P[g_\mu])=w(P)[g_{w(\mu)}]$ for $w\in\frakS_n$, $\mu\in \Lambda(n)$ and $P\in {\rm Frac}(\tR_\rmT)$, see \Cref{notnew}. Using the isomorphism $H^*_\rmT(\calG)=H^*_\rmG(\ccY)$ from \eqref{iso-FG/G-G/T} we get an induced $\frakS_n$-action which extends to $H_*^\rmT(\ccY)_{\loc}$. It satisfies \begin{equation}\label{inducedSn}
z(\tA^{-1}_{w,\mu}[x_{w,\mu}])=\tA^{-1}_{wz^{-1},z(\mu)}[x_{wz^{-1},z(\mu)}].
\end{equation}
To describe the action on $H^*_\rmG(\ccY)=\Bbbk[X_1,\ldots,X_n]\otimes \mywedge^\bullet(\omega_1,\ldots,\omega_n)$ concretely, write  
\begin{equation}\label{PinEuler}
P=\sum_{w,\mu}P_w\tA^{-1}_{w,\mu}[x_{w,\mu}], \quad\text{ in particular }\quad[\ccY]=\sum_{w,\mu}\tA^{-1}_{w,\mu}[x_{w,\mu}], 
\end{equation}
for $P\in \Bbbk[X_1,\ldots, X_n]$,  and then write $P\omega_\lambda$, for $\lambda\in\Lambda(n)$, as

\begin{equation}
\label{eq:pol-wedge-loc}
P\omega_\lambda=\sum_{w,\mu}P_w\tS^w_\lambda(\mu)\tA_{w,\mu}^{-1}[x_{w,\mu}].
\end{equation} 

\begin{lem}\label{problematic}
The group $\WW_n$ acts on $H_*^\rmG(\ccY)=\Bbbk[X_1,\ldots,X_n]\otimes \mywedge^\bullet(\omega_1,\ldots,\omega_n)$ by 
\begin{equation}\label{sraction}
s_r(P\omega_\lambda)=
\begin{cases}
s_r(P)\left(\omega_\lambda+(X_{r}-X_{r+1})\omega_{s_r(\lambda)}\right) &\mbox{ if }\lambda_r>\lambda_{r+1},\\
s_r(P)\omega_\lambda &\mbox{ otherwise}.
\end{cases}
\end{equation}
\end{lem}
\begin{rk}
Note that  \eqref{sraction} uses on the right hand side the usual action of $s_r$ on $P$. The lemma in fact implies that the action \eqref{sraction} is by algebra automorphisms which then justifies the notation on the left hand side of  \eqref{sraction}.
\end{rk}
\begin{proof}
We compute for $z\in\frakS_n$.
\begin{eqnarray*}
z(P\omega_\lambda)&\stackrel{\eqref{eq:pol-wedge-loc}}=&z(\sum_{w,\mu}P_w \tS^w_\lambda(\mu)\tA^{-1}_{w,\mu}[x_{w,\mu}])
\stackrel{\eqref{inducedSn}}{=}\sum_{w,\mu}P_w \tS^w_\lambda(\mu)\tA^{-1}_{wz^{-1},z(\mu)}[x_{wz^{-1},z(\mu)}]\\
&=&\sum_{w,\mu}P_{wz} \tS^{wz}_\lambda(z^{-1}(\mu))\tA^{-1}_{w,\mu}[x_{w,\mu}].
\end{eqnarray*}
Let us see what this gives for $z=s_r$. By \Cref{lem:Schubert-change-w} we have the following:

If $\lambda_r\leqslant \lambda_{r+1}$ then $\tS^{ws_r}_\lambda(s_r(\mu))=\tS^w_\lambda(\mu)$, thus $s_r(P\omega_\lambda)=P_{s_r}\omega_\lambda$.

If $\lambda_r> \lambda_{r+1}$, then $\tS^{ws_r}_\lambda(s_r(\mu))=\tS^w_\lambda(\mu)+(\ccT_{w(r)}-\ccT_{w(r+1)})\tS^w_{s_r(\lambda)}(\mu).$\qedhere
\end{proof}

Importantly, the $\WW_n$-action on \Cref{problematic} agrees with the action in \cite[\S 3.2.1]{NaisseVaz} and is the starting point of our geometric approach to the Naisse--Vaz algebra. We show later in this section that $\hhNH_n\cong H^G_*(\ccZ)$ by comparing the faithful actions of the two algebras on $\hPol_n\cong H^G_*(\ccY)$. For each generator $x$ of the algebra $\hhNH_n$ we will find some $x\in H^G_*(\ccZ)$ which acts by the same endomorphism.

\subsection{Demazure operators}
For $r\in[1;n-1]$, consider the subvariety $\Zr\subset \ccY\times \ccY$ defined by the condition that the  two Grassmannian spaces agree and the two flags are the same except maybe at the $r$th component. Define 
\begin{equation}
\label{eq:Zr-loc}
T_r=-[\Zr]\in H^G_*(\ccZ),\quad \text{and}\quad
[\Zr]=\sum_{w_1,w_2,\mu_1,\mu_2}\tA_{w_1,w_2,\mu_1,\mu_2}[x_{w_1,w_2,\mu_1,\mu_2}].
\end{equation}
\begin{lem}
Explicitly, $T_r$ is given by the formulas
\begin{equation}
\label{eq:Zr-coeff}
\tA_{w_1,w_2,\mu_1,\mu_2}=
\begin{cases}
\tA^{-1}_{w_1,\mu_1}(\ccT_{w_2(r)}-\ccT_{w_2(r+1)})^{-1} &\mbox{if } (w_1,\mu_1)\sim_r(w_2,\mu_2),\\
0 & \mbox{otherwise},
\end{cases}
\end{equation}
where $(w_1,\mu_1)\sim_r (w_2,\mu_2)$ if $(w_1,\mu_1)=(w_2,\mu_2)$ or $(w_1,\mu_1)=(w_2s_r,s_r(\mu_2))$.
\end{lem}

\begin{proof}
Let $\mu_1=s_r(\mu_2)$ and $w_1=w_2s_r$. Abbreviating $w=w_1$, $\mu=\mu_1$ we have 
\begin{align*}
\tA_{w_1,w_2,\mu_1,\mu_2}&\,\stackrel{\op{def}}=\,\;\op{eu}(\Zr,x_{w, ws_r, \mu, s_r(\mu)})^{-1}
&&=\;\op{eu}(\overline{O^{s_r}},f_{w, ws_r})^{-1}\cdot\tA^{-1}_{w(\mu)}
\\
&\stackrel{\ref{eunm}}=\;\op{eu}(\frakn^*_{w}\oplus \frakm_{w, ws_r})^{-1}\cdot \tA^{-1}_{w(\mu)}
&&\stackrel{\ref{Aw}}=\; \tA^{-1}_{w}\tA^{-1}_{w(\mu)}\op{eu}(\frakm_{w,ws_r})^{-1}
\\
&\,\,=\,\; \tA_{w,\mu}^{-1}(\ccT_{w(r)}-\ccT_{w(r+1)})^{-1}
&&\stackrel{\op{ass}}=\; \tA_{w_1,\mu_1}^{-1}(\ccT_{w_2(r+1)}-\ccT_{w_2(r)})^{-1}.
\end{align*}
Let $\mu=\mu_2$ and $w_1=w_2$. Abbreviating $w=w_1$, $\mu=\mu$ we have 
\begin{align*}
\tA_{w_1,w_2,\mu,\mu_2}
&\,\stackrel{\op{def}}=\,\;\op{eu}(\Zr,x_{w, w, \mu,\mu})^{-1}
&&=\;\op{eu}(\overline{O^{s_r}},f_{w,w})^{-1}\cdot \tA^{-1}_{w(\mu)}\\
&\stackrel{\ref{eunm}}=\;\op{eu}(\frakn^*_{w}\oplus \frakm_{ws_r,w})^{-1}\cdot \tA^{-1}_{w(\mu)}
&&\stackrel{\ref{Aw}}=\; \tA^{-1}_{w}\tA^{-1}_{w(\mu)}\op{eu}(\frakm_{ws_r,w})^{-1}\\
&\,\,=\,\; \tA_{w,\mu}^{-1}(\ccT_{w(r+1)}-\ccT_{w(r)})^{-1}
&&\stackrel{\op{ass}}=\; \tA_{w,\mu}^{-1}(\ccT_{w_2(r+1)}-\ccT_{w_2(r)})^{-1}.
\end{align*}
This shows the first case in \eqref{eq:Zr-coeff}. Clearly, we get zero otherwise. 
\end{proof}
We recover the action of the Demazure operators $\partial_r$, see \Cref{actionNH}:
\begin{prop}\label{TrasDemazure}
The action of $T_r\in H^G_*(\ccZ)$ on $H_*^\rmG(\ccY)$ agrees with  $\partial_r$.
\end{prop}
\begin{proof}
Abbreviate $s=s_r$. By \eqref{eq:pol-wedge-loc}, \eqref{eq:Zr-loc},  the action of $T_r$ is given by 
\begin{align*}
[\Zr]\star P\omega_\lambda&\;=(\sum_{w_1,w_2,\mu_1,\mu_2}\tA_{w_1,w_2,\mu_1,\mu_2}[x_{w_1,w_2,\mu_1,\mu_2}])\star (\sum_{w,\mu}P_w\tS^w_\lambda(\mu)\tA^{-1}_{w,\mu}[x_{w,\mu}])\\
&\stackrel{\eqref{unclear}}=
\sum_{w,w_1,\mu,\mu_1}\tA_{w_1,w,\mu_1,\mu}P_w\tS^w_\lambda(\mu)[x_{w_1,\mu_1}].
\end{align*}
By \eqref{eq:Zr-coeff}, the coefficient of $[x_{z,\nu}]$ is here equal to 
\begin{equation*}
\tA^{-1}_{z,\nu}\frac{P_z\tS^z_\lambda(\nu)}{(\ccT_{z(r)}-\ccT_{z(r+1)})}+\tA^{-1}_{z,\nu}\frac{P_{zs}\tS^{zs}_\lambda(s(\nu))}{(\ccT_{zs(r)}-\ccT_{zs(r+1)})}
=\tA^{-1}_{z,\nu} \frac{P_z\tS^z_\lambda(\nu)-P_{zs}\tS^{zs}_\lambda(s(\nu))}{\ccT_{z(r)}-\ccT_{z(r+1)}}.
\end{equation*}
Using \eqref{eq:pol-wedge-loc} and  \eqref{inducedSn} we deduce that this agrees with the coefficient of $[x_{w,\mu}]$ in the decomposition of $-\partial_r(P\omega_\lambda)=\frac{P\omega_\lambda-s_r(P\omega_{\lambda})}{X_{r+1}-X_{r}}$ in the basis of $\rmT$-fixed points. 
\end{proof}

\subsection{Creation operators}
\label{subs:creation}
For $k\in[1;n]$ we consider the following \emph{Grassmannian inclusion subvariety} of $\ccY_{k+1}\times\ccY_k$,
\begin{equation}\label{Grassind}
\ccZ_{\supse,k}=\{((\bV,W),(\tV,\tW))\in\ccY_{k+1}\times\ccY_k\mid W\supset \tW, \bV=\tV\}.
\end{equation}
If $\epsilon_r$ denotes the element $\lambda\in\Lambda_1(n)$ satisfying $\lambda_i=\delta_{i,r}$, the   
fundamental class $[\ccZ_{\supse,k}]$ can be written as $[\ccZ_{\supse,k}]=\sum_{\mu_1,\mu_2,w_1,w_2}\tB^+_{w_1,w_2,\mu_1,\mu_2}[x_{w_1,w_2,\mu_1,\mu_2}]$, where
\begin{equation*}
\tB^+_{w_1,w_2,\mu_1,\mu_2}=
\begin{cases}
\tA_{w_1,\mu_1}^{-1}\prod_{i\in \mu_2}(\ccT_{w_1(t)}-\ccT_{w_1(i)})^{-1} &\mbox{ if }w_1=w_2, ~\mu_1=\mu_2+\epsilon_t,\\
0 & \mbox{ otherwise}.
\end{cases}
\end{equation*}

\begin{prop}
\label{lem:creation+n}
Convolution with the fundamental class $[\ccZ_{\supse,k}]$ defines a map  
 which, via \Cref{Schurandomega}, is the creation operator $\omega_n^+\colon H_*^\rmG(\ccY_k)\to H_*^\rmG(\ccY_{k+1})$.
\end{prop}
\begin{proof}
Again by \eqref{unclear} and we get
$[\ccZ_{\supse,k}]\star P\omega_\lambda=
\sum_{w,w_1,\mu,\mu_1}\tB^+_{w_1,w,\mu_1,\mu}P_w\tS^w_\lambda(\mu)[x_{w_1,\mu_1}]$.
The coefficient of $[x_{z,\nu}]$ in this expression is
\begin{equation*}
\tA^{-1}_{z,\nu}P_z\sum_{t\in\nu}\frac{\tS^z_\lambda(\nu\backslash\{t\})}{\prod_{i\in \nu\backslash\{t\}}(\ccT_{z(t)}-\ccT_{z(i)})}.
\end{equation*}

Now, the statement follows from \Cref{lem:S-lambda+n} below and \Cref{rk-Sw}.
\end{proof}

\begin{lem} 
\label{lem:S-lambda+n}
The equivariant Schubert classes $\tS_\lambda:=\tS^\Id_\lambda$ satisfy the formula 
\begin{equation}\label{Schursforcreation}
\sum_{t\in \mu}\frac{\tS_\lambda(\mu\backslash\{t\})}{\prod_{i\in\mu\backslash\{t\}}(\ccT_t-\ccT_i)}=
\begin{cases}
\tS_{\lambda\cup \{n\}}(\mu) &\mbox{ if }n\not\in\lambda,\\
0 &\mbox{ if }n\in\lambda.
\end{cases}
\end{equation}
\end{lem}
\begin{proof}
Abbreviate the left hand side of \eqref{Schursforcreation} as $D_\lambda(\mu)$. Note that they are polynomials by \Cref{lem:caract-S}(ii) and, by the computation in the proof of \Cref{lem:creation+n}, they arise as restrictions of the class $x=[\ccZ_{\supse,k}]\star\omega_\lambda\in H^*_\rmG(\calF\times \calG_k)\cong H^*_\rmT(\calF_k).$ To show \eqref{Schursforcreation} in case  $n\not\in \lambda$, it therefore suffices to verify that $D_\lambda(\mu)$ satisfies the characterising properties (i)-(iii) from \Cref{lem:caract-S}. 

To see (i), assume $D_\lambda(\mu)\ne 0$ for some $\mu$. Then there exists $t\in \mu$ such that $\tS_\lambda(\mu\backslash\{t\})\ne 0$. Then, by  \Cref{lem:caract-S}(i), $\mu\backslash\{t\}\geqslant \lambda$ which implies $\mu\geqslant \lambda \cup\{n\}$ by definition of the ordering. Thus (i) holds.

To see (ii), note that 
\begin{equation}
D_\lambda(\lambda\cup\{n\})=\frac{\tS_\lambda(\lambda)}{\prod_{i\in \lambda}(\ccT_n-\ccT_i)}=
\tS_{\lambda\cup\{n\}}(\lambda\cup\{n\}).
\end{equation}
Here the second equality holds by \Cref{lem:caract-S}(ii), and the  first equality holds, since in the definition of $D_\lambda(\lambda\cup\{n\})$ only the summand for $t=n$ is nonzero. (The other terms vanish due to the support condition for the equivariant Schubert classes in \Cref{lem:caract-S}(i), because $\lambda\cup\{n\}\backslash\{t\}<\lambda$).  Thus (ii) holds.

Since (iii) is obviously true, we proved \eqref{Schursforcreation} in case  $n\not\in \lambda$.

Assume now $n\in \lambda$. Using the already proved case we can write
\begin{equation*}
\tS_\lambda(\mu\backslash\{t\})=\sum_{t'\in\mu\backslash\{t\}}\frac{\tS_{\lambda\backslash \{n\}}(\mu\backslash\{t,t'\})}{\prod_{i\in\mu\backslash\{t,t'\}}(\ccT_{t'}-\ccT_{i})}.
\end{equation*}
Plugging this into the definition of $D_\lambda(\mu)$ we obtain 
\begin{equation*}
D_\lambda(\mu)=\sum_{t,t'\in\mu,t\ne t'}\frac{\tS_{\lambda\backslash \{n\}}(\mu\backslash\{t,t'\})}{\prod_{i\in\mu\backslash\{t\}}(\ccT_t-\ccT_i)\prod_{i\in\mu\backslash\{t,t'\}}(\ccT_{t'}-\ccT_{i})}=0,
\end{equation*}
because the summand for $(t,t')$ cancels with the summand for $(t',t)$. 
\end{proof}

For the Grassmannian inclusion variety \eqref{Grassind} let  $\xi\in  H_*^\rmG(\ccZ_{\supse,k})$ be the Chern class of the line bundle given by $W/\tW$ and let $X_r\in H_*^\rmG(\ccZ_{\subse,k})$, for $r\in [1;n]$, be the Chern class of the line bundle corresponding to the tautological quotient $\bV^r/\bV^{r-1}$.

\begin{prop}
\label{lem:creation+r}
Let $r\in[1;n-1]$. The push-forward to $H_*^\rmG(\ccY\times\ccY)$ of 
\begin{equation}\label{fancycreators}
\prod_{p=r+1}^{n}(X_p-\xi)\in H_*^\rmG(\ccZ_{\supse,k})
\end{equation}
acts on $H_*^\rmG(\ccY)$ by the creation operator $\omega^+_r\colon H_*^\rmG(\ccY_k)\to H_*^\rmG(\ccY_{k+1})$.
\end{prop}

\begin{proof}
We can write the push-forward $\oX\in H_*^\rmG(\ccY\times\ccY)$ of the class \eqref{fancycreators} in the fixed point basis as
$
\oX=\sum_{\mu_1,\mu_2,w_1,w_2}\tB^{+,r}_{w_1,w_2,\mu_1,\mu_2}[x_{w_1,w_2,\mu_1,\mu_2}],  
$ where

\begin{equation*}
\tB^{+,r}_{w_1,w_2,\mu_1,\mu_2}=
\begin{cases}
\tA_{w_1,\mu_1}^{-1}\frac{\prod_{p=r+1}^{n}(\ccT_{w_1(p)}-\ccT_{w_1(t)}))}{\prod_{i\in \mu_2}(\ccT_{w_1(t)}-\ccT_{w_1(i)})} &\mbox{ if }w_1=w_2, ~\mu_1=\mu_2+\epsilon_t,\\
0 & \mbox{ otherwise}.
\end{cases}
\end{equation*}

The coefficient of $[x_{z,\nu}]$ in $\oX\star P\omega_\lambda$ is then, again by \eqref{unclear} and \eqref{eq:pol-wedge-loc}  equal to 
\begin{equation*}
\tA^{-1}_{w,\mu}P_w\sum_{t\in\mu}\frac{\prod_{p=r+1}^{n}(\ccT_{w(p)}-\ccT_{w(t)})\tS^w_\lambda(\mu\backslash\{t\})}{\prod_{i\in \mu\backslash\{t\}}(\ccT_{w(t)}-\ccT_{w(i)})}.
\end{equation*}
Now, the statement follows from \Cref{lem:S-lambda+r} below and \Cref{rk-Sw}.
\end{proof}

\begin{lem}
\label{lem:S-lambda+r}
Let $r\in[1;n-1]$. The equivariant Schubert classes $\tS_\lambda:=\tS^\Id_\lambda$ satisfy
\begin{equation}\label{formulaforfancycreator}
\sum_{t\in \mu}\frac{\prod_{p=r+1}^{n}(\ccT_p-\ccT_t)\tS_\lambda(\mu\backslash\{t\})}{\prod_{i\in\mu\backslash\{t\}}(\ccT_t-\ccT_i)}=
\begin{cases}
(-1)^{|\lambda_{>r}|}\tS_{\lambda\cup \{r\}}(\mu) &\mbox{ if }r\not\in\lambda,\\
0 &\mbox{ if }r\in\lambda.
\end{cases}
\end{equation}
\end{lem}

\begin{proof}
Denote by $D_\lambda(\mu)$ the left hand side of \eqref{formulaforfancycreator}. This is a polynomial, by \Cref{lem:caract-S}(ii) and arises by restriction of $x=\oX\star\omega_\lambda\in H^*_\rmG(\calF\times \calG_k)\cong H^*_\rmT(\calG_k).$

Assume $r\not\in \lambda$. To verify  \eqref{formulaforfancycreator} it suffices to verify again the characterising properties (i)-(iii) from \Cref{lem:caract-S} for $(-1)^{|\lambda_{>r}|}D_\lambda(\mu)$. 
For (i) assume that $D_\lambda(\mu)\ne 0$ for some $\mu$. Then there exists $t\in \mu$ such that $\tS_\lambda(\mu\backslash\{t\})\ne 0$. This implies $\mu\backslash\{t\}\geqslant \lambda$ and thus $\mu\geqslant \lambda \cup\{t\}$. For (ii) we have 
\begin{equation*}
(-1)^{|\lambda_{>r}|}D_\lambda(\lambda\cup\{r\})=\frac{\prod_{p\not\in\lambda,p>r}(\ccT_p-\ccT_r)\tS_\lambda(\lambda)}{\prod_{i\in \lambda,i<r}(\ccT_r-\ccT_i)}=\tS_{\lambda\cup\{r\}}(\lambda\cup\{r\}).
\end{equation*}
The second equality holds by \Cref{lem:caract-S} whereas the first one 
holds since only the summand for $t=r$ survives in the definition of  $D_\lambda(\lambda\cup\{n\})$. (The others are zero because $\lambda\cup\{r\}\backslash\{t\}<\lambda$ for $t<r$.) Thus (ii) holds, and (iii) is obvious.

Now, assume $t\in \lambda$. Using the already proved case we calculate
\begin{equation*}
\tS_\lambda(\mu\backslash\{t\})=(-1)^{|\lambda_{>r}|}\sum_{t'\in\mu\backslash\{t\}}\frac{\prod_{p=r+1}^{n}(\ccT_p-\ccT_{t'})\tS_{\lambda\backslash \{r\}}(\mu\backslash\{t,t'\})}{\prod_{i\in\mu\backslash\{t,t'\}}(\ccT_{t'}-\ccT_{i})},
\end{equation*}
\begin{equation*}
D_\lambda(\mu)=\sum_{t,t'\in\mu,t\ne t'}\frac{\prod_{p=r+1}^{n}(\ccT_p-\ccT_{t})\prod_{p=r+1}^{n}(\ccT_p-\ccT_{t'})\tS_{\lambda\backslash \{r\}}(\mu\backslash\{t,t'\})}{\prod_{i\in\mu\backslash\{t\}}(\ccT_t-\ccT_i)\prod_{i\in\mu\backslash\{t,t'\}}(\ccT_{t'}-\ccT_{i})}=0,
\end{equation*}
since the summands for $(t,t')$ and $(t',t)$ cancel each other.
\end{proof}
\begin{rk}
We also get a geometric construction of the element $\vartheta_r$ defined in \cite[Prop. 4.9]{AEHL} to give an exterior basis of the graded center of $\hNH_n$. Their $\vartheta_r$ corresponds to the push-forward of $\sum_{l=0}^{n-r}e_{n-r-l}(X_1,\ldots ,X_n)(-\xi)^l\in H_*^\rmG(\ccZ_{\supse,k})$ ($e_r$ is the $r$th elementary symmetric polynomial). 
\end{rk}
\subsection{Annihilation operators}
Similarly to $\ccZ_{\supse,k}$ we consider the \emph{opposite Grassmannian inclusion subvariety} 
\begin{equation}\label{Grassindopp}
\ccZ_{\subse,k}=\{((\bV,W),(\tV,\tW))\in\ccY_{k}\times\ccY_{k+1}\mid W\subset \tW, \bV=\tV\}.
\end{equation}

\begin{prop}
\label{lem:annihilation-1}
Convolution with the fundamental class $[\ccZ_{\subse,k}]$ induces, via \Cref{Schurandomega}, the (signed) annihilation operator $(-1)^{k}\omega_1^-\colon H_*^\rmG(\ccY_{k+1})\to H_*^\rmG(\ccY_{k})$.
\end{prop}

\begin{proof}
Writing 
$
[\ccZ_{\subse,k}]=\sum_{w_1,w_2,\mu_1,\mu_2}\tB^{-,k}_{w_1,w_2,\mu_1,\mu_2}[x_{w_1,w_2,\mu_1,\mu_2}] 
$ we  have 
\begin{align}\label{Bminusk}
\hspace{-3.5mm}\tB^{-,k}_{w_1,w_2,\mu_1,\mu_2}=
\begin{cases}
\tA_{w_1,\mu_1}^{-1}\!\prod_{i\not\in \mu_2}(\ccT_{w_1(i)}-\ccT_{w_1(t)})^{-1} &
\!\!\!\!\!\mbox{ if }(w_1,\mu_1)=(w_2, \mu_2-\epsilon_t),\\
0 &\!\!\!\!\!\mbox{ otherwise}.
\end{cases}
\end{align}
The coefficient of $[x_{w,\mu}]$ in the product $[\ccZ_{\subse,k}]\star P\omega_\lambda$ equals $$
\tA^{-1}_{w,\mu}\sum_{t\not\in\mu}\frac{\tS_\lambda^w(\mu\cup t)}{\prod_{i\not\in\mu\cup \{t\}}(\ccT_{w(i)}-\ccT_{w(t)})}.
$$
Now the statement follows from \Cref{lem:S-lambda-1} below and \Cref{rk-Sw}.
\end{proof}

\begin{lem}
\label{lem:S-lambda-1}
The equivariant Schubert classes $\tS_\lambda:=\tS^\Id_\lambda$ satisfy the formula 
\begin{equation}\label{Schursforann}
\sum_{t\not\in \mu}\frac{\tS_\lambda(\mu\cup\{t\})}{\prod_{i\not\in\mu\cup{t}}(\ccT_i-\ccT_t)}=
\begin{cases}
\tS_{\lambda\backslash \{1\}}(\mu) &\mbox{ if }1\in\lambda,\\
0 &\mbox{ if }1\not\in\lambda.
\end{cases}
\end{equation}
\end{lem}
\begin{proof}
The proof is similar to the proof of \Cref{lem:S-lambda+n}. We are therefore brief. Denote by $D_\lambda(\mu)$ the left hand side of \eqref{Schursforann} which arises by restriction from $x=[\ccZ_{\subse,k}]\star\omega_\lambda\in H^*_\rmG(\calF\times \calG_k)\cong H^*_\rmT(\calG_k),$ see the proof of \Cref{lem:annihilation-1}. 

Assuming $1\in \lambda$ we again check the properties (i)-(ii), since (iii) is obvious. 

If $D_\lambda(\mu)\ne 0$ for some $\mu$, then there exists $t\in \mu$ such that $\tS_\lambda(\mu\cup\{t\})\ne 0$. Thus $\mu\cup\{t\}\geqslant \lambda$ and therefore $\mu\geqslant \lambda \backslash \{1\}$ proving (i). For (ii) note that 
\begin{equation*}
D_\lambda(\lambda\backslash \{1\})=\frac{\tS_\lambda(\lambda)}{\prod_{i\not\in\lambda}(\ccT_i-\ccT_1)}=\tS_{\lambda\backslash \{1\}}(\lambda\backslash \{1\}).
\end{equation*}
Knowing the case $1\notin \lambda$, we deduce the case  $1\notin \lambda$ by computing
\begin{equation*}
\tS_\lambda(\mu\cup\{t\})=\sum_{t'\not\in\mu\cup\{t\}}\frac{\tS_{\lambda\cup \{1\}}(\mu\cup\{t,t'\})}{\prod_{i\not\in\mu\cup\{t,t'\}}(\ccT_{i}-\ccT_{t'})}\quad\text{and}
\end{equation*}
\begin{equation*}
D_\lambda(\mu)=\sum_{t,t'\not\in\mu,t\ne t'}\frac{\tS_{\lambda\cup \{1\}}(\mu\cup\{t,t'\})}{\prod_{i\in\mu\cup\{t\}}(\ccT_i-\ccT_t)\prod_{i\not\in\mu\cup\{t,t'\}}(\ccT_{i}-\ccT_{t'})}=0.\hfill\qedhere
\end{equation*}
\end{proof}

\bigskip

Abusing notation, we denote by 
$\xi\in  H_*^\rmG(\ccZ_{\subse,k})$ also the Chern class of the line bundle given by $\tW/W$ for the opposite Grassmannian inclusion variety, \eqref{Grassindopp}. 

\begin{prop}
\label{lem:annihilation-r}
Let $r\in[1;n-1]$. The push-forward to $H_*^\rmG(\ccY\times\ccY)$ of 
\begin{equation}\label{fancyann}
\prod_{p=1}^{r-1}(X_p-\xi)\in H_*^\rmG(\ccZ_{\subse,k})
\end{equation}
agrees with the (signed) annihilation operator $(-1)^{k}\omega^-_r\colon H_*^\rmG(\ccY_{k+1})\to H_*^\rmG(\ccY_{k})$.
\end{prop}
\begin{proof}
Write the push-forward $X\in H_*^\rmG(\ccY\times\ccY)$ of \eqref{fancyann} as 
\begin{equation*}
X=\sum_{\mu_1,\mu_2,w_1,w_2}\tB^{-,r}_{w_1,w_2,\mu_1,\mu_2}[x_{w_1,w_2,\mu_1,\mu_2}],
\end{equation*}
\begin{equation*}
\tB^{-,r}_{w_1,w_2,\mu_1,\mu_2}=
\begin{cases}
\tA_{w_1,\mu_1}^{-1}\frac{w_1(\prod_{p=1}^{r-1}(\ccT_p-\ccT_t))}{\prod_{i\not\in \mu_2}(\ccT_{w_1(i)}-\ccT_{w_1(t)})} &\mbox{ if }w_1=w_2, ~\mu_1=\mu_2-\epsilon_t,\\
0 & \mbox{ otherwise}.
\end{cases}
\end{equation*}
The coefficient of $[x_{w,\mu}]$ in the product $X\star P\omega_\lambda$ is 
\begin{equation*}
\tA^{-1}_{w,\mu}\sum_{t\not\in\mu}\frac{{\prod_{p=1}^{r-1}(\ccT_{w(p)}-\ccT_{w(t)})}\tS_\lambda^w(\mu\cup t)}{\prod_{i\not\in\mu\cup \{t\}}(\ccT_{w(i)}-\ccT_{w(t)})}.
\end{equation*}
The statement follows from \Cref{lem:S-lambda-r} below and \Cref{rk-Sw}.
\end{proof}

\begin{lem}
\label{lem:S-lambda-r}
Let $r\in[1;n-1]$. The polynomials $\tS_\lambda(\mu):=\tS^\Id_\lambda(\mu)$ satisfy
\begin{equation}\label{formulaforfancyann}
\sum_{t\not\in \mu}\frac{(\prod_{p=1}^{r-1}(\ccT_p-\ccT_t))\tS_\lambda(\mu\cup\{t\})}{\prod_{i\not\in\mu\cup{t}}(\ccT_i-\ccT_t)}=
\begin{cases}
(-1)^{|\lambda_{<r}|}\tS_{\lambda\backslash \{r\}}(\mu) &\mbox{ if }r\in\lambda,\\
0 &\mbox{ if }r\not\in\lambda.
\end{cases}
\end{equation}
\end{lem}
\begin{proof}
We argue as for \eqref{formulaforfancycreator}. Denote by $D_\lambda(\mu)$ the left hand side of \eqref{formulaforfancyann}.
First, assume $r\in \lambda$. We check that $(-1)^{|\lambda_{<r}|}D_\lambda(\mu)$ satisfies the properties of \Cref{lem:caract-S} (by the proof of \Cref{lem:annihilation-r} it arises by restriction from
$x=\prod_{p=1}^{r-1}(X_p-z)\star\omega_\lambda\in H^*_\rmG(\calF\times \calG_k)\cong H^*_\rmT(\calG_k)$). 
The support of $D_\lambda$ lies above $\lambda\backslash \{r\}$, since if  $D_\lambda(\mu)\ne 0$, then there exists $t\not\in \mu$ such that $\tS_\lambda(\mu\cup\{t\})\ne 0$ and $t\geqslant r$. This implies $\mu\cup\{t\}\geqslant \lambda$ and thus $\mu\geqslant \lambda \backslash \{r\}$ which shows property $(i)$. Now, 
\begin{equation*}
D_\lambda(\lambda\backslash \{r\})=(\prod_{i=1}^{r-1}(\ccT_i-\ccT_r))\tS_\lambda(\lambda)\prod_{i\not\in\lambda}(\ccT_i-\ccT_r)^{-1}
\end{equation*}
\begin{equation*}
=
(-1)^{|\lambda_{<r}|}(\prod_{i\in\lambda,i<r}(\ccT_r-\ccT_i))\tS_\lambda(\lambda)\prod_{i\not\in\lambda,i>r}(\ccT_i-\ccT_r)^{-1}=(-1)^{|\lambda_{<r}|}\tS_{\lambda\backslash \{r\}}(\lambda\backslash \{r\}). 
\end{equation*}
Here, the first equality holds, since only the summand for $t=r$ survives in $D_\lambda(\lambda\backslash \{r\})$, because $\lambda\backslash \{r\}\cup\{t\}<\lambda$ if $t>r$ and $\prod_{p=1}^{r-1}(\ccT_p-\ccT_t)=0$ if $t<r$. Thus (ii) and obviously (iii) holds. 
If now $r\not\in \lambda$ we can use 
\begin{equation*}
\tS_\lambda(\mu\cup\{t\})=(-1)^{|\lambda_{<r}|}\sum_{t'\not\in\mu\cup\{t\}}\frac{(\prod_{p=1}^{r-1}(\ccT_{p}-\ccT_{t'}))\tS_{\lambda\cup \{r\}}(\mu\cup\{t,t'\})}{\prod_{i\not\in\mu\cup\{t,t'\}}(\ccT_{i}-\ccT_{t'})}
\end{equation*}
and deduce 
\begin{equation*}
D_\lambda(\mu)=(-1)^{|\lambda_{<r}|}\sum_{t,t'\not\in\mu,t\ne t'}\frac{(\prod_{p=1}^{r-1}(\ccT_p-\ccT_t)(\ccT_{p}-\ccT_{t'}))\tS_{\lambda\cup \{r\}}(\mu\cup\{t,t'\})}{\prod_{i\in\mu\cup\{t\}}(\ccT_i-\ccT_t)\prod_{i\not\in\mu\cup\{t,t'\}}(\ccT_{i}-\ccT_{t'})}=0.\hfill\qedhere
\end{equation*}
\end{proof}

\subsection{The $\star$-product on $H^*_\rmT(\calG)$}
\label{subs:star-prod}
Let $a\in[1;n]$ and $X\in H^*_\rmT(\calG_a)$. Let $i_\mu$ be the inclusion $i_\mu\colon \{g_\mu\}\to \calG_a$ of the fixed point labelled by $\mu\in \Lambda_a(n)$. Consider the polynomial $X(\mu)= i_\mu^*X$ and let $[g_\mu]$ be the push-forwand of the fixed point $g_\mu$ to $\calG_a$. Since $(i_\mu)^*(i_\mu)_*=\tA_\mu$, we can write $\oX=\sum_{\mu\in\calG_a}\oX(\mu)\tA^{-1}_\mu[g_\mu]$ in $H^*_\rmT(\calG_a)_{\loc}$.

If now $\oX\in H_\rmT^*(\calG_a)_{\loc}$ we can similarly describe it via a family of rational functions $\oX(\mu)$ labelled by $\mu\in\Lambda_a(n)$. Given such $\oX$ and 
 $\oY\in H_\rmT^*(\calG_b)_{\loc}$ define
\begin{equation}\label{staralg}
(\oX\star \oY)(\mu)=\sum_{\stackrel{\mu_1\sqcup \mu_2=\mu}{|\mu_1|=a,|\mu_2|=b}}\frac{\oX(\mu_1)\oY(\mu_2)}{\catP_{\mu_1,\mu_2}}, \quad\text{where} \quad \catP_{\lambda,\nu}=\prod_{i\in\lambda,j\in \nu}(\ccT_i-\ccT_j).
\end{equation}
It is easy to see that $\star$ is associative and that $\star$ yields a graded algebra structure on $H_\rmT^*(\calG)_{\loc}$. 
We next show that $\star$ preserves $H^*_\rmT(\calG)\subset H^*_\rmT(\calG)_{\loc}$.  
\begin{lem}
\label{lem:S-star-S}
For $\lambda,\nu\in\Lambda(n)$, the following holds in $H^*_\rmT(\calG):$
\begin{equation}\label{SS}
\tS_{\lambda}\star \tS_{\nu}=
\begin{cases}
(-1)^{s(\lambda,\nu)}\tS_{\lambda\sqcup \nu}& \mbox{ if }\lambda\cap\nu=\emptyset,\\
0&\mbox{ otherwise,}
\end{cases}
\end{equation}
where $s(\lambda,\nu):=|\{(i,j)\mid i\in\lambda, j\in \nu, i<j\}$. 
\end{lem}
\begin{proof}
We prove \eqref{SS} by induction on $\nu$. If  $|\nu|=1$ then $\nu$ is of the form $\nu=\epsilon_r$ and the statement is equivalent to \Cref{lem:S-lambda+r}. Thus we have $\tS_\lambda\star \tS_{\nu\sqcup\{r\}}=\tS_{\nu}\star \tS_{\epsilon_r}$
for $(\lambda,\nu\sqcup\{r\})$ with $r\not\in \nu$. Assuming now \eqref{SS} for $(\lambda,\nu)$, we calculate
\begin{eqnarray*}
\tS_\lambda\star \tS_{\nu\sqcup\{r\}}=(-1)^{s(\lambda,\nu)+s(\nu,\epsilon_r)}\tS_{\lambda\sqcup\nu}\star \tS_{\epsilon_r}=
\begin{cases}
(-1)^{s(\lambda,\nu\sqcup\{r\})}\tS_{\lambda\sqcup\nu\sqcup\{r\}}& \mbox{ if }r\not\in\lambda,\\
0& \mbox{ otherwise.}
\end{cases}
\end{eqnarray*}
We used hereby the identity
$
(-1)^{s(\lambda,\nu)+s(\nu,\epsilon_r)+s(\lambda\sqcup\nu,\epsilon_r)}=(-1)^{s(\lambda,\nu\sqcup\{r\})},
$
which follows\footnote{Note that $\omega_\lambda\wedge\omega_\nu=(-1)^{s(\lambda,\mu)}\omega_{\lambda\cup\nu}$. 
}  from the associativity of wedges:
$
(\omega_\lambda\wedge\omega_\nu)\wedge\omega_r=\omega_\lambda\wedge(\omega_\nu\wedge\omega_r).
$
\end{proof}
\begin{coro}
The subspace $H^*_\rmT(\calG)$ of $H^*_\rmT(\calG)_{\loc}$ is a subalgebra. Moreover,
$$
H^*_\rmT(\calG)\cong\Pol_n\otimes\mywedge(\omega_1,\ldots,\omega_n)\qquad \text{such that}\qquad\tS_\lambda\mapsto \omega_\lambda.
$$
\end{coro}

\subsection{Geometric construction of the $\star$-product}
\label{subs:geom-star-prod}
In this section we give a geometric construction of the $\star$-product $H^*_\rmT(\calG_a)\times H^*_\rmT(\calG_b)\to H^*_\rmT(\calG_{a+b})$ from \eqref{staralg} for any $a,b\in[1;n]$ such that  $a+b\leqslant n$. The construction will use some extra hermitian scalar product which we introduce first.

Fix now a (hermitian) scalar product on $V$. Let $\rmU\subset \rmG=\rm{GL}(V)$ be the maximal compact subgroup formed by unitary transformations. 

\begin{rk}
\label{rk:G-vs-U}
Since there is no difference between $\rmG$-equivariant and $\rmU$-equivariant (co)homology, we will, to simplify notation, often write $\rmG$-equivariant (co)homology even if the variety is only $\rmU$-stable (and we really mean $\rmU$-equivariance). This is particularly convenient, since it is enough to find a $\rmU$-invariant isomorphism between two $\rmG$-varieties to be able to identify their $\rmG$-equivariant (co)homology. It is not even necessary to have an algebraic isomorphism, a homeomorphism suffices.
Similarly, for $\overline{\rm T}=\rmT\cap\rmU$ . When we mean $\overline{\rm T}$-equivariant (co)homology, we may write $\rmT$-equivariance  
instead $\overline{\rm T}$-equivariance, even if the variety does not carry a $\rmT$- action.
\end{rk}

Fix $k_1,k_2\in\bbZ$ such that $n\geqslant k_1\geqslant k_2\geqslant 0$.  Let $\calG_{k_1\supset k_2}$ denote the $3$-step partial flag variety in $V$ with subspaces of dimension $k_1$ and $k_2$. In other words,  $\calG_{k_1\supset k_2}$ is the subvariety of $\calG_{k_1}\times \calG_{k_2}$ where the subspaces are contained in each other.

\begin{lem}
\label{lem:perm-flags-3}
There is a diffeomorphism $\gamma_{k_1,k_2}:\calG_{k_1\supset k_2}\cong \calG_{k_1\supset k_1-k_2}$.
\end{lem}
\begin{proof}
Using the scalar product on $V$, we can define the following map
$$
\calG_{k_1\supset k_2}\to \calG_{k_1\supset k_1-k_2},\qquad (V\supset W\supset \tW\supset \{0\})\mapsto (V\supset W\supset \tW^\perp\cap W\supset \{0\}).
$$
In other words, given a flag $(V\supset W\supset \tW\supset \{0\})$ in $\calG_{k_1\supset k_2}$, we can replace the $k_2$-dimensional subspace $\tW$ by its orthogonal complement in $W$ and obtain a flag in $\calG_{k_1\supset k_1-k_2}$. 
This obviously defines a diffeomorphism.
\end{proof}

\begin{rk}
This diffeomorphism is only $\rmU$-invariant, and not 
$\rmG$-invariant. It still induces by   \Cref{rk:G-vs-U} an isomorphism in $\rmG$- or $\rmT$-equivariant cohomology. 

\end{rk}

\begin{df}
Let $\iota_{k_1,k_2}\colon H^*_\rmT(\calG_{k_1-k_2})\to H^*_\rmT(\calG_{k_1\supset k_2})$ be the  composition
$$
H^*_\rmT(\calG_{k_1-k_2})\stackrel{p^*}\longrightarrow H^*_\rmT(\calG_{k_1\supset k_1-k_2}) \xrightarrow{\gamma_{k_1,k_2}}\ H^*_\rmT(\calG_{k_1\supset k_2}),
$$
where the first map is the pull-back with respect to $p\colon\calG_{k_1\supset k_1-k_2}\to \calG_{k_1-k_2}$ which forgets one component of the flag.
Assume now $a,b\in[1;n]$. Pick integers $n\geqslant k_1\geqslant k_2 \geqslant k_3\geqslant 0$ such that $k_1-k_2=a$, $k_2-k_3=b$. We then have the inclusions
\begin{equation}\label{inclusions}
\begin{gathered}
\iota_{k_1,k_2}\colon H^*_\rmT(\calG_a)\to H^*_\rmT(\calG_{k_1\supset k_2}), \qquad \iota_{k_2,k_3}\colon H^*_\rmT(\calG_b)\to H^*_\rmT(\calG_{k_2\supset k_3}) 
\\
 \iota_{k_1,k_3}\colon H^*_\rmT(\calG_{a+b})\to H^*_\rmT(\calG_{k_1\supset k_3}).
\end{gathered}
\end{equation}
\end{df}
We view $\rmT$-fixed points in $\calG_{k_1\supset k_2}$ as $\rmT$-fixed point of $\calG_{k_1}\times \calG_{k_2}$. They are of the form $(g_{\mu_1},g_{\mu_2})$, where $\mu_1\in \Lambda_{k_1}(n)$, $\mu_2\in \Lambda_{k_2}(n)$ and $\mu_1\supset \mu_2$. Let us write $g_{\mu_1,\mu_2}$ instead of $(g_{\mu_1},g_{\mu_2})$ and similarly for $\calG_{k_2\supset k_3}$ and $\calG_{k_1\supset k_3}$. 
\begin{prop}
The usual convolution product 
$$
H^*_\rmT(\calG_{k_1\supset k_2})\times H^*_\rmT(\calG_{k_2\supset k_3})\to H^*_\rmT(\calG_{k_1\supset k_3})
$$
with respect to the inclusions
$$
\calG_{k_1\supset k_2}\subset \calG_{k_1}\times \calG_{k_2},\qquad \calG_{k_2\supset k_3}\subset \calG_{k_2}\times \calG_{k_3},\qquad \calG_{k_1\supset k_3}\subset \calG_{k_1}\times \calG_{k_3}
$$
restricts along \eqref{inclusions} to the $\star$-product  
$H^*_\rmT(\calG_a)\times H^*_\rmT(\calG_b)\to H^*_\rmT(\calG_{a+b})$ from \eqref{staralg}. 
\end{prop}
\begin{proof}
This follows directly from \Cref{redwine} below.
\end{proof}
\begin{lem}\label{redwine}
In the notation from \eqref{inclusions} we have for $\oX\in H^*_\rmT(\calG_a)$, $\oY\in H^*_\rmT(\calG_b)$,
\begin{equation*}
(\iota_{k_1,k_2}\oX)\star (\iota_{k_2,k_3}\oY)=\iota_{k_1,k_3}(\oX\star \oY).
\end{equation*}
\end{lem}
\begin{proof}
It is enough to verify this on the $\rmT$-fixed points. For any $\mu_1\in \Lambda_{k_1}(n)$, $\mu_2\in \Lambda_{k_2}(n)$ and $\mu_3\in \Lambda_{k_3}(n)$ we have
$$
(\iota_{k_1,k_2}\oX)(\mu_1,\mu_2)=\oX(\mu_1\backslash\mu_2) \qquad \mbox{and} \qquad \op{eu}(\calG_{k_1\supset k_2},g_{\mu_1,\mu_2})=\tA_{\mu_1}\catP_{\mu_1\backslash\mu_2,\mu_2} 
$$
with $\catP_{\mu_1\backslash\mu_2,\mu_2}$ as in \eqref{staralg}. Thus we can write 
$$
\iota_{k_1,k_2}\oX={\sum_{\mu_1\supset\mu_2}}\tA^{-1}_{\mu_1}\catP^{-1}_{\mu_1\backslash\mu_2,\mu_2}\oX(\mu_1\backslash\mu_2)[g_{\mu_1,\mu_2}],
$$
and similarly 
$$
\iota_{k_2,k_3}\oY={\sum_{\mu_2\supset\mu_3}}\tA^{-1}_{\mu_2}\catP^{-1}_{\mu_2\backslash\mu_3,\mu_3}\oY(\mu_2\backslash\mu_3)[g_{\mu_2,\mu_3}].
$$
Then, using $[g_{\mu_1,\mu_2}]\star [g_{\mu_2,\mu_3}]=\tA_{\mu_2}[g_{\mu_1,\mu_3}]$, which holds by \Cref{lem:comp-in-loc}~\ref{2}.), we get
$$
(\iota_{k_1,k_2}\oX)\star (\iota_{k_2,k_3}Y)={\sum_{\mu_1\subset\mu_2\subset \mu3}}\tA^{-1}_{\mu_1}\catP^{-1}_{\mu_1\backslash\mu_2,\mu_2}\catP^{-1}_{\mu_2\backslash\mu_3,\mu_3}\oX(\mu_1\backslash\mu_2)Y(\mu_2\backslash\mu_3)[g_{\mu_1,\mu_3}].
$$
On the other hand, we have 
\begin{align*}
\iota_{k_1,k_3}(\oX\star \oY)=&&&\sum_{\mu_1\subset\mu_3}\tA^{-1}_{\mu_1}\catP^{-1}_{\mu_1\backslash\mu_3,\mu_3}(\oX\star \oY)(\mu_1\backslash\mu_3)[g_{\mu_1,\mu_3}]\\
=&&&{\sum_{\mu_1\subset\mu_2\subset\mu_3}}\tA^{-1}_{\mu_1}\catP^{-1}_{\mu_1\backslash\mu_3,\mu_3}\catP^{-1}_{\mu_1\backslash\mu_2,\mu_2\backslash\mu_3}\oX(\mu_1\backslash\mu_2)\oY(\mu_2\backslash\mu_3)[g_{\mu_1,\mu_3}].
\end{align*}
The statement follows then from the following obvious equality 
\begin{eqnarray*}
\catP_{\mu_1\backslash\mu_2,\mu_2}\catP_{\mu_2\backslash\mu_3,\mu_3}=\catP_{\mu_1\backslash\mu_2,\mu_2\backslash\mu_3}\catP_{\mu_1\backslash\mu_2,\mu_3}\catP_{\mu_2\backslash\mu_3,\mu_3}=\catP_{\mu_1\backslash\mu_2,\mu_2\backslash\mu_3}\catP_{\mu_1\backslash\mu_3,\mu_3}.\hfill\qedhere
\end{eqnarray*}
\end{proof}
\subsection{The isomorphism $H^G_*(\ccZ)\cong \hhNH_n$}
We now are able to give a geometric construction of the doubly extended nil-Hecke algebra:  
\begin{thm}\label{doubly}
There is an isomorphism of algebras $H^G_*(\ccZ)\cong \hhNH_n$.
\end{thm}
\begin{proof}
The algebra $\hhNH_n$ acts faithfully on $\hPol_n$ by \Cref{df:hhNH}, and $H_*^\rmG(\ccZ)$ acts faithfully on $H_*^\rmG(\ccY)\cong \hPol_n$, by \Cref{lem:comp-in-loc} with  \Cref{rk:faith-from-strat}. 
Moreover, for each generator of the algebra  $\hhNH_n$ we have constructed 
in Propositions \ref{TrasDemazure}, \ref{lem:creation+n}, \ref{lem:creation+r}, 
\ref{lem:annihilation-1}, \ref{lem:annihilation-r},
 an element of $H_*^\rmG(\ccZ)$ acting on the polynomial representation by the same operator. This implies immediately that there is an injective homomorphism $\hhNH_n\hookrightarrow H^G_*(\ccZ)$. To prove that this is an isomorphism, we check that the graded dimensions of the algebras agree.

We already know by \Cref{prop:geom-NH},  that the graded dimensions of $\NH_n$ and $H^G_*(\calF\times\calF)$ agree. Moreover, by  \Cref{coro:decomp-hhNH}, we have an isomorphism of graded vector spaces $\hhNH_n\cong \NH_n\otimes \mywedge_n\otimes \mywedge_n$. Since  $\ccZ=\calF\times\calF\times\calG\times\calG$ and $H^*(\calG)\cong \mywedge_n$ as graded vector spaces, the graded dimensions of $H_*^\rmG(\ccZ)$ and $\hhNH_n$ agree. 
\end{proof}
\subsection{The Grassmannian quiver Hecke algebra}
\label{subs:geom-hNH}
We next construct the extended nil-Hecke algebra using the following \emph{Grassmannian--Steinberg variety}:
\begin{nota}\label{Zplus}
Set $\ccZplus=\calF\times\calF\times \calG=\coprod_{k=0}^n \ccZplus_k$, where $\ccZplus_k=\calF\times\calF\times \calG_k$.
\end{nota}
Our goal is to turn $H_*^\rmG(\ccZplus)$ into an algebra isomorphic to $\hNH_n\subset\hhNH_n$. 

Let $k\in[0;n]$ and fix integers $n\geqslant k_1\geqslant k_2\geqslant 0$ such that $k_1-k_2=k$ and recall the partial flag varieties from \S\ref{subs:geom-star-prod}. Consider the following diagram
\begin{equation*}
\ccZ_{k_1,k_2}\stackrel{\x}{\longleftarrow}\calF\times\calF\times \calG_{k_1\supset k_2}\stackrel{\y}\longrightarrow \calF\times\calF\times \calG_{k_1\supset k} \stackrel{\z}{\longrightarrow} \ccZplus_k,
\end{equation*}

Here, $\x$ and $\z$ are the obvious maps coming from $\calG_{k_1\supset k_2}\subset \calG_{k_1}\times \calG_{k_2}$ and from forgetting the $k_1$-dimensional flag component, and $\gamma$ is the diffeomorphism induced by $\gamma_{k_1,k+2}$ from \Cref{lem:perm-flags-3}. 
\begin{lem}
The linear map 
\begin{equation*}
\iota_{k_1, k_2}\colon H_*^\rmG(\ccZplus_k)\to H_*^\rmG(\ccZ_{k_1, k_2}), \quad \iota_{k_1, k_2}=\x_*\y^*\z^*
\end{equation*}
is injective and thus identifies $H_*^\rmG(\ccZplus_k)$ with a subspace of $H_*^\rmG(\ccZ_{k_1, k_2})$. 
\end{lem}
\begin{proof}
Indeed, $\z^*$ is injective, because $\z$ is a Grassmannian bundle. The map $\x_*$ is injective, because the push-forward of the inclusion $\calG_{k_1\supset k_2}\to \calG_{k_1}\times \calG_{k_2}$ is injective (we can see this using a cell decomposition).
\end{proof}

Taking the direct sum by all possible $k_1$ and $k_2$ such that $n\geqslant k_1\geqslant k_2\geqslant 0$, we get an identification of $H_*^\rmG(\ccZplus)$ with a vector subspace of $H_*^\rmG(\ccZ)$. We would like to show that this vector subspace is a subalgebra.

For this, we use that $H_*^\rmG(\ccZ)$ acts faithfully on $H_*^\rmG(\ccY)$. Abusing the terminology, we will say that $H_*^\rmG(\ccZplus)$ \emph{acts} on $H_*^\rmG(\ccY)$, by which we just mean the restriction of the $H_*^\rmG(\ccZ)$-action to the subspace $H_*^\rmG(\ccZplus)$. To relate the multiplication in $H_*^\rmG(\ccZ)$ with the $\star$-product from \S\ref{subs:geom-hNH}, we identify $H_*^\rmG(\ccY)\cong H_*^\rmT(\calG)$ as in \Cref{lem:isom-FG/G-G/T}. 

\begin{prop}The following diagram, relating the $H_*^\rmT(\calG\times \calG)$-action on $H_*^\rmT(\calG)$ at the top to the convolution action of $H_*^\rmG(\ccZ)$ on $H_*^\rmG(\ccY)$ the bottom, commutes.
$$
\begin{CD}
H_*^\rmT(\calG\times \calG)\times H_*^\rmT(\calG)@>>>H_*^\rmT(\calG)\\
@VV{\cong {\mbox{ \rm \eqref{iso-FG/G-G/T}}}}V @VV{\cong {\mbox{ \rm \eqref{iso-FG/G-G/T}}}}V\\
H_*^\rmG(\calF\times\calG\times \calG)\times H_*^\rmG(\ccY)@.H_*^\rmG(\ccY)\\
@VV{\Delta_*}V @VV\Id V\\
H_*^\rmG(\ccZ)\times H_*^\rmG(\ccY)@>>>H_*^\rmG(\ccY).\\
\end{CD}
$$
Here 
$
\Delta\colon \calF\times\calG\times \calG\to\calF\times\calF\times\calG\times \calG=\ccZ
$ 
is the diagonal inclusion.
\end{prop}
\begin{proof}
The statement follows from the commutativity of the localised diagram
$$
\begin{CD}
H_*^\rmT(\calG\times \calG)_{\rm loc}\times H_*^\rmT(\calG)_{\rm loc}@>>>H_*^\rmT(\calG)_{\rm loc}\\
@VVV @VVV\\
H_*^\rmT(\calF\times\calG\times \calG)_{\rm loc}\times H_*^\rmT(\ccY)_{\rm loc}@.H_*^\rmT(\ccY)_{\rm loc}\\
@VVV @VVV\\
H_*^\rmT(\ccZ)_{\rm loc}\times H_*^\rmT(\ccY)_{\rm loc}@>>>H_*^\rmT(\ccY)_{\rm loc}.\\
\end{CD}
$$
which can be checked on the basis of $\rmT$-fixed points using \Cref{rk:loc-X-vs-FX}.
\end{proof}
Projection onto the factors directly implies the following refinement:
\begin{coro}
\label{coro:diag-Zplus-act}
The following diagram commutes
$$
\begin{CD}
H_*^\rmT(\calG_a)\times H_*^\rmT(\calG_b)@>>> H_*^\rmT(\calG_{a+b})\\
@VV{\cong {\mbox{ \rm \eqref{iso-FG/G-G/T}}}}V                                    @VV{\cong {\mbox{ \rm \eqref{iso-FG/G-G/T}}}}V\\
H_*^\rmG(\ccY_a)\times H_*^\rmG(\ccY_b)@. H_*^\rmG(\ccY_{a+b})\\
@VV{\Delta_*}V                               @VV\Id V\\
H_*^\rmG(\ccZplus_a)\times H_*^\rmG(\ccY_b)@. H_*^\rmG(\ccY_{a+b})\\
@VV{\iota_{a+b,b}}V                               @VV\Id V\\
H_*^\rmG(\ccZ_{a+b,b}) \times H_*^\rmG(\ccY_b)@>>> H_*^\rmG(\ccY_{a+b}),\\
\end{CD}
$$
where the horizontal maps are the $\star$-product respectively the $H_*^\rmG(\ccZ)$-action on $H_*^\rmG(\ccY)$, and $\Delta\colon \ccY_a\to \ccZplus_a$ denotes the diagonal inclusion.
\end{coro}
As a direct consequence we obtain the geometric creation operators. For this let  $\omega_\lambda$ be the push-forward of $\omega_\lambda\in H_*^\rmG(\ccY)$ to $H_*^\rmG(\ccZplus)$ with respect to the diagonal inclusion 
$
\ccY=\calF\times\calG\to \ccZplus=\calF\times\calF\times\calG.
$ The notation is consistent:
\begin{coro}\label{actionomega}
Acting with $\omega_\lambda\in H_*^\rmG(\ccZplus)$ on $H_*^\rmG(\ccY)$   multiplies by $\omega_\lambda\in\mywedge_n$.
\end{coro}

\begin{prop}
\label{lem:basisHZplus}
The vector subspace $H_*^\rmG(\ccZplus)$ of $H_*^\rmG(\ccZ)$ has basis
$$
\{T_{w}X^{\bfa}\omega_\lambda; w\in \frakS_n, \bfa\in\bbZ_{\geqslant 0}^n,\lambda\in \Lambda(n)\}.
$$
In particular, this vector subspace is a subalgebra.
\end{prop}
\begin{proof}
The statement on the basis follows by standard arguments based on the cellular fibration lemma, see \cite[\S 5.5]{CG97}. The consequence holds then by \Cref{coro:diag-Zplus-act} and \Cref{actionomega}.
\end{proof}
We call $H_*^\rmG(\ccZplus)$ the \emph{Grassmannian quiver Hecke algebra} (attached to  $\mathfrak{sl}_2$ and $n$).
\subsection{The Naisse-Vaz algebra as Grassmannian quiver Hecke algebra}
The Grassmannian quiver Hecke algebra relates to the quiver Hecke algebra from \cite{VV} as the Naisse--Vaz algebra relates to the KLR algebras from \cite{KL}, \cite{Rou2KM}:
\begin{thm}
\label{thm:main-sl2}
The algebra $H_*^\rmG(\ccZplus)$ is isomorphic to $\hNH_n$.
\end{thm}
\begin{proof}
Both algebras act faithfully on $H_*^\rmG(\ccY)\cong \hPol_n$ (see  \Cref{lem:comp-in-loc} with  \Cref{rk:faith-from-strat}, and \Cref{df:hNH}). By Propositions \ref{TrasDemazure}, \ref{lem:creation+n}, \ref{lem:creation+r}, 
$\hNH_n$ embeds into $H_*^\rmG(\ccZplus)$. Now, compare \Cref{lem:basisHZplus} with \Cref{coro:basis-hNH}.
\end{proof}
\begin{rk}\label{shiftedgrading}
The isomorphism in \Cref{thm:main-sl2} becomes compatible with the grading from \Cref{rk:gr-hNH} if we use the \emph{shifted grading} on $H_*^\rmG(\ccZplus)$ with   $H_r^\rmG(\ccZplus_k)$ in degree $2\dim\ccY_k+k^2-k-r$. The correction $k^2-k$ encodes that $\omega_n\wedge\ldots\wedge\omega_{n-k+1}$, which corresponds to the fundamental class $[\ccY_k]$, has degree $k^2-k$.
\end{rk}
\begin{rk}\label{Warning}
The name \emph{Grassmannian quiver Hecke algebra} indicates that we extend the quiver Hecke algebra by extra Grassmannians. This notion should not be confused with the notion of quiver Grassmannians, see e.g. \cite{Feigin}, \cite{Irelli}, which would mean that we change the full flags to (Grassmannian) partial flags. The Grassmannian quiver Hecke algebra for $\mathfrak{sl}_2$, is a \emph{subalgebra} of $H^G_*(\ccZ)$.
\end{rk}


\section{Geometric construction of the differentials}
\label{sec:diff-geom}
In this section we construct the differential from \Cref{lem:dg} geometrically. 
\subsection{The differential $d_N$ on $\hPol_n$}
Recall the differential $d_N\colon \hPol_n\to\hPol_n$ from \Cref{defdN}.  We now reconstruct this differential geometrically using the push-forward 
$\tZ_{k,N}\in H_*^\rmG(\ccY\times\ccY)$ of $(-1)^k[\ccZ_{\subse,k}]\cap \xi^N\in H_*^\rmG(\ccZ_{\subse,k})$, see \eqref{Grassindopp}.
\begin{prop}
\label{lem:zN-acts-dN}
The element $\tZ_{k,N}$ acts as $d_N\colon H_*^\rmG(\ccY_{k+1})\to H_*^\rmG(\ccY_{k})$.
\end{prop}
\begin{proof}
\Cref{lem:annihilation-1} gives the statement for $N=0$. It also implies that the action satisfies the graded  Leibniz rule, see \S\ref{subs:DG-on-hNH}, and commutes\footnote{This is geometrically obvious: the correspondences deal with different factors of the varieties.}  with the action of the nil-Hecke algebra (i.e. with $\partial_1,\ldots,\partial_{n-1}$ and $X_1,\ldots, X_n$).
 By \Cref{defdN} it remains to show that $\tZ_{k,N}$ acts on $\omega_1$ by multiplication with $X_1^N$. This holds thanks to the extra twist $\xi^N$ in comparison to the $N=0$ case. 
\end{proof}

\begin{rk} Algebraically, $d_N$ is defined via the action of $\mathbb{d}_N\in \hhNH_n$ which is a $\Pol_n$-linear combination of the $\omega^-_r$'s. Geometrically, the situation is reversed and $d_N$ appears more natural than the $\omega^-_r$'s. \Cref{lem:annihilation-r} and \Cref{lem:zN-acts-dN} express then $\omega^-_r$ as $\Pol_n$-linear combination of the various differentials $d_N$. 
\end{rk}
To interpret $\bfd_N$ we incorporate the twist by $\xi^N$ into the construction. 
\subsection{The differential $\bfd_N$ on $\hNH$}
For the rest of this section fix $N\in\mathbb{Z}$. We first build a dependence on $N$ into our varieties. 
For this, consider the affine space $\mA=\Hom(V,\bbC^N)$ and define the following upgraded versions (all depending on $N$) 
\begin{equation}
\hY=\{(\bV,W,\gamma)\mid\gamma(W)=0\}\subset\ccY\times\mA\quad \text{and}\quad\hZ=\hY\times_\mA \hY
\end{equation}
of the extended flag varieties $\ccY$ and Steinberg variety $\ccZ$. Define for $k, k_1,k_2\in[1;n]$
\begin{equation*}
\hY_k=\hY\cap (\ccY_k\times \mA)\quad \text{and}\quad\hZ_{k_1,k_2}=\hY_{k_1}\times_\mA \hY_{k_2}.
\end{equation*}
Let $\ccZ_{\supset}$ be the subvariety of $\ccZ$ given by $W\supset \tW$ and set $\ccZ_{k_1\supset k_2}=\ccZ_\supset\cap \ccZ_{k_1,k_2}$. Let us define $\hZ_{\supset}$, $\hZ_{k_1\supset k_2}$, $\hZ_{\subse,k}$, $\hZplus_k=\hZ_{k,0}$ and $\hZplus=\coprod_{k=0}^n \hZplus_{k}$ as preimages of $\ccZ_{\supset}$, $\ccZ_{k_1\supset k_2}$, $\ccZ_{\subse,k}$, $\ccZplus_k=\ccZ_{k,0}$ and $\ccZplus=\coprod_{k=0}^n \ccZplus_{k}$ respectively with respect to the morphism $\hZ\to \ccZ$ which forgets the $\mA$-factor.

We equip $H_*^\rmG(\hZ)$  with the convolution product using the inclusion $\hZ\subset \hY\times \hY$. 
The algebras $H_*^\rmG(\hZ)$ and $H_*^\rmG(\ccZ)$ are usually different. Their interplay will be crucial, see \Cref{lem:same-prod-supset} and \Cref{differentials}.

\begin{prop}\label{lem:same-prod-supset}
The pull-back  $H_*^\rmG(\ccZ_\supset)\to H_*^\rmG(\hZ_\supset)$ is an algebra isomorphism. 
\end{prop}

\begin{proof}
The morphism $\hY_{k}\to \ccY_{k}$ is a vector bundle with the fibre over $(\bV,W)$ being $\Hom(V/W,\bbC^N)$. Similarly, the morphism $\hZ_{k_1\supset k_2}\to \ccZ_{k_1\supset k_2}$ is a vector bundle with the fibre $\Hom(V/W,\bbC^N)$ over $(\bV,\tilde\bV,W,\tW)$.  For $\mu\in \Lambda_k(n)$, let $W_
\mu\subset V$ be the vector space representing the $\rmT$-fixed point $g_\mu\in \calG_k$. Let $\catP_\mu\in \tR_\rmT \simeq \Bbbk[\ccT_1,\ldots\ccT_n]$ be the Euler class of $\Hom(V/W_\mu,\bbC^N)$. 	
	
We compare the products in the two algebras using the torus fixed point localisation. The $\rmT$-fixed points of $\ccY$ are $(w,\mu):=(f_w, g_\mu)$, $w\in\frakS_n$, $\mu\in \Lambda(n)$,  and of $\ccZ$ they are $(w_1,w_2,\mu_1,\mu_2):=(f_{w_1},f_{w_2},g_{\mu_1},g_{\mu_2})$. 
(Note we use a different labelling $(w,\mu)\not=x_{w,\mu}$ as before; now  $\mu$ is not twisted by $w$.)  The point $(w_1,w_2,\mu_1,\mu_2)$ is in $\ccZ_\supset$ if and only if $\mu_1\supset \mu_2$. The maps $\hZ_\supset\to \ccZ_\supset$ and $\hY\to \ccY$ induce bijections on $\rmT$-fixed points.
We can write $h\in H_*^\rmG(\ccZ_\supset)$ and its pull-back $\hh \in H_*^\rmG(\hZ_\supset)$ as 
\begin{eqnarray*}
h=\!\!\!\sum_{w_1,w_2,\mu_1,\mu_2}\!\!\! h_{w_1,w_2,\mu_1,\mu_2}[(w_1,w_2,\mu_1,\mu_2)],\;\;\hh=\!\!\!\sum_{w_1,w_2,\mu_1,\mu_2}\!\!\!\hh_{w_1,w_2,\mu_1,\mu_2}[(w_1,w_2,\mu_1,\mu_2)]%
\end{eqnarray*} 
for some $h_{w_1,w_2,\mu_1,\mu_2}\in \Bbbk(\ccT_1,\ldots,\ccT_n)$ and then $\hh_{w_1,w_2,\mu_1,\mu_2}=\catP^{-1}_{\mu_1} h_{w_1,w_2,\mu_1,\mu_2}$.

The coefficient of $[(w_1,w_3,\mu_1,\mu_3)]$ in $[(w_1,w_2,\mu_1,\mu_2)]\star [(w_2,w_3,\mu_2,\mu_3)]$, viewed as element in $H_*^\rmT(\hZ_\supset)_{\rm loc}$, agrees now up to the factor $\catP_{\mu_2}$ with the corresponding coefficient in $H_*^\rmT(\ccZ_\supset)_{\rm loc}$. Since  $\catP^{-1}_{\mu_1}\catP^{-1}_{\mu_2}\catP_{\mu_2}=\catP^{-1}_{\mu_1}$, the pull-back map preserves the multiplication, i.e. $H_*^\rmG(\ccZ_\supset)\to H_*^\rmG(\hZ_\supset)$ is an algebra isomorphism.
\end{proof}

\begin{rk}\label{vbrk}
\Cref{lem:same-prod-supset} is in general wrong if we remove the $\supset$-condition: The map $\hZ_{k_1,k_2}\to \ccZ_{k_1,k_2}$ is not a vector bundle; its fibre over $(\bV,\tilde\bV,W,\tW)$ is $\Hom(V/(W+\tW),\bbC^N)$ whose dimension is not even locally constant. The restriction to the part satifying the $(W\supset \tW)$-condition is however a vector bundle.
\end{rk}

In \Cref{lem:basisHZplus} we identified $H_*^\rmG(\ccZplus)$ with a subalgebra of $H_*^\rmG(\ccZ)$ using 
$$
\ccZplus_{k_1-k_2}\leftarrow \ccZ_{k_1\supset k_1-k_2}\simeq \ccZ_{k_1\supset k_2}\to \ccZ_{k_1,k_2}.
$$
Similarly, we can identify 
$H_*^\rmG(\hZplus)$ with a subalgebra of $H_*^\rmG(\hZ)$ using the upgrades
$$
\hZplus_{k_1-k_2}\leftarrow \hZ_{k_1\supset k_1-k_2}\simeq \hZ_{k_1\supset k_2}\to \hZ_{k_1,k_2}.
$$
The arguments from \Cref{lem:same-prod-supset} show that the two algebras agree: 
\begin{coro}\label{differentZs}
The pull-back $H_*^\rmG(\ccZplus)\to H_*^\rmG(\hZplus)$ is an isomorphism of algebras.
\end{coro}
We now use\footnote{The cohomology class used for this will live in the $\subset$-part of the algebra (not in the $\supset$- part). That is why the replacement of $\ccZ$ by $\hZ$ was important!}  $H_*^\rmG(\hZplus)$ to give a geometric construction of $\mathbb{d}_N\in \hhNH_n$. 

\begin{prop}\label{differentials}
The action of the fundamental class $[\hZ_{\subse,k}]\in H_*^\rmG(\hZplus)$ agrees with the differential $(-1)^k d_N: H_*^\rmG(\hY_{k+1})\to H_*^\rmG(\hY_k)$.
\end{prop}
\begin{proof}
The case $N=0$ holds by \Cref{lem:annihilation-1}. For general $N$ we get an extra multiplication by a power of the first Chern class of the line bundle given by the quotient $\tW/W$ of the Grassmannian spaces. Now argue as for \Cref{lem:zN-acts-dN}.
\end{proof}

\subsection{A geometric construction of cyclotomic nil-Hecke algebras}\label{geomcycNH}
We now construct $\NH_n^N$, see  \Cref{prop:resol-cycl-NH},  geometrically. 
\begin{nota}\label{shadow}
We abbreviate $\hF=\calF\times\mA$ and consider $\hF\times_\mA\hF=\calF\times\calF\times\mA$. Denote by $\mI\subset\mA$ the subvariety of injections in $\mA=\Hom(V,\bbC^N)$.
\end{nota}
Identify in this section $\NH_n=H_*^\rmG(\hF\times_\mA\hF)=H_*^\rmG(\calF\times\calF\times\mA)$. 
\begin{thm}\label{cycHecke}
	The pull-back map  
	\begin{equation}
	\label{eq:restr-to-Inj}
	H_*^\rmG(\calF\times\calF\times\mA)\to H_*^\rmG(\calF\times\calF\times \mI)
	\end{equation}
	is surjective and its kernel coincides with the kernel of $\NH_n\to \NH_n^N.$ In particular we get an isomorphism of algebras $\NH_n^N\cong H_*^\rmG(\calF\times\calF\times \mI)$.
\end{thm}

\begin{proof} If $n>N$, then the statement is trivial, because $\NH_n^N$ is zero and $\mI$ is empty. Let us assume $n\le N$. We first claim that $X_1^N$ is in the kernel of \eqref{eq:restr-to-Inj}. By definition, 
the element $X_1\in H_*^\rmG(\calF\times\calF\times \mA)$ is the first Chern class of the line bundle on $\calF\times\mA$ given by the first subspace of the flag, pushed forward to $H_*^\rmG(\calF\times\calF\times \mA)$ along the diagonal in $\calF\times\calF$. We now pull back. For $\gamma\in\mI$ we can identify, via $\gamma(\bV^1)\subset \bbC^N$, the restricted line bundle with a subbundle of a trivial bundle with fibre $\bbC^N$. By definition of $\mA$,  the $N$th power of this Chern class is zero and the claim follows.
Now we have  
\begin{equation}\label{freeact}
\dim H_*^\rmG(\calF\times\calF\times \mI)=\dim H_*(\calF\times\calF\times \mI/\rmG)=
\dim H_*(\calF\times\calF\times\op{Gr}_n(\bbC^N)), 
\end{equation}
since the $\rmG$-action is free by the definition of $\mI$.
This implies that the map in  \eqref{eq:restr-to-Inj} is surjective. It remains to show that the dimension of \eqref{freeact} equals $\dim\NH_n^N$. Since
$$
\dim\NH_n^N=n!\frac{N!}{(N-n)!}=(n!)^2\binom{N}{n}=\dim H_*(\calF\times\calF\times \op{Gr}_n(\bbC^N)), 
$$
by the basis theorem \cite[Thm. 2.34]{HuLiang}, see  \cite{Mathas} for an overview, we are done.
\end{proof}

\subsection{The geometric dg-model of the cyclotomic nil-Hecke algebras}
\label{subs:geom-mean-resol}
We encountered in \Cref{prop:resol-cycl-NH} a dg-model for $\NH_n^N$.  The dg-algebra $(\hNH_n,\bfd_N)$ has served there as a "resolution" of $\NH_n^N$. We  give now a geometric interpretation. 

\begin{nota}
For $k\in [0;n]$ let $\mA_k=\{\gamma\in \mA\mid \dim\Ker\;\gamma=k\}$, in particular, $\mA_0=\mI$. Let  $\IC(\overline{\mA}_k)$ be the simple perverse sheaf supported on the closure $\overline{\mA}_k$, whose restriction to $\mA_k$ is a shift on the trivial local system. 
\end{nota}
The variety $\overline{\mA}_k$ has a small resolution of singularities: 
\begin{equation*}
\mA\times \calG_k\supset\widetilde{\mA}_k=\{(\gamma,W)\mid \gamma(W)=0 \}\longrightarrow \overline{\mA}_k, \quad (\gamma,W)\mapsto\gamma.
\end{equation*}
Observe that $\widetilde{\mA}_k$ is reminiscent of the geometric construction from \S\ref{geomcycNH}. Indeed, we have $\hY_k=\calF\times \widetilde{\mA}_k$. So, the algebra $H_*^\rmG(\hZ)$ is the algebra of $(n!\times n!)$-matrices over the Ext- algebra of the intersection cohomology sheaf $\calL=\bigoplus_{k=0}^n \IC(\overline{X}_k)$. 

Recall the identification  $\hNH_n\cong H_*^\rmG(\ccZplus)\cong H_*^\rmG(\hZplus)$ from \Cref{thm:main-sl2} and \Cref{differentZs}. Now Propositions~\ref{differentials} and \ref{cycHecke} provide a geometric reformulation of the resolution in \Cref{prop:resol-cycl-NH}:
\begin{thm} \label{cyclogeo}
Consider the complex $H_*^\rmG(\hZplus_\bullet)$ supported in degrees, denoted $\bullet$, from $[0;n]$ with differential $\bfd_N$. This complex has homology concentrated in degree zero and this homology is isomorphic to $H_*^\rmG(\calF\times\calF\times \mI)$.
\end{thm}
 We work here with $\hZplus_k\cong\calF\times \calF\times \widetilde{\mA}_k$, but the statement  holds equivalently with $\calF\times\calF$ removed. Moreover, $H^*_\rmG(\widetilde \mA_k)\cong H^*_\rmG(\calG_k)$. By removing  $\calF\times\calF$ and replacing Borel--Moore homology by cohomology  we obtain the following
\begin{coro}\label{complex}
	There is a resolution of $H^*_\rmG(\mI)=H^*(Gr_n(\bbC^N))$ of the form  
	$$
	H^*_\rmG(\calG_n)\hookrightarrow H^*_\rmG(\calG_{n-1})\to\ldots\to H^*_\rmG(\calG_{1})\to H^*_\rmG(\calG_{0}).
	$$
\end{coro}

Using  $H^*_\rmG(\widetilde \mA_k)=H^*_\rmG(\IC(\overline \mA_k))$, this can be reformulated in terms of sheaves in the following way.

\begin{coro}\label{IC}
		There is a resolution of $H^*_\rmG(\mA_0)=H^*(Gr_n(\bbC^N))$ of the form  
	$$
	H^*_\rmG(\IC(\overline \mA_n))\hookrightarrow H^*_\rmG(\IC(\overline \mA_{n-1}))\to\ldots\to H^*_\rmG(\IC(\overline \mA_1))\to H^*_\rmG(\IC(\overline \mA_0)).
	$$
\end{coro}

\section{Geometric bases: Cells and Gells}
\label{sec:flags-cells}
In this section we introduce different decompositions of $\ccZplus=\calF\times\calF\times\calG$ which result in geometric constructions of different $\Pol_n$-bases of $\hNH_n$.

\subsection{Basic pavings}
\label{subs:B-strat}
Recall the isomorphism of algebras $\hNH_n\cong H_*^\rmG(\ccZplus)$ from \Cref{thm:main-sl2} using the varieties $\ccZplus$ from \Cref{Zplus}.  

Fix $k\in [0;n]$ and assume for this section that we have a finite decomposition 
\begin{equation}\label{dec}
\ccZplus_k=\coprod_{\theta\in \Theta}\calO_{\theta}
\end{equation}
with disjoint parts, labelled by some linearly ordered set $\Theta$. Set $\ccZplus_{k,\leqslant \theta}=\coprod_{\nu\leqslant \theta}\calO_{\nu}$.
\begin{df}
\label{def:b-paving}
We call the decomposition \eqref{dec} a \emph{basic paving} if the following holds:
\begin{enumerate}[\rm{(b-}i\rm{)}]
\item \label{b1}For each $\theta\in\Theta$, the subspace $\ccZplus_{k,\leqslant\theta}$ of $\ccZplus_k$ is closed.
\item \label{b2}Each $\calO_\theta$ is a smooth real\footnote{In the case of Gells studied below, $\calO_\theta$ is a priory a real and not a complex manifold. The fibres of $\calO_\theta\to \calF$ are however just complex affine spaces by \Cref{lem:cond-B-Gells}.} $\rmU$-stable manifold.
\item \label{cond-df:B-str-fibr} The projection $\calF\times\calF\times\calG_k\to\calF$ onto the first component restricted to $\calO_\theta$, is a locally trivial fibration 
$\calO_\theta\to\calF$ with (complex) affine fibres
(then the push-forward induces an isomorphism 
$H_*^\rmG(\calO_\theta)\cong H_*^\rmG(\calF)$).
\end{enumerate}
\end{df}
The word basic refers to the fact that it provides a basis in cohomology, namely assume the decomposition \eqref{dec} is a basic paving, then $H_*^\rmG(\ccZplus_k)$ is a free $\Pol_n$-module of rank $\lvert\Theta\rvert$.
Indeed, we get a filtration $H_*^\rmG(\ccZplus_k)^{\leqslant \theta}=H_*^\rmG(\ccZplus_{k,\leqslant\theta})$ on $H_*^\rmG(\ccZplus_k)$ such that each associated graded $H_*^\rmG(\ccZplus_k)^{\leqslant \theta}/H_*^\rmG(\ccZplus_k)^{<\theta}$ is isomorphic to $\Pol_n\cong H_*^\rmG(\calF)\cong H_*^\rmG(\calO_\theta)$.  Moreover, this filtration is stable by the left $\Pol_n$-action and, in view of \ref{cond-df:B-str-fibr}, the left $\Pol_n$-action on $H_*^\rmG(\ccZplus_k)^{\leqslant \theta}/H_*^\rmG(\ccZplus_k)^{<\theta}\cong H_*^\rmG(\calO_\theta)\cong \Pol_n$ is just the action by the left multiplication with polynomials. A $\Pol_n$-basis of the associate graded can then be lifted to a $\Pol_n$-basis in $H_*^\rmG(\ccZplus_k)$. 
\subsection{Adapted basic pavings}
From now on we identify $H_*^\rmG(\ccZplus)\cong\hNH_n$ via \Cref{thm:main-sl2} with the shifted grading from \Cref{shiftedgrading}.
We moreover fix a subset $\calB=\{b_\theta\mid \theta\in\Theta\}\subset H_*^\rmG(\ccZplus_k)$ of  homogeneous, pairwise distinct elements $b_\theta$. 

\begin{df} Assume \eqref{dec} is a basic paving. We call the decomposition \eqref{dec} \emph{strongly adapted} to $\calB$  if each $b_\theta$ is contained in the image of $H_*^\rmG(\ccZplus_{k,\leqslant\theta})$ and its restriction  to $H_*^\rmG(\calO_\theta)$ is a constant nonzero polynomial in $\Pol_n\cong H_*^\rmG(\calF)\cong H_*^\rmG(\calO_\theta)$.
\end{df}
\begin{lem} \label{strong}
	Assume there exists a  basic paving \eqref{dec} strongly adapted to $\calB$, 
	then $\calB$ is a $\Pol_n$-basis of $H_*^\rmG(\ccZplus_k)$.
\end{lem}
\begin{proof}
By the assumption, $\{b_\theta\}_{\theta\in\Theta}$ induces a $\Pol_n$-basis of the associated graded of $H_*^\rmG(\ccZplus_k)$, which implies that it is also a $\Pol_n$-basis in $H_*^\rmG(\ccZplus_k)$.
\end{proof}
Since strongly adaptedness is not easy to check we consider a weaker notion:\footnote{We denote by "$\dim$" always the complex dimension. In case of a real variety, "$\dim$" is half of the real dimension. This situation appears in the case of Gells in \S\ref{subsub:Gells}.}
\begin{df}
 We call a basic paving \eqref{dec} \emph{weakly adapted} to $\cB$ if, for each $\theta\in\Theta$,
 \begin{equation}\label{wadapt}
 \deg(b_\theta)=2\dim\ccY_k-2\dim\calO_\theta+ k^2-k.
 \end{equation}
\end{df}
By \Cref{shiftedgrading}, weakly adapted is indeed weaker than strongly adapted. 

\begin{lem}\label{checkbasis}
Assume that the basic paving \eqref{dec} is weakly adapted to $\cB$. If $\cB$ is $\Pol_n$-linearly independent or a $\Pol_n$-spanning set, then $\cB$ is a $\Pol_n$-basis of $H_*^\rmG(\ccZplus_k)$.
\end{lem}
\begin{proof}
For each $\theta\in\Theta$ we can find an element $b'_\theta\in H_*^\rmG(\ccZplus_{k,\leqslant\theta})\subset H_*^\rmG(\ccZplus_k)$ such that its restriction to $H_*^\rmG(\calO_\theta)\cong \Pol_n$ is $1$. Then \eqref{dec} is strongly adapted to the set $\calB'=\{b'_\theta\}_{\theta\in\Theta}$. Thus,  $\calB'$  is a $\Pol_n$-basis by \Cref{strong}. By \eqref{wadapt}, $b'_\theta$ and $b_\theta$ have the same degrees. By the assumptions, $\cB$ forms then a $\Pol_n$-basis as well.
\end{proof}

\subsection{Relative position of flags}
We will now construct some weakly adapted\footnote{We believe that they are in fact strongly adapted, but will not dicuss this.} pavings. 
First, we recall the notion of a relative position of two full flags in $V$. 
\begin{df}
Let $\bV, \tV\in\cF$.  Then $(\bV, \tV)$ is \emph{in relative position $w\in\frakS_n$}, denoted $w=\rel(\bV,\tV)$, if 
	$
	\dim (\bV^a\cap {\tV}{}^b)=\vert[1;a]\cap w([1;b])\rvert
	$ for all $a,b\in[1;n]$.
\end{df}

The relative position of partial flags can be reduced to this notion as follows.  Given a composition $\nu=(\nu_1,\nu_2,\ldots,\nu_t)$ of $n$, a partial flag $\bV$ in $V$ is of type $\nu$ if it is of the form
$
\bV=(\{0\}=\bV^0\subset \bV^1\subset\ldots \subset \bV^t=V)
$ 
with $\dim(\bV^r/\bV^{r-1})=\nu_r$ for $r\in [1;t]$.
Let $\frakS_\nu$ be the parabolic subgroups of $\frakS_n$ corresponding to $\nu$. 
We will use this notion for full flags (then $\nu=(1,1,\ldots,1)$ and $\frakS_\nu=\{\Id\}$) or for Grassmannians (then $\nu$ is of the form $(k,n-k)$ and $\frakS_\nu=\frakS_k\times\frakS_{n-k}$). 

Assume that $\bV$ is a partial flag of type $\nu$ and $\tV$ is a partial flag of type $\mu$. 
Complete them to full flags and let $w\in\frakS_n$ be their relative position. This depends on the chosen completions, but the double coset $\frakS_\nu w \frakS_\mu\in \frakS_\nu\backslash \frakS_n/\frakS_\mu$ does not and is by definition the relative position, $\rel(\bV,\tV)$, of $(\bV,\tV)$. We sloppily write $w=\rel(\bV,\tV)$.

\subsection{Cells and Gells}
We still fix $k\in [0;n]$. We equip $\frakS_n$ with the Bruhat order. We also fix the  induced ordering on any set $S$ of (double) cosets in $\frakS_n$ (obtained by embedding $S$ into $\frakS_n$ sending a coset to its representative of minimal length).
\label{subs:cells}
\subsubsection{Upper and lower cells}
\label{subsub:up-cells}
\begin{df}
Let $\Theta_{\rm up}$ be the set $\frakS_n\times (\frakS_n/(\frakS_k\times \frakS_{n-k}))$ with the lexicographic order on $\Theta_{\rm up}$. For $(w,\lambda)\in \Theta_{\rm up}$ define the \emph{upper cell}\footnote{Their images in $\calF\times\calF$ and $\cF\times\calG_k$ (for the left copy of $\calF$) are usual Schubert cells.}
\begin{equation}\label{upstrat}
\calO^{\rm up}_{w,\lambda}=\{(\bV,\tV, W)\mid \rel(\bV,\tV)=w, \rel(\bV,W)=\lambda\}\subset\ccZplus_k.
\end{equation}
\end{df}
For the following we identify $\frakS_n/(\frakS_k\times \frakS_{n-k})=\Lambda_k(n)$ via $x\mapsto x([1;k])$.
\begin{prop}
\label{prop:adapted-up} 
The decomposition $\ccZplus_k=\coprod_{(w,\lambda)\in \Theta_{\rm up}}\calO^{\rm up}_{w,\lambda}$ from \eqref{upstrat} is a basic paving and weakly adapted to $\calB^\up:=\{\omega_\lambda T_w\}_{w\in\frakS_n,\lambda\in\Lambda_k(n)}\subset\hNH_n$.
\end{prop} 
\begin{proof}
With the grading from \Cref{rk:gr-hNH}  and \Cref{shiftedgrading} we have
$$
\deg (T_w\omega_\lambda)=-2\ell(w)+2nk-k(k+1)-2\ell(\lambda).
$$
On the other hand
$
\dim\calO^\up_{w,\lambda}=\dim\calF+\ell(w)+\ell(\lambda),
$
where $\ell(\lambda)$ denotes the length of the shortest coset representative corresponding to $\lambda\in\Lambda_k(n)=\frakS_n/(\frakS_k\times \frakS_{n-k})$. Thus,
$$
2\dim\ccY_k+k^2-k-2\dim\calO^\up_{w,\lambda}=2\dim\calG_k+k^2-k-2\ell(w)-2\ell(\lambda)=\deg(T_w\omega_\lambda),
$$
and therefore \eqref{wadapt} holds which we wanted to show. 
\end{proof}
Varying over all $k$ gives what we call the \emph{geometric upper basis} of $H_*^\rmG(\ccZplus)$, since the diagrams of this basis have floating dots only on the top of the diagram. 
\begin{df} 
Denote by $\Theta_\down$ the set $\frakS_n\times(\frakS_k\times \frakS_{n-k})\backslash\frakS_n$ equipped with the lexicographic ordering. For $(w,\lambda)\in \Theta_\down^k$ define
\begin{equation}
\label{downstrat}
\calO^{\down}_{w,\lambda}=\{(\bV,\tV, W)\mid \rel(\bV,\tV)=w, \rel(W,\tV)=\lambda\}
\subset\ccZplus_k.
\end{equation}
\end{df}

Identifying $(\frakS_k\times \frakS_{n-k})\backslash\frakS_n=\Lambda_k(n)$,  $x\mapsto x^{-1}([1;k])$, the following holds.
\begin{prop}
The decomposition $\ccZplus_k=\coprod_{(w,\lambda)\in \Theta_{\down}}\calO^{\down}_{w,\lambda}$ from \eqref{downstrat} is a basic paving and weakly adapted to $\calB^\down:=\{T_w \omega_\lambda \}_{w\in\frakS_n,\lambda\in\Lambda_k(n)}\subset\hNH_n$.
\end{prop}
Varying over all $k$ gives what we call the \emph{geometric lower basis} of $H_*^\rmG(\ccZplus)$
\begin{rk}
	\label{rk:up-down-weak}
Since the dimensions and the degrees for the upper and lower basis agree, the basic paving  \eqref{upstrat} is weakly adapted to $\calB^\down$, similar for \eqref{downstrat} and $\calB^\up$. For strongly adaptedness this is different. 
 We expect that \eqref{upstrat} is strongly adapted to $\calB^\up$ and not to  $\calB^\down$; similarly for \eqref{downstrat} with respect to $\calB^\down$ (and not $\calB^\up$). \end{rk}

\subsubsection{$\calG$-cells} 
\label{subsub:W-cells}

So far we considered cells defined via relative positions with respect to the flag $\bV$ (upper cells) respectively $\tV$ (lower cells). We next construct basic pavings by  considering instead the relative position with respect to the Grassmannian subspaces $W$. For this we will introduce the notion of $\calG$-cells. 
\begin{df}\label{VW}
Assume $\bV\in \calF$ and $W\in \calG_k$ with $
\rel(\bV,W)=x\in \frakS_n/(\frakS_k\times \frakS_{n-k})$. View $x$ as element of $\frakS_n$ by taking the coset representative of minimal length, and let $\bV^W$ be the unique full flag such that $\rel(\bV,\bV^W)=x$.
\end{df}
\begin{rk} \label{rkVW} If 
$
\bV=(\bV^1\subset \bV^2 \subset\ldots\subset \bV^n=V) 
$
then $\bV^W$ is the full flag given by 
\begin{equation*}
\bV^W=(\bV^1\cap W\subset\cdots\subset \bV^{n-1}\cap W\subset W\subset \bV^1+W\subset \bV^2+W\subset \cdots\subset \bV^{n-1}+ W\subset V).
\end{equation*}
(The equalities appearing here must be removed to get a true $n$-step full flag.) Thus $\bV^W$ contains $W$ and it is "as close as possible" to $\bV$ amongst all such full flags. 
\end{rk}

\begin{df}\label{naivGstrat}
Set $\Theta_{\calG_k}=(\frakS_n/(\frakS_k\times \frakS_{n-k}))\times (\frakS_k\times \frakS_{n-k})\times ((\frakS_k\times \frakS_{n-k})\backslash \frakS_n)$ and define the \emph{$\calG$-cell} corresponding to $(x,y,z)\in \Theta_{\calG_k}$ as
\begin{equation*}
\calO^{\calG}_{x,y,z}=\{(\bV,\tV, W)\mid \rel(\bV,W)=x, \rel(\bV^W,\tV{}^W)=y, \rel(W, \tV)=z\}\subset\ccZplus_k.
\end{equation*}
We equip $\Theta_{\calG_k}$ with the twisted(!) lexicographic ordering which orders the triples  $(x,y,z)$ by comparing first $x$, then $z$, then $y$.
\end{df}

\begin{prop}\label{calcinGcells}
The decomposition $\ccZplus_k=\coprod_{(x,y,z)\in \Theta_{\calG_k}}\calO^{\calG}_{x,y,z}$  is a basic paving which is weakly adapted to $\calB^{\calG_k}$ defined as
\begin{equation}\label{twobases}
\{T_x \omega_1\omega_2\ldots\omega_k T_yT_z \}_{(x,y,z)\in \Theta_{\calG_k}} \quad\text{or}\quad \{T_x T_y \omega_1\omega_2\ldots\omega_k T_z \}_{(x,y,z)\in \Theta_{\calG_k}}. 
\end{equation}
\end{prop}

    Varying over all $k$ gives what we call\footnote{Using \Cref{checkbasis} one can see that they are indeed bases, justifying the name.} a \emph{$\calG$-cells basis} of $H_*^\rmG(\ccZplus).$

\begin{proof}
Only condition \ref{cond-df:B-str-fibr} for basic pavings is not obvious. It can be proved exactly as in \Cref{lem:cond-B-Gells} below which treats a harder situation. To show the weakly adaptedness note that for any two elements in  \eqref{twobases} with the same label the degrees agree. So it suffices to assume $\calB^{\calG_k}$  is the first set. We calculate 
\begin{equation*}
\deg(T_x \omega_1\omega_2\ldots\omega_k T_yT_z)=-2\ell(x)-2\ell(y)-2\ell(z)+2nk-k(k+1),
\end{equation*}
and
$
\dim\calO^{\calG}_{x,y,z}=\dim\calF+\ell(x)+\ell(y)+\ell(z).
$
Thus, $\deg (T_x \omega_1\omega_2\ldots\omega_k T_yT_z)$ equals 
\begin{equation*}
2\dim\calG_k+k^2-k-2\ell(x)-2\ell(y)-2\ell(z)=2\dim\ccY_k+k^2-k-2\dim \calO^\calG_{x,y,z}.
\end{equation*}
This proves the proposition.
\end{proof}

\subsubsection{Gells}
\label{subsub:Gells}
We next introduce an interesting twisted version of $\calG$-cells which we call \emph{Gells}. We give a special name to this notion because it will be very important in \S\ref{sec:gen-qv}, where we do general quivers. The choice of name comes from the fact that Gells are in some sense  "Grassmannian-based cells".

Using Gells we will construct a basic paving adapted to a basis containing only $\omega_1$ and no $\omega_2,\omega_3,\ldots,\omega_n$, cf. \Cref{onlyomega1}. Gells may seem to be unconvincing in case of $\mathfrak{sl}_2$, where we have many other, much nicer, bases of $\hNH_n$ available, see also \cite{AEHL}. For more general quivers only this complicated basis works and Gells will play a much more important role there. The definition of Gells requires to work with orthogonal complements of flags and actions of the group $\rmU$ instead of $\rmG$.

We fix a (hermitian) scalar product on $V.$ 
\begin{df}\label{VWperp}
Let  $\bV\in \calF$, $W\in \calG_k$. We denote by $\VWp$ the modification of the flag $\bV^W$ from \Cref{VW} where the part of the flag contained in $W$ is replaced by its dual flag inside of $W$, and the remaining part is kept; more precisely
\begin{eqnarray*}
\VWp&=&((\bV^{n-1}\cap W)^\perp\cap W\subset (\bV^{n-2}\cap W)^{\perp}\cap W\subset\cdots\subset (\bV^{1}\cap W)^{\perp}\cap W\\
&&\quad
\subset W\subset \bV^1+W\subset \bV^2+W\subset \cdots\subset \bV^{n-1}+ W\subset V).
\end{eqnarray*}
\end{df}
\begin{rk} \label{relVWp} If $\rel(\bV,W)=\rel(\bV, \bV^W)=x$, then  $\rel(\bV,\VWp)=xw_{0,k}$, where $w_{0,k}$ is the longest element in $\frakS_k$.
\end{rk}
\begin{df}
For $(x,y,z)\in \Theta_{\calG_k}$ define the corresponding \emph{Gell} as
\begin{equation*}
\calO^{\rm Gell}_{x,y,z}=\{(\bV,\tV, W)\mid \rel(\bV,W)=x;\; \rel(\VWp,\tV{}^W)=y;\; \rel(W, \tV)=z\}\subset\ccZplus_k.
\end{equation*}
\end{df}
\begin{lem}
\label{lem:cond-B-Gells}
The varieties $\calO^{\rm Gell}_{x,y,z}$ satisfy condition \ref{cond-df:B-str-fibr} from \Cref{naivGstrat}.
\end{lem}
\begin{proof}
We examine the fibres of the map $\calO^{\rm Gell}_{x,y,z}\to \calF$, $(\bV,W,\tV)\mapsto \bV$. For given $\bV$, the  $W$ satisfying $\rel(\bV, W)=x$ are  parameterised by $\bbC^{\ell(x)}$. For given $\bV$ and $W$ the $\tV$ satisfying $\rel(\VWp,\tV)=yz$ are parametrised by $\bbC^{\ell(yz)}$. For every such choice of $\tV$ we automatically have 
$
\rel(W,\tV)=\rel(\tV{}^W,\tV)=z$  and $\rel(\VWp,\tV{}^W)=y,
$ by definition of $\tV{}^W$.
We have a locally trivial fibration with complex affine fibres. 
\end{proof}
This shows that, with the ordering $\Theta_{\calG_k}$ from  \Cref{naivGstrat}, the decomposition into Gells is a basic paving (the conditions \ref{b1} and \ref{b2} are obvious).
\begin{ex} \label{minGell}
The minimal Gell $\calO^{\rm Gell}_{\Id,\Id,\Id}\subset\ccZplus_k$ contains all $(\bV,\tV,W)$ such that the subspaces in $\bV$ and $\tV$ agree in dimension $\geq k$, equal $W$ in dimension $k$ and form orthogonal flags in dimensions $\leq k$.
\end{ex}
The fundamental classes of minimal Gells motivate the special elements of $\hNH_n$: 
\begin{equation}\label{fundGell}
\Omega_{k,n}=\omega_1T_1\omega_1 T_2T_1\omega_1\ldots \ldots \omega_1T_{k-2}T_{k-2}\ldots T_1\omega_1T_{k-1}T_{k-2}\ldots T_1\omega_1.
\end{equation}
\begin{minipage}[t]{6cm}
Diagrammatically, this element is the permutation diagram of the longest element in $\frakS_k$, viewed inside $\frakS_n$, with $k$ floating dots inserted, at most one in each region. Each floating dot is separated from the left hand side by exactly one strand and different floating dots correspond to different strands.\\
\end{minipage}
\begin{minipage}[t]{0.3cm}
\hfill\\
\end{minipage}
\begin{minipage}[t]{6cm}
For example, for $(n,k)=(6,4)$ it is 	\\
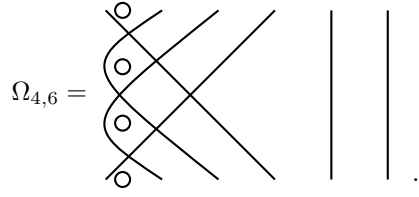
\tikz[thick,scale=.75,baseline={([yshift=1.0ex]current bounding box.center)}]{	
	\node at (-1,1.5){$\Omega_{4,6}=$};	
	\draw (0,3)-- (3,0); 	
	\draw (1,3)  .. controls (-0.5,2) ..  (2,0);
	\draw (2,3)  .. controls (-0.5,1) ..  (1,0);
    \draw (3,3)-- (0,0); 
    \draw (4,3)-- (4,0); 
    \draw (5,3)-- (5,0); 
    \fdot{}{0.3,3};
    \fdot{}{0.3,2};
    \fdot{}{0.3,1};
    \fdot{}{0.3,0};
    \node at (5.5,0){$.$}
} 
\end{minipage}
Gells can be used to provide a basis of $\calB^{\rm Gell}$ of $H_*^\rmG(\ccZplus)$:
\begin{prop}
For fixed $k$, the basic paving $\coprod_{(x,y,z)\in \Theta_{\calG_k}}\calO^{\rm Gell}_{x,y,z}$ is weakly adapted to $\calB^{\rm Gell}_k=\{T_x \Omega_{k,n} T_yT_z \}_{(x,y,z)\in \Theta_{\calG_k}}\subset\hNH_n$, i.e., we have
\begin{equation}
\label{eq-weaklyad-Gells-sl2}
\deg(T_x \Omega_{k,n} T_yT_z)=2\dim\ccY_k-2\dim\calO^{\rm Gell}_{x,y,z}+ k^2-k.
\end{equation}
In particular, $ \calB^{\rm Gell}=\bigcup_{0\leq k\leq n} \calB^{\rm Gell}_k$ is a $\Pol_n$-basis of $H_*^\rmG(\ccZplus).$
\end{prop}
\begin{proof}
We can just refer to the statement on $\calG$-cells from \Cref{calcinGcells}, since $\deg (T_x \Omega_{k,n} T_yT_z)=\deg (T_x \omega_1\omega_2\ldots\omega_k T_yT_z)$
 and $\dim \calO^{\rm Gell}_{x,y,z}=\dim \calO^\calG_{x,y,z}$.
 \end{proof}
This Gell basis is a special case of the $\Pol_n$-bases from \cite[Thm. 3.16]{NaisseVaz}.
\begin{rk}
Checking weakly adaptedness is merely a matching of degrees and dimensions. As a result, several basic pavings are weakly adapted to a given basis and a given paving can be used to construct many bases. For instance, the paving by Gells is also weakly adapted to the $\calG$-cells bases. In contrast to this example, we expect that the examples we gave proofs for are in fact strongly adapted. Strongly adaptedness reflects more of the geometry in algebra and more of the algebra in geometry. Currently we however do not have a proof for these stronger statements. 
\end{rk}

\addtocontents{toc}{{\textbf{Part II: The case of general quivers}}}
\section{Coloured geometry and its coloured combinatorics}
\label{sec:coloured}\label{generalnot}
From now on fix as in \S\ref{sec:KLR} a general finite quiver $\Gamma=(I,A)$ without loops together with a \emph{dimension vector} $\bfn=\sum_{i\in I}n_i\cdot i\in \bbZ_{\geqslant 0}I.$ Let $|\bfn|:=\sum_{i\in I}n_i=n.$ Finally fix  an $I$-graded complex vector space $V=\bigoplus_{i\in I}V_i$ of graded dimension $\bfn$. 

Let $\rmG_\bfn=\prod_{i\in I}\mathrm{GL}(V_i)$ with its standard torus $\rmT$ formed by diagonal matrices. We denote by $E_\bfn=\bigoplus_{a\in A}{\rm Hom}(V_{s(a)},V_{t(a)})$ the affine variety of representations of $\Gamma$ of dimension vector $\bfn$ equipped with the action of $\rmG_\bfn=\prod_{i\in I}\mathrm{GL}(V_i)$. 

We provide in this section combinatorial and geometric notions depending on $\Gamma$. In case $\Gamma=(\{\bullet\}, \emptyset)$ we recover many constructions from Part I (which were usually denoted by the same symbol there). We refer to the elements in $I$, i.e. the vertices of $\Gamma$, as \emph{colours}. By  \emph{coloured geometry} and 
        \emph{coloured combinatorics}, we mean geometry and combinatorics for more than one colour.
\subsection{Coloured geometry}
\label{subs:spec-var}
\begin{nota} Elements in $I^n$ will be denoted by bold letters, such as $\ui,\uj$, and their components by the same not bold letter, for instance  $\ui=(i_1,i_2,\ldots, i_n)$. 
Let $I^\bfn\subset I^n$ be given by the tuples $\ui$, where $\lvert\{j\mid i_j=k\}\rvert=n_k.$
The group $\frakS_n$ acts on $I^n$ and $I^\bfn$ such that $w(\uj)=\ui$ means $j_r=i_{w(r)}$ for $r\in[1;n]$.

For $\ui\in I^{\bfn}$, let $\calF_\ui$ be the variety of flags 
$
\bV=(\{0\}=\bV^0\subset \bV^1\subset\cdots\subset \bV^n=V),
$
which are homogeneous with respect to the decomposition $V=\bigoplus_{i\in I}V_i$, and for each $1\leq r\leq n$ the graded dimension of $\bV^r/\bV^{r-1}$ is equal to $i_r$. Set $\calF=\coprod_{\ui \in I^\bfn}\calF_\ui$.
\end{nota}
We say $\alpha\in E_\bfn$ {\emph{preserves $\bV\in\calF$} if $\bV$ is a sequence of subrepresentation of $\alpha$.
Set  $\ccF_\bfn=\prod_{i\in I}\calF(V_i)$, where $\calF(V_i)$ is the usual flag variety in the vector space $V_i$. Independently of $\ui\in I^\bfn$, the variety $\calF_\ui$ is always isomorphic to $\ccF_\bfn$. However whether $\alpha$ preserves $\bV$  depends essentially on $\ui$.
\begin{df} A \emph{Grassmannian dimension vector} is a vector $\bfk=(k_i)_{i\in I}$ with $k_i\in \bbZ$ such that $0\leqslant k_i\leqslant n_i$. Let $\Jbfn=\coprod_{i\in I}[0;n_i]$ be the set of such vectors.
\end{df}
Consider the following varieties associated with $\bfk\in \Jbfn$ where $|\bfk|=\sum_{i\in I}k_i$.
\begin{equation}\label{Grassvar}
\calG_\bfk=\prod_{i\in I}\op{Gr}_{k_i}(V_i), \quad \calG=\coprod_{\bfk\in \Jbfn}\calG_\bfk,\quad\text{and}\quad\calG_k=\coprod_{\bfk\in \Jbfn,|\bfk|=k}\calG_\bfk\quad\text{for fixed $k$}. 
\end{equation}
\begin{df}\label{Grassmannian Springer variety}
We define the \emph{Grassmannian--Springer quiver variety}
\begin{equation}
\rmY=\tF\times \calG\subset E_\bfn\times\calF\times\calG, \quad\text{where}\quad\tF=\{(\alpha,\bV)\mid \text{$\alpha$ preserves $\bV$}\}\subset E_\bfn\times\calF.
\end{equation}
We denote elements of $\rmY$ as triples $(\alpha,\bV, W)$, where $\alpha\in E_\bfn$, $\bV\in\calF$, $W\in\calG$. 
\end{df}
The decomposition $\calF=\coprod_{\ui\in I^\bfn}\calF_\ui$ induces decompositions $\tF=\coprod_{\ui\in I^\bfn}\tF_\ui$ and $\rmY=\coprod_{\ui\in I^\bfn}\rmY_\ui$. Moreover, \eqref{Grassvar} induces decompositions $\rmY=\coprod_{k=0}^n\rmY_k=\coprod_{\bfk\in \Jbfn}\rmY_\bfk$. Set also $\rmY_{\ui,k}=\rmY_{\ui}\cap \rmY_{k}$ and $\rmY_{\ui,\bfk}=\rmY_{\ui}\cap \rmY_{\bfk}$.
 
\begin{df}\label{Zdd}
Consider the following variety
\begin{equation}\label{GrassStein}
\GGZ=\rmY\times_{E_\bfn}\rmY\subset E_\bfn\times\calF\times\calG\times\calF\times \calG. 
\end{equation}
\end{df}
Points in  $\GGZ$ are usually denoted as tuples $(\alpha,\bV,\tV,W,\tW)$ (that is we first list the flags and then the Grassmannian spaces). Via the inclusion $\GGZ\subset \rmY\times \rmY$, the variety $\GGZ$ inherits many decompositions and subvarieties. We introduce notation for certain unions of connected components of $\GGZ$ as follows:

\begin{gather*}
\GGZ_{\ui,\uj}=\GGZ\cap (\rmY_\ui\times \rmY_\uj), \quad \GGZ_{\bfk_1,\bfk_2}=\GGZ\cap (\rmY_{\bfk_1}\times \rmY_{\bfk_2}), \quad \GGZ_{k_1,k_2}=\GGZ\cap (\rmY_{k_1}\times \rmY_{k_2}),\\
\GGZ_{(\ui,\bfk_1),(\uj,\bfk_2)}=\GGZ\cap (\rmY_{\ui,\bfk_1}\times \rmY_{\uj,\bfk_2}),\quad \GGZ_{(\ui,k_1),(\uj,k_2)}=\GGZ\cap (\rmY_{\ui,k_1}\times \rmY_{\uj,k_2}).
\end{gather*}

We define the \emph{Grassmannian--Steinberg variety} $\Znaiv$ involving only one copy of $\calG$:
\begin{df}
For $k\in[0;n]$ we define $\Znaiv_{k}=\GGZ_{k,0}$ and $\Znaiv=\coprod_{k=0}^n \Znaiv_k$. 
\end{df}
Alternatively, we could define $\Znaiv$ as $\rmY\times_{E_\bfn}\rmY_0$.  We denote points in $\Znaiv$ as tuples $(\alpha,\bV,\tV,W)$. Similar to $\GGZ$, we use additional notation for subvarieties of $\Znaiv$, namely
\begin{gather*}
\Znaiv_{\ui,\uj}=\Znaiv\cap (\rmY_\ui\times \rmY_{\uj,0}), \quad \Znaiv_{\bfk}=\Znaiv\cap (\rmY_{\bfk}\times \rmY_{0})=\GGZ_{\bfk,0},\\
\Znaiv_{\ui,\uj,\bfk}=\Znaiv\cap (\rmY_{\ui,\bfk}\times \rmY_{\uj,0})=\GGZ_{(\ui,\bfk),(\uj,0)},\quad \Znaiv_{\ui,\uj,k}=\Znaiv\cap (\rmY_{\ui,k}\times \rmY_{\uj,0})=\GGZ_{(\ui,k),(\uj,0)}.\nonumber
\end{gather*}

\subsection{Coloured permutations}
In this subsection we will identify permutations $w\in\frakS_n$ with the associated permutation diagram, for example $w\in\frakS_5$ given by $(w(1),\ldots, w(5))=(4,5,3,1,2)$ is identified with the diagram with five black strand connecting the $i$th point at the bottom with the $w(i)$ point at the top. We can take any other reduced diagram, i.e. where the number of crossings is equal to $\ell(w)$.

A \emph{coloured permutation} $D=(w,c)$ is a permutation diagram $w$ with a colouring $c$, which means each strand is labelled by an element of $I$, see \Cref{cross} for $I=\{g,b,r\}$. We identify the set of coloured permutations with 
\begin{eqnarray*}
I^n\times\frakS_n\quad=&\{\text{coloured permutations}\}&=\quad\frakS_n\times I^n,\\
(\ui,w)\quad\reflectbox{$\mapsto$}& D=(w,c)&\mapsto\quad(w,\uj)
\end{eqnarray*}
where $\uj$ respectively $\ui$ is the colour sequences read off at the bottom and at the top.  
\begin{ex} \label{cross} Here are instances of (coloured) permutation diagrams\\

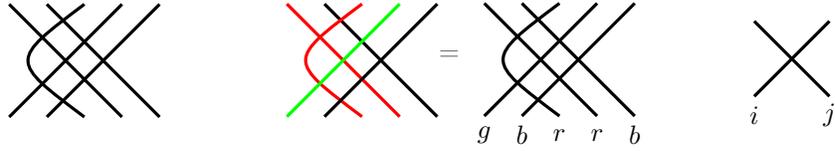
\begin{minipage}[c]{3.5cm}
\centering
\tikz[very thick,scale=0.5,baseline={([yshift=-8ex]current bounding box.center)}]{	
	\draw (1,3)-- (4,0); 	
    \draw (2,3) --  (5,0);
	\draw (3,3)  .. controls (1,1.5) ..  (3,0);
	\draw (4,3)-- (1,0); 
	\draw (5,3)-- (2,0); 
	}
\\ 
\end{minipage}
\begin{minipage}[c]{0.5cm}
\end{minipage}
$\quad$
\begin{minipage}[c]{5.5cm}
\centering
\tikz[very thick,scale=0.5,baseline={([yshift=0ex]current bounding box.center)}]{	
	\draw[color=red] (1,3)-- (4,0); 	
    \draw (2,3) --  (5,0);
	\draw[color=red] (3,3)  .. controls (1,1.5) ..  (3,0);
	\draw[color=green] (4,3)-- (1,0); 
	\draw (5,3)-- (2,0); 
	}=
\tikz[very thick,scale=0.5,baseline={([yshift=1.5ex]current bounding box.center)}]{	
	\draw (1,3)-- (4,0); 	
    \draw (2,3) --  (5,0);
	\draw (3,3)  .. controls (1,1.5) ..  (3,0);
	\draw (4,3)-- (1,0); 
	\draw (5,3)-- (2,0); 
	\node at (1,-.5){$g$};
	\node at (2,-.5){$b$};
	\node at (3,-.5){$r$};
	\node at (4,-.5){$r$};
	\node at (5,-.5){$b$};
	}	 
\end{minipage}
\begin{minipage}[c]{3cm}
\centering
\tikz[very thick,baseline={([yshift=+.6ex]current bounding box.center)}]{
		\draw   (-.5,-.5)-- (.5,.5);
		\draw   (.5,-.5)-- (-.5,.5);
		\node at (-.5,-.75){$i$};
		\node at (.5,-.75){$j$};
} 	
\end{minipage}
\hfill\\
\begin{minipage}[c]{4cm}
\centering
\emph{permutation diagram} $w$
\end{minipage}
\begin{minipage}[c]{5.5cm}
\centering
\emph{coloured permutation}
\end{minipage} 
\begin{minipage}[c]{3cm}
\centering
\emph{coloured crossing} 
\end{minipage}
In formulas, the coloured permutation here is $(w,(g,b,r,r,b))=((r,b,r,g,b),w).$ 
\end{ex}
\begin{rk}
One could define $I$-coloured permutations as reduced expressions of morphisms in the free symmetric monoidal category generated by the set $I$.
\end{rk}
\begin{rk} Purely diagrammatically, the coloured permutation $(w,\uj)$ looks like the diagram $\tau_w1_\uj$ from \S\ref{sec:KLR}, but the two should not be confused. 
Coloured permutation do for instance not depend on a choice of a reduced expression.
\end{rk}
The following count of crossings of the form \eqref{cross} does make sense for coloured permutations\footnote{One can easily check that the count is independent of the chosen drawing of the permutation. } as well as KLR diagrams. 
\begin{df} Assume $D$ is a coloured permutation.
\begin{itemize}
\item Let $\Xequal(D)$ be the number of crossings such that $i=j$.
\item Let $\Xright(D)$ be the multiplicity of crossings such that $i\to j$ in $\Gamma$.
\item Let $\Xleft(D)$ be the multiplicity of crossings such that $j\to i$ in $\Gamma$.
\end{itemize}
By \emph{multiplicity} we mean hereby that if there are $r$ arrows between $i$ and $j$ of the specified direction, we count the crossing  with multiplicity $r$. 
Sometimes we just write $\Xequal(w)$, $\Xright(w)$, $\Xleft(w)$, if $D=(w,\ui)$ and $\ui$ is clear clear from the context. 
\end{df}
\begin{ex}
We illustrate the counting for the coloured permutation $D=(w,\uj)$ from \Cref{cross}. Consider first the quiver $\Gamma_1:$ 
$
\begin{tikzcd}
{\color{red}\bullet} \arrow[r] & \bullet  & {\color{green}\bullet} \arrow[l] 
\end{tikzcd}.
$
The possible coloured crossings are of the form
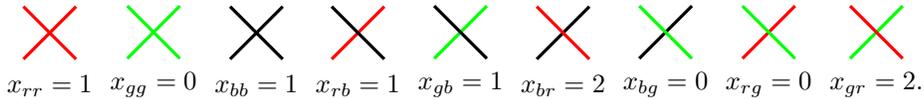
\begin{equation*}
\tikz[very thick, scale=0.7,baseline={([yshift=0ex]current bounding box.center)}]{
		\draw[color=red]   (-.5,-.5)-- (.5,.5);
		\draw[color=red]   (.5,-.5)-- (-.5,.5);
		\node at (0,-1){$x_{rr}=1$};
	}  	
\tikz[very thick,scale=0.7,baseline={([yshift=0ex]current bounding box.center)}]{
		\draw[color=green]   (-.5,-.5)-- (.5,.5);
		\draw[color=green]   (.5,-.5)-- (-.5,.5);
			\node at (0,-1){$x_{gg}=0$};
	}  	
	\tikz[very thick, scale=0.7,baseline={([yshift=0ex]current bounding box.center)}]{
		\draw[color=black]   (-.5,-.5)-- (.5,.5);
		\draw[color=black]   (.5,-.5)-- (-.5,.5);
			\node at (0,-1){$x_{bb}=1$};
	}  	
		\tikz[very thick, scale=0.7,baseline={([yshift=0ex]current bounding box.center)}]{
		\draw[color=red]   (-.5,-.5)-- (.5,.5);
		\draw[color=black]   (.5,-.5)-- (-.5,.5);
			\node at (0,-1){$x_{rb}=1$};
	}  	
\tikz[very thick,scale=0.7,baseline={([yshift=0ex]current bounding box.center)}]{
		\draw[color=green]   (-.5,-.5)-- (.5,.5);
		\draw   (.5,-.5)-- (-.5,.5);
			\node at (0,-1){$x_{gb}=1$};
	}  	
		\tikz[very thick,scale=0.7,baseline={([yshift=0ex]current bounding box.center)}]{
		\draw[color=black]   (-.5,-.5)-- (.5,.5);
		\draw[color=red]   (.5,-.5)-- (-.5,.5);
			\node at (0,-1){$x_{br}=2$};
	} 
		\tikz[very thick,scale=0.7,baseline={([yshift=0ex]current bounding box.center)}]{
		\draw[color=black]   (-.5,-.5)-- (.5,.5);
		\draw[color=green]   (.5,-.5)-- (-.5,.5);
			\node at (0,-1){$x_{bg}=0$};
	} 
	\tikz[very thick,scale=0.7,baseline={([yshift=0ex]current bounding box.center)}]{
		\draw[color=red]    (-.5,-.5)-- (.5,.5);
		\draw[color=green]   (.5,-.5)-- (-.5,.5);
			\node at (0,-1){$x_{rg}=0$};
	} 
	\tikz[very thick,scale=0.7,baseline={([yshift=0ex]current bounding box.center)}]{
		\draw[color=green]    (-.5,-.5)-- (.5,.5);
		\draw[color=red]   (.5,-.5)-- (-.5,.5);
			\node at (0,-1){$x_{gr}=2.$};
	} 	
\end{equation*}
Below every crossing we displayed how often it appears in $D=(w,\uj)$. We calculate
\begin{equation*}
\Xequal(w,\uj)=x_{rr}+x_{gg}+x_{bb}=2,\quad \Xright(w,\uj)=x_{rb}+x_{gb}=2,\quad \Xleft(w,\uj)=x_{br}+x_{bg}=2.
\end{equation*}
Note that crossings which involve one red and one green strand are irrelevant. 

To illustrate how sensitive these numbers are to the given quiver, let us first add a second arrow from the red vertex to the black vertex and then also consider the quiver which has two arrows from red to black and one arrow from black to red (no arrows at the green vertex). Thus we consider the following quivers
\begin{equation*}
\Gamma_2:\quad
\begin{tikzcd}
{\color{red}\bullet} \arrow[r,bend left] \arrow[r,bend right] & \bullet  & {\color{green}\bullet} \arrow[l] 
\end{tikzcd}
\quad\text{and}\quad \Gamma_3:\quad
\begin{tikzcd}
{\color{red}\bullet} \arrow[r,bend left] \arrow[r,bend right] & \bullet \arrow[l]  & {\color{green}\bullet}  
\end{tikzcd}.
\end{equation*}
For $\Gamma_2$ we get extra multiplicities, namely
\begin{equation*}
\Xequal(w,\uj)=x_{rr}+x_{gg}+x_{bb}=2,\;\; \Xright(w,\uj)=2x_{rb}+x_{gb}=3,\;\; \Xleft(w,\uj)=2x_{br}+x_{bg}=4.
\end{equation*}
For $\Gamma_3$ crossings involving black and red contribute twice and we get
\begin{equation*}
\Xequal(w,\uj)=x_{rr}+x_{gg}+x_{bb}=2,\;\; \Xright(w,\uj)=2x_{rb}+x_{br}=4,\;\; \Xleft(w,\uj)=2x_{br}+x_{rb}=5.
\end{equation*}
\end{ex}
\begin{rk}\label{countrk}
We will use such countings to obtain various dimensions formulas in geometry. For KLR diagrams, the counts encode their degrees in the KLR algebra, see \Cref{gradingKLR}. We have for example $\deg(\tau_w1_\uj)=2\cdot\Xright(w,\uj)-2\cdot\Xequal(w,\uj)$.
\end{rk}
\subsection{Some dimension formulas}
\label{subs:dim-formulas}
Given $\bV\in \calF_\ui$, denote by $X(\bV)$ its fibre under the morphism $\tF_\ui\to \calF_\ui$, i.e. the subvariety of $E_\bfn$ given by the representations that preserve the flag $\bV$. We will use now coloured permutations to do calculations with dimensions related to the fibres $X(\bV)$.

If  $\bV\in \calF_\ui$ and $\tV\in\calF_\uj$ then the relative position $\rel(\bV,\tV)$ of the two flags is a permutation $w\in\frakS_n$ such that $w(\uj)=\ui$. It is useful to think of the relative position as being encoded by the coloured permutation $(w,\uj)$, since it carries the information about the coloured types ($\ui$ and $\uj$) of the flags. We will do this from now on.

\begin{df} \label{defsum}
Given quivers  $(I,A_1)$ and  $(I,A_2)$ with the same set of vertices, we define\footnote{It is indeed a coproduct in an appropriate category of quivers with vertex set $I$.} their \emph{sum} as the quiver $(I,A)$ where $A$ is the disjoint union of $A_1$ and $A_2$. 
\end{df}
Each quiver having at least one arrow can be presented as a sum of quivers with one arrow.  Thus, this notion allows to do reduction to a one-arrow quiver in proofs involving quantities which are additive with respect to the sum of quivers.

\begin{prop}
\label{lem:diff-dim-KLR}
Assume $\bV\in \calF_\ui$ and $\tV\in \calF_\uj$ with relative position $\rel(\bV,\tV)=w$. Then we have
$
\dim X(\bV)-\dim (X(\bV)\cap X(\tV))=\Xright(w,\uj).
$ 
\end{prop}
\begin{proof}
Since both sides of the proposed equation are additive with respect to the sum of quivers, it is enough to assume that the quiver has at most one arrow. The case of no arrow is obvious. Thus we assume that the quiver has exactly one arrow $i\to j$ (but possibly more vertices). Then $\Xright(w,\uj)$ is the number of pairs 
 $(r,t)\in [1;n]^2$ such that $j_r=i$ and $j_t=j$ and $r<t$, $w(r)>w(t)$.
 By definition, the vector spaces $X(\bV)$, $X(\tV)$ can be described as spaces of $(n\times n)$-matrices specified by the vanishing of certain matrix entries. The $(r,t)$-matrix entries for the above pairs $(r,t)$ are exactly the ones which are zero for $X(\tV)$ but not zero for $X(\bV)$.
\end{proof}

Let $\J(w,\uj)=\J(\ui,w,\uj)\subset \calF_\ui\times\calF_\uj$ be the subvariety containing the pairs of flags in relative position $w$. Let $\JJ(w,\uj)$ be the preimage of $\J(w,\uj)$ under the projection $\tF_\ui\times_{E_\bfn}\tF_\uj\to \calF_\ui\times\calF_\uj$. Both,  $\tF_\ui\to\calF_\ui$ and $\JJ(w,\uj)\to\J(w,\uj)$, are vector bundles.
Let $\dim_\alpha\tF_\ui$ and  $\dim_\alpha \JJ(w,\uj)$ be their ranks, that is the dimensions of the fibres.\footnote{The $\alpha$ should indicate that we take  dimensions of the vector spaces of the possible $\alpha\in E_\bfn$.} 

\begin{coro}
\label{coro:dim-alpha-cell-KLR}
For $w\in \frakS_n$ such that $w(\uj)=\ui$, we have 
$$
\dim_\alpha\tF_\ui-\dim_\alpha \JJ(w,\uj)= \Xright(w,\uj).
$$
\end{coro}
\begin{proof}
	This is just a reformulation of the statement in \Cref{lem:diff-dim-KLR}.
\end{proof}

\begin{coro}
\label{coro:dim-cell-KLR}
For $w\in \frakS_n$ such that $w(\uj)=\ui$, we have 
$$
\dim\tF_\ui-\dim \JJ(w,\uj)=\frac{1}{2}\deg (\tau_w1_\uj).
$$
\end{coro}
\begin{proof}
We have
$\dim\tF_\ui=\dim_\alpha\tF_\ui+\dim\calF_\ui$, $\dim \JJ(w,\uj)=\dim_\alpha \JJ(w,\uj)+\dim \J(w,\uj).$ Now, the statement follows from \Cref{coro:dim-alpha-cell-KLR} and from the obvious equality
$
\dim\calF_\ui-\dim \J(w,\uj)=-\Xequal(w,\uj)
$ with $
\frac{1}{2}\deg (\tau_w1_\uj)=\Xright(w,\uj)-\Xequal(w,\uj).$
\end{proof}

Assume $\bV\in \calF_\ui$. As explained in \S\ref{subs:spec-var}, there is an isomorphism $\calF_\ui\cong \ccF_\bfn$.  In other words, the data of a flag $\bV\in\calF_\ui$ is the same as the data of a flag $\bV_i\in \calF(V_i)$ for each $i$. The components of the flag $\bV_i$ are just the intersection of the components of the flag $\bV$ with $V_i$, with repetitions removed. 

Now take a second flag $\tV\in\calF_\uj$ and describe it again by a family of flags $\tV_i\in\calF(V_i)$. Let $w=\rel(\bV,\tV)$ be the relative position of $\bV$ and $\tV$. The permutation $w$ is naturally coloured by $\ui$ on the top. Moreover, knowing the relative position $w\in\frakS_n$ of the flags $\bV$ and $\tV$ is the same thing as knowing the relative position $w^{(i)}\in \frakS_{n_i}$ of the flags $\bV_i$ and $\tV_i$ for each $i$. We can see the permutation $w_i$ in the coloured permutation $w$ just by ignoring all strands which are not of colour $i$. The definitions imply the following.

\begin{lem}
	\label{lem:rel-triple-flags}
	Assume $\ui,\uj,\ut\in I^\bfn$ and $\bV\in \calF_\ui$, $\tV\in \calF_\uj$, $\ttV\in\calF_\ut$. Then the following two statements are equivalent 
	
	\begin{enumerate}[I.\rm{)}]
		\item We have $\rel(\bV,\tV)\cdot \rel(\tV,\ttV)=\rel(\bV,\ttV)$.
		\item For each $i\in I$, we have $\rel(\bV_i,\tV_i)\cdot \rel(\tV_i,\ttV_i)=\rel(\bV_i,\ttV_i)$.
	\end{enumerate}
\end{lem}

\subsection{Coloured subsets}
\label{subs:col-subsets}
We introduce now coloured version of the sets $\Lambda(n)=\coprod_{k=0}^n\Lambda_k(n)\subset [1;n]$ from \Cref{notnew} to parametrize coloured cells in  \S\ref{subs:col-cells} .

Fix $\uj\in I^\bfn$ and think of the elements $1,2,\ldots,n$ from $[1;n]$ as being coloured by $j_1,j_2,\ldots,j_n$ respectively.  We obtain a decomposition  
$\Lambda_k(n)=\coprod_{\bfk\in \Jbfn,|\bfk|=k}\Lambda_\bfk(n,\uj)$, where $\Lambda_\bfk(n,\uj)$ consists of all subsets of $[1;n]$ of cardinality $k$, such that exactly $k_i$ elements have colour $i$. Obviously the parts of this decomposition depend on $\uj$ and 
\begin{equation}\label{cLambda}
\Lambda(n)=\coprod_{\bfk\in \Jbfn}\Lambda_\bfk(n,\uj)=\coprod_{k\in [0;n]}\quad\coprod_{\bfk\in \Jbfn,|\bfk|=k}\Lambda_\bfk(n,\uj).
\end{equation} 
We obtain a bijection $\Lambda_\bfk(n,\uj)\cong \Lambda_\bfk(\bfn):=\prod_{i\in I}\Lambda_{k_i}(n_i)$ by considering each colour separately. 
Indeed, given  $\mu\in \Lambda_\bfk(n,\uj)$ and $i\in I$ let $1\leqslant r_1<\ldots<r_{n_i}\leqslant n$ be the elements in $[1;n]$ coloured $i$. Identify this set order preserving with $[1;n_i]$ and let $\mu^{(i)}\in \Lambda_{k_i}(n_i)$ be its intersection with $\mu$. 
Obvioulsy $\mu\mapsto \prod_{i\in I}\mu^{(i)}$ is a bijection.
\begin{ex}
	Assume $\uj=(ijiijiijj)$, $\mu=({\color{red}1}0{\color{red}01}0{\color{red}00}10)$. Then $n_i=5$, $n_j=4$, $k=3$, $k_i=2$ and $k_j=1$. We get $\mu^{(i)}=(10100)\in \Lambda_2(5)$ and $\mu^{(j)}=(0010)\in \Lambda_1(4)$.
\end{ex} 
In \S\ref{subs:cells} we used minimal coset representatives to label the parts of basic pavings. Recall the identification $\frakS_n/(\frakS_k\times\frakS_{n-k})\cong \Lambda_k(n)$, $x\mapsto x([1,k])$. Still fixing $\uj$, we assign to a minimal coset representative $x$ the coloured permutation $(\uj,x)$ and let $x^{(i)}$ be the permutation obtained by ignoring all strands not coloured by $i$. With the assignment $x\mapsto (x^{(i)})_{i\in I}$ as left vertical map we obtain the following diagram:
\begin{lem}
\label{lem:CD-left-coset-col}
The identifications fit into a commutative diagram of bijections
$$
\begin{CD}
\frakS_n/(\frakS_k\times\frakS_{n-k})@>>> \Lambda_k(n)\\
@VVV                                      @VVV\\
\coprod_{\bfk\in \Jbfn,|\bfk|=k}(\prod_{i\in I}(\frakS_{n_i}/(\frakS_{k_i}\times\frakS_{n_i-k_i})))@>>> \coprod_{\bfk\in \Jbfn,|\bfk|=k}\Lambda_\bfk(\bfn).
\end{CD}
$$ 
\end{lem}
We get an analogous commutative diagram when we use right cosets instead. For this consider the identification 
$(\frakS_k\times\frakS_{n-k})\backslash\frakS_n\cong \Lambda_k(n), z\mapsto z^{-1}([1,k])$ and attach to a minimal right coset representative $z$ now the coloured permutation $(z,\uj)$.  We obtain the following (with the assignment  $z\mapsto (z^{(i)})_{i\in I}$ as the left vertical map).

\begin{lem}
\label{lem:CD-right-coset-col}
The identifications fit into a commutative diagram of bijections
$$
\begin{CD}
(\frakS_k\times\frakS_{n-k})\backslash\frakS_n@>>> \Lambda_k(n)\\
@VVV                                      @VVV\\
\coprod_{\bfk\in \Jbfn,|\bfk|=k}(\prod_{i\in I}((\frakS_{k_i}\times\frakS_{n_i-k_i})\backslash\frakS_{n_i}))@>>> \coprod_{\bfk\in \Jbfn,|\bfk|=k}\Lambda_\bfk(\bfn).
\end{CD}
$$ 
\end{lem}

\subsection{Coloured cells, $\calG$-cells and Gells}
\label{subs:col-cells}
Similar to \S\ref{sec:flags-cells}, we like to have cells, $\calG$-cells and Gells for the variety $\calF_\ui\times\calF_\uj\times \calG_\bfk$. Recall from \S\ref{subs:spec-var} the varieties
$$
\calF_\ui\cong \prod_{l\in I}\calF(V_l)\cong\calF_\uj, \qquad  \calG_\bfk=\prod_{i\in I}Gr_{k_i}(V_i).
$$
Using these identifications our task is easy to achieve: we can take the definition of cells, $\calG$-cells or Gells from  \S\ref{subs:cells}, apply them separately for each colour and then pull the resulting decompositions through the following isomorphism 
\begin{equation}\label{coloursintro}
\calF_\ui\times\calF_\uj\times \calG_\bfk\cong \prod_{i\in I}(\calF(V_i)\times \calF(V_i)\times \op{Gr}_{k_i}(V_i)).
\end{equation}
We obtain a decomposition whose parts we call (coloured) \emph{cells}, \emph{$\calG$-cells} and \emph{Gells}, respectively.  The parametrization of the parts is however not very convenient from a more global point of view, since it is just a naive combination of the parameterizations inside each $i\in I$. We improve this using the combinatorics from \S\ref{subs:col-subsets}.

Arguing as for \eqref{cLambda} we obtain a bijection
\begin{equation}\label{identi1}
\{w\in\frakS_n\mid w(\uj)=\ui\}\cong\prod_{i\in I}\frakS_{n_i}. 
\end{equation}
Namely, we view $w$ as a coloured permutation $(w,\uj)$ and then consider each colour $i\in I$ separately. Similarly we have identifications 
\begin{equation}\label{identi2}
\prod_{i\in I}(\frakS_{n_i}/(\frakS_{k_i}\times \frakS_{n_i-k_i}))\cong \prod_{i\in I}\Lambda_{k_i}(n_i)=\Lambda_{\bfk}(\bfn).
\end{equation}
In summary, the isomorphism \eqref{coloursintro} together with the identifications \eqref{identi1},  \eqref{identi2}, \eqref{cLambda} give the \emph{global labelling sets} (in the second case we specify additionally $\uj$):
\begin{itemize}
\item Upper and lower cells in  $\calF\times\calF_\uj\times \calG$ are labelled by $\frakS_n\times \Lambda(n)$.
\item Upper and lower cells in $\calF\times\calF\times \calG$ are labelled by $\frakS_n\times\Lambda(n)\times I^\bfn$.
\end{itemize}
\medskip
The combinatorics is a bit more delicate for $\calG$-cells and Gells. We have that 
\begin{itemize}
\item $\calG$-cells and Gells in $\calF\times\calF\times \calG_k$ are labelled by the quadruples
\begin{equation*}
\quad\quad(x,y,z,\uj)\in (\frakS_n/(\frakS_k\times \frakS_{n-k}))\times (\frakS_k\times \frakS_{n-k})\times ((\frakS_k\times \frakS_{n-k})\backslash \frakS_n)\times I^\bfn,
\end{equation*}
\end{itemize}
where the $\calG$-cell $\calO^{\calG}_{x,y,z,\uj}$ respectively the Gell $\calO^{\rm Gell}_{x,y,z,\uj}$ is defined via the following bijection to the colour-wise labelling.
 Depending whether we consider $\calG$-cells or Gells, pick
 $w=e$ or $w=w_{0,k}$, by which we mean the shortest or the longest element in $\frakS_k$ respectively. We view them as element of $\frakS_n$ via the standard inclusion $\frakS_k\subset \frakS_n$.  Then we colour  the permutation $xwyz\in\frakS_n$ by $\uj$ at the bottom. This induces colours on the premutations $x$, $y$, $z$ and $w$. Let $\ui=xwyz(\uj)$ and $\bfk$ such that $z\in\Lambda_\bfk(n,\uj)\subset\Lambda_k(n)\cong  (\frakS_{k}\times \frakS_{n-k})\backslash \frakS_{n}$, see \eqref{cLambda}.  Considering each colour $i\in I$ separately, we obtain permutations 
$$
x^{(i)}\in \frakS_{n_i}/(\frakS_{k_i}\times \frakS_{n_i-k_i}),\qquad y^{(i)}\in \frakS_{k_i}\times \frakS_{n_i-k_i},\qquad z^{(i)}\in (\frakS_{k_i}\times \frakS_{n_i-k_i})\backslash \frakS_{n_i}.
$$
\begin{df}\label{DefGellgl}
Define the \emph{$\calG$-cell} respectively \emph{Gell corresponding to $(x,y,z,\uj)$} as
\begin{equation}\label{DefglobalGells}
\calO^\calG_{x,y,z,\uj}=\prod_{i\in I}\calO^\calG_{x^{(i)},y^{(i)},z^{(i)}}
\quad\text{and}\quad
\calO^{\rm Gell}_{x,y,z,\uj}=\prod_{i\in I}\calO^{\rm Gell}_{x^{(i)},y^{(i)},z^{(i)}}.
\end{equation}
Here the equalities are inside $\calF_\ui\times\calF_\uj\times \calG_\bfk=\prod_{i\in I}(\calF(V_i)\times\calF(V_i)\times \op{Gr}_{k_i}(V_i)).$
\end{df}
\begin{lem}\label{eq:dim-col-Gell}
The Gell $\calO^{\rm Gell}_{x,y,z,\uj}$ has dimension $\dim\calO^{\rm Gell}_{x,y,z,\uj}$ equal to 
\begin{equation*}
\dim(\calF_\ui\times \calG_\bfk)+ \left(\frac{1}{2}\sum_{i\in I}(k_i^2-k_i)\right)+(\Xequal(x)+\Xequal(y)+\Xequal(z)+\Xequal(w_{0,k}))-\sum_{i\in I}(n_i-1)k_i.
\end{equation*}
\end{lem}
\begin{proof}
We consider equation \eqref{eq-weaklyad-Gells-sl2} for the Gell $\calO^{\rm Gell}_{x^{(i)},y^{(i)},z^{(i)}}$ for each $i\in I$ and we sum these equations. 
We also use the fact that for a coloured permutation $w$ we have $\Xequal(w)=\sum_{i\in I}\ell(w^{(i)})$. Note that the terms with $\Xequal$ come from the degrees of the crossings and the sum $\sum_{i\in I}(n_i-1)k_i$ is the half-degree of the floating dots.
\end{proof}

\section{The problem: extending quiver Hecke algebras}
\label{sec:philsophy}
We assume the setup from the opening paragraph of \S\ref{sec:coloured}. In particular, $\Gamma=(I,A)$ is a quiver with no loops and $\bfn$ is a dimension vector.  In this section we set up the basics towards our main goal: a geometric construction of the algebra $\hR_\bfn$. We start by recalling the faithful representation of $\hR_\bfn$ from \cite{NaisseVaz} and the geometric construction of KLR. Based on these and the construction of the Grassmannian quiver Hecke algebra in the $\mathfrak{sl}_2$-case, we discuss an intuitive, although too naive, strategy towards the main goal.  We describe its deficiencies in some detail, since  this analysis helps to understand the subtleties in our final geometric construction. 

\subsection{Polynomial representation of $\hR_\bfn$}
\label{ssec:polyaction}
Recall $\hPol_n$ from \S\ref{subs:hNH} and set
\begin{equation}
\label{Epolcoloured}
\hPol_\bfn = \bigoplus_{\ui \in I^\bfn} \hPol_n 1_\ui.
\end{equation}
Here, $\hPol_n 1_\ui=\hPol_n$ as vector space, but we add $1_\ui$ to the notation to distinguish the different copies of $\hPol_n$ in $\hPol_\bfn$. Consider the $\frakS_n$-action on $\hPol_\bfn$ given by
\begin{gather}
\begin{aligned}
s_r : \hPol_n 1_\ui &\rightarrow \hPol_n 1_{s_r\ui},\\
X_{p} 1_\bi &\mapsto X_{s_r(p)}1_{s_r \bi},\\
\omega_{p} 1_\bi &\mapsto  \begin{cases}
\left( \omega_{r} + (X_{r} - X_{r+1}) \omega_{r+1} \right)   1_{s_r \bi} &\text{ if $p = r$ and $i_r = i_{r+1}$}, \\
\omega_{p} 1_{s_r\bi} &\text{ if $p \ne r$ and $i_r = i_{r+1}$}, \\ 
\omega_{s_r(p)}   1_{s_r \bi} &\text{ if $i_r \ne i_{r+1}$}. 
\end{cases}
\end{aligned}
\end{gather}
\smallskip
For $i,j\in I$, $i\ne j$, we consider the polynomial $\cP_{ij}(u,v)=(u-v)^{h_{i,j}}$. Then we have $\cQ_{i,j}(u,v)=\cP_{i,j}(u,v)\cP_{j,i}(v,u)$,  see \S\ref{sec:KLR}.
We assign to the generators, \eqref{diagdotcross} and \eqref{diagfloatdot} of $\hR_\bfn$ the following endomorphisms of $\hPol_n 1_\ui$ extended by zero to $\hPol_\bfn$.
\begin{equation}\label{faithfulcole}
\begin{gathered}
X_r1_\bi
\quad\longmapsto\qquad f1_\ui\mapsto X_{r}f1_\ui ,
\quad\quad
\Omega 1_\ui\quad \longmapsto\qquad
f1_\ui \mapsto \omega_1f1_\ui,\\
	\tau_r 1_\bi \quad\longmapsto\qquad
f 1_\bi \mapsto
\begin{cases}
   \dfrac{f1_\bi - s_{r}(f1_\bi)}{X_{r} - X_{r+1}}  & \text{if $i_r = i_{r+1}$},
  \\[1.5ex]
 \cP_{i_r,i_{r+1}}(X_{r},X_{r+1}) 
 s_{r} (f  1_{\bi}) & \text{if $i_r \ne i_{r+1}$}.
 \\[1.5ex]
\end{cases}
\end{gathered}
\end{equation}
\begin{prop}
\label{prop:faithrep-KLR-Gr}
The rules \eqref{faithfulcole} define a faithful action of $\hR_\bfn$ on $\hPol_\bfn$.
\end{prop}
\begin{proof} This is Proposition~3.8 and Theorem 3.15 in~\cite{NaisseVaz}. 
\end{proof}
It will sometimes be useful to enumerate the omegas colour by colour. For this assume $\ui\in I^\bfn$ and $i\in I$ and let then $t_1<t_2<\cdots<t_{n_i}$ be the indices coloured by $i$, that is $i=i_{t_1}=i_{t_2}=\cdots=i_{t_{n_i}}$. We then denote $\omega_{r,i}1_\ui=\omega_{t_r}1_\ui$, the $\omega$ corresponding to the $r$th index of colour $i$ with the conventions $\omega_{0,i}=0$. Similarly, for polynomial variables, set $X_{r,i}1_\ui=X_{t_r}1_\ui$.
\subsection{Bases of $\hR_\bfn$}
For each $w\in\frakS_n$ fix a reduced decomposition  $w=s_{r_1}\ldots s_{r_t}$. Recalling $\tau_w$ from \Cref{tau} the following holds, see e.g. \cite[Thm. 3.7]{Rou2KM}.
\begin{lem}
\label{lem:basis-KLR}
	The set $\{\tau_w 1_\uj\mid w\in\frakS_n,\uj\in I^\bfn\}$ is a $\Pol_n$-basis of $R_\bfn$.
\end{lem}
To formulate a similar result for $\hR_\bfn$, set for $k\in[0;n]$, 
 analogously to \eqref{fundGell}, 
\begin{equation}
\label{omegakn}
\Omega_{k,n}=\Omega\tau_1\Omega \tau_2\tau_1\Omega\ldots \ldots \Omega\tau_{k-2}\tau_{k-2}\ldots \tau_1\Omega\tau_{k-1}\tau_{k-2}\ldots \tau_1\Omega.
\end{equation}
The following gives a geometric version of a basis of type \cite[Thm. 3.16]{NaisseVaz}. 
\begin{prop}
\label{prop:NVbasis-gen}
	There is a  $\Pol_n$-basis of $\hR_\bfn$ as follows:
	\begin{equation}
	\label{eq:NV-basis-canon}
\coprod_{k=0}^n	\left\{\tau_x\Omega_{k,n}\tau_y\tau_z1_\uj\mid (x,y,z)\in\frakS,\, \uj\in I^\bfn\right\},
	\end{equation}
where $\frakS=\left(\frakS_n/(\frakS_k\times \frakS_{n-k})\right)\times\left(\frakS_k\times \frakS_{n-k}\right)\times\left((\frakS_k\times \frakS_{n-k})\backslash \frakS_n\right)$. 
\end{prop}
\begin{proof}
Consider the basis of $\hR_\bfn$ given in \cite[Thm. 3.16]{NaisseVaz}. To check that \eqref{eq:NV-basis-canon} is indeed a basis, it is enough to show that the elements of \eqref{eq:NV-basis-canon} may be written in term of the elements of the basis in \cite[Thm. 3.16]{NaisseVaz} with a triangular change-of-basis matrix with units on the diagonal. However, to get an element of the basis in \cite{NaisseVaz} from an element of \eqref{eq:NV-basis-canon} we need to move some strands through crossings, this would create extra terms. This means that an element of \eqref{eq:NV-basis-canon} is equal to an element of the basis in \cite[Thm. 3.16]{NaisseVaz} plus smaller terms.
\end{proof}

\subsection{The geometric construction of the KLR algebras}
\label{subs:KLR-geom}
We will now use the varieties introduced in \S\ref{sec:coloured}.
We consider the algebra structure on $H^{\rmG_\bfn}_*(\tF\times_{E_\bfn}\tF)$, given by the convolution product with respect to the inclusion $\tF\times_{E_\bfn}\tF\subset \tF\times\tF$. The following important result is proven in \cite[Thm. 3.6]{VV} and in \cite{Rou2Lie}.
\begin{prop}
	\label{prop:KLR-geom}
There is an isomorphism of algebras $R_\bfn\cong H_*^{\rmG_\bfn}(\tF\times_{E_\bfn}\tF)$.
\end{prop}
We sketch the proof of this statement to prepare the reader for the proof of a similar statement for the Naisse-Vaz extended version $\hR_\bfn$. 
\begin{proof}[Idea of the proof]
The algebra $R_\bfn$ acts faithfully on $\Pol_\bfn=\bigoplus_{\ui\in I^\bfn}\Pol_n1_\ui$ and $H_*^{\rmG_\bfn}(\tF\times_{E_\bfn}\tF)$ acts faithfully on  $H_*^{\rmG_\bfn}(\tF)\cong \Pol_\bfn$. To get the desired isomorphism one can proceed as follows. For each generator $g$ of $R_\bfn$ one can construct by hand some $g'$ $H_*^{\rmG_\bfn}(\tF\times_{E_\bfn}\tF)$ such that the action of $g$ and $g$ agree. 
This defines a monomorphism $R_\bfn\to H_*^{\rmG_\bfn}(\tF\times_{E_\bfn}\tF)$. To show that it is an isomorphism, it is enough to check that a $\Pol_n$-basis of $R_\bfn$ goes to a $\Pol_n$-basis of $H_*^{\rmG_\bfn}(\tF\times_{E_\bfn}\tF)$. This can be done following \S\ref{subs:B-strat}. Namely, on can show that the $\Pol_n$-basis of $R_\bfn$ from \Cref{lem:basis-KLR} is weakly adapted\footnote{We define basic pavings similarly to the $\mathfrak{sl}_2$ case. For condition \ref{cond-df:B-str-fibr} in \Cref{def:b-paving}, we take the projection to the first flag component $\calF_\ui$.}
to the basic paving given by the decomposition with respect to relative positions of flags. To check a condition similar to \eqref{wadapt}, we use \Cref{coro:dim-cell-KLR}.
\end{proof}
\begin{term}
\label{rk:KLR-vs-quivHecke}
To keep track of different versions of (possibly isomorphic) algebras, we refer to the algebraic/diagrammatic algebra $R_\bfn$ as \emph{KLR algebra} and to the geometric algebra $H_*^{\rmG_\bfn}(\tF\times_{E_\bfn}\tF)$ as \emph{quiver Hecke algebra}. Then \Cref{prop:KLR-geom} says that the KLR algebra is isomorphic to the quiver Hecke algebra. 

Following the same philosophy, we call the Naisse-Vaz extension $\hR_\bfn$ of $R_\bfn$ the \emph{extended KLR algebra}. Its geometric counterpart that we introduce in \Cref{def:Gr-quiver-Hecke} is then the \emph{Grassmannian quiver Hecke algebra}. The naming reflects the fact that we add the variety of Grassmannians to the previous geometric construction. Our main result \Cref{mainthm}  then says that the extended KLR algebra is isomorphic to the Grassmannian quiver Hecke algebra. 
\end{term}

\subsection{An intuitive approach and its problem}\label{naiveapproach}
\label{subs:intuit-approch-philo}
Intuitively, \S\ref{sec:sl2-geom}  and \Cref{prop:KLR-geom} suggest the following strategy to construct $\hR_\bfn$ geometrically.

We consider the coloured versions $\rmY=\tF\times \calG$ and $\GGZ=\rmY\times_{E_\bfn}\rmY$ of the extended flag variety and the $\ccZ$  from \S\ref{setupsl2}. Hereby, the fibre product is defined using the obvious forgetting map $\rmY\to E_\bfn$.

By viewing $\GGZ$ as a subvariety of $E_\bfn\times\calF\times\calG\times \calF\times\calG$ we can turn $H^{\rmG_\bfn}_*(\GGZ)$ into an algebra for the convolution product with respect to the inclusion $\GGZ\subset \rmY\times \rmY$, which  acts on $H^{\rmG_\bfn}_*(\rmY)$, see \Cref{lem:conv}.   

On the other hand, $H_*^{\rmG_\bfn}(\rmY)$ can be identified with the space $\hPol_\bfn$ from \eqref{Epolcoloured} underlying the faithful representation of $\hR_\bfn$. Indeed, we have $\rmY=\coprod_{\ui\in I^\bfn}\rmY_\ui$ and 
$$
H^*_{\rmG_\bfn}(\rmY)=\bigoplus_{\ui\in I^\bfn} H^*_{\rmG_\bfn}(\rmY_\ui)=\bigoplus_{\ui\in I^\bfn} H^*_{\rmG_\bfn}(\calF_\ui\times \calG)\cong\bigoplus_{\ui\in I^\bfn} H^*_\rmT(\calG)\cong\bigoplus_{\ui\in I^\bfn}\hPol_n=\hPol_\bfn.
$$
The third equality here follows by applying \eqref{eq:isom-G/B*Gr} for each vertex of the quiver.

Again following \S\ref{sec:sl2-geom} we can then consider the variety $\Znaiv=\rmY\times_{E_\bfn}\rmY_0$ containing only one Grassmannian instead of two Grassmannians and similarly to \S \ref{subs:geom-hNH} identify $H_*^{\rmG_\bfn}(\Znaiv)$ with a subalgebra of $H_*^{\rmG_\bfn}(\GGZ)$. This will be done in \S\ref{subs:Z'-subalg-Z}. One could expect that the algebra $H_*^{\rmG_\bfn}(\Znaiv)$ is isomorphic to $\hR_\bfn$. However, we will see that it is not true in general. We face the problem that \emph{the algebra $H_*^{\rmG_\bfn}(\Znaiv)$ is too big}. 

The problem is that in the definitions of the varieties $\GGZ$ and $\Znaiv$ from \S\ref{naiveapproach}, we \emph{naively} added copies of $\calG$ without any interaction with $E_\bfn$ nor $\calF$. 
In general, we are forced to introduce some coupling of $\calG$ with $E_\bfn$ and $\calF$. For this some (non-obvious) modifications of the varieties $\GGZ$ and $\Znaiv$ are necessary. In \S\ref{sec:sl2-geom} this was obsolete, because of  
the absence of arrows in the quiver. 

\subsection{Vector bundles over Gells?}
Let us explain what fails if we consider the naive coloured analogue $\Znaiv$  of $\ccZplus$.  Imagine that we can copy the proof of  \Cref{prop:KLR-geom}. The final step requires a check that the inclusion $\hR_\bfn\hookrightarrow H_*^{\rmG_\bfn}(\Znaiv)$ is an isomorphism. We would do this using bases in cohomology arising from basic pavings. Now recall from \S\ref{sec:sl2-geom},  that Gells provide bases which algebraically generalize to arbitrary $\Gamma$, see \S\ref{subs:col-cells}. Thus it looks promising to consider the decomposition of $\Znaiv$ induced by the preimages  $p^{-1}(\calO^{\rm Gell}_{x,y,z,\uj})$ of the Gells under the projection $p:\Znaiv\to \calF\times\calF\times \calG$.  These preimages seem to be not well-behaved, e.g. not smooth. 

Moreover, to be able to deduce that the image of the basis in \Cref{prop:NVbasis-gen} is weakly adapted to the decomposition in $\Znaiv$ we need to verify the dimension formula
\begin{equation}\label{problem}
2\dim p^{-1}(\calO^{\rm Gell}_{x,y,z,\uj})=
2\dim \rmY_{\ui,\bfk}+\sum_{i\in I}(k_i^2-k_i)-
\deg(\tau_x\Omega_{k,n}\tau_y\tau_z1_\uj).
\end{equation}
We have already seen in \eqref{eq-weaklyad-Gells-sl2} that this formula holds in case of $\mathfrak{sl}_2$. In the general we compute the dimension by computing the dimensions on the base (which is just an $\mathfrak{sl}_2$ calculation) and of the fibres. If we divide \eqref{problem} by two and subtract the equality from \Cref{eq:dim-col-Gell}  we obtain with \Cref{countrk} that \eqref{problem} is equivalent to
\begin{equation}\label{problem2}
\dim_\alpha p^{-1}(\calO^{\rm Gell}_{x,y,z,\uj})=\dim_\alpha \rmY_{\ui,\bfk}-(\Xright(x)+\Xright(y)+\Xright(z)+\Xright(w_{0,k})),
\end{equation}
where we colour the permutations $x$, $y$, $z$, $w_{0,k}$ as in \S\ref{subs:col-cells}. As above, $\dim_\alpha$ denotes the dimension of the fibre of the vector bundle given by forgetting the $E_\bfn$-component.

While trying to verify this, we however encountered a bad behaviour: the fibres of $p$ are vector spaces, but their dimensions are not constant over a fixed Gell.  
We therefore want to modify $\Znaiv$  (more precisely, the fibres of $\Znaiv\to\calF\times\calF\times\calG$) to a less naive coloured analogue such that the map $p$ defines a vector bundle over each Gell and the dimensions of the fibres are such that  formula \eqref{problem2} holds. Making this precise will be the purpose of the next sections.

\section{Geometric construction of the Naisse--Vaz generators}
\label{sec:gen-qv}
We still assume the setup from \S\ref{sec:coloured}. In particular, $\Gamma=(I,A)$ is a quiver with no loops and $\bfn$ is a dimension vector. In this section we will construct geometrically the action of the generators of the algebra $\hR_\bfn$. They arise from the action of a subspace $H_*^{\rmG_\bfn}(\Znaiv)$ inside the convolution algebra $H_*^{\rmG_\bfn}(\GGZ)$ from \S\ref{subs:intuit-approch-philo}. We show that this subspace is in fact an algebra and realise $\hR_\bfn$ as a subalgebra of $H_*^{\rmG_\bfn}(\Znaiv)$.

\subsection{Fixed points}\label{sec:fixedpoints}
To understand $H_*^{\rmG_\bfn}(\GGZ)$ better, we compute the action of certain elements from $H_*^{\rmG_\bfn}(\GGZ)$ on the polynomial representation. We do this using the torus localisation and compare the results with \eqref{faithfulcole}.

Fix once and for all a basis $\{e_1,e_2,\ldots, e_n\}$ of $V$ consisting of $\rmT$-eigenvectors. Let  $\bU$ be the corresponding \emph{standard flag} given by $\bU^i=\langle  e_1,e_2,\ldots,e_{i} \rangle\subset V.$ Let $\ui_0$ be the colour type of this flag, i.e., $\bU\in\calF_{\ui_0}$. 

Set $f_w=w(\bU)$, and $\ui_w=w^{-1}(\ui_0)$. We have $f_w\in \calF_{\ui_w}$ and $\calF^\rmT=\{f_w;\,w\in\frakS_n\}$. 
Given a $\rmT$-fixed point $f_w\in\calF$. Then $f_w\in\calF_\uj$ if and only if $w(\uj)=\ui_0$.

The obvious inclusion $\calG_k\subset \op{Gr}_k(V)$ from \eqref{Grassvar}, induces a bijection on the sets of $\rmT$-fixed points. As in \S \ref{subs:more-Euler}, we thus have 
$\calG_k^\rmT=\op{Gr}_k(V)^\rmT=\{g_\mu;\, \mu\in \Lambda_k(n)\}$.
Since $\calG_\bfk=\prod_{i\in I}\op{Gr}_{k_i}(V_i)$, the $\rmT$-fixed points can also naturally be labelled by $\Lambda_\bfk(\bfn)$. The inclusion $\calG_\bfk^\rmT\subset \calG_k^\rmT$ induces then an inclusion $\Lambda_\bfk(\bfn)\subset \Lambda_k(n)$ which is in fact the inclusion from \S\ref{subs:col-subsets} (with respect to the fixed sequence $\ui_0\in I^\bfn$).

The $\rmT$-fixed points in $\calF\times \calG_k$ by $\frakS_n\times \Lambda_k(n)$ can be parametrized as in \S\ref{subs:more-Euler}, namely $(\calF\times \calG_k)^\rmT=\{x_{w,\lambda};\,w\in \frakS_n, \lambda\in \Lambda_k(n)\}$, where $x_{w,\lambda}=(f_w,g_{w(\lambda)})$. Note that $\rmT$-fixed points in $\rmY_k$ have zero $E_\bfn$-component. Therefore, the $\rmT$-fixed points in $\rmY_k$ can be viewed as $\rmT$-fixed points in $\calF\times \calG_k$. Abusing the notation, we also write $x_{w,\lambda}$ for the elements of $\rmY_k^\rmT$. 

We also identify the $\rmT$-fixed points via the inclusion $\GGZ\subset \rmY\times \rmY$ and denote them as tuples $x_{w_1,w_2,\mu_1,\mu_2}=(x_{w_1,\mu_1},x_{w_2,\mu_2})$, where $w_1,w_2\in \frakS_n$, $\mu_1,\mu_2\in \Lambda(n)$. We will moreover need the labelling which does not involve a twist, namely 
\begin{equation*}
(w_1,w_2,\mu)=(f_{w_1},f_{w_2},g_\mu)\in\Znaiv, \quad\text\quad (w_1,w_2,\mu_1,\mu_2)=(f_{w_1},f_{w_2},g_{\mu_1},g_{\mu_2})\in \GGZ,
\end{equation*}
of the $\rmT$-fixed points  in $\Znaiv$ and $\GGZ$ respectively.
\subsection{The representation $H_*^{\rmG_\bfn}(\rmY)$} 
Similarly to \S\ref{subs:rep-H(FG)} we consider for $\lambda\in\Lambda_k(n)$ subvarieties $C_\lambda\subset \calF\times \calG_k$ defined as 
$$
C_\lambda=\{(\bV,W)\in \calF\times \calG_k\mid\dim((\bV^r\cap W)/(\bV^{r-1}\cap W))=\lambda_r,~\forall r\in[1;n]\}.
$$
Denote $\omega_\lambda=[\overline {C_\lambda}]\in H^*_{\rmG_\bfn}(\calF\times \calG_k)\cong H^*_{\rmG_\bfn}(\rmY_k)$. The notation makes sense, since the identification of $H_*^{\rmG_\bfn}(\rmY)$ with $\bigoplus_{\uj\in I^\bfn}\hPol_n$ in \S\ref{naiveapproach}  sends $\omega_\lambda$ to $\omega_\lambda$. We can write
\begin{equation}
\label{eq:omega-geom-gen}
\omega_\lambda=\sum_{w,\mu}\ovS^w_{\lambda}(\mu)\tA^{-1}_{w,\mu}
[x_{w,\mu}]\quad\text{where}\quad \tA_{w,\mu}=\op{eu}(\rmY,x_{w,\mu}).
\end{equation}
We would like to have a description of the coefficients $\ovS^w_\lambda(\mu)$. 
We are going to use the notation defined in \S\ref{subs:col-subsets}. Moreover, if $\lambda,\mu\in \Lambda_k(n)$, we say that $\mu$ is $\uj$-above $\lambda$, denoted $\mu\geqslant^\uj\lambda$, if $\lambda,\mu\in \Lambda^\uj_\bfk(n)$ for some $\bfk\in \Jbfn$ and $\mu^{(i)}\geqslant\lambda^{(i)}$ for all $i\in I$. Here we identify $\Lambda^\uj_\bfk(n)=\Lambda_\bfk(\bfn)$ as in \S\ref{subs:col-subsets} and we consider the partial order on $\Lambda_{k_i}(n_i)$ defined in \S \ref{subs:rep-H(FG)}. 

\begin{lem}
\label{lem:T-to-S}
Recalling \Cref{lem:caract-S},  the coefficients in \eqref{eq:omega-geom-gen}  can be expressed as
$$
\ovS^w_\lambda(\mu)=\prod_{i\in I}\tS^{w^{(i)}}_{\lambda^{(i)}}(\mu^{(i)}).
$$
We use here the notation from \S\ref{subs:col-subsets} with $w$ being coloured by $\ui_0$ on the top, and thus by $\ui_w$ on the bottom, and with $[1;n]$ being coloured by $\ui_w$. 
\end{lem}
\begin{proof}
This follows from the K\"unneth formula applied to $\calG_\bfk=\prod_{i\in I}\op{Gr}_{k_i}(V_i)$.
\end{proof}

\subsection{Coloured Demazure operators}
For $r\in[1;n-1]$, consider the subvariety $\Zr\subset \GGZ$ defined by the condition that the two Grassmannian spaces are the same and the two flags are the same except maybe at the $r$th component. For $\ui,\uj\in I^\bfn$, set $\Zr_{\ui,\uj}=\Zr\cap \GGZ_{\ui,\uj}$ and write
\begin{equation}
\label{eq:Z-loc-gen}
[\Zr_{s_r(\uj),\uj}]=\sum_{w_1,w_2,\mu_1,\mu_2}\tA_{w_1,w_2,\mu_1,\mu_2}[x_{w_1,w_2,\mu_1,\mu_2}].
\end{equation}
The following lemma is proved in \cite[Lemma 2.19]{VV}.
\begin{lem}
\label{lem:coeff-Zr-gen} The coefficient $\tA_{w_1,w_2,\mu_1,\mu_2}$  in the expansion  \eqref{eq:Z-loc-gen} can only be nonzero if   
	$
	w_1s_r(\uj)=\ui_0, w_2(\uj)=\ui_0$, and $w_1(\mu_1)=w_2(\mu_2); 
	$
in which case we moreover have
\begin{eqnarray*}
\tA_{w_1,w_2,\mu_1,\mu_2}=
\begin{cases}
\tA^{-1}_{w_1,\mu_1}(\ccT_{(w_1(r))}-\ccT_{w_1(r+1)})^{h_{\uj_{r},\uj_{r+1}}} &\mbox{if } s_r(\uj)\ne \uj, w_2=w_1s_r,\\ 
\tA^{-1}_{w_1,\mu_1}(\ccT_{(w_2(r))}-\ccT_{w_2(r+1)})^{-1} &\mbox{if } s_r(\uj)=\uj, w_2\in\{w_1, w_1s_r\},\\
0, & \mbox{otherwise}.
\end{cases}
\end{eqnarray*}
\end{lem}
Identifying  $H^*_{\rmG_\bfn}(\rmY)\cong\hPol_\bfn$ as in \S\ref{naiveapproach}, we construct the Naisse--Vaz operators: 
  \begin{prop}
The fundamental class $[\Zr_{s_r(\uj),\uj}]$ acts on $h\in H_*^{\rmG_\bfn}(\rmY)$ by 
$$
[\Zr_{s_r(\uj),\uj}]h=
\begin{cases}
(X_{r}-X_{r+1})^{h_{\uj_{r},\uj_{r+1}}}s_r(h)\in H_*^{\rmG_\bfn}(\rmY_{s_r(\uj)}) &\mbox{ if } s_r(\uj)\ne\uj,\\
-\partial_r(h)\in H_*^{\rmG_\bfn}(\rmY_{s_r(\uj)}) &\mbox{ if } s_r(\uj)=\uj.
\end{cases}
$$
In particular, this action agrees (up to sign) with the action of $\tau_r1_\uj$ from \eqref{faithfulcole}.
\end{prop}
\begin{proof}
Let $s=s_r$ and $h= P\omega_\lambda$, where $P\in\Pol_n$. Using \eqref{eq:Z-loc-gen} and \eqref{eq:omega-geom-gen} we get
\begin{equation*}
[\Zr_{s(\uj),\uj}]\star P\omega_\lambda=(\sum_{w_1,w_2,\mu_1,\mu_2}\tA_{w_1,w_2,\mu_1,\mu_2}[x_{w_1,w_2,\mu_1,\mu_2}])\star (\sum_{w\in\WW_n}P_w\ovS^w_\lambda(\mu)\tA^{-1}_{w,\mu}[x_{w,\mu}]).
\end{equation*}
We compute the coefficient of $[x_{w,\mu}]$ using  \eqref{unclear} and  \Cref{lem:coeff-Zr-gen}.
If $s(\uj)\ne \uj$ we get
\begin{equation*}
\tA_{w,ws,\mu,s(\mu)}P_{ws}\ovS^{ws}_\lambda(s(\mu))=\tA^{-1}_{w,\mu}(\ccT_{(w(r))}-\ccT_{w(r+1)})^{h_{\uj_{r},\uj_{r+1}}}P_{ws}\ovS^{ws}_\lambda(s(\mu)).
\end{equation*}
This agrees with the coefficient of $[x_{w,\mu}]$ when we write $(X_{r}-X_{r+1})^{h_{\uj_{r},\uj_{r+1}}}s(P\omega_\lambda)$ in the basis of $\rmT$-fixed points, since we have
$P\omega_\lambda=\sum_{w,\mu}P_w\ovS^w_{\lambda}(\mu)\tA^{-1}_{w,\mu}
[x_{w,\mu}]$ by \eqref{eq:omega-geom-gen}. If $s(\uj)= \uj$ we get
\begin{eqnarray*}
&&\tA_{w,w,\mu,\mu}P_w\ovS^w_\lambda(\mu)+\tA_{w,ws,\mu,s(\mu)}P_{ws}\ovS^{ws}_\lambda(s(\mu))\\
&=&\tA^{-1}_{w,\mu}P_w\ovS^w_\lambda(\mu)(\ccT_{w(r+1)}-\ccT_{w(r)})^{-1}+\tA^{-1}_{w,\mu} P_{ws}\ovS^{ws}_\lambda(s(\mu))(\ccT_{ws(r+1)}-\ccT_{ws(r)})^{-1}\\
&=&\tA^{-1}_{w,\mu}\cdot (P_w\ovS^w_\lambda(\mu)-P_{ws}\ovS^{ws}_\lambda(s(\mu)))(\ccT_{w(r+1)}-\ccT_{w(r)})^{-1}.
\end{eqnarray*}
This is the coefficient of $[x_{w,\mu}]$ in $
-\partial_r(P\omega_\lambda)
$
written in the $\rmT$-fixed points basis. 
 \end{proof}

\subsection{Creation operators}
Fix $j\in I$ and $k\in[0;n-1]$.
Consider the subvariety 
\begin{equation*}\label{Zsubsetk}
\GGZ^j_{\supse,k}=\{(\alpha,\bV,\tV,W,\tW)\in \GGZ_{k+1,k}\mid W\supset \tW, \bV=\tV, \grdim(W/\tW)=j\}
\end{equation*}
and write
\begin{equation}
\label{eq:Z_k+loc-gen}
[\GGZ^j_{\supse,k}]=\sum_{w_1,w_2,\mu_1,\mu_2}\tB^+_{w_1,w_2,\mu_1,\mu_2}[x_{w_1,w_2,\mu_1,\mu_2}]. 
\end{equation}
Using the abbreviation $\ui=\ui_{w_1}$ we have
\begin{equation}
\label{eq:coef-Z_k-gen}
\tB^+_{w_1,w_2,\mu_1,\mu_2}=
\begin{cases}
\tA_{w_1,\mu_1}^{-1}\!\!\displaystyle\prod_{r\in \mu_2,i_r=j}\!\!(\ccT_{w_1(t)}-\ccT_{w_1(r)})^{-1} &\mbox{if } (w_1,\mu_1,i_t)=(w_2, \mu_2+\epsilon_t, j),\\
0 & \mbox{otherwise}.
\end{cases}
\end{equation}

We recover the multiplication by the elements $\omega_{r,i}$ from \S\ref{ssec:polyaction}.
\begin{prop}\label{colouredwn}
The action of $[\GGZ^j_{\supse,k}]$ on $H_*^{\rmG_\bfn}(\rmY)$ is by multiplication with $\omega_{n_j,j}$.
\end{prop}
\begin{proof}
Using  \eqref{eq:Z_k+loc-gen} and \eqref{eq:omega-geom-gen}, we get
$$
[\GGZ^j_{\supse,k}]\star P\omega_\lambda=(\sum_{w_1,w_2,\mu_1,\mu_2}\tB^+_{w_1,w_2,\mu_1,\mu_2}[x_{w_1,w_2,\mu_1,\mu_2}])\star (\sum_{w,\mu}P_w\ovS^w_\lambda(\mu)\tA^{-1}_{w,\mu}[x_{w,\mu}]).
$$
By \eqref{unclear} and \eqref{eq:Z_k+loc-gen}, the coefficient of $[x_{w,\mu}]$ in this product is
$$
\tA^{-1}_{w,\mu}P_w\sum_{t\in\mu,i_t=j}\frac{\ovS^w_\lambda(\mu\backslash\{t\})}{\prod_{r\in \mu\backslash\{t\},i_r=j}(\ccT_{w(t)}-\ccT_{w(r)})}.
$$

Now, the statement follows from \Cref{lem:S-lambda+n} and \Cref{lem:T-to-S}.
\end{proof}

As in \Cref{lem:creation+r}, we can construct multiplications by other $\omega_i$'s. Similar to \S\ref{subs:creation}, denote by $\xi$ the first Chern class of the line bundle on $\GGZ^j_{\supse,k}$ given by $W/\tW$.
\begin{prop}
	\label{lem:creation+r-gen}
Fix $\uj\in I^\bfn$ and $r\in[1;n]$ such that $j_r=j$. The push-forward to $H_*^{\rmG_\bfn}(\GGZ)$ of the element 
	$$
	\prod_{p\in [r+1;n],j_p=j}(X_p-\xi)\in H_*^{\rmG_\bfn}(\GGZ^j_{\supse,k})
	$$
	acts on $H_*^{\rmG_\bfn}(\rmY_\uj)$ by the multiplication with $\omega_r$.
\end{prop}
\begin{proof}
	We can argue as for \Cref{colouredwn} except that we use \Cref{lem:S-lambda+r}. 
\end{proof}

\subsection{The algebra $H^{\rmG_\bfn}_*(\Znaiv)\subset H^{\rmG_\bfn}_*(\GGZ)$}
\label{subs:Z'-subalg-Z}
Fix for each $i\in I$ a (hermitian) scalar product on $V_i$. They combine to a scalar product on $V$ such that all $V_i$ are orthogonal. Let $\rmU_\bfn\subset \rmG_\bfn$ be the maximal compact subgroup of unitary transformations. Fix $\bfk_1,\bfk_2\in \Jbfn$ such that $\bfk_1-\bfk_2\in \Jbfn$.  We define the \emph{inclusion variety} $\GGZ_{\bfk_1\supset \bfk_2}$ as the subvariety of $\GGZ_{\bfk_1,\bfk_2}$ defined by the condition $W\supset \tW$. 

Similarly to \Cref{lem:perm-flags-3}, there is a diffeomorphism 
\begin{equation}\label{gamma}
\gamma_{\bfk_1,\bfk_2}\colon\GGZ_{\bfk_1\supset \bfk_2}\cong \GGZ_{\bfk_1\supset \bfk_1-\bfk_2},\quad  (\bV,\tV,W,\tW)\mapsto (\bV,\tV,W,\tW^\perp\cap W).
\end{equation}
This diffeomorphism is not $\rmG_\bfn$-invariant, but it is $\rmU_\bfn$-invariant, see \Cref{rk:G-vs-U}.
\begin{df}\label{imap}
	Let $\iota_{\bfk_1,\bfk_2}\colon H_*^{\rmG_\bfn}(\Znaiv_{\bfk_1-\bfk_2})\to H_*^{\rmG_\bfn}(\GGZ_{\bfk_1, \bfk_2})$ be the composition
	$$
	H_*^{\rmG_\bfn}(\Znaiv_{\bfk_1-\bfk_2})\to H_*^{\rmG_\bfn}(\GGZ_{\bfk_1\supset \bfk_1-\bfk_2}) \to H_*^{\rmG_\bfn}(\GGZ_{\bfk_1\supset \bfk_2})\to H_*^{\rmG_\bfn}(\GGZ_{\bfk_1, \bfk_2}),
	$$
	where the first map is the pull-back with respect to $\GGZ_{\bfk_1\supset \bfk_1-\bfk_2}\to \Znaiv_{\bfk_1-\bfk_2}$ which forgets the component $W$, the second map is induced by $\gamma_{\bfk_1,\bfk_2}$ and the third map is the push-forward with respect to the inclusion.
\end{df}

Assume $r\in [1;n]$. Consider the subvarieties $\dot\Delta_j$ of $\Znaiv_j$ given by $\bV=\tV$ and $\dot\Delta_{j,r}$ of $\dot\Delta_j$ given by $W\subset \bV^r$. Assume that $\uj\in I^\bfn$ is such that $j_r=j$ and assume $\bfk\in\Jbfn$.
\begin{lem}
The image of the fundamental class $[\dot\Delta_{j,r}]\in H_*^{\rmG_\bfn}(\Znaiv_j)$ under $\iota_{\bfk+j,\bfk}$ acts on $H_*^{\rmG_\bfn}(\rmY_{\uj,\bfk})$ as multiplication with $\omega_r$. 
\end{lem}
\begin{proof}
 We have $\iota_{\bfk+j,\bfk}([\dot\Delta_j])=[\GGZ^j_{\supse,\bfk}]$, where $\GGZ^j_{\supse,\bfk}=\GGZ^j_{\supse,k}\cap \GGZ_{\bfk+j,\bfk}$ for $k=|\bfk|$. 
Since we have $[\dot\Delta_{j,r}]=[\dot\Delta_j]\prod_{p\in [r+1;n],j_p=j}(X_p-\xi)$, \Cref{lem:creation+r-gen} implies then that the image of $[\dot\Delta_{j,r}]$ acts on $H_*^{\rmG_\bfn}(\rmY_{\uj,\bfk})$ by the multiplication with $\omega_r$. 
\end{proof}

}

Now, we construct a product on $H^{\rmG_\bfn}_*(\Znaiv)$ and we identify this algebra with a subalgebra of $H^{\rmG_\bfn}_*(\GGZ)$. This could be done exactly as in \S\ref{sec:sl2-geom} because $\Znaiv$ contains $\calG$ as a direct factor (since $\Znaiv=(\tF\times_{E_\bfn}\tF)\times\calG)$ and $\GGZ$ contains $\calG\times \calG$ as a direct factor (namely $\GGZ=(\tF\times_{E_\bfn}\tF)\times\calG\times\calG)$. We prefer however to give an argument that also works for the modifications of $\GGZ$ and $\Znaiv$ that we will define later.

\begin{df}
For $\bfk_1,\bfk_2\in \Jbfn$ such that  $\bfk_1+\bfk_2\in \Jbfn$ construct a \emph{product map}
\begin{equation}\label{eq:star-G-gen}
H_*^{\rmG_\bfn}(\Znaiv_{\bfk_1})\times H_*^{\rmG_\bfn}(\Znaiv_{\bfk_2})\to H_*^{\rmG_\bfn}(\Znaiv_{\bfk_1+\bfk_2})
\end{equation}
as the composition of $\iota_{\bfk_1+\bfk_2,\bfk_2}\times \Id$ with the convolution product
\begin{equation*}
H_*^{\rmG_\bfn}(\Znaiv_{\bfk_1})\times H_*^{\rmG_\bfn}(\Znaiv_{\bfk_2})\to H_*^{\rmG_\bfn}(\GGZ_{\bfk_1+\bfk_2,\bfk_2})\times H_*^{\rmG_\bfn}(\GGZ_{\bfk_2,0})\to H_*^{\rmG_\bfn}(\GGZ_{\bfk_1+\bfk_2,0})
\end{equation*}
using the identification $H_*^{\rmG_\bfn}(\GGZ_{\bfk_1+\bfk_2,0})= H_*^{\rmG_\bfn}(\Znaiv_{\bfk_1+\bfk_2})$. 
\end{df}
\begin{prop}\label{closedmulti}
The maps \eqref{eq:star-G-gen} define an algebra structure on  $H^{\rmG_\bfn}_*(\Znaiv)$. 
\end{prop}
After some preparation we will in fact prove a stronger statement in \Cref{closedmultistrong}. We start with some cohomology calculations using localisation.  
For this we upgrade the notation from \S\ref{subs:star-prod} to coloured subsets. For $\lambda,\nu\in \Lambda(n)$ coloured by $\ui_0$, we set $\catP_{\lambda,\nu}=\prod_{r\in\lambda,t\in \nu}(\ccT_i-\ccT_j)$, where the product runs only over pairs $(r,t)$ such that $r$ and $t$ have the same colour (in $\ui_0$). 

Recalling the notation for $\rm T$-fixed points from \S\ref{sec:fixedpoints} we write  $\oX\in H_*^\rmT(\Znaiv_{\bfk_1})_{\rm loc}$ as
\begin{equation}\label{X}
\oX=\sum_{w_1,w_2,\mu}\oX(w_1,w_2,\mu)[(w_1,w_2,\mu)].
\end{equation}
The morphism $\iota_{\bfk_1,\bfk_2}$ from \Cref{imap} lifts to $H_*^\rmT(\bullet)_{\rm loc}$ by setting 
$$
\iota_{\bfk_1,\bfk_2}\oX={\sum_{\mu_1\supset\mu_2}}\catP^{-1}_{\mu_2,\mu_1^c}\oX(w_1,w_2,\mu_1\backslash\mu_2)[(w_1,w_2,\mu_1,\mu_2)].
$$
Set $\Theta_{w,\mu}=\op{eu}(\rmY,(f_w,g_\mu))$ and $\Theta_w=\op{eu}(\tF,f_w)$. Then $\Theta_{w,\mu}=\Theta_w\catP_{\mu,\mu^c}$, where here and in the following $\mu^c=[1;n]\backslash \mu$ denotes the complement of $\mu$. 
\begin{df} Define the following product on $H_*^\rmT(\Znaiv)_{\rm loc}$ using notation \eqref{X}. For $\oX\in H_*^\rmT(\Znaiv_{\bfk_1})_{\rm loc}$ and $\oY\in H_*^\rmT(\Znaiv_{\bfk_2})_{\rm loc}$ let $\oX\star \oY\in H_*^\rmT(\Znaiv_{\bfk_1+\bfk_2})_{\rm loc}$ be given by
	\begin{equation*}
	\label{eq:star-loc-gen}
	{(\oX\star\oY)}(w_1,w_2,\mu)=\sum_{w_2,\mu=\mu'\coprod\mu''} \Theta_{w_2}\catP_{\mu'',(\mu')^c}\oX(w_1,w_2,\mu')\oY(w_2,w_3,\mu'').
	\end{equation*}
\end{df}

\begin{lem}\label{starintertwines}
Assume we are given $\bfk_1,\bfk_2,\bfk_3\in \Jbfn$ such that\footnote{In this case, the element $\bfk_{13}:=\bfk_{12}+\bfk_{23}=\bfk_1-\bfk_3$ is then automatically in $\Jbfn$.} 	
$$
\bfk_{12}:=\bfk_1-\bfk_2\in \Jbfn \quad \mbox{ and } \quad \bfk_{23}:=\bfk_2-\bfk_3\in \Jbfn.
$$
Then the following diagram commutes
$$
\begin{CD}
H_*^\rmT(\Znaiv_{\bfk_{12}})_{\rm loc}\times H_*^\rmT(\Znaiv_{\bfk_{23}})_{\rm loc} @>{\star}>> H_*^\rmT(\Znaiv_{\bfk_{13}})_{\rm loc}\\
@V{\iota_{\bfk_1,\bfk_2}\times \iota_{\bfk_2,\bfk_3}} VV                                                             @V{\iota_{\bfk_1,\bfk_3}}VV\\
H_*^\rmT(\GGZ_{\bfk_1,\bfk_2})_{\rm loc}\times H_*^\rmT(\GGZ_{\bfk_2,\bfk_3})_{\rm loc} @>{\star}>> H_*^\rmT(\GGZ_{\bfk_1,\bfk_3})_{\rm loc}.\\
\end{CD}
$$
Here the bottom map is the product on $H_*^\rmT(\GGZ)_{\rm loc}$ from \Cref{lem:comp-in-loc}.
\end{lem}
\begin{rk}
By \Cref{lem:comp-in-loc}, we have explicitly (here $\delta$ is the Kronecker delta)
\begin{equation}
\label{eq:mulp-Tfixed-GGZ}
[(w_1,w_2,\mu_1,\mu_2)]\star [(w'_2,w_3,\mu'_2,\mu_3)]=\delta_{\mu_2,\mu_2'}\delta_{w_2,w_2'}\Theta_{w_2,\mu_2}[(w_1,w_3,\mu_1,\mu_3)].
\end{equation}
\end{rk}
\begin{proof}
Assume $\oX\in H_*^\rmT(\Znaiv_{\bfk_{12}})_{\rm loc}$ and $\oY\in H_*^\rmT(\Znaiv_{\bfk_{23}})_{\rm loc}$. Then we can write
\begin{align*}
\iota_{\bfk_1,\bfk_2}\oX&={\sum_{\mu_1\supset\mu_2}}\catP^{-1}_{\mu_2,\mu_1^c}\oX(w_1,w_2,\mu_1\backslash\mu_2)[(w_1,w_2,\mu_1,\mu_2)],\\
\iota_{\bfk_2,\bfk_3}\oY&={\sum_{\mu_2\supset\mu_3}}\catP^{-1}_{\mu_3,\mu_2^c}\oY(w_2,w_3,\mu_2\backslash\mu_3)[(w_2,w_3,\mu_2,\mu_3)],\\
\iota_{\bfk_1,\bfk_3}(\oX\star \oY)&=\sum_{\mu_1\subset\mu_3}\catP^{-1}_{\mu_3,\mu_1^c}{(\oX\star \oY)}(w_1,w_3,\mu_1\backslash\mu_3)[(w_1,w_3,\mu_1,\mu_3)].
\end{align*}
Using \eqref{eq:mulp-Tfixed-GGZ}, we can then compute $(\iota_{\bfk_1,\bfk_2}X)\star (\iota_{\bfk_2,\bfk_3}Y)$ as
$$
{\sum_{\mu_1\supset\mu_2\supset \mu_3}}\Theta_{w_2,\mu_2}\catP^{-1}_{\mu_2,\mu_1^c}\catP^{-1}_{\mu_3,\mu_2^c}\oX(w_1,w_2,\mu_1\backslash\mu_2)\oY(w_2,w_3,\mu_2\backslash\mu_3)[(w_1,w_3,\mu_1,\mu_3)].
$$
The desired equality $(\iota_{\bfk_1,\bfk_2}X)\star (\iota_{\bfk_2,\bfk_3}Y)=\iota_{\bfk_1,\bfk_3}(\oX\star \oY)$ follows, since we have 
\begin{align*}
&\Theta_{w_2,\mu_2}\catP^{-1}_{\mu_2,\mu_1^c}\catP^{-1}_{\mu_3,\mu_2^c}\catP_{\mu_3,\mu_1^c}=\Theta_{w_2}\catP_{\mu_2,(\mu_2)^c}\catP^{-1}_{\mu_2,\mu_1^c}\catP^{-1}_{\mu_3,\mu_2^c}\catP_{\mu_3,\mu_1^c}\\
=\quad&\Theta_{w_2}\catP_{\mu_2,(\mu_1\backslash \mu_2)^c}\catP^{-1}_{\mu_3,(\mu_1\backslash \mu_2)^c}=\Theta_{w_2}\catP_{\mu_2\backslash\mu_3,(\mu_1\backslash \mu_2)^c}.
\end{align*}
by the formula $\Theta_{w,\mu}=\Theta_w\catP_{\mu,\mu^c}$ and the definitions. 
\end{proof}
We observe that \Cref{starintertwines} applied in case $\bfk_3=0$ implies that the product on $H_*^\rmT(\Znaiv)_{\rm loc}$ defined in \eqref{eq:star-loc-gen} restricts to the product on $H_*^{\rmG_\bfn}(\Znaiv)$ defined in \eqref{eq:star-G-gen}. \begin{coro}\label{coromult}
	Assume that $\bfk_1,\bfk_2,\bfk_3\in \Jbfn$ are as in \Cref{starintertwines}.	
	Then the following diagram commutes
	$$
	\begin{CD}
	H_*^{\rmG_\bfn}(\Znaiv_{\bfk_{12}})\times H_*^{\rmG_\bfn}(\Znaiv_{\bfk_{23}}) @>>> H_*^{\rmG_\bfn}(\Znaiv_{\bfk_{13}})\\
	@V{\iota_{\bfk_1,\bfk_2}\times \iota_{\bfk_2,\bfk_3}} VV                                                             @V{\iota_{\bfk_1,\bfk_3}}VV\\
	H_*^{\rmG_\bfn}(\GGZ_{\bfk_1,\bfk_2})\times H_*^{\rmG_\bfn}(\GGZ_{\bfk_2,\bfk_3}) @>>> H_*^{\rmG_\bfn}(\GGZ_{\bfk_1,\bfk_3}).
	\end{CD}
	$$
\end{coro}
\Cref{coromult} identifies $H_*^{\rmG_\bfn}(\Znaiv)$ with a subalgebra of $H_*^{\rmG_\bfn}(\GGZ)$. 
Moreover, the algebra $H_*^{\rmG_\bfn}(\Znaiv)$ acts faithfully on $H_*^{\rmG_\bfn}(\rmY)=\hPol_\bfn$ by \Cref{rk:faith-from-strat}.  
For each generator of $\hR_\bfn$, namely for $X_{r}1_{\bi}$,  $\tau_{r}1_{\bi}$ from \eqref{diagdotcross} and for $\Omega1_{\ui}$ from \eqref{diagfloatdot}, 
we have constructed an element of $H_*^{\rmG_\bfn}(\Znaiv)$ acting by the same operator. This gives a homomorphism $\hR_\bfn\to H_*^{\rmG_\bfn}(\Znaiv)$ of algebras. With the following grading it is homogeneous of degree zero.

\begin{df}\label{geomgradhR} 
Define a $\bbZ$-grading on the algebra $H_*^{\rmG_\bfn}(\Znaiv)$ by putting $H_r^{\rmG_\bfn}(\Znaiv_{\ui,\uj,\bfk})$ in degree $
-r+2\dim\rmY_{\ui,\bfk}+\sum_{i\in I}(k_i^2-k_i).$ 
\end{df}

In summary we obtain the following strengthening of \Cref{closedmulti}.
\begin{thm}\label{closedmultistrong}
The product maps  \eqref{eq:star-G-gen}  equip $H_*^{\rmG_\bfn}(\Znaiv)$ with an algebra structure. Moreover, there is an embedding $\hR_\bfn\hookrightarrow H_*^{\rmG_\bfn}(\Znaiv)$ of algebras, even of graded algebras with the gradings given by \Cref{DefgradhR} and \Cref{geomgradhR} respectively.
\end{thm}
\begin{rk}\label{Rkonembedding} 
The algebra homomorphism $\hR_\bfn\to H_*^{\rmG_\bfn}(\Znaiv)$ is in general not surjective, although it is indeed an isomorphism in case of $\mathfrak{sl}_2$. It is not hard to show, that $H_*^{\rmG_\bfn}(\Znaiv)$ is the algebra of operators on $\hPol_\bfn$ generated by $\hR_\bfn$ and the multiplications with $\omega_2,\omega_3,\ldots,\omega_n$. In the nil-Hecke case the floating dot $\Omega_r$ acts on the 
polynomial representation by $\omega_r$ for any $1\leq r\leq n$. In general $\Omega_11_\ui$ acts by $\omega_1$, but $\Omega_r1_\ui$ acts only by $\omega_r1_\ui$ multiplied by some polynomial plus some other terms involving some $\omega_11_\ui,\ldots,\omega_{r-1}1_\ui$.
\end{rk}
\begin{rk} We like to stress that the overall strategy of this section is rather general and might be pursued for other varieties. We use three ingredients: 
\begin{enumerate}[\rm{(I-}i\rm{)}]
\item We used the embeddings of $H_*^{\rmG_\bfn}(\GGZ)$ into $H_*^{\rmT}(\GGZ)_{\rm loc}$ and of $H_*^{\rmG_\bfn}(\Znaiv)$ into $H_*^{\rmT}(\Znaiv)_{\rm loc}$. This for instance always holds, when the variety has a nice paving, see \Cref{rk:faith-from-strat}.
	\item We used that $\GGZ_{\bfk_1\supset \bfk_2}\to \Znaiv_{\bfk_1-\bfk_2}$ is a Grassmannian bundle whose relative Euler class at the $\rmT$-fixed point $(w_1,w_2,\mu_1,\mu_2)$ is $\catP_{\mu_2,\mu_1^c}$. 
	\item The Euler class $\op{eu}(\rmY,(w,\mu))$ factorizes as a product of $\catP_{\mu,\mu^c}$ and of something which depends only on $w$.
\end{enumerate}
\end{rk}

\subsection{Higher floating dots}\label{Higherfloatingdots}
In \Cref{Rkonembedding}  we elaborated on the difference between the action of the $\Omega_r1_\ui$ and the multiplication with $\omega_r$. Algebraically the higher floating dots $\Omega_{r,i}^a$ (that is floating dots with index $r>1$ or with nonzero twist $a$) are rather mysterious and seem to appear ad hoc. We finish now this section by explaining their geometric meaning. 

\begin{rk}
The floating dot $\Omega_{r,j_0}$ will finally correspond to the fundamental class of some variety and $\Omega^a_{r,j_0}$ will arise by \emph{twisting} this class by the $a$th power of a certain Chern class. This observation motivated our terminology (\emph{twist}).
\end{rk}

Whereas, by  \cite[\S 3.2.1]{NaisseVaz}, the element $\Omega_{r,i}^a$ acts on $\hPol_\bfn$ by some complicated recursively defined operator,  the floating dots appear geometrically more directly. We describe next the meaning of a single higher floating dot of colour $j_0$ in $H_*^{\rmG_\bfn}(\Znaiv_{j_0})$.  In this case, the vector space $W=W_{j_0}$ in $\Znaiv_{j_0}$ is $1$-dimensional. We will see that $\Omega_{r,j_0}^a$ acts on $\hPol_\bfn$ by some $\omega_{r'}$ multiplied by a product of Chern classes, see \Cref{highdot} and \eqref{eq:Omega-via-z}. We will upgrade our geometric framework so that this product of Chern classes appears naturally from certain geometric conditions. To run the construction we add a new copy of $E_\bfn$ to $\Znaiv$. We  identify $H_*^{\rmG_\bfn}(\Znaiv)=H_*^{\rmG_\bfn}(\Znaiv\times E_\bfn)$ as vector spaces and will interpret the product of Chern classes from \eqref{eq:Omega-via-z} below in terms of geometric conditions on elements in the new $E_\bfn$.

As above, we fix a (hermitian) scalar product on each $V_i$, $i\in I$. This allow to identify $\Hom(V_i,V_j)$ with the the space of sesquilinear forms $V_i\times V_j\to \bbC$. So $\beta$ can be thought as a family of sesquilinear forms $\beta_h\colon V_i\times V_j\to \bbC$ for each arrow $h\colon i\to j$ of the quiver.

Fix $\ui\in I^\bfn$.  Algebraically the vector space $H^*_{\rmG_\bfn}(\rmY_{\ui,j_0})$ can be identified with the vector subspace of $\hPol_\bfn$ whose elements are of the form
\begin{equation}
\label{eq:Y-j0-alg}
\sum_{r=1}^{n_{j_0}}P_r\,\omega_{r,j_0}1_{\ui},\qquad P_r\in\Pol_n.
\end{equation}
 Moreover, we can identify $\rmY_{\ui,j_0}$ with the subvariety $\dot\Delta_{\ui,j_0}$ of $\Znaiv_{\ui,j_0}$ given by $\bV=\tV$. The action of $H_*^{\rmG_\bfn}(\dot\Delta_{\ui,j_0})\subset H_*^{\rmG_\bfn}(\Znaiv_{\ui,j_0})$ on $H^*_{\rmG_\bfn}(\rmY)$ corresponds in  $\hPol_\bfn$ to the multiplication by the elements \eqref{eq:Y-j0-alg}.

\begin{df}\label{Omegafancy}
Denote by $\mathbb{\Omega}_{r,j_0}$ the subvariety of $\dot\Delta_{j_0}\times E_\bfn$ given by the conditions 
\begin{itemize}
	\item $W\subset \bV^r$,
	\item for each arrow $h$ in the quiver of the form $j_0\to j$ or $j\to j_0$, the vector space $W=W_{j_0}$ is orthogonal to $\bV^r_j$ with respect to $\beta_h$.
\end{itemize}
\end{df}

Let $\xi\in H^*_{\rmG_\bfn}(\rmY_{\ui,j_0})$ be the first Chern class of the line bundle with fibre $W$. Then, abusing notation,  the operator $\xi\colon H_*^{\rmG_\bfn}(\rmY_{\ui,j_0})\to H_*^{\rmG_\bfn}(\rmY_{\ui,j_0})$  of multiplication by $\xi$  is characterized by the fact that it commutes with multiplications by polynomials and satisfies 
\begin{equation}
\label{eq:z-acts-omega}
(X_{r,j_0}-\xi)\omega_{r,j_0}1_\ui=\omega_{r-1,j_0}1_\ui,\qquad\text{for all } r\in[1;n_{j_0}].
\end{equation}
Now higher floating dots and their twists have a natural geometric interpretation.
 \begin{prop}\label{highdot}
Via the  embedding $\hR_\bfn\hookrightarrow H_*^{\rmG_\bfn}(\Znaiv)$ of algebras, the (higher) floating dot $\Omega^a_{r,j_0} 1_\ui$ corresponds to the twisted fundamental class $(-\xi)^a[\mathbb{\Omega}_{r,j_0}]$  pushed forward to $H_*^{\rmG_\bfn}(\Znaiv\times E_\bfn)$.
\end{prop}
\begin{proof}
	In case it exists, let $r'$ be maximal such that $r'\leqslant r$ and $i_{r'}=j_0$. Set $r'=0$ otherwise.
	By \Cref{lem:creation+r-gen}, the push-forward of the fundamental class of the subvariety given by the condition  $W\subset \bV^r$ in \Cref{Omegafancy}
	acts just by $\omega_{r'}1_\ui$. The second condition adds an additional multiplication by 
	$\prod_{t=1}^{r}\cQ_{i_t,j_0}(X_t,\xi)$. Thus, the push-forward in the proposition acts on the polynomial representation by
	\begin{equation}
	\label{eq:Omega-via-z}
	(-\xi)^a\prod_{t=1}^{r}\cQ_{i_t,j_0}(X_t,\xi)\omega_{r'}1_\ui.
	\end{equation}
	This operator should be understood as a polynomial in $\xi$ with coefficient in \eqref{eq:Y-j0-alg}. The power of $\xi$ should be applied to the coefficient.
		
All $\Omega^a_{r,j_0} 1_\ui$ can be reduced to $\Omega_{1,j_0}^01_\ui$ using the relations \eqref{eq:fdots}-\eqref{eq:ExtR2}  and  \eqref{eq:extra-rel} (with the obvious commutativity relations making the relations well-defined). Since the commutativity relations are obvious,  it suffices to check that the operators \eqref{eq:Omega-via-z} satisfy the 
 formulas \eqref{eq:fdots}--\eqref{eq:ExtR2}  and \eqref{eq:extra-rel} and act as required  in case $a=0$, $r=1$. Indeed, they act by $\omega_1$ if $j_0=i_1$ and by zero otherwise. Algebraically, \eqref{eq:fdots} equals
\begin{equation}
\label{eq:rel1-flot-alg}
\Omega_{r,j_0}^a1_\ui=
\begin{cases}
\Omega_{r-1,j_0}^{a-1}1_\ui-X_r\Omega_{r,j_0}^{a-1}1_\ui, &\mbox{if $a>0$, $i_r=j_0$},\\
\sum_{t,p} (-1)^p q_{i_rj_0}^{tp}\Omega_{r-1,j_0}^{a+p}X_r^t1_\ui&\mbox{if $i_r\ne j_0$.}
\end{cases}
\end{equation}
To verify this geometrically assume first  $r=1$. If $i_1\ne j_0$, then the relation \eqref{eq:rel1-flot-alg} becomes $0=0$, since $\omega_{r'}=\omega_0=0$. (We interpret the operator \eqref{eq:Omega-via-z} as zero for $r=0$.)
If $i_1\ne j_0$, then \eqref{eq:rel1-flot-alg} holds, since for the operator from \eqref{eq:Omega-via-z} we have  
$$
(-X_1)^a\omega_11_\ui=0-X_1(-X_1)^{a-1}\omega_1 1_\ui.
$$
Next assume $r>1$. In case $j_0\ne i_r$ the relation \eqref{eq:rel1-flot-alg} translates into geometry as 	
	\begin{equation*}
	(-\xi)^a\prod_{t=1}^{r}\cQ_{i_t,j_0}(X_t,\xi)\omega_{r'}1_\ui=\cQ_{i_t,j_0}(X_r,\xi)\cdot(-\xi)^a\prod_{t=1}^{r-1}\cQ_{i_t,j_0}(X_t,\xi)\omega_{r'}1_\ui,
	\end{equation*}
	which is obviously true. If $j_0=i_r$ then the relation \eqref{eq:rel1-flot-alg} translates into 	
		\begin{gather*}
	(-\xi)^a\prod_{t=1}^{r}\cQ_{i_t,j_0}(X_t,\xi)\omega_r1_\ui\\
	=(-\xi)^{a-1}\prod_{t=1}^{r-1}\cQ_{i_t,j_0}(X_t,\xi)\omega_{(r-1)'}1_\ui-(-\xi)^{a-1}X_r\prod_{t=1}^{r}\cQ_{i_t,j_0}(X_t,\xi)\omega_r1_\ui,
	\end{gather*}
	which holds by \eqref{eq:z-acts-omega}.  Thus, \eqref{eq:fdots}  holds for the operators \eqref{eq:Omega-via-z}. 
	
	To show that the operators \eqref{eq:Omega-via-z}
 satisfy also formula \eqref{eq:ExtR2}, it suffices to verify    
\begin{equation}\label{lost}
\tau_r \Omega_{r,j_0}^a\tau_r1_\ui=\Omega_{r+1,j_0}^a1_\ui+ \sum_{t,p} q_{i_ri_{r+1}}^{tp} \sum_{\substack{h+\ell=\\p-1}} (-1)^h\Omega_{r-1,j_0}^{a+h}X_r^tX_{r+1}^\ell
\end{equation}
	geometrically in case $i_r\ne i_{r+1}=j_0$. Since $\cQ_{i,j}(u,v)=\sum_{t,p}q_{ij}^{tp}u^tv^p$, we have
	\begin{equation*}
	\frac{\cQ_{i,j}(u,v_2)-\cQ_{i,j}(u,v_1)}{v_2-v_1}=\sum_{t,p}q_{ij}^{tp}u^t\sum_{h+\ell=p-1}v_1^hv_2^\ell
	\end{equation*}
	and thus \eqref{lost} turns into
	\begin{equation}\label{lost2}
	\tau_r^21_\ui=Q_{i_r,i_{r+1}}(X_r,\xi)1_\ui+\frac{Q_{i_r,i_{r+1}}(X_r,X_{r+1})-Q_{i_r,i_{r+1}}(X_r,\xi)}{X_{r+1}-\xi}(X_{r+1}-\xi)1_\ui
	\end{equation}
	multiplied by $(-\xi)^a\prod_{t=1}^{r-1}\cQ_{i_t,j_0}(X_t,\xi)\omega_{r+1}$. However, \eqref{lost2} is true, because $\tau_r^21_\ui=Q_{i_r,i_{r+1}}(X_r,X_{r+1})1_\ui$. Thus we verified \eqref{eq:ExtR2}. Finally we consider  \eqref{eq:extra-rel}. Diagrammatically it is a consequence of the already checked formulas and  the fact that $\Omega_{r,j_0}^a$ commutes with $\tau_p$ for $r\ne p$ which is also clear geometrically.
\end{proof}
	
\begin{rk}
The element $\xi$ does not make sense globally in $H_*^{\rmG_\bfn}(\Znaiv)$, but only when we work in $H_*^{\rmG_\bfn}(\Znaiv_{j_0})$. This fact explains why,  algebraically,  the formula for the action of $\Omega_{r,j_0}^a$ cannot be written simply as $\omega_{r'}$ times a polynomial as in \eqref{eq:Omega-via-z}. 
\end{rk}

\section{The source algebras and the sink algebras}

\label{sec:source-sink}
We still assume the setup from the opening paragraphs of \S\ref{sec:KLR} and \S\ref{sec:coloured}. In particular, $\Gamma=(I,A)$ is a quiver with no loops and $\bfn$ is a dimension vector. 

We approach now our problem of defining a variety $\Zplus$ whose equivariant Borel-Moore homology is the algebra $\hR_\bfn$ by first realising some important subalgebras of $\hR_\bfn$ geometrically and then combine the constructions. 

For this assume that $K\subset I$ is a subset of vertices of $\Gamma$ and consider the corresponding set of Grassmannian dimension vectors $\bfk$ that are supported on $K$, i.e., we only take those $\bfk$ where $k_i=0$ for $i\in I\backslash K$.
We define the \emph{subalgebra of $\hR_\bfn$ Grassmannian supported on $K$}, as the subalgebra of $\hR_\bfn$ generated by $R_\bfn$ and by the floating dot with colours from $K$. 
\begin{rk} The subalgebras of $\hR_\bfn$ Grassmannian supported on some $K$ are  the algebras called \emph{$\mathfrak p$-KLR algebras} in  the context of categorified parabolic Verma modules,  see \cite[Def. 4.1, Def.~7.6]{NaisseVaz}. 
\end{rk}

We will see in \S \ref{subs:source}, that if we pick for $K$ the set of sources in $\Gamma$, then  
the subalgebra Grassmannian supported on $\vecI$ can be constructed using a certain subvariety $\Zplus$ of $\Znaiv$ defined by the \emph{Grassmannian--Steinberg conditions}, \Cref{sinkcondition}.  In case $K$ is the set of sinks in $\Gamma$, the same holds, but now for a variety $\Zplus$ not inside $\Znaiv\subset E_\bfn\times \calF\times \calF\times \calG$, but rather inside $E_\bfn\times E_\bfn\times \calF\times \calF\times \calG$, i.e., we need to use a second copy of $E_\bfn$. We start by explaining these source and sink cases.  We believe that the geometric source/sink algebras are interesting on their own.

\subsection{The source algebra}
\label{subs:source}
We say that a vertex $i\in I$ of $\Gamma$ is a \emph{source}, if there is no arrow which ends in $i$. 
Let $\vecI\subset I$ be the set of sources of $\Gamma$ and let $\vec J_{\bfn}\subset\Jbfn$ be the set of Grassmannian dimension vectors $\bfk$ that are supported on $\vecI$, i.e., we only take those $\bfk$ where $k_i=0$ for $i\in I\backslash \vecI$.
\begin{df} The \emph{source algebra}  $\vechR_{\bfn}$ is the subalgebra of $\hR_\bfn$ Grassmannian supported on $\vecI$, that is the subalgebra generated by $R_\bfn$ and by the floating dots with colours in $\vecI$. 
\end{df}
We denote by $\vec\rmY$, $\vec\GGZ$, $\vec\Znaiv$ the modification of $\rmY$, $\GGZ$, $\Znaiv$ obtained by only allowing the Grassmannian dimension vectors $\bfk\in \vec J_{\bfn}$. We define the replacement $\vec\Zplus$ of $\vec\Znaiv$. 
\begin{df}\label{sinkcondition} Define the subvariety 
$\vec\Zplus=\{\bV,\tV,W)\mid \alpha(W)=0\}\subset\vec\Znaiv$,
where the condition $\alpha(W)=0$ means $\alpha_h(W_i)=0$ for each arrow $h\colon i\to j$. 
\end{df}

\begin{df}
Define a $\bbZ$-grading on the algebra $H_*^{\rmG_\bfn}(\vec\Znaiv)$ as in \Cref{geomgradhR}.
\end{df}

The embedding from \Cref{closedmultistrong} restricts to an injective algebra homomorphism $\vechR_{\bfn}\hookrightarrow H_*^{\rmG_\bfn}(\vec\Znaiv)$ and the following holds:
\begin{thm}\label{geomsourcealgebra}
The push-forward of the inclusion $\vec\Zplus\subset \vec\Znaiv$ identifies $H_*^{\rmG_\bfn}(\vec\Zplus)$ with the source algebra $\vechR_{\bfn}$ inside $H_*^{\rmG_\bfn}(\vec\Znaiv)$.
\end{thm}
We call $\vec\Zplus$  the \emph{Grassmannian--Steinberg variety Grassmannian supported on sources}.
\begin{rk}Despite the name, we should warn the reader that this variety $\vec\Zplus$ makes only sense when we work in the specific setup of this section. Although the homology of $\vec\Zplus$ is an idempotent truncation of the homology of the Grassmannian--Steinberg variety $\Zplus$ defined later, the variety $\vec\Zplus$ is \emph{not} a  union of connected components of  $\Zplus$ (whereas $\vec\Zplus\times E_\bfn$ would be of this form and would be also  generalizable). 
\end{rk}
\begin{proof}
Consider the variety $\vec\ZplusGG=\{(\alpha,\bV,W,\tV,\tW)\mid \alpha(W)=0\}\subset\vec\GGZ$. Then $H_*^{\rmG_\bfn}(\vec\ZplusGG)$ is a subalgebra of $H_*^{\rmG_\bfn}(\vec\GGZ)$ by \Cref{convcirc} and the same argument as in \S\ref{subs:Z'-subalg-Z} shows that $H_*^{\rmG_\bfn}(\vec\Zplus)$ is a subalgebra of $H_*^{\rmG_\bfn}(\vec\Znaiv)$. By looking at the generators, is is not hard to see that the image of $\vechR_{\bfn}\hookrightarrow H_*^{\rmG_\bfn}(\vec\Znaiv)$ is contained in $H_*^{\rmG_\bfn}(\vec\Zplus)$. (We only need to check this for $\Omega$.) To conclude it suffices to show that 
the algebras $\vechR_{\bfn}$ and  $H_*^{\rmG_\bfn}(\vec\Zplus)$ have the same graded dimension, which is  \Cref{blacktea}. 
\end{proof}

\begin{lem} \label{blacktea}
The algebras $\vechR_{\bfn}$ and  $H_*^{\rmG_\bfn}(\vec\Zplus)$ have the same graded dimension.
\end{lem}
\begin{proof}
For each Gell $\calO^{\rm Gell}_{x,y,z,\uj}\subset \calF_\ui\times\calF_\uj\times\calG_\bfk$, denote by $p^{-1}(\calO^{\rm Gell}_{x,y,z,\uj})$ its preimage with respect to the projection $\vec\Zplus\to \calF\times\calF\times \vec\calG$.
Note that this projection is a vector bundle over each Gell.
The preimages $p^{-1}(\calO^{\rm Gell}_{x,y,z,\uj})$ of Gells form an analogue of a basic paving of $\vec\Zplus$. As in \S\ref{sec:flags-cells}, the classes $[p^{-1}(\calO^{\rm Gell}_{x,y,z,\uj})]$ form then a $\Pol_n$-basis of the associated graded of $H_*^{\rmG_\bfn}(\vec\Zplus)$, and this basis can be lifted to a homogeneous basis of $H_*^{\rmG_\bfn}(\vec\Zplus)$. We would like to show that the elements of this basis have the same degrees as the elements $\tau_x\Omega_{k,n}\tau_y\tau_z1_\uj$ of the $\Pol_n$-basis of $\vechR_{\bfn}$ in \Cref{prop:NVbasis-gen} corresponding to the Grassmannian dimension vectors supported on $\vecI$. 

Using \Cref{checkbasis}, it is enough to show that the paving $\vec\Zplus=\coprod_{x,y,z,\uj} p^{-1}(\calO^{\rm Gell}_{x,y,z,\uj})$ is weakly adapted to $\{\tau_x\Omega_{k,n}\tau_y\tau_z1_\uj\}_{x,y,z,\uj}$.
  For weakly  adaptedness we need to verify
	\begin{equation}\label{gradeddim}
     \deg(\tau_x\Omega_{k,n}\tau_y\tau_z1_\uj)=2\dim \rmY_{\ui,\bfk}+\sum_{i\in I}(k_i^2-k_i)-2\dim p^{-1}(\calO^{\rm Gell}_{x,y,z,\uj}).
	\end{equation}
Assume first that the quiver has no arrows.  Then the equality \eqref{gradeddim} holds by \eqref{eq-weaklyad-Gells-sl2}. For a general $\Gamma$ we argue as for \eqref{problem}: using \Cref{eq:dim-col-Gell}, we reformulate \eqref{gradeddim} as 
\begin{equation}\label{reformulate}
\Xright(x)+\Xright(y)+\Xright(z)+\Xright(w_{0,k})=\dim_\alpha \rmY_{\ui,\bfk}-\dim_\alpha p^{-1}(\calO^{\rm Gell}_{x,y,z,\uj}).
\end{equation}
To verify this pick $\bV\in \calF_\ui$, $\tV\in \calF_\uj$ and  $W\in\calG_\bfk$ such that $(\bV,\tV,W)\in \calO^{\rm Gell}_{x,y,z,\uj}$. 

We have $\dim_\alpha \rmY_{\ui,\bfk}=\dim X(\bV)$. Moreover, $\dim_\alpha p^{-1}(\calO^{\rm Gell}_{x,y,z,\uj})=\dim (X(\bV^W)\cap X(\tV))$, since $X(\bV^W)$ is the subvariety of $X(\bV)$ given by the condition $\alpha(W)=0$.  

On the other hand, $\rel(\bV,\bV^W)=x$ and $\rel(\bV^W,\tV)=yz$ by the definition of $\calO^{\rm Gell}_{x,y,z,\uj}$. From  \Cref{lem:diff-dim-KLR} we get then 
\begin{equation*}
\dim X(\bV)-\dim X(\bV^W)=\Xright(x),\quad\dim X(\bV^W)-\dim (X(\bV^W)\cap X(\tV))=\Xright(yz).
\end{equation*}
Altogether we obtain
\begin{align*}
&\dim_\alpha \rmY_{\ui,\bfk}-\dim_\alpha p^{-1}(\calO^{\rm Gell}_{x,y,z,\uj})\;=\;\dim X(\bV)-\dim (X(\bV^W)\cap X(\tV))\\
=\quad& (\dim X(\bV)-\dim X(\bV^W))+(\dim X(\bV^W)- \dim (X(\bV^W)\cap X(\tV)))\\ 
=\quad &\Xright(x)+\Xright(yz)\;=\;\Xright(x)+\Xright(y)+\Xright(z).
\end{align*}
Since we assumed that $\bfk\in \vec J_{\bfn}$, i.e. Grassmannian vetors are supported on sources, we have $\Xright(w_{0,k})=0$ and therefore \eqref{reformulate} holds. Thus we showed weakly adaptedness, \eqref{gradeddim}. Finally, the equality of the graded dimensions follows by comparing bases, since $\cB$ is a basis of the source algebra $\vechR_{\bfn}$.
\end{proof}
\subsection{The sink algebras algebraically}\label{problemtarget}
Analogous to the source algebra in \S\ref{subs:source} we like to restrict now to sink vertices.
We say that a vertex $i\in I$ of $\Gamma$ is a \emph{sink}, if there is no arrow which starts in $i$. 
Let $\cevI\subset I$ be the set of sinks of $\Gamma$ and let $\cev J_{\bfn}\subset\Jbfn$ be the set of Grassmannian dimension vectors $\bfk$ that are supported on $\cevI$.
\begin{df} The \emph{sink algebra}  $\cevhR_{\bfn}$ is the subalgebra of $\hR_\bfn$ Grassmannian supported on $\cevI$, that is generated by $R_\bfn$ and the floating dots with colours in $\cevI$. 
\end{df}

Similarly, let $\cevhPol_{\bfn}$ be the subalgebra of $\hPol_{\bfn}$ generated by the $X_r$, $1_\ui$ and by the omegas having colours in $\cevI$ (i.e. by $\omega_r 1_\ui$ such that $i_r\in I_t$). Now consider $\hPol_{\bfn}$ as a representation of the sink algebra  $\cevhR_{\bfn}$. Then $\cevhPol_{\bfn}\subset \hPol_{\bfn}$ is clearly a subrepresentation which is still faithful.

To get the geometric source algebra we have replaced the naive variety $\vec\Znaiv$ by a smaller variety $\vec\Zplus$. This strategy however cannot work for sink algebras:
\begin{ex}\label{ex:sinkcase}
	\label{ex:contr-target}
	Let $\Gamma$ be the quiver with two vertices: $i$ and $j$, and one arrow $i\to j$. Set $\bfn=i+j$, $\ui=(i,j)$. Algebraically, we have the floating dot $\Omega_2 1_\ui$ which acts on the polynomial representation by multiplication with $(X_2-X_1)\omega_21_\ui$.
	In contrast, the fundamental class of $[\GGZ_{\ui,\ui,j}]$ acts on $\hPol_\bfn$ by $\omega_21_\uj$. We observe that it is not possible to match the algebraic formula by considering a smaller variety because $\Znaiv_{\ui,\ui,j}$ is already just a point! 
	
To compare with what happens for the source algebra, consider the same example but with the arrow reversed, thus $j\to i$ instead of $i\to j$. Now $\Omega_2 1_\ui$ acts on the polynomial representation by $(X_1-X_2)\omega_21_\ui$. The fundamental class $[\GGZ_{\ui,\ui,j}]$ acts on $\hPol_\bfn$ by $\omega_21_\uj$. But now $\Znaiv_{\ui,\ui,j}\cong \Hom(V_j,V_i)$ and the variety $\Zplus\subset \Znaiv$ is cut out by imposing $\alpha=0$. Now the subvariety is again a point, but the action differs from the sink case by an extra factor $X_1-X_2$ coming from the Euler class of $\Hom(V_j,V_i)$.
\end{ex}

\Cref{ex:sinkcase} shows that sink algebras need a different geometric treatment. Our main idea is, in comparison to the source algebras,  to add first an additional copy of $E_\bfn$ to the variety $\rmY$ and then also some additional twist mixing the two copies of $E_\bfn$. This approach was inspired by the Naisse--Vaz diagrams, more precisely by the \emph{diagram varieties} which we can attach to each such diagram; see \Cref{app:diag-var} for a definition and more motivation. More precisely, we believe that in the naive version \S\ref{subssubs:diag-naive} some transversality property holds for every basis diagram which has floating dots only for source vertices, and that this is the reason why the source case works in an easy way. This expected transversality in the source case fails however in the general case, see \Cref{ex:contrex-transv} for an example, and \S\ref{subsubs:diag-twisted} for a fix.

\subsection{$\calG$-twists of vector bundles}
\label{subs:W-twist}
We provide now some basic linear algebra facts which will
be useful to describe vector bundles which appear in our constructions.
\begin{df} \label{defWtwist}
Let $E$ be a complex vector space with a fixed direct sum decomposition $E=U\oplus U'$. Then the \emph{$U$-flip map} is the linear map 
\begin{equation}\label{Wtwist}
\phi_U\colon E\oplus E\to E\oplus E,\qquad
(u_1+u_1',u_2+u_2')\mapsto (u_2+u_1',u_1+u_2').
\end{equation}
where $u_1,u_2\in U$,  $u_1',u_2'\in U'$.
\end{df}
\begin{rk}
If $U=E$, then $\phi_U$ just flips the summands. It is the
identity map if $U=\{0\}$. In general,  $\phi_U$ exchanges the two
copies of $U$ in $E\oplus E$ without touching the $U'$ parts.
\end{rk}
One can easily compute the \emph{$U$-twists} $\phi_U(E'\oplus E)$ for a subspace $E'\subset E$: 
\begin{lem}\label{Wtwistlemma}
For any vector subspace $E'\subset E$, we have
\begin{equation}\label{lem:twist-E'E}
\phi_U(E'\oplus E)=(E'+U)\oplus(E'\cap U\oplus U').
\end{equation}
\end{lem}
\begin{ex} \label{exWtwist}
Back to our general setup in \S\ref{sec:coloured}, consider the $I$-graded vector space $V=\oplus_{i\in I}V_i$ with fixed subspaces $W_i\subset V_i$ for all $i$. We equip $V$ with a hermitian scalar product which is compatible with the $I$-grading. 
This induces a decomposition $V_i=W_i\oplus W_i^\perp$ for any $i$. Now we can take 
\begin{eqnarray*}
&E:=E_\bfn=\bigoplus_{i\to j}\Hom(V_i,V_j)=U\oplus U', \text{ where } & \\
& U=\bigoplus_{i\to j}\Hom(V_i,W_j)\quad\text{and} \quad U'=\bigoplus_{i\to j}\Hom(V_i,W_j^\perp). &
 \end{eqnarray*}
 For fixed $\bV\in\calF$ consider the corresponding \emph{Springer conditions}
 \begin{equation}\label{KLR}
 E':=\{\alpha\in E\mid \alpha_h(\bV^r_i)\subset \bV_j^r \text{ for all } h\colon i\to j\}.
 \end{equation}
By \Cref{Wtwistlemma} the $U$-twist of $E'\oplus E\subset E\oplus E=\bigoplus_{i\to j}\Hom(V_i,V_j\oplus V_j)$ equals
\begin{equation*}
\phi_U(E'\oplus E)=\{\overline\alpha\in E_\bfn\oplus E_\bfn\mid \overline\alpha (\bV_i^r)\subset (\bV_j^r+W_j)\oplus ((\bV^r_j\cap W_j)+W_j^\perp)\}.
\end{equation*}
More explicitly, $(\alpha,\beta)\in \phi_U(E'\oplus E)$ satisfies the \emph{twisted Springer conditions}
\begin{equation*}
\label{eq:twisted-preserve-cond}
\phi_U(E'\oplus E)=\{(\alpha,\beta)\in E_\bfn\oplus E_\bfn\mid 
\alpha_h(\bV_i^r)\subset \bV_j^r+W_j,\; \beta_h(\bV_i^r)\subset (\bV^r_j\cap W_j)+W_j^\perp\}. 
\end{equation*}
\end{ex}

We will now generalize the flip map to vector bundles. For this let $X$, $Y$ be smooth complex varieties and let 
$\calV\to X\times Y$ be the trivial vector bundle on $X\times Y$ with
fibre $V$, and denote by $\calV_X$ and $\calV_Y$ the trivial vector
bundles with fibre $V$ on $X$ and $Y$ respectively. These vector bundles have an induced scalar product on each fibre. We can view $\calV$ as the pull-backs of $\calV_X$ and of  $\calV_Y$ along the projection to $X$ respectively $Y$.
Similarly, for subbundles $\calK_X$ of $\calV_X\oplus\calV_X$ and $\calU_Y$ of
$\calV_Y$ we denote by  $\calK$ and $\calU$ the subbundles of $\calV\oplus\calV$ respectively  $\calV$ 
which are the pull-backs to $X\times Y$ along the respective projection map.
Since the fibre over $(x,y)\in X\times Y$ of the vector bundle $\calU$
is the fibre of $\calU_Y$ over $y$ (independently of $x$), we denote it
just by $U_y$. Similarly, for the fibre $K_x$ of $\calK$ over $x$.
We now define vector bundles  by glueing $U$-flip maps:
\begin{df} 
Consider $\calV$ as a real vector bundle. The {$\calU$-twist} $\phi_\calU(\calK)$ of $\calK$ for the triple $(X,Y,V)$ is the subbundle of
$\calV\oplus\calV$  with fibre $\phi_{U_y}(K_x)$ over $(x,y)$. 
\end{df}
\begin{rk}
As a bundle of real varieties,  $\phi_\calU(\calK)$ is isomorphic to $\calK$ (via
$\phi_\calU$).  The base and the fibres are complex varieties, although it is in general not a bundle of complex varieties (since the map is not necessarily algebraic).
\end{rk}
In practise we will use $\calU$-twists to show that certain manifolds are (real) vector bundles by realizing them as some $\phi_\calU(\calK)$. Since for us prominent examples will be such that $\calU$ is the
tautological bundle of a Grassmannian we call these twists
$\calG$-twists or \emph{Grassmannian twists}. 
\begin{ex}\label{exWtwistfancy}
As an upgrade of \Cref{exWtwist}, we can define the \emph{extended KLR twist bundle} $\phi_\calU(\calK)$. For this take $X=\calF$, $Y=\calG$, $V=E_\bfn$. Then define $\calU_X$ as subbundle of $\calV_Y$ and $\calK_X$ as the subbundle of  $\calV_X\oplus\calV_X$  with fibres $U$ respectively $E'\oplus E$ as in \Cref{exWtwist}. 
\end{ex}

\subsection{The sink algebras geometrically}
\label{subs:target}

Now, we explain how to construct geometrically the subalgebra $\cevhR_\bfn$ of $\hR_\bfn$.
We denote by $\cev\rmY$, $\cevGGZ$ the modification of $\rmY$, $\GGZ$ from \Cref{generalnot} by only allowing the Grassmannian dimension vectors $\bfk\in \cev J_{\bfn}$. To define the replacement $\cev\Zplus$ of $\Znaiv$ we first enlarge $\cev\rmY$ by  an extra copy of $E_\bfn.$ 

Set $\bfY_{\ui,\bfk}=\rmY_{\ui,\bfk}\times E_\bfn$, $\bfY_\bfk=\rmY_\bfk\times E_\bfn$ and $\cev\bfY=\cev\rmY\times E_\bfn$. We denote elements of the original copy of $E_\bfn$ by $\alpha$ and of the new copy by $\beta$, thus $(\alpha, \beta)\in E_\bfn\oplus E_\bfn=\bigoplus_{i\to j}\Hom(V_i,V_j\oplus V_j)$. Finally  fix a (hermitian) scalar product on $V$ such that different $V_i$'s are orthogonal.
\begin{df} \label{crucialdef}
Let $\Ytw_{\ui,\bfk}$ be the (real) vector bundle\footnote{Although $\Ytw$ is by definition a real algebraic variety,  it is homeomorphic to $\bfY$, and we can view it as a complex algebraic variety. The variety $\Ztw$ below is however real.}  over $\calF\times\calG$ whose fiber over $(\bV,W)$ is the subspace of 
$E_\bfn\times   E_\bfn\times\calF_\ui\times \calG_\bfk$ determined by the tuples
\begin{equation}\label{twistSprcond}
\{(\alpha,\beta,\bV,W)\mid 
\alpha_h(\bV_i^r)\subset \bV_j^r+W_j,\;\beta_h(\bV_i^r)\subset (\bV^r_j\cap W_j)+W_j^\perp\; 
\text{for } h:i\rightarrow j\},
\end{equation}
Set also
\begin{eqnarray*}
&\Ytw_\bfk=\coprod_{\ui\in I^\bfn}\Ytw_{\ui,\bfk}\quad\text{and}\quad\cevYtw=\coprod_{\ui\in I^\bfn,\bfk\in \cev J_{\bfn}}\bfY_{\ui,\bfk},&\\
&\Ztw_{\bfk_1,\bfk_2}=\Ytw_{\bfk_1}\times_{E_\bfn\oplus E_\bfn}\Ytw_{\bfk_2},\qquad \cevZtw=\coprod_{\bfk_1,\bfk_2\in \cev J_{\bfn}}\Ztw_{\bfk_1,\bfk_2},&\\
&\Zplus_{\bfk}=\Ztw_{\bfk,0},\qquad\cev\Zplus=\coprod_{\bfk\in \cev J_{\bfn}}\Zplus_\bfk.&
\end{eqnarray*}

\end{df}
\begin{rk}
Note that $\Ytw_{\ui,\bfk}$ is the $\calU$-twist of the bundle with fibres $E'\oplus E$ given by the Springer conditions \eqref{KLR} for the triple $(\calF,\calG, E_\bfn)$, see \Cref{exWtwistfancy}. We therefore call the conditions from \eqref{twistSprcond} \emph{twisted Springer conditions}.
\end{rk}

Consider now the algebra $H_*^{\rmG_\bfn}(\cevZtw)$, with product given by convolution defined with respect to the inclusion $\cevZtw\subset \cevYtw\times\cevYtw$. We next define generalizations of the Grassmannian inclusion subvarieties from \S\ref{subs:creation}. Namely we consider
\begin{equation}\label{geninclvar}
\cevZtw_\supset:=\{(\alpha,\beta,\bV,\tV,W,\tW)\mid W\supset \tW\}\subset\cevZtw\quad\text{and}  \quad\Ztw_{\bfk_1\supset\bfk_2}=\Ztw_{\bfk_1,\bfk_2}\cap \cevZtw_\supset. 
\end{equation}
If $\bfk_1,\bfk_2,\bfk_1-\bfk_2\in \cev J_{\bfn}$, then we have 
\begin{equation*}
\Ztw_{\bfk_1\supset\bfk_2}\to \Zplus_{\bfk_1-\bfk_2}, 
\quad 
(\alpha,\beta, \bV,\tV,W,\tW)\mapsto (\phi_{U}(\alpha,\beta),\bV,\tV,W,W\cap \tW^\perp),
\end{equation*}
where $U= \bigoplus_{i\to j}\Hom(V_i,\tW_j)$ and $U'= \bigoplus_{i\to j}\Hom(V_i,\tW_j^\perp)$. 
By definition of the flip map, this is a locally trivial fibration with Grassmannians as fibers; the fibers are  the same fibres as for
$\GGZ_{\bfk_1\supset\bfk_2}\to \Zplus_{\bfk_1-\bfk_2}$.

We have a linear map $H_*^{\rmG_\bfn}(\cev\Zplus)\to H_*^{\rmG_\bfn}(\cevZtw)$ defined as the composition of the pull-back $H_*^{\rmG_\bfn}(\cev\Zplus)\to H_*^{\rmG_\bfn}(\cevZtw_\supset)$ and the push-forward $H_*^{\rmG_\bfn}(\cevZtw_\supset)\to H_*^{\rmG_\bfn}(\cevZtw)$.

\begin{prop}
The space $H_*^{\rmG_\bfn}(\cev\Zplus)$ is a subalgebra of $H_*^{\rmG_\bfn}(\cevZtw)$.
\end{prop}
\begin{proof}
We just copy the arguments of \S \ref{subs:Z'-subalg-Z}.  In that case, we knew (using Gells) that $H_*^{\rmG_\bfn}(\cev\Zplus)$ embeds into $H_*^{\rmT}(\cev\Zplus)_{\rm loc}$ but the injectivity for the similar map for $\cevZtw$ is not clear. So, the argument above allows to identify $H_*^{\rmG_\bfn}(\cev\Zplus)$ with a subalgebra of the quotient of $H_*^{\rmG_\bfn}(\cevZtw)$ by the kernel of the localisation map. Moreover, since the localisation map is injective on the image of $H_*^{\rmG_\bfn}(\cev\Zplus)$ in $H_*^{\rmG_\bfn}(\cevZtw)$, we see that $H_*^{\rmG_\bfn}(\cev\Zplus)$ is identified with a subalgebra of $H_*^{\rmG_\bfn}(\cevZtw)$ (not just of its quotient).
\end{proof}

\begin{df}\label{geomgradhR-sink} 
	Define a $\bbZ$-grading on the algebra $H_*^{\rmG_\bfn}(\cev\Zplus)$ by putting $H_r^{\rmG_\bfn}(\Zplus_{\ui,\uj,\bfk})$ in degree $
	-r+2\dim\Ytw_{\ui,\bfk}+\sum_{i\in I}(k_i^2-k_i).$ 
\end{df}
\begin{rk}
We omit the straight-forward check that  \Cref{geomgradhR-sink} is a well-defined grading. The main difference to  \Cref{geomgradhR} is that we use now $\dim\Ytw_{\ui,\bfk}$ instead of $\dim\rmY_{\ui,\bfk}$. In fact, we have $\dim\Ytw_{\ui,\bfk}=\dim\bfY_{\ui,\bfk}=\dim\rmY_{\ui,\bfk}+\dim E_\bfn$.
\end{rk}
\begin{thm} \label{geomsinkalgebra} 
There is an isomorphism of graded algebras $\cevhR_\bfn\cong H_*^{\rmG_\bfn}(\cev\Zplus)$.
\end{thm}
\begin{proof}
The algebra $H_*^{\rmG_\bfn}(\cevZtw)$ acts on $H_*^{\rmG_\bfn}(\cevYtw)\cong H_*^{\rmG_\bfn}(\cev\rmY)\cong\cevhPol_{\bfn}$.  We can restrict this action to $H_*^{\rmG_\bfn}(\cev\Zplus)$. Similarly to what we did before, for each generator $x\in\cevhR_{\bfn}$ we can construct some $x\in H_*^{\rmG_\bfn}(\cev\Zplus)$ acting by the same operator on $\cevhPol_\bfn$. Indeed, there is nothing to do for $\tau_r$ and $X_p$, we only have to consider $\Omega 1_\ui$ with $i_1\in \cevI$. It is easy to see (doing a modification of the computations above) that it acts by the same operator (namely multiplication by $\omega_11_\ui$), as the fundamental class of the closed subset of $\Zplus_{\ui,\ui,i_1}$ given by $\bV=\tV$ and  $W=\bV^1$. This is explained in more detailed in the more general situation below, see \Cref{lem:gen-mixed}.
 Moreover the $H_*^{\rmG_\bfn}(\cev\Zplus)$-action on $H_*^{\rmG_\bfn}(\cevYtw)$ is faithful, as can be shown using the approach of \Cref{rk:faith-from-strat} and the paving of $\cev\Zplus$ by preimages of Gells. Thus,  we get an injective algebra homomorphism $\cevhR_{\bfn}\hookrightarrow H_*^{\rmG_\bfn}(\cev\Zplus)$. This is an isomorphism, since the two algebras have the same graded dimension by \Cref{lem:check-dim-target} below. 
\end{proof}
Let $\calO^{\rm Gell}_{x,y,z,\uj}$ be a Gell in $\calF_\ui\times\calF_\uj\times\calG_\bfk$ and let $p^{-1}(\calO^{\rm Gell}_{x,y,z,\uj})$ be its preimage under the projection $\cev\Zplus\to\calF\times\calF\times\cev\calG$.
\begin{lem}
\label{lem:check-dim-target}
	For $\bfk\in\cev J_\bfn$, we have
	\begin{equation}	\label{eq:check-dim-target}
	\deg(\tau_x\Omega_{k,n}\tau_y\tau_z1_\uj)=	2\dim \Ytw_{\ui,\bfk}+\sum_{i\in I}(k_i^2-k_i)-2\dim \calO^{\rm Gell}_{x,y,z,\uj}.
	\end{equation}
\end{lem}
\begin{proof}

If the quiver has no arrows, then the equality \eqref{eq:check-dim-target} holds by \eqref{eq-weaklyad-Gells-sl2}. For a general quiver $\Gamma$ we reformulate \eqref{eq:check-dim-target}, as  we did for \eqref{problem}, using \Cref{eq:dim-col-Gell} as  
\begin{equation}
\label{eq:to-check-alphabar-target}
\Xright(x)+\Xright(y)+\Xright(z)+\Xright(w_{0,k})=\dim_{\alpha,\beta} \Ytw_{\ui,\bfk}-\dim_{\alpha,\beta}\left(p^{-1}(\calO^{\rm Gell}_{x,y,z,\uj})\right).
\end{equation}
Here $\dim_{\alpha,\beta}$ denotes the dimension of the fibre of the vector bundle given by forgetting both $E_\bfn$-components. To make sure that this notation makes sense consider 
\begin{equation}
\begin{array}[t]{rcccrccc}
\label{eq:2bundles-target}
q\colon&
\Ytw_{\ui,\bfk}&\to& \calF_\ui\times \calG_\bfk,& \text{and}& p^{-1}(\calO^{\rm Gell}_{x,y,z,\uj})&\to& \calO^{\rm Gell}_{x,y,z,\uj}\\
&(\alpha,\beta,\bV,W)&\mapsto&(\bV,W),&
&(\alpha, \beta,\bV,\tV,W)&\mapsto &(\bV,\tV,W).
\end{array}
\end{equation}
These are subbundles of the trivial vector bundles with fibre $E_\bfn\oplus E_\bfn$. Moreover, it is easy to see from the construction, that in both cases every fibre $F\subset E_\bfn\oplus E_\bfn$ is homogeneous with respect to the decomposition $E_\bfn\oplus E_\bfn$, i.e., we always have $F=F_\alpha\oplus F_\beta$, where $F_\alpha$ and $F_\beta$ are the intersections of $F$ with the first $E_\bfn$ and second copy of $E_\bfn$ respectively. Note that $\dim F=\dim F_\alpha+\dim F_\beta$ is independent of the choice of fibre, but $\dim F_\alpha$ and  $\dim F_\beta$ heavily depend on the choice.

Pick therefore $(\bV,W)\in\calF_\ui\times \calG_\bfk$ and $(\bV,\tV,W)\in\calO^{\rm Gell}_{x,y,z,\uj}$, and consider their fibres in \eqref{eq:2bundles-target}. We get  decompositions (dependent of our choice)
\begin{eqnarray*}
&\dim_{\alpha,\beta} \Ytw_{\ui,\bfk}=\dim_{\alpha} \Ytw_{\ui,\bfk}+\dim_{\beta} \Ytw_{\ui,\bfk},&\\ &\dim_{\alpha,\beta}p^{-1}(\calO^{\rm Gell}_{x,y,z,\uj})=\dim_{\alpha}p^{-1}(\calO^{\rm Gell}_{x,y,z,\uj})+\dim_{\beta}p^{-1}(\calO^{\rm Gell}_{x,y,z,\uj}).&
\end{eqnarray*}
Note that we have $\dim_{\beta}\Ytw_{\ui,\bfk}=\dim_{\beta}p^{-1}(\calO^{\rm Gell}_{x,y,z,\uj})$. Thus, \eqref{eq:to-check-alphabar-target} becomes
\begin{equation}\label{toshow}
\Xright(y)+\Xright(z)=\dim_{\alpha} \Ytw_{\ui,\bfk}-\dim_{\alpha}p^{-1}(\calO^{\rm Gell}_{x,y,z,\uj}),
\end{equation}
since by assumption (sink case), we have $\Xright(x)=\Xright(w_{0,k})=0$.
On the other hand,
\begin{equation}\label{toshow2}
\dim_{\alpha} \Ytw_{\ui,\bfk}=\dim X(\bV^W), \qquad \dim_{\alpha}p^{-1}(\calO^{\rm Gell}_{x,y,z,\uj})=\dim (X(\bV^W)\cap X(\tV)),
\end{equation}
by \Cref{rkVW} and the arguments after \eqref{reformulate}. By \Cref{lem:diff-dim-KLR}, these formulas imply \eqref{toshow}. Since \eqref{toshow} is equivalent to \eqref{lem:check-dim-target}, this finishes the proof.
\end{proof}
\subsubsection*{Summary} We treated sources and sinks quite differently. To obtain the source algebra we did not do anything special with the variety $\vec\rmY$, but defined $\vec\Zplus$ by putting conditions cutting $\cev\rmY\times_{E_\bfn}\vec\rmY_0$. To obtain the sink algebra, we modified $\cev\rmY$ to $\cevYtw$ (by adding a copy of $E_\bfn$ and a twist), but defined $\cev\Zplus$ as $\cevYtw\times_{E_\bfn\times E_\bfn}\cevYtw_0$ (without passing to a proper closed subset). We fusion now these two approaches.

\section{Grassmannian quiver Hecke algebras}
\label{sec:isom-thm-GqHecke-gen}
In this section we prove the main isomorphism theorem which realises the Naisse--Vaz algebras geometrically in terms of Grassmannian quiver Hecke algebras. 

We still assume the setup from the opening paragraphs of \S\ref{sec:KLR} and \S\ref{sec:coloured}. In particular, $\Gamma=(I,A)$ is a quiver with no loops and $\bfn$ is a dimension vector. 

\subsection{The Grassmannian--Steinberg varieties}\label{111}
 Let $\bfk$ be an arbitrary element of $\Jbfn$.
We define $\bfY_{\ui,\bfk}$ and $\Ytw_{\ui,\bfk}$ as in \Cref{crucialdef}, except we do not assume $\bfk\in \cev J_{\bfn}$. Set $\Ztw=\Ytw\times _{E_\bfn\oplus E_\bfn}\Ytw$ and $\Znaivtw=\Ytw\times _{E_\bfn\oplus E_\bfn}\Ytw_0$. We also consider certain unions of connected components of $\Ztw$ exactly as for $\GGZ$, see \Cref{Zdd} and the definitions afterwards. We now can define our main player:
\begin{df}\label{DefGrassSt}
The \emph{Grassmannian--Steinberg variety}  is the subvariety $\Zplus\subset\Znaivtw$ of    
all $(\alpha,\beta,\bV,\tV,W)$ satisfying the following \emph{Grassmannian--Steinberg conditions}
\begin{equation} 
\label{eq:cond-bfZ'-gen}
\alpha_h((\bV^r\cap W_i)^\perp\cap W_i)\subset (\bV^r\cap W_j)^\perp\cap W_j,
\end{equation}   
for any $h\colon i\to j$ and $r\in [0;n]$; that is, the flag $\VWp$ is preserved by $\alpha$. 
\end{df}
\begin{rk}
In case $W$ is supported on sources, \eqref{eq:cond-bfZ'-gen} recovers \Cref{sinkcondition}.
In case $W$ is supported on sinks, \eqref{eq:cond-bfZ'-gen} is no condition  (compatible with \Cref{subs:target}).
\end{rk}
Let $\calO^{\rm Gell}_{x,y,z,\uj}$ be a Gell in $\calF_\ui\times\calF_\uj\times\calG_\bfk$ and $p^{-1}(\calO^{\rm Gell}_{x,y,z,\uj})$ its preimage in $\Zplus$. The next statement, illustrated in \Cref{exWinM}, will help to compute their dimensions.
\begin{lem} \label{relpositioncalc}
Assume that $\Gamma$ has exactly one arrow $i\to j$. Let  $\bV\in\calF_\ui$ and $W\in\calG_\bfk$ with $\rel(\bV,W)=x$. 
Then $\rel(\bV,\VWp)=xw_{0,k}$ and $xw_{0,k}=x_1x_2$ with $x_1=\rel(\bV,\bV^{W_j})$, $x_2=\rel(\bV^{W_j},\VWp)$. Moreover, $\ell(x_1x_2)=\ell(x_1)+\ell(x_2).$
\end{lem}
\begin{proof}
The first statement follows from the definitions, see \Cref{relVWp}. By \Cref{lem:rel-triple-flags}, we have $\rel(\bV,\bV^{W_j})\cdot \rel(\bV^{W_j},\VWp)=\rel(\bV,\VWp)$ which implies $xw_{0,k}=x_1x_2$. Since $x_1^{-1}(a)>x_1^{-1}(b)$ implies $(xw_{0,k})^{-1}(a)>(xw_{0,k})^{-1}(b)$ for any $1\leqslant a<b \leqslant n$, we also have $\ell(x_1x_2)=\ell(x_1)+\ell(x_2).$
\end{proof}

\begin{prop}
\label{lem:dim-deg-gen}
We have
\begin{equation}
\label{eq:dim-deg-gen}
\deg(\tau_x\Omega_{k,n}\tau_y\tau_z1_\uj) =	2\dim \Ytw_{\ui,\bfk}+\sum_{i\in I}(k_i^2-k_i)-2\dim \calO^{\rm Gell}_{x,y,z,\uj}.
\end{equation}
\end{prop}
\begin{proof}
In case the quiver has no arrows, \eqref{eq:dim-deg-gen} holds by \eqref{eq-weaklyad-Gells-sl2}. For a general quiver $\Gamma$ we argue as for \eqref{problem}: using \Cref{eq:dim-col-Gell}, we reformulate \eqref{eq:dim-deg-gen} as 
\begin{equation}\label{toshowmixed}
\Xright(x)+\Xright(y)+\Xright(z)+\Xright(w_{0,k})=\dim_{\alpha} \rmY_{\ui,\bfk}-\dim_{\alpha} p^{-1}(\calO^{\rm Gell}_{x,y,z,\uj}),
\end{equation} 	
where we reduce, as in the proof of \Cref{lem:check-dim-target} from $\dim_{\alpha,\beta}$ to $\dim_{\alpha}$.

Now assume that $\Gamma$ contains only one arrow, say $h:i\to j$. Then 
\begin{align}
&\dim_\alpha \Ytw_{\ui,\bfk}-\dim_\alpha p^{-1}(\calO^{\rm Gell}_{x,y,z,\uj})=\dim X(\bV^{W_j})-\dim (X(\VWp)\cap X(\tV))\nonumber\\
=&\quad(\dim X(\bV^{W_j})-\dim X(\VWp)+(\dim X(\VWp)-\dim (X(\VWp)\cap X(\tV)))\nonumber\\
=&\quad\Xright(xw_{0,k})+\Xright(yz).\label{toshowmixed2}
\end{align}
The third equality here follows from \Cref{lem:diff-dim-KLR} applied to $(\VWp,\tV)$ (which gives $\dim X(\VWp)-\dim (X(\VWp)\cap X(\tV))=\Xright(yz)$) and  \Cref{lem:diff-dim-KLR} applied to $(\bV^{W_j},\VWp)$. Indeed, 
with $x_1$, $x_2$ as in \Cref{relpositioncalc} we have $\Xright(x_1)=0$ and then $\Xright(x_2)=\Xright(xw_{0,k})$, and therefore 
$\dim X(\bV^{W_j})-\dim X(\VWp)=\dim X(\bV^{W_j})-\dim X(\bV^{W_j})\cap X(\VWp)=\Xright(x_2)=\Xright(xw_{0,k})$.
We thus showed \eqref{toshowmixed} in case $\Gamma$ has one arrow. It is clear if there is now arrow.
Now \eqref{toshowmixed} follows by writing $\Gamma$ as a sum of quivers with one arrow, see \Cref{defsum}, since all terms are additive. 
\end{proof}
We illustrate the main step of the proof and \Cref{relpositioncalc} in an example.
\begin{ex}\label{exWinM}
	Consider the quiver $i\to j$. Let $\bfn=2i+3j$, $\bfk=i+2j$, $\ui=(i,j,i,j,j)$ and $x=s_1s_2s_3$. For $\bV\in \calF_\ui$, $W\in\calG_\bfk$, the condition $\rel(\bV,W)=x$ means there exist  orthogonal bases $(v_{i,1},v_{i,2})$ of $V_i$ and $(v_{j,1},v_{j,2},v_{j,3})$ of $V_j$ such that
	\begin{equation}\label{exV}
	\bV^1=\langle v_{i,1}\rangle, \bV^2=\langle v_{i,1},v_{j,1}\rangle,\bV^3=\langle v_{i,1},v_{j,1},v_{i,2}\rangle, \bV^4=\langle v_{i,1},v_{j,1},v_{i,2},v_{j,2}\rangle,
	\end{equation}
	\begin{equation}\label{exW}
	 W_i=\langle v_{i,2}\rangle, W_j=\langle v_{j,1}, v_{j,2}\rangle.
\end{equation}
The flags  $\bV^W$, $\VWp$, $\bV^{W_j}$ are then (here each step changes the dimension by $1$):
\begin{align*}
&\bV^W:\quad&\{0\}\subset \langle v_{j,1}\rangle\subset \langle v_{j,1},v_{i,2}\rangle\subset \langle v_{j,1},v_{i,2},v_{j,2}\rangle \subset \langle v_{j,1},v_{i,2},v_{j,2},v_{i,1}\rangle \subset V.\\
&\VWp:\quad&\{0\}\subset \langle v_{j,2}\rangle\subset \langle v_{j,2},v_{i,2}\rangle\subset \langle v_{j,2},v_{i,2},v_{j,1}\rangle \subset \langle v_{j,2},v_{i,2},v_{j,1},v_{i,1}\rangle \subset V.\\
&\bV^{W_j}:\quad&\{0\}\subset \langle v_{j,1}\rangle\subset \langle v_{j,1},v_{j,2}\rangle\subset \langle v_{j,1},v_{j,2},v_{i,1}\rangle \subset \langle v_{j,1},v_{j,2},v_{i,1},v_{i,2}\rangle \subset V.
\end{align*}
To illustrate this diagrammatically we draw the coloured permutation $(\ui, xw_{0,k})$, 
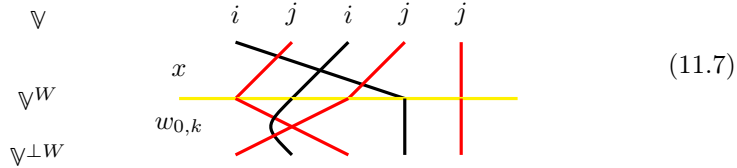
	\begin{equation}\label{exWinM1}
	\tikz[very thick,scale=.75,baseline={([yshift=1.0ex]current bounding box.center)}]{	
	        \node at (-2.5,2.4){$\bV$};
		\node at (1,2.5){$i$};
		\node at (2,2.5){$j$};
		\node at (3,2.5){$i$};
		\node at (4,2.5){$j$};
		\node at (5,2.5){$j$};
		
		\draw (1,2)-- (4,1); 	
		\draw (2,2)-- (1,1)[color=red]; 
		\draw (3,2)-- (2,1); 
		\draw (4,2)-- (3,1)[color=red]; 
		\draw (5,2)-- (5,1)[color=red]; 
		
		\draw (0,1)-- (6,1)[color=yellow];
		  \node at (-2.5,1){$\bV^W$};
		\draw (1,1)-- (3,0)[color=red]; 	
		\draw (2,1) .. controls (1.5,0.5) .. (2,0); 
		\draw (3,1)-- (1,0)[color=red]; 
		\draw (4,1)-- (4,0); 
		\draw (5,1)-- (5,0)[color=red]; 
	  \node at (-2.5,0){$\VWp$};
        \node at (0,1.5){$x$};
        \node at (0,0.5){$w_{0,k}$};
	} 
	\end{equation}
with $w_{0,k}$ at the bottom (which is the longest element in $\frakS_k$ with $k=2+1$) and $x$, the given shortest coset representative at the top. The colour sequence $\bi$ at the top tells us $\bV\in\calF_\ui$ and the strands indicate the flag $\bV$ from \eqref{exV} (each strand stands for a new basis vector). Then $\bV^W$, and $W$ from \eqref{exW},  can be read off by looking at the intersection with the yellow horizontal line (the strands still stand for the basis vectors).  Finally, $\VWp$ is encoded by the strands at the bottom of the diagram.

Now we want to factor $xw_{0,k}$ as $x_1x_2$ with $x_1=\rel(\bV,\bV^{W_j})$, $x_2=\rel(\bV^{W_j},\VWp)$. In practise, we can take $\bV^{W_j}$ and encode it diagrammatically (see the yellow horizontal line). Then $x_1$ can be read off by connecting with strands in a minimal way to the top,  and $x_2$ appears when we extend afterwards at the bottom to get $xw_{0,k}$.	
\begin{equation}
	\tikz[very thick,scale=.75,baseline={([yshift=1.0ex]current bounding box.center)}]{	
		\node at (1,2.5){$i$};
		\node at (2,2.5){$j$};
		\node at (3,2.5){$i$};
		\node at (4,2.5){$j$};
		\node at (5,2.5){$j$};
		   \node at (-2.5,2.4){$\bV$};
		\draw (1,2)-- (3,1); 	
		\draw (2,2)-- (1,1)[color=red]; 
		\draw (3,2)-- (4,1); 
		\draw (4,2)-- (2,1)[color=red]; 
		\draw (5,2)-- (5,1)[color=red]; 
	        \node at (-2.5,1){${\bV^{W_j}}$};
		\draw (0,1)-- (6,1)[color=yellow];
		
		\draw (1,1)-- (3,0)[color=red]; 	
		\draw (2,1)-- (1,0)[color=red]; 
		\draw (3,1)-- (4,0); 
		\draw (4,1)-- (2,0); 
		\draw (5,1)-- (5,0)[color=red]; 
		  \node at (-2.5,0){$\VWp$};
	
	    \node at (0,1.5){$x_1$};
	    \node at (0,0.5){$x_2$};
	} 
\end{equation}
We can also argue purely diagrammatically by connecting the top and bottom sequences (determined by $\ui$ and $xw_{0,k}$) by a product $x_1x_2=xw_{0,k}$ of permutations. Hereby  $x_1$ only moves some  strands coloured $j$ (corresponding to $W_j$) to the left, and $x_2$ does the rest.
In particular, $\Xright(x_1)=0$, and the flag  ${\bV^{W_j}}$ can then be read off in the middle of the diagram.   
We see that $x_1=\rel(\bV,\bV^{W_j})=s_3s_1s_2$, $x_2=\rel(\bV^{W_j},\bV'^{W})=s_3s_1s_2$, although the two elements play very different roles. 		
\end{ex}
\subsection{Main theorems}
In this subsection we formulate the two main results: the \emph{subalgebra theorem} (\Cref{mainprop}) and the \emph{isomorphism theorem} (\Cref{mainthm}).

Similarly to what we did in \S\ref{subs:target}, we have the linear map $H_*^{\rmG_\bfn}(\Znaivtw)\to H_*^{\rmG_\bfn}(\Ztw)$ defined as the composition of the pull-back $H_*^{\rmG_\bfn}(\Znaivtw)\to H_*^{\rmG_\bfn}(\Ztw_\supset)$ and the push-forward $H_*^{\rmG_\bfn}(\Ztw_\supset)\to H_*^{\rmG_\bfn}(\Ztw)$. We also have the push-forward map $H_*^{\rmG_\bfn}(\Zplus)\to H_*^{\rmG_\bfn}(\Znaivtw)$.
	Similarly to \Cref{geomgradhR-sink}, we consider the following grading. 
	\begin{df}\label{geomgradhR-gen} 
	We equip the vector space $H_*^{\rmG_\bfn}(\Zplus)$ with a $\bbZ$-grading by putting $H_r^{\rmG_\bfn}(\Zplus_{\ui,\uj,\bfk})$ in degree $
		-r+2\dim\Ytw_{\ui,\bfk}+\sum_{i\in I}(k_i^2-k_i).$ 
	\end{df}
	
\begin{thm}\label{mainprop}
The convolution algebra structure on $H_*^{\rmT}(\Ztw)_{\rm loc}$ induces a graded algebra structure on $H_*^{\rmG_\bfn}(\Zplus)$ via the maps $H_*^{\rmG_\bfn}(\Zplus)\to H_*^{\rmG_\bfn}(\Znaivtw)\to H_*^{\rmG_\bfn}(\Ztw)\to H_*^{\rmT}(\Ztw)_{\rm loc}$. It turns $H_*^{\rmG_\bfn}(\Zplus)$ into a graded subalgebra of $H_*^{\rmT}(\Ztw)_{\rm loc}$.  
\end{thm}

The second main result is the following \emph{isomorphism theorem}:	
\begin{thm}\label{mainthm} 
There is an isomorphism of graded algebras $\hR_\bfn\cong H_*^{\rmG_\bfn}(\Zplus)$. \hfill\\
    Moreover, the isomorphism can be chosen such that the following holds:
\begin{align*}
1_\uj&\mapsto[\dot\Delta_{\uj,\uj,0}], \quad \tau_r1_\uj\mapsto{(-1)^{\delta_{j_r,j_{r+1}}}}[\mathbb{T}_{r,\uj}],\quad \Omega1_\uj\mapsto [\mathbb{\Omega}_\uj],\\
X_r1_\uj&\mapsto\text{the first Chern class of the line bundle on $\dot\Delta_{\uj,\uj,0}$ given by $\bV^r/\bV^{r-1}$,} 
\end{align*}
where $\mathbb{\Omega}_\uj:=\{(\alpha,\beta,\bV,\tV,W)\mid \bV=\tV, W=\bV^1\}\subset\Zplus_{\uj,\uj,1}$
and 
$\mathbb{T}_{r,\uj}:=\{(\alpha,\beta,\bV,\tV)\mid \bV^p=\tV{}^p \text{ for }p\ne r\}\subset \Zplus_{s_r(\uj),\uj,0}$, and $\dot\Delta_{\uj,\uj,0}:=\{(\alpha,\beta,\bV,\tV)\mid \bV=\tV\}\subset\Zplus_{\uj,\uj,0}$.
\end{thm}
\begin{df}
	\label{def:Gr-quiver-Hecke}
We call $H_*^{\rmG_\bfn}(\Zplus)$ the \emph{Grassmannian quiver Hecke algebra}.
\end{df}
\Cref{mainthm}  says now that the extended KLR algebra is isomorphic to the Grassmannian quiver Hecke algebra (and \Cref{mainprop} says that that the Grassmannian quiver Hecke algebra is well-defined). See also \Cref{rk:KLR-vs-quivHecke}.

\begin{rk}
For the proof, we like to proceed as for the source and sink algebras and
deduce that $H_*^{\rmG_\bfn}(\Zplus)$ is an algebra isomorphic to  $\hR_\bfn$, but there are some subtleties in the general situation.  The variety $\Zplus$ is nice, because it is a vector bundle over each Gell, and we get an inclusion $H_*^{\rmG_\bfn}(\Zplus)\hookrightarrow H_*^\rmT(\Zplus)_{\rm loc}$. However, it is not clear if this also holds for $\Znaivtw$ and $\Ztw$ and we do not see a reason why $H_*^{\rmG_\bfn}(\Ztw)$ should act faithfully on $H_*^{\rmG_\bfn}(\Ytw)$. Since the map from $H_*^{\rmG_\bfn}(\Znaivtw)$  to the localisation is not necessarily injective, we cannot argue as in \S\ref{subs:Z'-subalg-Z} 
to define a product on $H_*^{\rmG_\bfn}(\Znaivtw)$ from the product on $H_*^{\rmG_\bfn}(\Ztw)$.

We overcome these problems as follows. We know that for each generator $x$ of the algebra $\hR_\bfn$ there exists some $x\in H_*^{\rmG_\bfn}(\Ztw)$ which acts in the same way on the polynomial representation. Without faithfulness, this does not allow to embed $\hR_\bfn$ into $H_*^{\rmG_\bfn}(\Ztw)$, but allows to embed $\hR_\bfn$ into the quotient $\overline{H}_*^{\rmG_\bfn}(\Ztw)$ of $H_*^{\rmG_\bfn}(\Ztw)$ by the kernel of the action. We then use the argument of \S\ref{subs:Z'-subalg-Z}  
to identify a quotient of $H_*^{\rmG_\bfn}(\Znaivtw)$ with a subalgebra of $\overline{H}_*^{\rmG_\bfn}(\Ztw)$ and show that the image of $\hR_\bfn$ is contained in this subalgebra. Finally, we show that the image of $\hR_\bfn$ is in the image of $H_*^{\rmG_\bfn}(\Zplus)$ and we identify them by counting degrees. 
\end{rk}
\subsection{The proof of the main theorem}
As a preparation of the proof consider the following commutative diagram: 
\begin{equation}
\label{eq:CD0-Zplus-loc}
\begin{CD}
H_{*}^\rmT(\Zplus)_{\rm loc} @>>> H_{*}^\rmT(\Znaivtw)_{\rm loc} @>>> H_{*}^\rmT(\Ztw)_{\rm loc}\\
@AAA    @AAA   @AAA\\
H_{*}^{\rmG_\bfn}(\Zplus) @>>> H_{*}^{\rmG_\bfn}(\Znaivtw) @>>> H_{*}^{\rmG_\bfn}(\Ztw).
\end{CD}
\end{equation}
Here, the left vertical map is injective, but we do not know if the other vertical maps are injective. Denoting by $\overline{H}_{*}^{\rmG_\bfn}(\Znaivtw)$ and $\overline{H}_{*}^{\rmG_\bfn}(\Ztw)$ the quotients of $H_{*}^{\rmG_\bfn}(\Znaivtw)$ and $H_{*}^{\rmG_\bfn}(\Ztw)$ respectively by the kernels of the vertical maps we obtain: 
\begin{lem}
All maps in the following commutative diagram are injective.
\begin{equation}
\label{eq:CD1-Zplus-loc}
\begin{CD}
H_{*}^\rmT(\Zplus)_{\rm loc} @>>> H_{*}^\rmT(\Znaivtw)_{\rm loc} @>>> H_{*}^\rmT(\Ztw)_{\rm loc}\\
@AAA     @AAA   @AAA\\
H_{*}^{\rmG_\bfn}(\Zplus) @>>> \overline{H}_{*}^{\rmG_\bfn}(\Znaivtw)@>>> \overline{H}_{*}^{\rmG_\bfn}(\Ztw).
\end{CD}
\end{equation}
\end{lem}
\begin{prop}
\label{lem:gen-mixed}
For each generator $x$ of $\hR_\bfn$, there exists $x\in H_{*}^{\rmG_\bfn}(\Znaivtw)$ whose image in  $H_{*}^{\rmG_\bfn}(\Ztw)$ acts on $H_{*}^{\rmG_\bfn}(\Ytw)\cong \hPol_{\bfn}$ by the same operator.
\end{prop}
\begin{proof} We claim that the assignments in \Cref{mainthm} (with $\Zplus$ replaced by $\Znaivtw$) satisfy the desired property.
We already checked in \S\ref{sec:gen-qv} that this works for the naive version $\Znaiv$. Let us explain what happens when we add the twist. 
Since $\Znaivtw_0=\Znaiv_0\times E_\bfn$ and $\Ztw_{k\supset k}\cong\GGZ_{k\supset k}\times E_\bfn$, the twist changes nothing for $1_\uj$, $x_r1_\uj$, $\tau_r1_\uj$. 

We are left with $\Omega1_\uj$. We have $\Znaivtw_1\ne \Znaiv_1\times E_\bfn$, but these varieties become the same when we add the conditions $\bV=\tV$ and $W=\bbV^1$. Moreover, the image in $H_*^{\rmG_\bfn}(\Ztw_{k+1,k})$ of the class of the subvariety given by these conditions is the class of the subvariety defined by $\bV=\tV$, $W=\tW\oplus \bV^1$ and this direct sum is orthogonal. This subvariety is the same as in the non-twisted case. Thus, a computation using $\rmT$-fixed points gives the same coefficients in the twisted as in the untwisted case.
\end{proof}

\begin{coro}\label{Cormono}
	There is a monomorphism of algebras $\Phi\colon\hR_\bfn\to H_{*}^{\rmT}(\Ztw)_{\rm loc}$. 
\end{coro}
\begin{proof}
	An element of $H_{*}^{\rmG_\bfn}(\Ztw)$ acts by zero on $H_{*}^{\rmG_\bfn}(\Ytw)$ if and only if it is zero in $H_{*}^{\rmT}(\Ztw)_{\rm loc}$.
\end{proof}

	Similarly to the computation in \Cref{starintertwines}, we see that the algebra structure on  $H_{*}^\rmT(\Ztw)_{\rm loc}$ restricts to an algebra structure on $H_{*}^\rmT(\Znaivtw)_{\rm loc}$. 
\begin{lem}
		The image of $\Phi$ is contained in the image of $H_{*}^{\rmT}(\Znaiv)_{\rm loc}$.
	\end{lem}
\begin{proof}
	It is clear from \Cref{lem:gen-mixed} that the image of every generator of $\hR_\bfn$ is in $H_{*}^{\rmT}(\Znaiv)_{\rm loc}$. Since $H_{*}^{\rmT}(\Znaiv)_{\rm loc}$ is a subalgebra of $H_{*}^\rmT(\Ztw)_{\rm loc}$, the claim follows.
	\end{proof}

An argument similar to \Cref{coromult} shows that the image of $\overline{H}_{*}^{\rmG_\bfn}(\Znaivtw)$ in $H_{*}^{\rmT}(\Ztw)_{\rm loc}$ is a subalgebra. This yields in particular an algebra structure on $\overline{H}_{*}^{\rmG_\bfn}(\Znaivtw)$. 
\begin{lem}
	\label{lem:image-Hbardot}
	The image of $\Phi$ is in the image of $\overline{H}_{*}^{\rmG_\bfn}(\Znaivtw)$.
\end{lem}

As explained above, the element $\Omega 1_\uj$ is given in $H_*^{\rmG_\bfn}(\Ztw_{k+1,k})$  (see \Cref{lem:gen-mixed}) by the fundamental class of the subvariety described in the proof of \Cref{lem:gen-mixed}. For any point $(\alpha,\beta,\bV,\tV,W,\tW)$ of this subvariety we have the following.
\begin{lem}
	\label{lem:corr-Omega-image}
	 Let $h\colon i\to j$ be an arrow in the quiver $\Gamma$.  Then it holds
	\begin{equation}
	\label{eq:corr-Omega-image}
	\alpha_h(W_i\cap \tW_i^\perp)\subset \tW_j. 
	\end{equation}
\end{lem}

\begin{proof}
	We have $W_i\cap \tW_i^\perp=\bV^1\cap V_i$ by the definition of the subvariety. Now  \eqref{eq:corr-Omega-image} follows from the twisted Springer condition in $\Ytw_{\uj,k}$, see \Cref{crucialdef}.
\end{proof}

\begin{prop}
	\label{lem:fact-thr-bfZ'}
The image of $\Phi$
is contained in the image of $H_*^{\rmG_\bfn}(\Zplus)$. 
\end{prop}
\begin{proof}
By \Cref{lem:image-Hbardot}, we have a monomorphism of algebras $\hR_\bfn\hookrightarrow \overline{H}_*^{\rmG_\bfn}(\Znaivtw)$. We would like to check that its image is in the image of $H_*^{\rmG_\bfn}(\Zplus)$. It is enough to check this for every basis element $\tau_x\Omega_{k,n}\tau_y\tau_z1_\uj$ of $\hR_\bfn$. Let us write this basis element as a product $G_1G_2\cdots G_a$ of elementary generators of the algebra (crossings, polynomials, floating dots). Each $G_p$ is sent to some homology class in $\overline{H}_*^{\rmG_\bfn}(\Znaivtw)$ defined as a push-forward of some homology class on some elementary correspondence $C_p$.
\color{black}
More precisely,
\begin{itemize}[leftmargin=5mm]
    \item if $G_p=P1_\ui$ with $P\in\Bbbk[X_1,\ldots,X_n]$, then $C'_p$ is the diagonal ($\bV=\tV$) in $\Znaivtw_{\ui,\ui,0}$,
    \item if $G_p=\tau_r1_\ui$, then $C'_p$ is "almost diagonal" ($\bV^q=\tV{}^q$ for $q\ne r$) in  $\Znaivtw_{s_r(\ui),\ui,0}$,
    \item if $G_p=\Omega 1_\ui$, then $C'_p \subset \Znaivtw_{\ui,\ui,1}$ is given by $\bV=\tV$ and $W=\bV^1$.
\end{itemize}
We lift now $C'_1,\ldots,C'_a$ to correspondences $C_1,\ldots,C_a$ in $\Ztw$. For $p\in[1;a]$ let  $f(p)$ be the number of floating dots among the $G_{p+1}, G_{p+2},\ldots, G_a$. 
We define next $C_p\subset \Ztw_{f(p)+1,f(p)}$ if $G_p$ is a floating dot, and $C_p\subset \Ztw_{f(p),f(p)}$ otherwise. Namely,  let $C_p$ be the preimage of $C'_p$ under the map $\Ztw_{f(p)+1\supset f(p)}\to \Znaivtw_1$ respectively $\Ztw_{f(p)\supset f(p)}\to \Znaivtw_0$. Then the basis element $G_1G_2\cdots G_a$ in $\overline{H}_*^{\rmG_\bfn}(\Znaivtw_\bfk)= \overline{H}_*^{\rmG_\bfn}(\Ztw_{\bfk,0})$ is the push-forward of some homology class on $C_1\circ C_2\circ\cdots\circ C_a$. (We use the notation of \S\ref{subs:conv}.) We claim that  
$C_1\circ C_2\circ\ldots\circ C_a\subset \Zplus$. We need to show that if $(y_0,y_{1},\ldots,y_a)\in (\Ytw)^{a+1}$ such that $(y_{p-1}, y_p)\in C_p$ for $p\in[1;a]$, then $y_0$ satisfies Grassmannian--Steinberg conditions \eqref{eq:cond-bfZ'-gen}.

To see this we need some notation. Let $y_0=(\alpha,\beta, \bV,\tV,W=\bigoplus_{i\in I}W_i)$. Moreover, let $p_1<p_2<\ldots<p_k$ be the indices such that $G_{p_r}$ are floating dots. Then denote by $W^r$ the Grassmanian space from $y_{p_r}$. Then $W=:W^0\supset W^1\supset W^2\supset\cdots\supset W^k=\{0\}$ is a full flag in $W$. For $r\in[1;k]$, pick some nonzero $w_r\in W^{r-1}\cap (W^r)^\perp$. The vectors $w_1,w_2,\ldots,w_k$ form then an orthogonal basis of $W$ and we have $W^r=\langle w_{r+1},w_{r+2},\ldots,w_k\rangle$ for  $r\in[0;k]$.

In case $\Gamma$ has no arrow, there is nothing to do, otherwise pick any arrow $h\colon i\to j$. 
Then, since we have $w_r\in W^{r-1}\cap (W^r)^\perp$, \Cref{lem:corr-Omega-image} applied to $C_{p_r}$ yields
$$
\alpha_h(\langle w_r\rangle \cap W_i)\subset W^r \cap W_j=\langle w_{r+1},w_{r+2},\ldots,w_k\rangle\cap W_j.
$$
We see from the formula above that $\alpha$ preserves each $W^r=\langle w_{r+1},w_{r+2},\ldots,w_k\rangle$.

To verify the Grassmannian--Steinberg conditions \eqref{eq:cond-bfZ'-gen} for $y_0$, let $r\in [1;n]$. By definition of the correspondences, there is some index $b$ such that $w_1,\ldots,w_b\in\bV^r\cap W$ and $w_{b+1},w_{b+2},\ldots,w_k$ are orthogonal to $\bV^r\cap W$. Thus, $\bV^r\cap W_i=(W^b)^\perp\cap W_i$ and $(\bV^r\cap W_i)^\perp \cap W_i=W^b\cap W_i$. Now, \eqref{eq:cond-bfZ'-gen} follows, since $\alpha$ preserves $W^r$.
\end{proof}

\begin{rk}
	The proof above becomes more natural if we think of it in terms of \emph{diagram varieties}, a notion we introduce in \Cref{app:diag-var}. In fact, this proof explains that for each basis diagram $D$, the image of the map $\dvp_D\colon V(D)\to \Znaivtw$ (we use the notation of \S\ref{subsubs:diag-twisted}) satisfies \eqref{eq:cond-bfZ'-gen}. In this diagram varieties, the vectors $w_1,\ldots,w_k$ span the spaces $\langle w_1\rangle,\ldots, \langle w_k\rangle$  corresponding to the regions of the diagram $D$ with floating dots (counted from top to bottom), see also \Cref{ex:diag-twisted}.
\end{rk}

\begin{proof}[Proof of \Cref{mainprop} and \Cref{mainthm}]	
	
	By \Cref{lem:fact-thr-bfZ'}, the monomorphism $\Phi$ factors through $H_*^{\rmG_\bfn}(\Zplus)$. On the other hand, the graded dimensions of $\hR_\bfn$ and $H_*^{\rmG_\bfn}(\Zplus)$ agree by \Cref{lem:dim-deg-gen}. Thus the image of $\Phi$ is exactly (the image of) $H_*^{\rmG_\bfn}(\Zplus)$. In particular, $H_*^{\rmG_\bfn}(\Zplus)\subset {H}_*^{\rmT}(\Ztw)_{\rm loc}$ is a subalgebra.
	

The explicit formulas for the image of the generators of $\hR_\bfn$ in $H_*^{\rmG_\bfn}(\Zplus)$ are given by the choice of images of the generators in \Cref{lem:gen-mixed}, see \Cref{mainthm}.
\end{proof}

\subsection{Positive characteristics}
\label{subs:pos-char}
We finish by showing that our main theorem (\Cref{mainthm}) still holds in positive characteristics. When going through the arguments of the paper, one can observe that the characteristic zero assumption on  $\Bbbk$ was only relevant for the calculations in the convolution algebra. More precisely, the theory of torus fixed points localisation requires in general the characteristic zero assumption. The purpose of this section is to verify that in our specific setup the arguments can, at least with some extra care,  still be applied.  

We allow now the field $\Bbbk$ to be of arbitrary characteristic. Let still $\rmU_\bfn$ be the compact subgroup of $\rmG_\bfn$ given by unitary matrices. Set $\compT=\rmU_\bfn\cap \rmT$. If not specified otherwise, $\abY$ denotes a topological $\rmU_\bfn$-space.

\begin{prop}\label{cond:loc1} 
 The map $H_*^{\rmU_\bfn}(\abY)\to H_*^{\compT}(\abY)$ is always injective.
\end{prop}
\begin{proof}
By \cite[Thm.~2.10]{HolmSjamaar} this holds independently of the characteristic of $\Bbbk$, since $\rmU_\bfn$ is a subgroup of a reductive group of type $A$.
\end{proof}

Recall the varieties $\Zplus,\Ztw,\Znaivtw,\Ytw$ from \S\ref{111}.
The decomposition of $\Zplus$ into preimages of Gells, see \Cref{DefGellgl}, directly implies the injectivity of the localisation map from \Cref{locmap}: 
\begin{prop}\label{cond:loc2} 
If $\abY\in\{\Zplus,\Ytw\}$ then the localisation map $H_*^{\compT}(\abY)\to H_*^{\compT}(\abY)_{\rm loc}$ is injective (in any characteristic).
\end{prop}

\begin{rk}
We like to stress that it is not clear (to us) whether \Cref{cond:loc2}  holds also for $\Ztw$ and $\Znaivtw$.
\end{rk}

The existence of an  isomorphism of ${\rm Frac}(\tR_\compT)$-modules $H_*^{\compT}(\abY)_{\rm loc}\cong H_*^{\compT}(\abY^\compT)_{\rm loc}$ requires extra assumptions. To formulate this we call  $\compT_1\subset \compT$ a \emph{standard subtorus} if $\compT_1$ is given by imposing equalities between some coordinates after identifying $T\cong (S^1)^n$. For example $\{(a,b,a,c,b,c,c)\mid a,b,c\in S^1\}\subset (S^1)^7$ is a standard subtorus. Let $\St(\compT)$ be the set of standard subtori of $\compT$, equipped with the partial ordering given by inclusion, that is $\compT_1\leq \compT_2$ if $\compT_1\subset \compT_2$. The following holds for any characteristic of $\Bbbk$.

\begin{lem}
	\label{lem:loc-Y-YT-gen}
If the  $\compT$-stabiliser of any point in $\abY$ is a standard subtorus, then the push-forward induces an isomorphism of ${\rm Frac}(\tR_\compT)$-modules $H_*^{\compT}(\abY)_{\rm loc}\cong H_*^{\compT}(\abY^\compT)_{\rm loc}$.
\end{lem}
\begin{proof}
Consider  $\compT_1\in\St(\compT)$.  Denoting by $\abY^{(\compT_1)}\subset \abY$ the subset of points whose stabiliser is $\compT_1$ we have
	$$
	\abY^{(\geq \compT_1)}:=\bigcup_{\compT_2\geq \compT_1}\abY^{(\compT_2)}=\abY^{\compT_1}.
	$$
	This is a closed subset of $\abY$.  This still holds for any completion of the partial order $\geq$ on $\compT$ to a total order. 
		
	Given now  $\compT_1\in\St(\compT)$, we can find some (not necessarily standard) subtorus $\compT'_1\subset \compT$ such that $\compT=\compT_1\times \compT'_1$. Consequently, $H_*^{\compT}(Y^{(\compT_1)})\cong H_*^{\compT_1}(Y^{(\compT_1)}/\compT'_1)$, on which the action of $\compT_1$ is trivial. In particular, $H_*^{\compT}(Y^{(\compT_1)})$ is a $\tR_\compT$-torsion module in case $\compT_1\ne \compT$. 
	Now, using the long exact sequences 
	$$
	\ldots \to H_r^\compT(\abY^{(> \compT_1)})\to H_r^\compT(\abY^{(\geq\compT_1)})\to H_r^\compT(\abY^{(\compT_1)})\to\ldots,
	$$
	we prove by induction that $H_r^\compT(\abY\backslash \abY^\compT)$ is also a torsion $\tR_\compT$-module for any $r$. Thus, the corresponding modules in the long exact sequence 
	$$
	\ldots \to H_{r+1}^\compT(\abY\backslash \abY^\compT)\to H_r^\compT(\abY^\compT)\to H_r^\compT(\abY)\to  H_r^\compT(\abY\backslash \abY^\compT)\to\ldots,
	$$
	vanish after localisation inducing isomorphisms $H_r^{\compT}(\abY)_{\rm loc}\cong H_r^{\compT}(\abY^\compT)_{\rm loc}$ for any $r$. We get the desired  isomorphism $H_*^{\compT}(\abY)_{\rm loc}\cong H_*^{\compT}(\abY^\compT)_{\rm loc}$ of ${\rm Frac}(\tR_\compT)$-modules.
\end{proof}
\begin{prop}\label{cond:loc3} 
The assumptions of  \Cref{lem:loc-Y-YT-gen} hold for $\abY\in \{\Zplus,\Ztw,\Znaivtw,\Ytw\}$.
\end{prop}
\begin{proof}
It is clear that that the $\compT$-spaces $E_\bfn$ and $\op{Gr}_t(V_i)$ (for $t\in [1;n_i]$) satisfy the assumption of \Cref{lem:loc-Y-YT-gen}. Then their products and $\compT$-stable subsets therein also satisfy this assumption, and so does in particular $\abY\in \{\Zplus,\Ztw,\Znaivtw,\Ytw\}$.
\end{proof}

In summary, the push-forward map $H_*^{\compT}(\abY^\compT)\to H_*^{\compT}(\abY)$ induces an isomorphism after localisation for $\abY\in \{\Zplus,\Ztw,\Znaivtw,\Ytw\}$. Note that the Euler classes at the $\compT$-fixed points of $\Ytw$ are products of polynomials of the form $\ccT_r-\ccT_t$. This implies, similarly to \cite[Prop. 2.4]{MakMin}, that for these varieties the pull-backs $H_*^{\compT}(\abY)\to H_*^{\compT}(\abY^\compT)$ with respect to the inclusions  $\Zplus\subset \Ytw\times\Ytw_0$,  $\Ztw\subset \Ytw\times\Ytw$,  $\Znaivtw\subset \Ytw\times\Ytw_0$ are also isomorphisms after localisation. This allows us in particular, to use the argument in \cite[Prop. 2.10]{MakMin} to reduce the computation to $\compT$-fixed points. 
The arguments of \S\ref{sec:isom-thm-GqHecke-gen} can now be applied to prove the main theorem in any characteristic: 
\begin{thm}
Theorem \ref{mainthm} holds over arbitrary fields $\Bbbk$. 
\end{thm}
\begin{proof}
Thanks to Propositions \ref{cond:loc1} ,  \ref{cond:loc2},   \ref{cond:loc3}, and the reduction argument, this follows from \eqref{eq:CD1-Zplus-loc} and the usual $\rmT$-fixed points computations.
\end{proof}

\appendix
\setlength{\multicolsep}{5pt}
\section{Presentations in terms of generators and relations}

\label{app:gen-rel}
We present here $\NH_n$, $\hNH_n$ and $\hhNH_n$  as algebras with generators and relations. The isomorphism to the definitions in \S\ref{sec:hNH} is canonically given by identifying the abstract generators with the concrete operator with the same name.
\begin{lem}
\label{lem:gen-rel-NH}
The nil-Hecke algebra $\NH_n$ is isomorphic to the algebra  generated by elements $X_i$ for $i\in[1;n]$ and $T_r$  for $r\in [1;n-1]$ modulo the \emph{nil-Hecke relations}
\begin{multicols}{2}
\begin{enumerate}
    \item $T_rT_t=T_tT_r$, if $|r-t|>1$,
    \item $T_rT_{r+1}T_r=T_{r+1}T_rT_{r+1}$,
    \item$T_r^2=0$,
    \item $X_iX_j=X_jX_i$,
    \item $T_rX_{r}=X_{r+1}T_r+1$,
    \item $X_{r}T_r=T_rX_{r+1}+1$,
    \item $T_rX_i=X_iT_r$, if $t\ne i,i+1$,
\end{enumerate}
where $i,j\in [1;n]$ and $r,t\in [1;n-1]$.
\end{multicols}
\end{lem}
\begin{lem}
\label{lem:gen-rel-hNH}
The extended nil-Hecke algebra $\hNH_n$ is isomorphic to the algebra generated by $X_i$, $\omega_i$, $T_r$, $i\in[1;n]$, $r\in [1;n-1]$, modulo the nil-Hecke relations and 
\begin{multicols}{2}
\begin{enumerate} [start=8]
    \item $\omega_i\omega_j=-\omega_j\omega_i$,
    \item $X_i\omega_j=\omega_jX_i$,
    \item $T_r\omega_i=\omega_i T_r$, if $r\ne i$,
    \item $[T_r,\omega_r+X_{r}\omega_{r+1}]=0$,
    \item $\omega_i^2=0$ \rm{(}in case $\op{char}(\Bbbk)=2$\rm{)},
\end{enumerate}
where $i,j\in [1;n]$ and $r,t\in [1;n-1]$.
\end{multicols}
\end{lem}

It is not difficult to verify in $\hNH_n$ the formula  
\begin{equation*}
\omega_{i+1}=T_i\omega_iT_iX_{i+1}-X_iT_i\omega_iT_i, \mbox{ for }0\leq i\leq n-1.	
\end{equation*}
This allows a smaller presentation using only $\omega_1$ instead of $\omega_1,\omega_2,\ldots,\omega_n$ as follows. 
\renewcommand{\labelenumi}{\theenumi'.\rm{)}}
\begin{lem}
	\label{lem:gen-rel-hNH-omega1}
	The extended nil-Hecke algebra $\hNH_n$ is isomorphic to the algebra  generated by elements $\omega_1$, $X_i$, $T_r$, $i\in[1;n]$, $r\in [1;n-1]$ modulo the nil-Hecke relations and
	\begin{multicols}{2}
	\begin{enumerate}[start=8]
	\item $\omega_1^2=0$,
	\item $X_i\omega_1=\omega_1X_i$,
	\item $T_r\omega_1=\omega_1 T_r$, if $r\ne 1$,
	\item $T_1\omega_1T_1\omega_1+\omega_1T_1\omega_1T_1=0$.
	\end{enumerate} 
	\end{multicols}
\end{lem}
\begin{proof}[Proofs of Lemmas~\ref{lem:gen-rel-NH}, \ref{lem:gen-rel-hNH}, \ref{lem:gen-rel-hNH-omega1}, \ref{pol-rep-plus-hNH}, \ref{coro:basis-hNH}]
The statement of \Cref{lem:gen-rel-NH}  is well-known, see e.g. \cite[Ex.~2.2(3)]{KL1}. 
Let $\hNH_n'$ denote the asserted presentation of $\hNH_n$.  It is easy to verify that the relations $1.)-12.)$ hold in $\hNH_n$ (note $\omega_r+X_r\omega_{r+1}$ is $s_r$-invariant, see \eqref{Esr}). We get an algebra morphism $\hNH_n'\rightarrow \hNH_n$ which is surjective by \Cref{df:hNH}. We see from \cite[Prop. 8.1]{NVapproach} that there is an isomorphism of vector spaces $\hNH_n'\cong \mywedge_n\otimes \NH_n$, using the identification from \Cref{lem:gen-rel-NH}. This implies that the $\hNH_n'$-action on $\hPol_n$ is faithful because of the faithfulness of the $\NH_n$-action on $\Pol_n$. Then the map $\hNH_n'\rightarrow \hNH_n$ is also injective. Thus we proved \Cref{lem:gen-rel-hNH}. 
\Cref{lem:gen-rel-hNH-omega1} holds by \cite[Cor. 3.17]{NaisseVaz}.
To see \Cref{pol-rep-plus-hNH}, we can write any $h\in \hNH_n$ uniquely as $h=\sum_{\lambda\in\Lambda(n)}\omega_\la h_\la$ with $h_\la\in\NH_n$.  \Cref{pol-rep-plus-hNH}  then follows by definition of $\hNH_n$ and the fact that $\NH_n$ acts faithfully on $\Pol_n$. Note that there is an antiautomorphism of $\hNH_n'$ sending each generator to itself, this implies \Cref{coro:basis-hNH}. 
\end{proof}
\renewcommand{\labelenumi}{\theenumi''.\rm{)}}
\begin{lem}
\label{lem:gen-rel-hhNH}
The algebra $\hhNH_n$ is isomorphic to the algebra  generated by $X_i$, $\omega^+_i$, $\omega^-_i$, $T_r$, $i\in[1;n]$, $r\in [1;n-1]$ modulo the nil-Hecke relations and additionally 

\begin{multicols}{2}
\begin{enumerate}[start=8]
     \item \label{Estart}$\omega^+_i\omega^-_j+\omega^-_j\omega^+_i=\delta_{i,j}$,
     \item $\omega^\pm_i\omega^\pm_j=-\omega^\pm_j\omega^\pm_i$,
       \item $X_i\omega^\pm_j=\omega^\pm_jX_i$,
    \item $T_r\omega^+_i=\omega^+_i T_r$, if $r\ne i$,
    \item $T_r\omega^-_i=\omega^-_i T_r$, if $r\ne i+1$,
    \item $[T_r,\omega^+_r+X_{r}\omega^+_{r+1}]=0$,
        \item \label{Eend} $[T_r,\omega^-_{r+1}-X_r\omega^-_r]=0$.
    \end{enumerate}

\end{multicols}
where $i,j\in [1;n]$ and $r,t\in [1;n-1]$.
\end{lem}

\begin{proof}[Proof of \Cref{lem:gen-rel-hhNH} and \Cref{lem:basis-hhNH}]
It is easy to check that the operators from \Cref{df:hhNH} satisfy the relations \ref{Estart}''.)-\ref{Eend}''.) and of course the nil-Hecke relations. To prove \Cref{lem:gen-rel-hhNH}, it suffices to verify that the algebra $\hhNH_n'$ given by generators and relations in \Cref{lem:gen-rel-hhNH} acts faithfully on $\hPol_n$. 

The relations directly imply that the  set $S$ of elements analogous to \Cref{lem:basis-hhNH} span the algebra $\hhNH_n'$. To show $\hhNH_n'\cong \hhNH_n$, it is enough to check that the elements of $S$ act by linearly independent operators. 
For this, we verify that $\hhNH_n$ is a free left $\hNH_n$-module with a basis $\{\omega^-_\lambda;\mbox{ } \lambda\in\Lambda(n)\}$. Note that each element $x\in\hhNH_n$ can be written as $x=\sum_{\lambda\in\Lambda(n)}x_\lambda\omega^-_\lambda$ where $x_\lambda\in \hNH_n$. We claim that a nonzero coefficient $x_\lambda$ implies that such a linear combination is nonzero, since it acts on $\hPol_n$ by a nonzero operator. Indeed, pick $\lambda_0$ with $|\lambda_0|$ as small as possible such that $x_{\lambda_0}\ne 0$. By \Cref{pol-rep-plus-hNH}, we can find a polynomial $P\in \Pol_n$ such that $x_{\lambda_0}P\ne 0$, and thus $x\cdot P\omega_{\lambda_0}=x_{\lambda_0}P\ne 0$ in $\hPol_n$. Therefore, $\hhNH_n'\cong \hhNH_n$. We implicitly also showed \Cref{lem:basis-hhNH}.
\end{proof}

\section{Diagram varieties}
\label{app:diag-var}
In this section we explain some of the motivation and intuition behind the constructions of the paper. 
The main theme hereby is that, up to technical details, a lot of the geometry can be read off the diagrams defining the nil-Hecke algebra, its extended version and their coloured version from \Cref{sec:KLR}. More explicitly, we will define what we call diagram varieties, which are varieties assigned to the diagrams and basis elements from the algebras we like to describe geometrically. In some sense we \emph{build} varieties and (co)homology classes from diagrams and in this way construct step by step the geometry described in the paper.  We think that the underlying combinatorics is interesting on its own. It lead us naturally to the definition of upper/lower cells, $\calG$-cells, Gells (see \S\ref{subs:col-cells}) and the definitions of the modifications of the variety $\Zplus$ given in \S\ref{subs:source}, \S\ref{subs:target} and \S\ref{sec:isom-thm-GqHecke-gen}.

Since the purpose of this appendix is to give some intuition, we will not get into technical details. As a result we formulate instead several \emph{beliefs} which we believe are true and most likely not too difficult to prove. These statements are not used in the main part of the paper. Assuming them however 
allows a more conceptual explanation of why the overall construction given in the paper really works. In the main part of the paper we do proofs in a more economical (but less conceptual way) based on computations of dimensions.

\subsection{The general idea}
\label{subs:diag-var-gen}

Let  $L,\abX,\abY,\abZ$ be as in \S\ref{subs:conv}. Assume that $C\subset \abZ$ is a (smooth) $L$-stable closed submanifold. Abusing the notation, we denote by $[C]$ the push-forward to $H_*^{L}(\abZ)$ of the fundamental class of $C$. Imagine now that we want to multiply fundamental classes in the convolution algebra $H_*^L(\abZ)$: 
\begin{equation}
\label{eq:prod-fund-cl}
[C_1]\cdot [C_2]\cdot \ldots \cdot [C_a].
\end{equation}
For $1\leqslant i<j\leqslant a+1$, denote by $p_{i,j}$ the projection $p_{i,j}\colon \abY^{a+1}\to \abY^2$ to the $i$th and $j$th components. By the definition of the convolution product (see also the proof of \cite[(2.7.19)]{CG97}), the product \eqref{eq:prod-fund-cl}  is equal to
\begin{equation}\label{CG}
(p_{1,a+1})_*([(p_{1,2})^{-1}(C_1)]\cap [(p_{2,3})^{-1}(C_2)]\cap\ldots\cap [(p_{a,a+1})^{-1}(C_a)]),
\end{equation}
where the $\cap$-product is taken inside the variety $\abY^{a+1}$. \emph{If these intersections \eqref{CG} are transversal}, they can be replaced by the corresponding set-theory intersections
$$
(p_{1,a+1})_*([(p_{1,2})^{-1}(C_1)\cap (p_{2,3})^{-1}(C_2)\cap\ldots\cap (p_{a,a+1})^{-1}(C_a)]).
$$
Moreover, we have 
$$
(p_{1,2})^{-1}(C_1)\cap (p_{2,3})^{-1}(C_2)\cap\ldots\cap (p_{a,a+1})^{-1}(C_a)\cong C_1\times_{\abY}C_2\times_{\abY}\ldots\times_{\abY}C_a.
$$
Thus, if the intersection is transversal, the product \eqref{eq:prod-fund-cl} is equal to 
\begin{equation}
(p_{1,a+1})_*([C_1\times_{\abY}C_2\times_{\abY}\ldots\times_{\abY}C_a]).
\end{equation}
In general, if there is no transversality, the product  \eqref{eq:prod-fund-cl}  is the push-forward $(p_{1,a+1})_*(c)$ of a more complicated class $c\in H_*^L(C_1\times_{\abY}C_2\times_{\abY}\ldots\times_{\abY}C_a)$ which is not always a fundamental class.

We will now connect such calculations with the diagrams from \Cref{sec:KLR}. We will see in concrete examples how the varieties 
$C_1\times_{\abY}C_2\times_{\abY}\ldots\times_{\abY}C_a$ can explicitly be described in terms of a corresponding diagram $D$. We will formalise this observation in the notion \emph{the diagram variety $V(D)$ attached to $D$}.
\subsection{The case of $\mathfrak{sl}_2$}
\subsubsection{Nil-Hecke case}\label{AppNH}
We first consider the nil-Hecke case from \S\ref{sec:hNH}, \S\ref{sec:sl2-geom}.

Consider a word $h_D:=T_{r_1}T_{r_2}\ldots T_{r_p}\in \NH_n$ with corresponding diagram $D$ (it has $n$ strands built from $p$ crossings and no dots, see \S\ref{subs:des-KLR-NV}). 
 We now explain how to see from the diagram which element of $H_*^G(\calF\times\calF)$ corresponds to $h_D$. For this we consider $D$ as a diagram drawn in a horizontal strip $\Xi$ in the plane connecting $n$ dots in the bottom boundary of the strip with $n$ dots at the top boundary. 
 \begin{df}
For a diagram  $D\in \NH_n$ without dots we denote by $\calA=\calA(D)$ the set of connected components in the complement of $D$ in $\Xi$. 
For $a\in\calA$ let $d(a)$ be the distance\footnote{The line segments given by the connected components of the intersection of the closure of $a\in\calA$ with $D$ are the \emph{walls} of $a$. 
Moreover, $a,b\in\calA$ are neighbours if they share a common wall. The \emph{distance} between $a\in\calA$ and the leftmost component $a_0$ is the minimal $d$ such that there exists a sequence $a=a_0,a_1,\ldots ,a_d=a$ in $\calA$ such that successive elements are neighbours.
} from $a$ to the unique leftmost component. (This number is $0$ for the leftmost region and it is $n$ for the unique rightmost region).
\end{df}
For example, the diagrams for $h_{D}=T_2T_1T_3T_2\in \NH_4$ and $h_{D'}=T_1 T_1\in \NH_2$ are drawn in  \Cref{VD}  in black. The additional blue numbers in the regions, i.e. the elements in  $a\in\calA$, are the distances $d(a)$.
\begin{figure}[htb]
\begin{equation*}
	\tikz[very thick,scale=.75,baseline={([yshift=1.0ex]current bounding box.center)}]{	
	\draw (0,3)-- (2,0); 	
	\draw (1,3)-- (3,0); 	
	\draw (2,3)-- (0,0); 	
	\draw (3,3)-- (1,0); 		
   	\node at (0.75,2.5){$\color{blue}{1}$};	
	\node at (1.5,3){$\color{blue}{2}$};	
	\node at (2.25,2.5){$\color{blue}{3}$};	
	\node at (0.75,0.5){$\color{blue}{1}$};	
	\node at (1.5,0){$\color{blue}{2}$};	
	\node at (2.25,0.5){$\color{blue}{3}$};	
	\node at (1.5,1.5){$\color{blue}{2}$};	
	\node at (0,1.5){$\color{blue}{0}$};	
	\node at (3,1.5){$\color{blue}{4}$};
	   \node at (-1,1.5){$D:=$};	
} 
\quad\quad 
	\tikz[very thick,scale=.75,baseline={([yshift=1.0ex]current bounding box.center)}]{	
	\draw (0,3)-- (2,0); 	
	\draw (1,3)-- (3,0); 	
	\draw (2,3)-- (0,0); 	
	\draw (3,3)-- (1,0); 		
	\node at (0.5,3){$\bV^1$};	
	\node at (1.5,3){$\bV^2$};	
	\node at (2.5,3){$\bV^3$};	
	\node at (0.5,0){$\tV{}^1$};	
	\node at (1.5,0){$\tV{}^2$};	
	\node at (2.5,0){$\tV{}^3$};	
	\node at (1.5,1.5){$U$};	
	\node at (0,1.5){$0$};	
	\node at (3,1.5){$V$};	
} 
\quad \quad \quad
	\tikz[very thick,scale=.75,baseline={([yshift=1.0ex]current bounding box.center)}]{	
	\draw (0,0) -- (1,1) -- (0,2); 	
	\draw (1,0) -- (0,1) -- (1,2); 	
	       \node at (-1,1){$D':=$};	
	       \node at (0.5,1){$\color{blue}{1}$};	
	\node at (0.5,2){$\color{blue}{1}$};	
	\node at (0,1.5){$\color{blue}{0}$};	
	\node at (1,1.5){$\color{blue}{2}$};
	\node at (0.5,0){$\color{blue}{1}$};	
} 
\quad \quad 
	\tikz[very thick,scale=.75,baseline={([yshift=1.0ex]current bounding box.center)}]{	
	\draw (0,0) -- (1,1) -- (0,2); 	
	\draw (1,0) -- (0,1) -- (1,2); 	
	\node at (0.5,2){$U_1$};	
	\node at (0.5,1){$U_2$};	
	\node at (0.5,0){$U_3$};	
} 
\end{equation*}
\caption{The distance function in $\calA$ and points in $\prod_{a\in\calA}\op{Gr}_{d(a)}(V)$.}
\label{VD}
\end{figure}
 \begin{df}
For a diagram  $D\in \NH_n$ without dots, the \emph{diagram variety}  $V(D)$ is the subvariety of $\prod_{a\in\calA}\op{Gr}_{d(a)}(V)$ given by all tuples $(U_a)_{a\in\calA}$ satisfying the \emph{diagram conditions}: $U_{a_1}\subset U_{a_2}$ whenever $a_1$ and $a_2$ are neighboured and $d(a_1)<d(a_2)$.
\end{df}
\begin{df}
To $U_\bullet\in V(D)$, we can associate two full flags in $V$: the flag $U^{\rm top}$ given by reading the vector spaces in the regions touching the top boundary of the strip, and the flag $U^{\rm bot}$ given by reading the vector spaces in the bottom regions. We get a projective morphism
$
\dvp_D\colon V(D)\to \calF\times \calF, U\mapsto (U^{\rm top},U^{\rm bot}).
$
\end{df}

\begin{ex}\label{exD}
For $D$ as in \Cref{VD}, the diagram variety $V(D)$ is given by all tuples 
$(0,\bV^1, \bV^2,\bV^3,U,\tV{}^1,\tV{}^2,\tV{}^3,V)$ 
 in $\op{Gr}_0(V)\times\op{Gr}_1(V)\times \op{Gr}_2(V)\times \op{Gr}_3(V)\times \op{Gr}_2(V)\times \op{Gr}_1(V)\times \op{Gr}_2(V)\times \op{Gr}_3(V)\times \op{Gr}_4(V)$
such that, see the second diagram of \Cref{VD},
$$
\bV^1\subset  \bV^2\subset \bV^3, \quad\tV{}^1\subset \tV{}^2\subset\tV{}^3,\quad \bV^1\subset U \subset \bV^3,\quad \tV{}^1\subset U\subset\tV{}^3.
$$
the morphism $\dvp_D$ is 
$
(\bV^1, \bV^2,\bV^3,U,\tV{}^1,\tV{}^2,\tV{}^3)\mapsto (\bV^1\subset \bV^2\subset \bV^3,\tV{}^1\subset\tV{}^2\subset\tV{}^3).
$
\end{ex}
Note that the subspace corresponding to the leftmost region is always $0$ and it is $V$ for the rightmost region. The inclusions forced by the diagram conditions on these vector spaces are automatic. In the following we will mostly omit these components which do not change the variety. 
\begin{ex}\label{exD'}
For $D'$ as in \Cref{VD}, we have $V(D')\cong \op{Gr}_1(V)\times \op{Gr}_1(V)\times \op{Gr}_1(V)$ and  
$
h_{D'}\colon V(D')\cong \op{Gr}_1(V)\times \op{Gr}_1(V)\times \op{Gr}_1(V)\to \calF\times \calF=\op{Gr}_1(V)\times \op{Gr}_1(V)
$
is given by $(U_1,U_2,U_3)\mapsto (U_1,U_3)$.
The push-forward of this map sends the fundamental class to zero by dimension reasons. So, we have $(\dvp_{D'})_*([V(D')])=0=T_1T_1$.
\end{ex}

It is clear from the definition of the convolution algebra (see \S\ref{subs:diag-var-gen}), that for any diagram $D$ from $\NH_n$, we have $h_D=(\dvp_D)_*(c)$  for some $c\in H_*^G(V(D))$. We believe that it is a feature of $\mathfrak{sl}_2$ that $c=[V(D)]$ works (see Examples~\ref{exD}  and \ref{exD'}).

\begin{belief}
We believe that $h_D=(\dvp_D)_*([V(D)])$ for any diagram $D$ in $\NH_n$.
\end{belief}
For general $\Gamma$ this may however fail, see \Cref{ex:contrex-transv}.

\subsubsection{The version for $\hNH_n$}

Consider an element $h_D\in \hNH_n$ given by a word with $T$'s and $\omega$'s, but no $X$'s. The corresponding diagram $D$ is allowed to have intersections and floating dots, but no ordinary dots. 

Fix a (hermitian) scalar product on $V$. Let $\calA=\calA(D)$ be as before and let $\calB=\calB(D)$ be the set of floating dots in $D$. 
\begin{df}
For a diagram $D\in\hNH_n$ without dots, the associated \emph{diagram variety}\footnote{Since the third condition uses the scalar product, this condition is not algebraic over $\bbC$, but it is algebraic over $\bbR$. So, the diagram varieties are in general real algebraic varieties.}  $V(D)$ is the subvariety in $\prod_{a\in\calA}\op{Gr}_{d(a)}(V)\times \prod_{b\in\calB}\op{Gr}_{1}(V)$ given by all points satisfying the following \emph{diagram conditions}
\begin{itemize}
\item if  $a_1,a_2\in\calA$ are neighboured with $d(a_1)<d(a_2)$, then $U_{a_1}\subset U_{a_2}$, 
\item if the floating dot $b$ is in the region $a$, then $W_b\subset U_a$,
\item all $W_b$ are orthogonal with respect to the scalar product.
\end{itemize}
Here, $U_a\in\op{Gr}_{d(a)}(V)$, $W_b\in \op{Gr}_{1}(V)$ are the elements labelled by $a$ respectively $b$.
\end{df}
As in \S\ref{AppNH}, there is a projective morphism 
\begin{equation*}
\dvp_D\colon V(D)\to \ccZplus=\calF\times\calF\times \calG,\qquad (U_\bullet,W_\bullet)\mapsto (U^{\rm top},U^{\rm bot}, \oplus_{b\in\calB}W_b).
\end{equation*}
Clearly $h_D=(\dvp_D)_*(c)$ for some $c\in H_*^G(V(D))$. We expect nicer behaviour:
\begin{belief}\label{belief2}
We believe that $h_D=(\dvp_D)_*([V(D)])$.
\end{belief}
\begin{ex}
Consider the diagram $D$ representing the element $\Omega_{3,6}$:
\begin{equation}\label{flagF}
	\tikz[very thick,scale=.75,baseline={([yshift=1.0ex]current bounding box.center)}]{	
	\draw (0,2) -- (2,0); 	
	\draw (1,2)  .. controls (-0.5,1) ..  (1,0);
	\draw (2,2) -- (0,0); 
    \draw (3,2)-- (3,0); 
    \draw (4,2)-- (4,0); 
    \draw (5,2)-- (5,0);     
        \node at (-1,1){$D:=$}; 
    \fdot{}{0.3,2};
    \fdot{}{0.3,1};
    \fdot{}{0.3,0}; 
} \quad\quad\quad\quad
	\tikz[very thick,scale=.75,baseline={([yshift=1.0ex]current bounding box.center)}]{	
	\draw (0,2) -- (2,0); 	
	\draw (1,2)  .. controls (-0.5,1) ..  (1,0);
	\draw (2,2) -- (0,0); 
    \draw (3,2)-- (3,0); 
    \draw (4,2)-- (4,0); 
    \draw (5,2)-- (5,0);     
    \fdot{}{0.3,2};
    \fdot{}{0.3,1};
    \fdot{}{0.3,0};
    \node at (0.7,2){$\bV^1$};
    \node at (1,1.5){$\bV^2$};    
    \node at (2,1){$\bV^3$};    
      \node at (1,0.5){$U$};    
    \node at (3.5,1){$\bV^4$};   
    \node at (4.5,1){$\bV^5$};   
    }
\end{equation}
We claim that picking out the top flag defines an isomorphism $V(D)\cong\calF$. We show that any full flag $\bV^1\subset \bV^2\subset\cdots\subset \bV^5$ can be extended in a unique way to a point in $V(D)$ by adding $1$-dimensional vector spaces  $W_1,W_2,W_3$  associated to the floating dots (counted from top to bottom) and a $2$-dimensional space $U$ as indicated in \eqref{flagF}. Clearly,  $(W_1,W_2,W_3)=(\bV^1,\bV^2\cap W_1^\perp,\bV^3\cap W_2^\perp)$ is the only possibility and then $U=W_2\oplus W_3$. Then we also see that the map $\dvp_D$ in this case is an inclusion of a closed subvariety with image isomorphic to the triples $(\bV,\tV,W)$ satisfying 
$$
\bV^5=\tV{}^5, \quad \bV^4=\tV{}^4,\quad \bV^3=W=\tV{}^3, \quad \bV^2\perp \tV{}^1,\quad \bV^1\perp \tV{}^2.
$$
These conditions are exactly the conditions defining the minimal Gell $\calO^{\op{Gell}}_{\Id,\Id,\Id}$ from \Cref{minGell} associated to the basis element $\Omega_{3,6}$. With some extra work one can show that \Cref{belief2} holds, but we omit the arguments.
\end{ex}

Next we describe a general strategy how to assign a stratum of a basic paving (more precisely  an upper or lower cell, a $\calG$-cell or a Gell) to a basis element \S\ref{subs:B-strat}: 

We take the diagram $D$ of the basis element and consider its diagram variety $V(D)$ with the morphism $\dvp_D\colon V(D)\to \ccZplus$. Then we construct a certain open subset $\colJ_D$ in the image of $\dvp_D$ such that $\dvp_D$ restricts to an isomorphism  $\dvp_D^{-1}(\colJ_D)\to \colJ_D$. If we manage to get such a $\colJ_D$ for each basis diagram $D$ such that the $\colJ_D$ give rise to a basic paving in the sense of \Cref{def:b-paving} and $h_D=(\dvp_D)_*([V(D)])$ for each $D$, then the resulting basic paving is automatically strongly adapted to the basis. We believe, that this strategy works for every type of bases introduced in \S\ref{subs:cells} and that the obtained pavings are respectively upper cells, lower cells, $\calG$-cells and Gells. We give some examples.

\begin{ex} Let $n=3$ and $k=2$. Consider now the diagrams $D$, $D'$ as follows
\begin{equation}\label{DD'}
\tikz[very thick,scale=.75,baseline={([yshift=1.0ex]current bounding box.center)}]{	
	\draw (0,0) -- (2,1); 
	\draw (1,0) -- (0,1); 
	\draw (2,0) -- (1,1); 	
	\draw (0,1) -- (1,2); 
	\draw (1,1) -- (0,2);
	\draw (2,1) -- (2,2);  
	\draw (0,2) -- (0,3);  
	\draw (1,2) -- (2,3);  
	\draw (2,2) -- (1,3);  
    \fdot{}{0.5,1};
    \fdot{}{0.5,2};
    \node at (-1,1.5){$D:=$};
} \quad\quad
\tikz[very thick,scale=.75,baseline={([yshift=1.0ex]current bounding box.center)}]{	
	\draw (0,0) -- (2,1); 
	\draw (1,0) -- (0,1); 
	\draw (2,0) -- (1,1); 	
	\draw (0,1) -- (1,2); 
	\draw (1,1) -- (0,2);
	\draw (2,1) -- (2,2);  
	\draw (0,2) -- (0,3);  
	\draw (1,2) -- (2,3);  
	\draw (2,2) -- (1,3);  
    \fdot{}{0.5,1};
    \fdot{}{0.5,2};
    \node at (0.5,3){$\bV^1$};
    \node at (1.5,3){$\bV^2$};
    \node at (0.5,-0.2){$\tV{}^1$};
    \node at (1.5,-0.2){$\tV{}^2$};
    \node at (1.5,1.5){$W$};
} 
\quad\quad\quad
	\tikz[very thick,scale=.75,baseline={([yshift=1.0ex]current bounding box.center)}]{	
	\draw (0,0) -- (2,1); 
	\draw (1,0) -- (0,1); 
	\draw (2,0) -- (1,1); 	
	
	\draw (0,1) -- (1,2); 
	\draw (1,1) -- (0,2);
	\draw (0,2) -- (1,3); 
	\draw (1,2) -- (0,3);
	\draw (2,1) -- (2,3);  
	
	\draw (0,3) -- (0,4);  
	\draw (1,3) -- (2,4);  
	\draw (2,3) -- (1,4);

    \fdot{}{0.5,2};
    \fdot{}{0.5,3};
   
    \node at (-1,2){$D':=$};

} 
\quad\quad\quad
	\tikz[very thick,scale=.75,baseline={([yshift=1.0ex]current bounding box.center)}]{	
	\draw (0,0) -- (2,1); 
	\draw (1,0) -- (0,1); 
	\draw (2,0) -- (1,1); 	
	
	\draw (0,1) -- (1,2); 
	\draw (1,1) -- (0,2);
	\draw (0,2) -- (1,3); 
	\draw (1,2) -- (0,3);
	\draw (2,1) -- (2,3);  
	
	\draw (0,3) -- (0,4);  
	\draw (1,3) -- (2,4);  
	\draw (2,3) -- (1,4);

    \fdot{}{0.5,2};
    \fdot{}{0.5,3};
    
    \node at (0.5,4){$\bV^1$};
    \node at (1.5,4){$\bV^2$};
    \node at (0.5,-0.2){$\tV{}^1$};
    \node at (1.5,-0.2){$\tV{}^2$};
    \node at (1.5,2){$W$};

} 
\end{equation}
That is $(x,y,z)=(s_2,\Id,s_2s_1)$ for $D$ and $(x,y,z)=(s_2, s_1, s_2s_1)$ for $D'$. 
Consider first the diagram $D$. It is clear that the elements in the image of $\dvp_D$ satisfy 
$
\bV^1\subset W, (\bV^1)^\perp\cap W\subset \tV{}^2
$
in the notation indicated in \eqref{DD'}. Inside of the closed subvariety of $\ccZplus$ defined by these conditions, consider the open subvariety $\colJ_D$ given by the additional open conditions 
$
\tV{}^1\not\subset W, \quad \bV^2\ne W\ne\tV{}^2.
$
Then we also automatically have  $W\cap (\bV^1)^\perp=W\cap \tV{}^2$. Note that $\colJ_D$  is exactly the Gell $\calO^{\op{Gell}}_{x,y,z}$ defined by the conditions $\rel(\bV,W)=s_2$, $\rel(W,\tV)=s_2s_1$ and $\VWp=\tV{}^W$. It is not hard to see that $\dvp_D$ restricts to an isomorphism $\dvp_D^{-1}(\colJ_D)\to \colJ_D$. 

Consider now the diagram $D'$. In the notation indicated in \eqref{DD'},  it is clear that the elements in the image of $\dvp_{D'}$ satisfy the condition $\bV^1\subset W$. Inside of the closed subvariety of $\ccZplus$ determined by this condition, consider the open subvariety $\colJ_{D'}$ given by the additional open conditions 
$$
\tV{}^1\not\subset W, \quad\bV^2\ne W\ne \tV{}^2,\quad \bV^1\not\perp (W\cap \tV{}^2).
$$ 
It is not hard to see that $\dvp_{D'}$ restricts to an isomorphism $\dvp_D^{-1}(\colJ_{D'})\to \colJ_{D'}$. Note that $\colJ_{D'}$ given by these conditions is exactly the Gell $\calO^{\op{Gell}}_{x,y,z}$ defined by the conditions
$\rel(\bV,W)=s_2$, $\rel(W,\tV)=s_2s_1$, $\rel(\VWp,\tV{}^W)=s_1$.
\end{ex} 

In fact, is was not really necessary to add the first and the third open condition to get an isomorphism. However, it is necessary if we want the different pieces to be disjoint. For example, if we do not put the third open condition here, then we get a nonzero intersection of $\colJ_{D'}$ with $\colJ_D$. 

\subsection{The case of general quivers}
Now, we assume that $\Gamma$ is an arbitrary quiver without loops. This subsection should motivate the definitions of the modifications of the variety ${\Zplus}$ given in \S\ref{subs:source}, \S\ref{subs:target} and \S\ref{sec:isom-thm-GqHecke-gen}.
 We use notation of  \S\ref{sec:coloured} -- \S\ref{sec:isom-thm-GqHecke-gen}.
 
\subsubsection{The version for $R_\bfn$}
Consider a word of the form $h_D=\tau_{r_1}\tau_{r_2}\ldots \tau_{r_p}1_\uj\in R_\bfn$. Diagrammatically, it is  a diagram  $D$ with $n$ strands coloured by elements of $I$, with $p$ crossings of these strands but no dots. We would like to see geometrically from the diagram which element of $H_*^G(\tF\times_{E_\bfn}\tF)$ corresponds to $h_D$. 

\begin{df}
Let $\calA=\calA(D)$ as be before the set of connected components of the complement of $D$. For each $a\in\calA$ and $i\in I$, let $d_i(a)$ be the distance to the leftmost component in the part coloured $i$, i.e. number of strands of colour $i$ to the left of $a$. Let $\bfd(a)\in I_\bfn$ be the resulting dimension vector, that is at vertex $i$ we have  dimension $d_i(a)$.
\end{df}
\begin{df}
Let $D\in R_\bfn$ be a diagram without dots. Consider the variety $E_\bfn\times\prod_{a\in\calA}\calG_{\bfd(a)}$. The \emph{diagram variety}  $V(D)$ is the subvariety in $E_\bfn\times\prod_{a\in\calA}\calG_{\bfd(a)}$ given by the points $(\alpha, (U_a)_{a\in\calA})$ satisfying
\begin{itemize}[leftmargin=5mm]
\item if $a_1,a_2\in\calA$ are neighbours with $d_i(a_1)\leq d_i(a_2)$ for any $i\in I$, then  $U_{a_1}\subset U_{a_2}$,
\item each $U_a$ is a subrepresentation of the representation $V$ given by $\alpha\in E_\bfn$.
\end{itemize}
To each element $(\alpha,U_\bullet)\in V(D)$, we can associate two full flags in $V$, namely the flag $U^{\rm top}$ given by reading the vector spaces for the top regions of the diagram, the flag $U^{\rm bot}$ given by reading the vector spaces for the bottom regions of the diagram. We get a projective morphism
$$
\dvp_D\colon V(D)\to \tF\times_{E_\bfn} \tF,\qquad (\alpha,U)\mapsto (\alpha,U^{\rm top},U^{\rm bot}).
$$
\end{df}
It is clear from the definition of the convolution algebra, that $h_D=(\dvp_D)_*(c)$  for some $c\in H_*^G(V(D))$. If some transversality conditions hold, then $c=[V(D)]$. 
\begin{belief}
We believe that $h_D=(\dvp_D)_*([V(D)])$ holds if the diagram has no pair of strands intersecting twice (see however also \Cref{ex:contrex-transv}).
\end{belief}

\begin{ex}
	\label{ex:contrex-transv}
Let $\Gamma$ be of the form $r\to b$. Set $\bfn=r+b$, $\ui=(r,b)$, $\uj=(b,r)$. 

The varieties $\calF_\ui$ and $\calF_\uj$ are single points. We identify $\tF_\ui$ and $\tF_\uj$ with subvarieties of $E_\bfn\cong \bbC$, namely $\tF_\uj=E_\bfn=\bbC$ and $\tF_\ui\cong \{0\}\subset \bbC\cong E_\bfn$.

The elements $\tau_1 1_\ui$ and $\tau_1 1_\uj$ are represented by the fundamental classes $[C_1]$ and  $[C_2]$ respectively,  where $C_1=\tF_\uj\times_{E_\bfn}\tF_\ui$ and $C_2=\tF_\ui\times_{E_\bfn}\tF_\uj$. The varieties $C_1$ and $C_2$ are again single points and can be identified with $\{0\}\subset \bbC\cong E_\bfn$. We compare now $h_D=\tau_1\tau_11_\uj$ and $h_{D'}=\tau_1\tau_11_\ui$ given by the diagrams
\begin{equation}
\tikz[very thick,scale=.75,baseline={([yshift=1.0ex]current bounding box.center)}]{	
	\draw[blue] (0,0) -- (1,1) -- (0,2); 	
	\draw[red] (1,0) -- (0,1) -- (1,2); 		
	\node at (-1,1){$D:=$};			
        	\node at (0,-0.3){$b$};	
	\node at (1,-0.3){$r$};
}\quad\quad
\tikz[very thick,scale=.75,baseline={([yshift=1.0ex]current bounding box.center)}]{	
	\draw[blue] (0,0) -- (1,1) -- (0,2); 	
	\draw[red] (1,0) -- (0,1) -- (1,2); 	
	\node at (0.5,2){$\bV^1$};	
	\node at (0.5,1){$U$};	
	\node at (0.5,0){$\tV{}^1$};	
	\node at (0,-0.3){$b$};	
	\node at (1,-0.3){$r$};		
},
\quad\quad\quad
\tikz[very thick,scale=.75,baseline={([yshift=1.0ex]current bounding box.center)}]{	
	\draw[red] (0,0) -- (1,1) -- (0,2); 	
	\draw[blue] (1,0) -- (0,1) -- (1,2); 		
	\node at (-1,1){$D':=$};			
        	\node at (0,-0.3){$r$};	
	\node at (1,-0.3){$b$};
}\quad\quad
\tikz[very thick,scale=.75,baseline={([yshift=1.0ex]current bounding box.center)}]{	
	\draw[red] (0,0) -- (1,1) -- (0,2); 	
	\draw[blue] (1,0) -- (0,1) -- (1,2); 	
	\node at (0.5,2){$\bV^1$};	
	\node at (0.5,1){$U$};	
	\node at (0.5,0){$\tV{}^1$};	
	\node at (0,-0.3){$r$};	
	\node at (1,-0.3){$b$};		
}.
\end{equation}
The element $h_D=\tau_1\tau_1 1_\uj$ corresponds geometrically to the product $[C_1]\cdot [C_2]$ in the convolution algebra. We multiply the fundamental classes using the formula in \Cref{lem:conv}\ref{1}. When we do the $\cap$-product $[p_{1,2}^{-1}(C_1)]\cap [p_{2,3}^{-1}(C_2)]$ inside of $\tF_\uj\times \tF_\ui\times \tF_\uj\cong \bbC\times \{0\}\times \bbC$, we intersect $\{0\}\times \{0\}\times \bbC$ with $\bbC\times \{0\}\times \{0\}$. This intersection is obviously transversal. Thus we have $[p_{1,2}^{-1}(C_1)]\cap [p_{2,3}^{-1}(C_2)]=[V(D)]$ and then $h_D=(\dvp_D)_*([V(D)])$. We have the relation $\tau_1\tau_1 1_\uj=(X_2-X_1)1_\uj$ in the KLR algebra. The element $(X_2-X_1)1_\uj$ is geometrically just the fundamental class of $\{0\}\subset \bbC\cong \tF_\uj\times_{E_\bfn}\tF_{\uj}$. This fundamental class is equal to $h_D=(\dvp_D)_*([V(D)])$.
\smallskip

The element $h_{D'}=\tau_1\tau_11_\ui$ corresponds to the opposite product $[C_2]\cdot [C_1]$. This time, we do the $\cap$-product $[p_{1,2}^{-1}(C_2)]\cap [p_{2,3}^{-1}(C_1)]$ inside of $\tF_\ui\times \tF_\uj\times \tF_\ui\cong \{0\} \times \bbC\times \{0\}$. We intersect $\{0\}\times \{0\}\times \{0\}$ with $\{0\}\times \{0\}\times \{0\}$, this intersection is \emph{not} transversal. We get $c=[p_{1,2}^{-1}(C_1)]\cap [p_{2,3}^{-1}(C_2)]\ne [V(D)]$. Now we have $\tau_1\tau_11_\ui=(X_1-X_2)1_\ui\ne 1_\ui=(\dvp_D)_*[V(D)]$. Moreover, we see that $(X_1-X_2)1_\ui$ is however  \emph{not a fundamental class of a subvariety} of $\tF_\ui\times_{E_\bfn}\tF_\ui$, because $\tF_\ui\times_{E_\bfn}\tF_\ui$ is already a singleton.
\end{ex}

\subsubsection{The naive version for $\hR_\bfn$}
\label{subssubs:diag-naive}
In this section we deal with the geometric construction of \S\ref{subs:Z'-subalg-Z}, where the algebra $\hR_\bfn$ is realized inside of $H_*^{G_\bfn}(\Znaiv)$. Consider a word $h_D$ in $\hR_\bfn$ on the generators $\tau_r 1_\ui$ or $\Omega 1_\ui$. It is represented by a diagram $D$ which has crossings and floating dots, but the floating dots are to the right of the leftmost strand and have the same colour that the leftmost strand (and no twist).
\begin{df}
	Let $\calA=\calA(D)$ be the same as before and let $\calB=\calB(D)$ be the set of floating dots in the diagram. For $b\in\calB$, denote by $c(b)\in I$ the colour of $b$.\
	\end{df}
	Fix a (hermitian) scalar product on each $V_i$. It induces a scalar product on $V$ such that different $V_i$ are orthogonal. Consider now the variety $E_\bfn\times\prod_{a\in\calA}\calG_{\bfd(a)}\times \prod_{b\in\calB}\op{Gr}_{1}(V_{c(b)})$. For a point in this variety denote by $\alpha$ the element of $E_\bfn$, by $U_a$ the element of $\calG_{\bfd(a)}$ and by $W_b$ the element of $\op{Gr}_{1}(V_{c(b)})$.
\begin{df}	\label{VDsource}
	The \emph{diagram variety}  $V(D)$ attached to $D$ is the subvariety in $E_\bfn\times\prod_{a\in\calA}\op{Gr}_{\bfd(a)}\times \prod_{b\in\calB}\op{Gr}_{1}(V_{c(b)})$ given by the following conditions:
	\begin{itemize}[leftmargin=5mm]
		\item if $a_1,a_2\in\calA$ are neighbours with $d_i(a_1)\leq d_i(a_2)$ for any $i\in I$, then  $U_{a_1}\subset U_{a_2}$.		
		\item if the floating dot $b$ is in the region $a$, then we have\footnote{This implies $W_b=U_a$ because here we allow currently only floating dots of the form $\Omega1_\ui$.
} $W_b\subset U_a$, 	
		\item all $W_b$ are orthogonal with respect to the scalar product,
		\item each $U_a$ is a subrepresentation of the representation $V$ given by $\alpha\in E_\bfn$.
	\end{itemize}
As before, there is a projective morphism 
$$
\dvp_D\colon V(D)\to \Znaiv\subset E_\bfn\times\calF\times\calF\times \calG,\qquad (\alpha,U_\bullet,W_\bullet)\mapsto (\alpha,U^{\rm top},U^{\rm bot}, \oplus_{b\in\calB}W_b).
$$
\end{df}
\begin{belief}
We believe that $h_D=(\dvp_D)_*([V(D)])$ if $D$ is a basis element from \Cref{prop:NVbasis-gen} containing only floating dots coloured by sources of the quiver.
\end{belief}
For a counterexample where the source assumption does not hold see \Cref{ex:contrex-transv}. Take $D'$ with an additional floating dot inside the double-crossing.

Now, we explain the conceptual reason of why the construction of \S\ref{subs:source} works. Points in the  image of $\dvp_D$  always satisfy the condition $\alpha(W)=0$ in the source case. That is why we define $\vec{\Zplus}\subset \vec{\Znaiv}$ by the Grassmannian--Steinberg conditions, see \Cref{sinkcondition}. 

Let $D$ be a basis diagram as in \Cref{prop:NVbasis-gen} allowing floating dots only for source vertices. Let $\colJ_D\subset\calF\times\calF\times\vec\calG$ be the corresponding (coloured) Gell. Let $\naivJ_D$ be the preimage of $\colJ_D$ in $\vec{\Znaiv}$. We believe that the following happens. 

Both $\dvp_D^{-1}(\naivJ_D)\to\colJ_D$ and $\naivJ_D\to\colJ_D$ are vector bundles. The first one is a subbundle of the second one. 
If we consider $\vec{\Zplus}$ instead of $\vec{\Znaiv}$, let $\plusJ_D$ be the preimage of $\colJ_D$ in $\vec\Zplus$. 
We believe that in this case, the subbundle $\plusJ_D\subset \naivJ_D$ coincides with the subbundle $\dvp_D^{-1}(\naivJ_D)=\dvp_D^{-1}(\plusJ_D)\subset \naivJ_D$. This is the main reason why the paving by Gells should be strongly adapted to the basis. (Under the assumption that floating dots are coloured by source vertices only.)

\subsubsection{The twisted version for $\hR_\bfn$}
\label{subsubs:diag-twisted}

In this section we give a version corresponding to the geometric construction of \S\ref{sec:isom-thm-GqHecke-gen}, where the algebra $\hR_\bfn$ is realized inside of $\overline{H}_*^{G_\bfn}(\Znaivtw)$.  The twist $\phi$ involved in the construction is crucial and necessary to solve the non-transversality problem appearing for floating dots coloured with vertices which are not sources.

Let $\calA$ and $\calB$ be as before. Fix a (hermitian) scalar product on each $V_i$. It induces a scalar product on $V$ such that different $V_i$ are orthogonal. 
 
Consider the variety $E_\bfn\times E_\bfn\times\prod_{a\in\calA}\calG_{\bfd(a)}\times \prod_{b\in\calB}\op{Gr}_{1}(V_{c(b)})$. Concerning points in this variety, we denote by $(\alpha,\beta)$ the element of $E_\bfn\times E_\bfn$, by $U_a$ the element of $\calG_{\bfd(a)}$ and by $W_b$ the element of $\op{Gr}_{1}(V_{c(b)})$. A point $\star$ in some $a\in\cA$ is \emph{generic}, if there is no floating dot in $\cA$ centred on the same height. For generic $\star\in a\in\calA$  let $W_{\downarrow \star}=\oplus_b W_b\subset V$, where the sum runs over all floating dots $b\in\calB(D)$ below the chosen point $\star$.
	
\begin{df}\label{VDtwist}
Consider\footnote{One might want to allow more general floating dots, but to get a basis this is sufficient.} a word $h_D$ in $\hR_\bfn$ on the generators $\tau_r 1_\ui$ or $\Omega 1_\ui$ with corresponding diagram $D$. Then the \emph{diagram variety}\footnote{The word "variety" here is slightly delicate,  since it is not a complex algebraic variety (the conditions using the scalar product are not algebraic over $\bbC$), but rather a real manifold.}  $V(D)$ is the subvariety in $E_\bfn\times E_\bfn\times\prod_{a\in\calA}\op{Gr}_{d(a)}(V)\times \prod_{b\in\calB}\op{Gr}_{1}(V_{c(b)})$ given by the following conditions:
	\begin{itemize}[leftmargin=5mm]
		\item if $a_1,a_2\in\calA$ are neighbours with $d_i(a_1)\leq d_i(a_2)$ for any $i\in I$, then  $U_{a_1}\subset U_{a_2}$.		
		\item if the floating dot $b$ is in the region $a$, then we have $W_b\subset U_a$, 	
		\item all $W_b$ are orthogonal with respect to the scalar product,
		\item the following conditions are satisfied for each arrow $h\colon i\to j$, $a\in\calA$  and for any generic point in $a$: 
		\begin{equation}\label{VDtwisteq}
	       \alpha_h((U_a)_i)\subset (U_a)_j+(W_{\downarrow \star})_j,\qquad \beta_h((U_a)_i)\subset ((U_a)_j\cap (W_{\downarrow \star})_j)+(W_{\downarrow \star})_j^\perp. 
		\end{equation}
	\end{itemize}
\end{df}
\begin{rk}
The condition \eqref{VDtwisteq} should be compared with \Cref{eq:twisted-preserve-cond}.  It is the twist of the condition $\alpha_{h}(U_a)_i\subset (U_a)_j$ with respect to the vector space $W_{\downarrow \star}$.
\end{rk}
\begin{rk}
If in \Cref{VDtwist} the Grassmannian dimension vectors are supported on sources only, we recover the conditions from \Cref{VDsource}. In particular,  
 \eqref{VDtwisteq} boils down to the subrepresentation condition there. 
 \end{rk}
 As before we get a projective morphism to $\Znaivtw\subset E_\bfn\times E_\bfn\times\calF\times\calF\times \calG$
$$
\dvp_D\colon V(D)\to \Znaivtw,\qquad (\alpha,\beta,U_\bullet,W_\bullet)\mapsto (\alpha,\beta,U^{\rm top},U^{\rm bot}, \oplus_{b\in\calB}W_b).
$$
\begin{belief}\label{Beliefmain}
Using \Cref{VDtwist}, we believe that $h_D=(\dvp_D)_*([V(D)])$ for every diagram representing a basis element from \eqref{eq:NV-basis-canon}. 
\end{belief}
The twisted conditions \eqref{VDtwisteq} in the definition of diagram varieties should be seen as providing a resolution of non-transversalities.
However, we should not expect an analogue of \Cref{Beliefmain} for all (non-basis) diagrams. Counterexamples, similar to $D'$ in \Cref{ex:contrex-transv} still exist. \Cref{Beliefmain} gives a new hint on the importance of the basis \eqref{eq:NV-basis-canon} which is algebraically not visible.

\begin{ex}
	\label{ex:diag-twisted}
	Let $\Gamma$ be a quiver with two vertices and one arrow $b\to r$. Set $k_b=k_r=2$ and $n_b=n_r=3$. Take $x=y=\Id$ and let $z\in (\frakS_4\times\frakS_2)\backslash\frakS_6$ be the longest possible (among shortest coset representatives). We have the following basis diagram $D$ corresponding to the basis element $h_D=\Omega_{4,6}\tau_z1_\uj$ with $\uj=(b,r,b,r,b,r)$.   
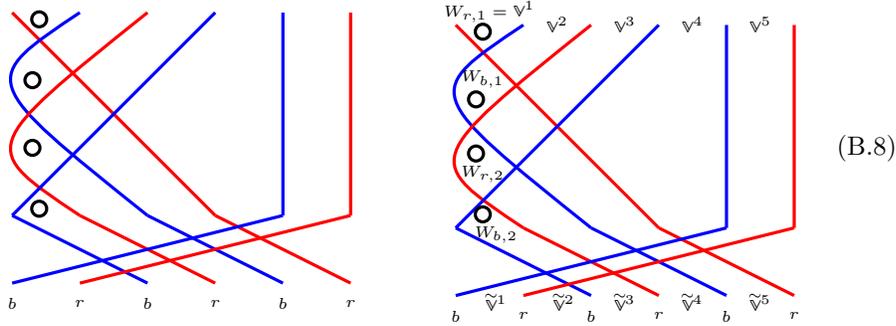
\begin{equation}\label{bigdiagram}
\tiny
\tikz[very thick,scale=0.9,baseline={([yshift=1.0ex]current bounding box.center)}]{		
	\draw[red] (0,3)-- (3,0); 	
	\draw[blue] (1,3)  .. controls (-0.5,2) ..  (2,0);
	\draw[red] (2,3)  .. controls (-0.5,1) ..  (1,0);
	\draw[blue] (3,3)-- (0,0); 
	\draw[blue] (4,3)-- (4,0); 
	\draw[red] (5,3)-- (5,0); 
	
	\fdot{}{0.4,2.9};
	\fdot{}{0.3,2};
	\fdot{}{0.3,1.0};
	\fdot{}{0.4,0.1};
	\node at (0,-1.3){$b$};
	\node at (1,-1.3){$r$};
	\node at (2,-1.3){$b$};
	\node at (3,-1.3){$r$};
	\node at (4,-1.3){$b$};
	\node at (5,-1.3){$r$};

	\draw[blue] (0,0) -- (2,-1);
	\draw[red] (1,0) -- (3,-1);
	\draw[blue] (2,0) -- (4,-1);
	\draw[red] (3,0) -- (5,-1); 
	\draw[blue] (4,0)-- (0,-1); 
	\draw[red] (5,0)-- (1,-1); 
}
\quad\quad\quad
\tiny
\tikz[very thick,scale=0.9,baseline={([yshift=1.0ex]current bounding box.center)}]{		
	\draw[red] (0,3)-- (3,0); 	
	\draw[blue] (1,3)  .. controls (-0.5,2) ..  (2,0);
	\draw[red] (2,3)  .. controls (-0.5,1) ..  (1,0);
	\draw[blue] (3,3)-- (0,0); 
	\draw[blue] (4,3)-- (4,0); 
	\draw[red] (5,3)-- (5,0); 
	
	\fdot{}{0.4,2.9};
	\fdot{}{0.3,1.9};
	\fdot{}{0.3,1.1};
	\fdot{}{0.4,0.2};
	\node at (0,-1.3){$b$};
	\node at (1,-1.3){$r$};
	\node at (2,-1.3){$b$};
	\node at (3,-1.3){$r$};
	\node at (4,-1.3){$b$};
	\node at (5,-1.3){$r$};
	
	\node at (0.5,3.2){$W_{r,1}=\bV^1$};
	\node at (0.4,2.2){$W_{b,1}$};
	\node at (0.4,0.8){$W_{r,2}$};
	\node at (0.6,-0.1){$W_{b,2}$};
	
	\node at (1.5,3){$\bV^2$};
	\node at (2.5,3){$\bV^3$};
	\node at (3.5,3){$\bV^4$};
	\node at (4.5,3){$\bV^5$};
	
	\node at (0.6,-1.1){$\tV{}^1$};
	\node at (1.6,-1.1){$\tV{}^2$};
	\node at (2.5,-1.1){$\tV{}^3$};
	\node at (3.5,-1.1){$\tV{}^4$};
	\node at (4.5,-1.1){$\tV{}^5$};
	
	\draw[blue] (0,0) -- (2,-1);
	\draw[red] (1,0) -- (3,-1);
	\draw[blue] (2,0) -- (4,-1);
	\draw[red] (3,0) -- (5,-1); 
	\draw[blue] (4,0)-- (0,-1); 
	\draw[red] (5,0)-- (1,-1); 
}
\normalsize
\end{equation}
Denote the $1$-dimensional vector spaces that we associate to the floating dots (from top to bottom) by $W_{r,1}$, $W_{b,1}$, $W_{r,2}$, $W_{b,2}$. In $V(D)$, we then obviously have
\begin{equation*}
\bV^1=W_{r,1}, \bV^2=W_{r,1}\oplus W_{b,1},\bV^3=W_{r,1}\oplus W_{b,1}\oplus W_{r,2},  \bV^4=W_{r,1}\oplus W_{b,1}\oplus W_{r,2}\oplus W_{b,2}.
\end{equation*}
Also the remaining subspaces are then completely determined, see \eqref{bigdiagram}.
By condition \eqref{VDtwisteq} of $V(D)$ we have $\alpha(W_{b,2})=0$ and $\alpha(W_{b,1})\subset W_{r,2}$. Using
$$
W_{b,2}=(\bV^3\cap W_b)^\perp\cap W_b,\quad W_{b,1}\oplus W_{b,2}=(\bV^1\cap W_b)^\perp\cap W_b, \quad W_{r,2}=(\bV^1\cap W_r)^\perp\cap W_r,
$$
these conditions can be rewritten as
\begin{equation}
\label{eq:ex-cond}
\alpha((\bV^3\cap W_b)^\perp\cap W_b)=0,\quad \alpha((\bV^1\cap W_b)^\perp\cap W_b)\subset (\bV^1\cap W_r)^\perp\cap W_r.
\end{equation}

Let $\colJ_D\subset\calF\times\calF\times\calG$ be the corresponding (coloured) Gell. Let $\naivJ_D$ be the preimage of $\colJ_D$ in $\Znaivtw$. There is no reason why this conditions \eqref{eq:ex-cond} should be satisfied for $\naivJ_D$, so the map $\dvp_D^{-1}(\naivJ_D)\to \naivJ_D$ is not surjective. 

Now let us see what happens when we replace $\Znaivtw$ by $\Zplus$ and we consider the preimage $\plusJ_D$ of $\colJ_D$ in $\Zplus$. The subvariety $\Zplus\subset \Znaivtw$ is defined by \eqref{eq:cond-bfZ'-gen} (compare this with \eqref{eq:ex-cond}!). Now, the map $\dvp_D^{-1}(\naivJ_D)=\dvp_D^{-1}(\plusJ_D)\to \plusJ_D$ becomes surjective. It is even an isomorphism!
\end{ex}

 We believe that the same thing happens for any basis diagram $D$ and that for this reason we get a basic paving by Gells strongly adapted to the basis.
 
 In fact, Grassmannian--Steinberg conditions \eqref{eq:cond-bfZ'-gen} defining $\Zplus$ inside of $\Znaivtw$ are in fact the conditions coming from the image of $\dvp_D$.

\bibliographystyle{abbrv}
\bibliography{RefgeometricNaisseVaz}

\address{R. M.: Laboratoire Analyse, G\'eom\'etrie et Mod\'elisation, CY Cergy Paris Universit\'e, 95302  Cergy-Pontoise (France),}
\email{ruslmax@gmail.com}

\address{C. S.: Mathematisches Institut, Universit\"at Bonn, 53115 Bonn (Germany),\\}
\email{stroppel@math.uni-bonn.de}
\end{document}